\documentclass[11pt]{amsart}
\usepackage{mathrsfs}
\usepackage{amsfonts,amssymb,amsmath,amsthm}
\usepackage{cite}
\usepackage[colorlinks=true, citecolor=red]{hyperref}
\usepackage{titletoc}
\usepackage{geometry}
\usepackage{bm}
\usepackage{indentfirst}
\usepackage{graphicx}
\usepackage{float} 
\usepackage{booktabs}
\usepackage{longtable}
\usepackage{subfigure}
\usepackage{stmaryrd}
\usepackage{enumerate}
\usepackage{tikz}
\usepackage{ulem}
\usepackage{color,soul}
\usepackage{setspace}
\usepackage{stmaryrd}
\allowdisplaybreaks[4]
\usepackage{appendix}
\usepackage{cases}
\usepackage{todonotes}
\usepackage{enumitem}

\usepackage{amsmath,amsthm}
\usepackage {latexsym}
\usepackage{amssymb}

\usepackage{xcolor}

\newcommand{\blackhyperref}[2]{%
  {%
    \hypersetup{linkcolor=black}%
    \hyperref[#1]{#2}%
  }%
}

\newcommand{\beal}{\begin{align}}
\newcommand{\enal}{\end{align}}
\newcommand{\bealn}{\begin{align*}}
\newcommand{\enaln}{\end{align*}}
\newcommand{\bear}{\begin{eqnarray}}
\newcommand{\eear}{\end{eqnarray}}
\newcommand{\beeq}{\begin{equation}}
\newcommand{\eneq}{\end{equation}}

\newcommand{\supp}{\mbox{\rm supp}}

\newcommand{\eps}{{\varepsilon}}
\newcommand{\R}{{\mathbb R}}

\def\bm{\left[ \begin{array}{cc}}
\def\endm{\end{array}\right]}

\def\eps{\varepsilon}

\def\bm{\left[\begin{matrix} }
\def\endm{\end{matrix}\right]}

\def\R{{\mathbb R}}

\numberwithin{equation}{section}
\newtheorem{theorem}{Theorem}[section]
\newtheorem{lemma}[theorem]{Lemma}
\newtheorem{cor}[theorem]{Corollary}
\newtheorem{prop}[theorem]{Proposition}

\newtheorem*{op}{Open Problem}
\newtheorem*{pp}{Problem}

\newtheorem{defi}[theorem]{Definition}
\newtheorem{remark}[theorem]{Remark}

\theoremstyle{definition}

\renewcommand{\epsilon}{\eps}
\renewcommand{\tilde}{\widetilde}

\begin{document}

\title[Global control of wave maps]{Wave maps from  circle to Riemannian manifold: \\global controllability is equivalent to homotopy}

\author{Jean-Michel Coron}
\address{Sorbonne Universit\'{e}, Universit\'{e} Paris-Diderot SPC, CNRS, INRIA, Laboratoire Jacques-Louis Lions, LJLL,  \'{e}quipe CAGE, F-75005 Paris, France}
\email{\texttt{jean-michel.coron@sorbonne-universite.fr}}
\thanks{}


\author{Joachim Krieger}
\address{B\^{a}timent des Math\'{e}matiques, EPFL\\Station 8, CH-1015 Lausanne, Switzerland}
\email{\texttt{joachim.krieger@epfl.ch}}

\author{Shengquan Xiang}
\address{School of Mathematical Sciences, Peking University, 100871, Beijing, P. R. China}
\email{\texttt{shengquan.xiang@math.pku.edu.cn}}
\thanks{}

\begin{abstract}
We study wave maps from the circle to a general compact Riemannian manifold. We prove that the global controllability of this geometric equation is characterized precisely by the homotopy class of the data, thereby resolving the conjecture posed in \cite{KX, CKX}. As a remarkable intermediate result, we establish uniform-time global controllability between steady states, providing a partial answer to an open problem raised in \cite{Dehman-Lebeau-Zuazua}.  Finally, we obtain quantitative exponential stability around closed geodesics with negative sectional curvature. This work highlights the rich interplay between partial differential equations, differential geometry, and control theory. 
\end{abstract}
\subjclass[2010]{35L05, 
35B40, 
58J45, 
93C20 
}
\thanks{\textit{Keywords.} wave maps, controllability,  stability, dynamics, homotopy, geodesics, curvature.}

\maketitle

    \setcounter{tocdepth}{2}
	\tableofcontents

\section{Introduction}

Global controllability of the wave maps equation from a circle to a sphere has been established in \cite{KX, CKX}. This raises the question of controllability for this model in the case where the target is a general Riemannian manifold.
Clearly, in order to go from a given initial state to a given target 
state, it is necessary that these two states are homotopic; see Figure \ref{fig:wavemap}. This leads 
to the natural conjecture that this necessary condition for controllability is 
also sufficient.

\begin{pp}
The wave maps equation from a circle to a general compact Riemannian manifold is  controllable in the sense that 
\begin{equation*}
    \textrm{global controllability is equivalent to homotopy.}
\end{equation*}
\end{pp}

This work resolves the conjecture concerning controllability, and investigates the long-time dynamics of the locally damped wave maps equation as well as stability around geodesics.

\begin{figure}[H]\label{fig:wavemap}
\centering
\includegraphics[width=0.8\textwidth]{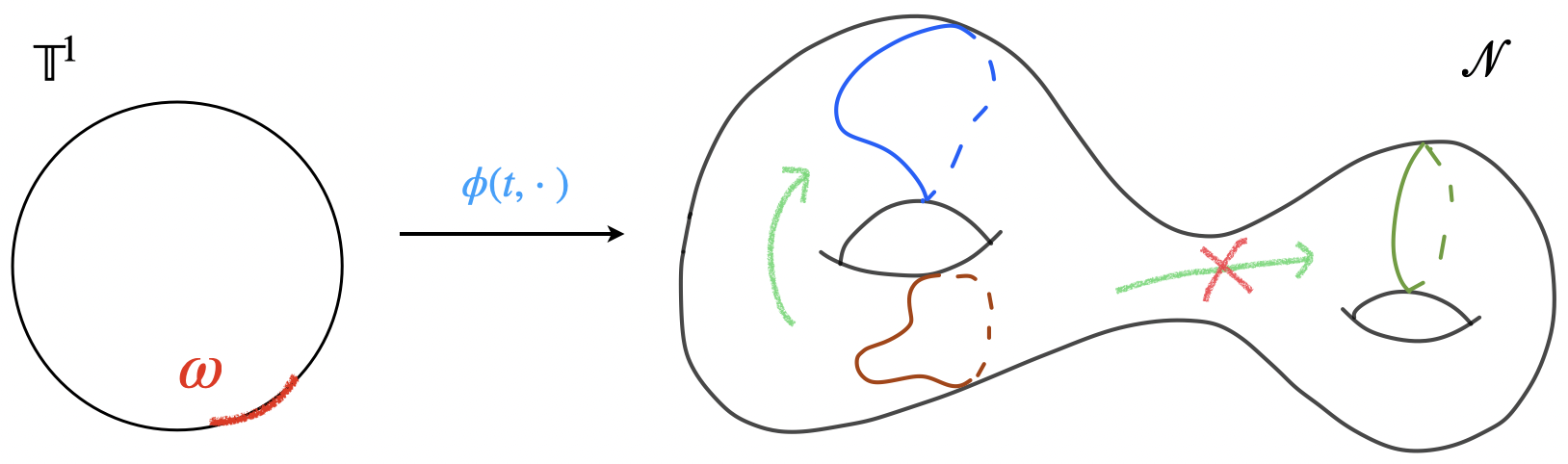}
\caption{Given a state $(\phi, \phi_t): \mathbb{T}^1\rightarrow T\mathcal{N}$. The first component $\phi: \mathbb{T}^1\rightarrow \mathcal{N}$ can be regarded as a closed curve in $\mathcal{N}$. We say that two states are  homotopic if their spatial components (the curves $\phi$) are homotopic as maps from $\mathbb{T}^1$ to $\mathcal{N}$, that is, if one can be continuously deformed into the other.  }
\end{figure}

Wave maps arise as the hyperbolic counterpart of harmonic maps and play an important role in mathematical physics. They coincide with the nonlinear sigma models of quantum field theory, where fields are represented as maps from spacetime into a target manifold. Wave map systems also appear in general relativity through symmetry reductions of the Einstein equations, and are closely related to models in ferromagnetism and liquid crystals.   In these contexts, control problems naturally emerge, where external influence is used to steer field configurations or stabilize dynamics, thereby linking geometric control to concrete physical phenomena.

\vspace{2mm}
Throughout this paper, we fix the following basic setting. 

\vspace{2mm}

\begin{enumerate}[leftmargin=2em]
\item[$\mathbf{(S)}$]\label{setting(S)} 
Let $(\mathcal{N}, g)$ be a smooth, compact, orientable Riemannian manifold without boundary.\\
Let  $\omega\subset \mathbb{T}^1$ be a non-empty open set. Let $a: \mathbb{T}^1\rightarrow \mathbb{R}_{\geq 0}$ be a non-trivial smooth function supported in $\omega$.
\end{enumerate}
    
By Nash's embedding theorem, $(\mathcal{N}, g)$ can be isometrically embedded into some Euclidean space $\mathbb{R}^N$. Hence, throughout this paper, we treat $\mathcal{N}\hookrightarrow \mathbb{R}^N$ as a submanifold.  For simplicity, we focus on smooth compact manifolds, although the results and techniques can be extended to $C^3$ manifolds that are possibly non-compact, and maybe even have a boundary.

\vspace{2mm}

The wave maps equation  from $\mathbb{T}^1$ to $\mathcal{N}$ is
\begin{equation*}\label{eq:wm:ori}
   \Box \phi -  \Pi(\phi)\left(\partial_{\nu}\phi, \partial^{\nu}\phi\right)= 0, 
\end{equation*}
where $\Box= -\partial_t^2+ \Delta$ is the d'Alembert operator,  $\Pi$ is the second fundamental form, and we used the Einstein summation convention for $\nu\in \{0, 1\}$ with $(\partial_0, \partial_1, \partial^0, \partial^1)= (\partial_t, \partial_x, -\partial_t, \partial_x)$. The state $(\phi, \phi_t)(t, \cdot)$ at time $t$ is simply denoted by $\phi[t]$.
We are interested in the natural physical  $H^1$-topology for the system.  Define the usual energy space for wave maps
\begin{equation*}
 \mathcal{H}(\mathbb{T}^1; \mathcal{N}):= \left\{(\phi, \phi_t):  \phi\in H^1(\mathbb{T}^1; \mathcal{N}), \; \phi_t\in L^2(\mathbb{T}^1; \phi^* T\mathcal{N})  \right\}.
\end{equation*}
Since $\mathcal{N}\hookrightarrow \mathbb{R}^N$ is a submanifold, we also define the extrinsic energy space
\begin{gather*}
    \mathcal{H}:= \left\{(f, g)\in  H^1(\mathbb{T}^1; \mathbb{R}^N)\times L^2(\mathbb{T}^1; \mathbb{R}^N)\right\} \; \textrm{ with } \\
    \|(f, g)\|_{\mathcal{H}}^2:= \|f\|_{H^1(\mathbb{T}^1)}^2+ \|g\|_{L^2(\mathbb{T}^1)}^2,
\end{gather*}
and the energy functional for every $(f, g)\in \mathcal{H}$,
\begin{equation*} 
    E(f, g):= \int_{\mathbb{T}^1} \left(|\partial_x f|^2+ |g|^2\right)(x) \, dx.
\end{equation*}

We say that two maps $\phi_1, \phi_2\in H^1(\mathbb{T}^1; \mathcal{N})$ are {\it homotopic}, if  there exists a continuous map $Q:[0,1]\times \mathbb{T}^1\rightarrow \mathcal{N}$ such that, for all $x\in \mathbb{T}^1$, $Q(0,x)=\phi_1(x)$ and $Q(1,x)=\phi_2(x)$. Similarly, we say that two states of wave maps $(\phi_1, \phi_{1t}), (\phi_2, \phi_{2t})\in \mathcal{H}(\mathbb{T}^1; \mathcal{N})$ are {\it homotopic}, if their spatial component $\phi_1$ and $\phi_2$ are homotopic.
\vspace{2mm}

{\it The locally damped wave maps equation} is
\begin{equation}\label{eq:dwm:ori}
   \Box \phi -    \Pi(\phi)\left(\partial_{\nu}\phi, \partial^{\nu}\phi\right)= a(x) \partial_t \phi.
\end{equation}
Notice that {\it harmonic maps} are steady states of both the free wave maps equation and the damped wave maps equation \eqref{eq:dwm:ori},
\begin{equation*}
   \Delta \phi  -   \Pi(\phi)\left(\partial_{x}\phi, \partial_{x}\phi\right)= 0.
\end{equation*}
In the setting $\mathbb{T}^1\rightarrow \mathcal{N}$, harmonic maps are {\it closed geodesics}\footnote{Throughout this paper,  a closed geodesic always indicates a harmonic map $\gamma: \mathbb{T}^1\rightarrow \mathcal{N}$.}. Given a closed geodesic $\gamma: \mathbb{T}^1\rightarrow \mathcal{N}$, there is a family of harmonic maps via rotation along this geodesic:
\begin{equation*}
    \gamma_{p}(\cdot):= \gamma(\cdot+ p), \; \forall p\in [0, 2\pi).
\end{equation*}
Natural questions are understanding the long-time dynamics of the damped wave maps equations and their stability around closed geodesics.
\vspace{2mm}

{\it The controlled wave maps equation} can be expressed as
\begin{equation}\label{eq:wminhomo}
   \Box \phi -  
   \Pi(\phi)\left(\partial_{\nu}\phi, \partial^{\nu}\phi\right)= \chi_{\omega} \Pi_T(\phi) f 
\end{equation}
where the extra force $f:\mathbb{T}^1\rightarrow \mathbb{R}^N$ can be regarded as a control that is localized in $\omega$ via a cutoff $\chi_{\omega}$,  and $\Pi_T(\phi) f$ is the projection of $f$ on the tangent space $T_{\phi}\mathcal{N}$ at $\phi$ in order to obey the geometric constraints and guarantee that the flow stays in the manifold $\mathcal{N}$. 
\vspace{2mm}

\subsection{The main results}
The main purpose of this paper consists of the following  results.
\begin{theorem}[Global controllability is equivalent to homotopy]\label{thm1}
 Let $\mathcal{N}$ and $\omega$ satisfy  \blackhyperref{setting(S)}{$({\bf S})$}.   The controlled wave maps equation \eqref{eq:wminhomo} is globally exactly controllable.
    
    More precisely, for every $M>0$ there exists some  $T>0$, such that for every homotopic states $(\phi_0, \phi_{0t}), (\phi_1, \phi_{1t})\in \mathcal{H}(\mathbb{T}^1; \mathcal{N})$ with their energy smaller than $M$, there exists a control $f\in C([0, T]; L^2(\mathbb{T}^1))$ such that the unique solution of \eqref{eq:wminhomo} with initial state $(\phi_0, \phi_{0t})$ satisfies $\phi[T]= (\phi_1, \phi_{1t})$. 
\end{theorem}

\begin{theorem}[Dynamics of the locally damped equation]\label{thm:dynamics}
Let $\mathcal{N}$ and $a(\cdot)$ satisfy  \blackhyperref{setting(S)}{$({\bf S})$}. 
 For any $M>0$ and $\delta>0$ there exists  $T_c= T_c(M, \delta)$ such that, for any initial state with energy smaller than $M$ the solution of \eqref{eq:dwm:ori} satisfies
    \begin{align}\label{eq:thm12}
\|\phi[t_{\delta}] - (\gamma, 0)\|_{\mathcal{H}}<\delta,
\end{align}
for some time $t_{\delta}\in [0, T_c]$ and some closed geodesic  $\gamma$.  Moreover,
    for any fixed initial state $\phi[0]$, there exists a  closed geodesic   $\gamma$ with the property that for each $\delta>0$, there exists $t_{\delta}>0$  such that \eqref{eq:thm12} holds.
\end{theorem}

To the best of our knowledge, Theorem \ref{thm1} is the first {\it global controllability} result for geometric equations with general targets. The  global controllability  problem\footnote{The term ``global" refers to large-state control, as opposed to small perturbations around equilibrium.} is a central topic in PDE control theory. 
It fundamentally depends on the nonlinear structure of the PDE and requires sophisticated nonlinear control structures.
\vspace{1mm}

Theorem \ref{thm:dynamics} can be viewed as the analogue of the Eells--Sampson theorem \footnote{The homotopy problem asks whether a given map between two manifolds can be homotopically deformed into a harmonic map. Various approaches have been developed to address this problem; we refer to \cite{Eells-Lemaire, Eells-Lemaire-2} for a comprehensive survey. Among these, a particularly natural method is the heat flow approach due to Eells--Sampson, which suggests that, under suitable conditions, the heat flow may converge to a harmonic map.} for the harmonic map heat flow, in the framework of locally damped wave maps.
This result highlights how localized control mechanisms can drive global deformations, drawing parallels between stabilization techniques in control theory and the asymptotic behavior of geometric flows.

\vspace{2mm}

\noindent {\bf Remark.}
{\it   Our strategy stems from the idea of ``global stabilization-local controllability'' which has been used for KdV, NLW, and  NLS \cite{Laurent-Rosier-Zhang-kdv, Dehman-Lebeau-Zuazua, Laurent-2011, Joly-Laurent-2, DGL-2006, MR2644360}. Usually the problem  involves only one steady state, namely zero. However, for  geometric models, there are infinitely many steady states, which significantly complicates  dynamical behaviors and  analysis.
    
    To tackle this, we introduced a new approach composed of four stages, which emphasizes different tools in each stage: global stabilization toward steady sets, local controllability around steady states, global controllability between steady states, and the return method; see Section \ref{sec:strategy1} for details. We believe this approach along with the tools can be applied to global controllability  problems for dispersive models with solitons and other geometric models, such as the harmonic map heat flow,  H-systems, Schrödinger maps, Landau–Lifshitz equation, and Yang-Mills, etc.}
\vspace{1mm}

\begin{theorem}[Stability around geodesics of the locally damped equation]\label{thm:stability}
    
Let $\mathcal{N}$ and $a(\cdot)$ satisfy  \blackhyperref{setting(S)}{$({\bf S})$}.   Let $\gamma: \mathbb{T}^1\rightarrow \mathcal{N}$ be a closed geodesic. Assume the sectional curvature is strictly negative on $\gamma$. Then  \eqref{eq:dwm:ori} is exponentially stable around $\gamma$. 

    More precisely, there exist $\varepsilon>0$, $b>0$ and $C>0$ such that for any initial state $\phi[0]$ satisfying 
    \begin{equation}
        \|\phi[0]- (\gamma, 0)\|_{\mathcal{H}}\leq \varepsilon, \notag
    \end{equation}
    there exists a constant $p$ such that 
    \begin{gather}
     |p|\leq C   \|\phi[0]- (\gamma, 0)\|_{\mathcal{H}},  \label{ine:thm2:first} \notag\\
       \|\phi[t]- (\gamma_p, 0)\|_{\mathcal{H}}\leq C e^{- b t}  \|\phi[0]- (\gamma, 0)\|_{\mathcal{H}}\;\; \forall t\in (0, +\infty). \label{ine:thm2:second} \notag
    \end{gather}
\end{theorem}

Although damped wave equations have been extensively studied in the literature, their stability around non-trivial steady states remains less understood. To our knowledge, this is the first quantitative result for geometric wave equations. The curvature assumption is essentially necessary, since without it the closed geodesic is not even necessarily isolated (see Remark \ref{rmk:iso:negcur}).

\vspace{2mm}

\noindent {\bf Methodology:} 
To address the main results, we have developed several novel tools and combined ideas from various research topics. Here we summarize several of them, while further discussions can be found in Section \ref{sec:strategy}.
\begin{itemize}[leftmargin=2em]
     \item[\tiny$\bullet$]  {\it  The return method.} {\rm This method was introduced by the first author for nonlinear global controllability  problems \cite{Coron-1996-Euler}. As we are going to see, it turns out to be powerful to get controllability results around non steady states.}
   \vspace{1mm}
   
   \item[\tiny$\bullet$] {\it Propagation of smallness.}   This property is related to unique continuation and linear observability inequalities. The current quantitative version for geometric waves was introduced by the last two authors in \cite{KX}. 
   \vspace{1mm}
   
   \item[\tiny$\bullet$]  {\it  Dynamics of geometric equations.}  {\rm The long time dynamics of the locally damped wave maps were first investigated by the authors in \cite{CKX, KX} with sphere target case. Here we furnish a wave analogue of the Eells--Sampson argument for the harmonic map heat flow, but with localized damping.}
   \vspace{1mm}

   \item[\tiny$\bullet$]  {\it  Control around geodesics.} {\rm In this paper, we introduce a reduction approach based on the intrinsic moving frame method and iteration schemes that transforms the geometric control problem into a linear control problem. }
   \vspace{1mm}

\item[\tiny$\bullet$]  {\it  Global controllability  between homotopic geodesics.}   This remarkable property emphasizes control and dynamics on geodesics. It is proved by a method recently introduced by the first and last author in \cite{CX-2024} for the harmonic map heat flow. 
   \vspace{1mm}

    \item[\tiny$\bullet$]  {\it Stability around closed geodesics with negative curvature.}   In this paper, we propose a five-step strategy to address this stability problem. This approach emphasizes the role of the propagation of smallness and negative curvature. 
\end{itemize}

\vspace{2mm}    

\noindent {\bf Further directions:} 
We believe that this work opens up several avenues for further exploration: 
\begin{itemize}[leftmargin=2em]
    \item[(i)] How can these results be extended to higher-dimensional wave maps equations? In particular, is it possible to adapt these techniques to the important phenomenon of singularity formation? 
        \item[(ii)] Other geometric models that are strongly related to physics, such as the harmonic map heat flow and the Yang-Mills equations?
         \item[(iii)] Control properties are strongly related to stochastic equations.  Can these results lead to statistical properties of random geometric equations?
        \item[(iv)] In Theorem \ref{thm1}, the control time depends on the scale of the states. Can this be improved to achieve uniform or even optimal-time global controllability? 
        \item[(v)] Other common types of control conditions, such as boundary controls or rough controls, and other classical control problems, such as reachable sets and optimal control?
\end{itemize}

\subsection{Review on the literature}\label{sec:literature}

\subsubsection{Wave maps equations and related control problems} 
Geometric wave equations have been extensively studied over the past few decades, particularly in relation to well-posedness and singularity formation. A vast body of literature is dedicated to well-posedness results; see, for instance, the works of Christodoulou and Tahvildar-Zadeh \cite{Christodoulou-Tahvildar-Zadeh-1993-Duke, Christodoulou-Tahvildar-Zadeh-1993}, Klainerman and Machedon \cite{Klainerman-Machedon-1997}, Tao \cite{Tao-wwm-1, Tao-wwm-2, tao2009global3}, Tataru \cite{Tataru-2001}, Sterbenz and Tataru \cite{Ster-Tat1, Ster-Tat2}, and Krieger and Schlag \cite{Krieger-Schlag-2012}. The study of singularity formation has also been a central theme in the analysis of dispersive equations. In the context of wave maps, we refer to the works \cite{Krieger-Schlag-Tataru, R-R-2012, Rod-Ster}, which provide significant insights into blow-up dynamics and critical behavior. For a broader perspective, we also recommend the survey by Tataru \cite{Tataru-2004} and the references therein.

The study of wave maps is closely connected to the theory of harmonic maps, an important topic in differential geometry with deep links to mathematical physics. Harmonic maps have broad applications, ranging from physics and fluid dynamics to materials science and even computer vision. For an introduction to harmonic maps, we refer to the lecture notes by Schoen and Yau \cite{Schoen-Yau-harmonicmap}, and for the theory of the harmonic map heat flow, to the monograph by Lin and Wang \cite{Lin-Wang-book}.

More recently, the control theory of geometric  wave equations has attracted growing attention. 
In particular, the authors in \cite{CKX, KX} have investigated the controllability and stabilization of wave maps from a circle to a sphere. The control of the harmonic map heat flow is studied in \cite{CX-2024, Liu-2018}.  Unlike control problems for wave equations with Euclidean targets, where the primary focus is often on dispersive properties and energy estimates, these works emphasize the geometric and topological aspects of the target manifold, which play a crucial role in the analysis. Understanding how curvature, geodesics, and topology influence control properties is essential for advancing control techniques for these PDEs.

\subsubsection{Global controllability problems}
Unlike local or linear control problems, where the main challenges stem from observability and the spectral properties of the underlying linear operator, global controllability  problems largely rely on the role and use of nonlinear terms.  In general, three main approaches have been developed in the literature to tackle global controllability  problems.

The first approach is the return method, originally introduced by the first author for the global controllability  problem of the incompressible Euler equations \cite{Coron-1996-Euler}. A comprehensive reference on this method can be found in \cite{Glass-bourbaki}. Since then, this method has been further developed and successfully applied to various systems, including the Euler equations \cite{Glass-2000}, the Navier--Stokes equations \cite{1996-Coron-Fursikov-RJMP, Coron-Marbach-Sueur, Coron-Marbach-Sueur-Zhang, NR-2025}, as well as the viscous Burgers equation \cite{Marbach-burgers, Coron-Xiang-2018}.

A second important approach is damping stabilization, which is particularly useful in the control of defocusing dispersive equations. The key idea is that the presence of a damping term ensures that the energy of the nonlinear system decays over time, eventually leading to stabilization towards a zero equilibrium state. This method has been successfully employed in various settings, including the Benjamin-Ono equation \cite{Laurent-Linares-Rosier-2015}, the KdV equation \cite{Laurent-Rosier-Zhang-kdv}, NLW \cite{Dehman-Lebeau-Zuazua, Laurent-2011}, and NLS   \cite{DGL-2006, MR2644360}. More recently, this global dissipation property has been used by the third author and his coauthors for the investigation of ergodicity properties of randomly forced dispersive equations \cite{WaveMixing, NLSMixing}.

A third major approach is the geometric control method, also known as the Agrachev-Sarychev method in the PDE setting.  This method typically leads to global approximate controllability with the help of finitely many Fourier modes as control. It emphasizes the role of Lie bracket for nonlinear terms and relies on Hörmander-type conditions.  It has been applied to various models under different settings, see for instance,  \cite{AS-NS, beauchard2025smalltimeapproximatecontrollabilitybilinear,  NR-2025, coron-xiang-zhang, Duca-Nersesyan}, and the references therein.

As discussed in the remark after Theorem \ref{thm1}, in this paper, we introduce a new approach to establishing global controllability  results by combining the return method and damping stabilization, leveraging the strengths of both techniques to handle nonlinear PDEs with infinitely many equilibrium states.

\vspace{1mm}

\subsubsection{Control of wave equations}\label{sec:contro:wave}

The controllability of (semi)linear wave equations has been an important and active topic in PDEs control theory over the past several decades. One of the foundational approaches is the multiplier method, first introduced by Lions, which yields quantitative controllability results for star-shaped control regions \cite{Lions}. This technique was subsequently refined and extended in the works \cite{MR1745475, Zuazua-1990, Zuazua-1993}. Under similar geometric conditions,  another widely used approach to obtain observablity is the global Carleman estimates  \cite{BDE-Carleman, DZZ-2008, Zhang-2000-2, Shao-2019}. In a one-dimensional setting, direct control via characteristics also proves to be highly effective (see \cite{Li-Yu-2006, Li-2010}).

Another major approach is based on microlocal analysis,  tied to the so-called Geometric Control Condition, introduced by Bardos, Lebeau, and Rauch \cite{Bardos-Lebeau-Rauch}. Numerous further developments have followed along this line of research, see, for example, \cite{Burq-Gerard-wave, Laurent-2011, Dehman-Lebeau-Zuazua, DGL-2006, Dehman-Lebeau-2009, LL-2019, LLTT-2017}. 
There are also many important works on the stabilization of wave equations. These results also heavily rely on the aforementioned methods, which include for example the works \cite{Anantharaman-Leautaud-2014, Burq-Gerard-wave, Burq-Gerard-2021, Lebeau-Robbiano-2, KX2023, LT-1993, Koch-Tataru-1995, Sun-2023} and the references therein.

\subsection{Strategy of the proofs}\label{sec:strategy}

In this section,  we outline our approach to the proofs of the main theorems. The whole section is composed of two parts. 
\subsubsection{{\bf Ideas for Theorem \ref{thm1} and Theorem \ref{thm:dynamics} }}\label{sec:strategy1}
\; \vspace{2mm}

Achieving this global controllability is challenging due to four issues: 1) global controllability  problems are more difficult than local ones;  2) typically, local exact controllability is studied along trajectories, with limited results around non-steady states; 3) infinitely many implicit steady states of this geometric equation make the dynamics more complicated; and 4) when the target is a general manifold, the control is subject to implicit geometric constraints, making the flow hard to analyse. 
\vspace{1mm}

To solve this problem, we follow the strategy of the return method. The idea is to construct some non-trivial trajectory with given initial and final states, such that the system is controllable around this trajectory. The major difficulty consists in finding such a trajectory. This task is handled by introducing three intermediate properties, each of which being interesting in its own right.    See Sections  \ref{sec:propsmall}--\ref{Sec:gc}.

We first establish the local controllability around any given {\it non-steady}  states by constructing a well-designed return trajectory.  Qualitatively, this property easily leads to the global exact controllability by continuous deformation within a homotopy class. Next, we also directly construct a special trajectory that connects any two given homotopic states. This latter construction provides uniform quantitative bounds on the control time, see  Section \ref{Sec:5}.
\vspace{2mm}

\noindent {\bf Three intermediate properties.}
\begin{itemize}[leftmargin=2em]
    \item[(1)] {\it Dynamics of the locally damped equation};  Theorem \ref{thm:dynamics}. \\
This result demonstrates the {\it dynamics} of the damped equation towards closed geodesics.  The presence of infinitely many steady states complicates the dynamical  behavior.   The proof is inspired by our earlier work \cite{KX, CKX}, while the lack of an explicit formula for closed geodesics makes the analysis more complicated.  
Refer to Sections \ref{sec:propsmall}--\ref{sec:3}.
\vspace{1mm}

 \item[(2)]  {\it Local exact controllability around closed geodesics}; Proposition \ref{pro:controlaroundgeodesic}. \\
 The primary difficulty arises from the geometric constraints imposed during the control process. To address this, we introduce a novel three-step method. We believe the first two steps, reducing the semilinear control problem with geometric constraint to a linear control problem in Euclidean space, are of independent interest. For this see Section \ref{Sec:4}. 
 \vspace{1mm}

  \item[(3)]  {\it Uniform-time global exact controllability between closed geodesics}; Proposition \ref{pro:controlbetween}.\\
Given two (steady) states that are far apart, is there an effective and uniform-time control that connects them? This question has remained a central and challenging open problem for Navier-Stokes, nonlinear heat, and semilinear wave equations. In a recent work, the first and third authors introduced a novel method that resolves this problem in the context of the harmonic map heat flow \cite{CX-2024}.
This approach highlights the role of {\it geometry} and {\it geodesics}.  Using this new approach, we solve the problem for wave maps (a semilinear wave equation). See Section \ref{Sec:gc}.
\end{itemize}
\vspace{2mm}

A more detailed explanation of these three intermediate properties is provided below.

\vspace{2mm}
\noindent \textbf{First property.} This is a stabilization result that contributes to Theorem \ref{thm:dynamics}.
The dissipation is due to the damping effect. Indeed, simple integration by parts yields
\begin{equation*}
   E(\phi[T])+ 2\int_0^T\int_{\mathbb{T}^1} a(x)|\phi_t|^2(t, x) \, dx dt= E(\phi[0]).
\end{equation*}
For (defocusing) wave equations, the exponential stabilization is equivalent to the observability inequality
\begin{equation*}
   \int_0^T\int_{\mathbb{T}^1} a(x)|\phi_t|^2(t, x) \, dx dt\geq c E(\phi[0]), c>0.
\end{equation*}
This type of estimate has been extensively studied in the literature during the past decades. We refer to the references given in Section \ref{sec:contro:wave}. However, for geometric equations, one can not simply reduce the dissipation problem to an observability estimate, since there are infinitely many non-trivial steady states (namely, harmonic maps in our case).

The authors have introduced a method in \cite{CKX, KX} to show that the solution will converge to a closed geodesic in case that $\mathcal{N}= \mathbb{S}^k$, at least along a sequence of times.  Our proof of Theorem \ref{thm:dynamics} is based on this method, which consists of two ingredients: 
\begin{itemize}
    \item[\tiny$\bullet$] {\it quantitative propagation of smallness};
    \item[\tiny$\bullet$] {\it characterization of approximate closed geodesics}. 
\end{itemize}

 We first show in Section \ref{sec:propsmall} the propagation of smallness property; Lemma \ref{lem:2}, which is a technical generalization of \cite[Proposition 2.2]{KX} to the Riemannian manifold case. It shows that if a solution of the damped equation is small in the sense that 
 \begin{equation*}
     \int_{I} \int_{\omega} |\phi_t(t, x)|^2 \, dx dt\ll E(\phi[0]),
 \end{equation*}
 then the smallness for $x\in \omega$ can be propagated to the whole spatial space $x\in \mathbb{T}^1$: 
  \begin{equation*}
   \sup_{x\in \mathbb{T}^1}  \int_{\tilde I} |\phi_t(t, x)|^2 \, dx \ll E(\phi[0]) \; \textrm{ 
 for some  $\tilde I\subset I$.}
 \end{equation*}

Next, in Section \ref{sec:3} we focus on the characterization of approximate closed geodesics. While the preceding estimates suggest that the solution behaves similarly to a steady state in $(t, x)\in \tilde I\times \mathbb{T}^1$, we demonstrate that it is, in fact, close to a closed geodesic. The main obstacle lies in the fact that the closed geodesics themselves are not explicitly known in the general manifold case.

\vspace{2mm}

\noindent \textbf{Second property. }  This is devoted to the following local controllability result.

 \begin{prop}\label{pro:controlaroundgeodesic}  Let $\mathcal{N}$ and $\omega$ satisfy  \blackhyperref{setting(S)}{$({\bf S})$}.   Let $T_1= 64\pi$. Let $M>0$. There exists $\delta_1= \delta_1(M)$ such that for any 
   closed geodesic $\gamma: \mathbb{T}^1\rightarrow \mathcal{N}$ with energy smaller than $M$, the system \eqref{eq:wminhomo} on the time invertal $(0, T_1)$ is locally exactly controllable around $(\gamma, 0)$. More precisely, for any given  pair of data $(\phi_0, \phi_{0t}), (\phi_1, \phi_{1t})\in \mathcal{H}(\mathbb{T}^1; \mathcal{N})$
 satisfying 
 \[
 \big\|(\phi_0, \phi_{0t}) - (\gamma, 0)\big\|_{\mathcal{H}} +  \big\|(\phi_1, \phi_{1t}) - (\gamma, 0)\big\|_{\mathcal{H}}<\delta_1,
 \]
 there is a control $f\in C([0, T_1]; L^2(\mathbb{T}^1))$ such that the flow associated to \eqref{eq:wminhomo} carries the data $(\phi_0, \phi_{0t})$ at time $t = 0$ into $(\phi_1, \phi_{1t})$ at time $t = T_1$. 
\end{prop}
The equation is a semilinear one subject to geometric constraints. When dealing with a physically localized control force, the force must first be projected onto the tangent spaces to meet these constraints, making it hard to explicitly characterize the evolution of the flow. Consequently, standard techniques such as duality methods and fixed point arguments cannot be directly applied. 

We propose a three-step method, with each step leveraging distinct properties of the system.
\begin{itemize}[leftmargin=1em]
    \item[] {\it Step 1. On the reduction to a semilinear equation without geometric constraint}; Section  \ref{Sec:cont:1} 
    
    We start with the extrinsic formulation  (see  Section \ref{subsec:geometricsetting} and equation \eqref{eq:controlgeneral}), 
\begin{equation}
\Box \phi^i + S^i_{jk}(\phi)\partial_{\alpha}\phi^j \partial^{\alpha}\phi^k =  \chi_{\omega} \left(\Pi_T(\phi)f\right)^i,  \notag
\end{equation}
for every $i = 1, 2,\ldots N$. This is a system with $N$-coupled equations and implicit controls, namely $(\phi^1,...\phi^N)$ and the projected control force.
   Using the moving frame method,  we reformulate the equation from an intrinsic point of view (see  equation \eqref{eq:visystem11})
  \begin{equation}
\begin{split}
    \Box w^p = \sum_{j,k= 1}^R \sum_{\beta= 0}^1
A_{jk}^{p}(x; {\bf{w}})\partial_{\beta} w^j \partial^{\beta} w^k +  \sum_{j= 1}^R  B_{j}^{p}(x; {\bf{w}}) \partial_{x}w^j+ \sum_{j= 1}^R  C_j^{p}(x; {\bf w}) w^j\\
+ \chi_{\omega} \alpha^p(t, x) +  \chi_{\omega} \sum_{j= 1}^R  D_j^p(x; {\bf{w}}) \alpha^j(t, x), \notag
\end{split}
\end{equation}
for every $i = 1, 2,\ldots R$. It suffices to control the latter system, see Proposition \ref{prop:congeo:1st}.  Note that the new system only has $R$-coupled terms and takes value in Euclidean space.  
\vspace{1mm}

    \item[]   {\it Step 2. On the reduction to a linear control problem}; Section \ref{Sec:cont:2}
    
In the precdeding semilinear controlled equation for ${\bf w}$, the control force is not directly given by $\alpha^j$ but is slightly modified. Due to this, the fixed point technique still remains difficult to apply. Instead, we employ an iterative scheme to determine the controlled trajectory. This approach reduces the nonlinear controllability problem to a linear control problem, as demonstrated by equation \eqref{eq:visystem11li}
\begin{equation}\label{eq:int:lin:cont}
\begin{split}
    \Box w^p- \sum_{1\leq j\leq R} B_{0, j}^{p}(x) \partial_{x}w^j- \sum_{1\leq j\leq R} C_{0, j}^{p}(x) w^j= 
\chi_{\omega} \alpha^p, 
\end{split}
\end{equation}
for every $p = 1, 2,\ldots R$. 
It is worth noting that we have taken advantage of the null structure to treat the nonlinear term $A_{jk}^{p}(x; {\bf{w}})\partial_{\beta} w^j \partial^{\beta} w^k$.
\vspace{1mm}

    \item[] {\it Step 3. Controllability of the linearized system}; Section \ref{subsec:con:3}

   In view of the Hilbert Uniqueness Method \cite{Lions}, the controllability of \eqref{eq:int:lin:cont}  is equivalent to the observability problem: for every  $\varphi[0]\in L^2\times H^{-1}(\mathbb{T}^1; \mathbb{R}^R)$ the solution of 
  \begin{equation*}
        \Box \varphi+ {\bf b}(x) \partial_x \varphi+ {\bf c}(x) \varphi= 0 \;  \textrm{ where ${\bf b}(x),{\bf c}(x)\in M_{R\times R}$},
\end{equation*} 
 satisfies
\begin{equation}
    \int_0^T \|\chi_{\omega} \varphi\|_{L^2}^2 \, dt\geq C_0 \|\varphi[0]\|_{L^2\times H^{-1}}^2. \notag
    \end{equation}
 Although the observability has been extensively studied over the past decades, we were unable to find a specific reference to this equation.  Thus we provide a proof based on the quantitative propagation of smallness argument introduced in our earlier work \cite{KX}. 

\end{itemize}

\vspace{2mm}

\noindent \textbf{Third property.} Next, we present the following uniform-time global controllability result.

\begin{prop}\label{pro:controlbetween}
 Let $\mathcal{N}$ and $\omega$ satisfy  \blackhyperref{setting(S)}{$({\bf S})$}.    Let  $T_2= 6\pi$.  Then for any homotopic closed geodesics\footnote{Throughout this paper,  two closed geodesics are said to be homotopic if their maps are homotopic. Note that generally $\gamma(\cdot)$ and $\gamma(2\cdot)$ are  not homotopic.} $\gamma_0, \gamma_1: \mathbb{T}^1\rightarrow \mathcal{N}$,
 there is a control $f\in C([0, T_2]; L^2(\mathbb{T}^1))$ such that the flow associated to \eqref{eq:wminhomo} carries the data $(\gamma_0, 0)$ at time $t = 0$ into $(\gamma_1, 0)$ at time $t = T_2$. 
\end{prop}

\begin{remark}
Note that the steady states of the free wave maps equation are closed geodesics. The result can be immediately improved to the set of controlled steady states for the wave maps, namely, the set
\begin{equation*}
    \mathcal{S}:= \{(\phi, 0); \Delta \phi  - \Pi(\phi)\left(\partial_{x}\phi, \partial_{x}\phi\right)= \chi_{\omega}f\}.
\end{equation*}
Thus, for every two states $(\phi_1, 0), (\phi_2, 0)\in \mathcal{S}$ that are homotopic, there is a control such that the flow associated to \eqref{eq:wminhomo} carries the data $(\phi_0, 0)$ at time $t = 0$ into $(\phi_1, 0)$ at time $t = 6\pi$.
\end{remark}

We emphasize that in Proposition \ref{pro:controlbetween}, the control time $T_2$ is {\it independent} of the state’s scale.  Roughly speaking, looking at Figure \ref{fig:movethreegeo}, our goal is  to move the state from $\gamma_0$ to $\gamma_1$. Note that these two states may lie far apart from each other, and the topology of the maps $\gamma_0, \gamma_1: \mathbb{T}^1\rightarrow \mathcal{N}$ can be highly nontrivial. 

\begin{figure}[H]
    \centering
    \includegraphics[width=0.8\linewidth]{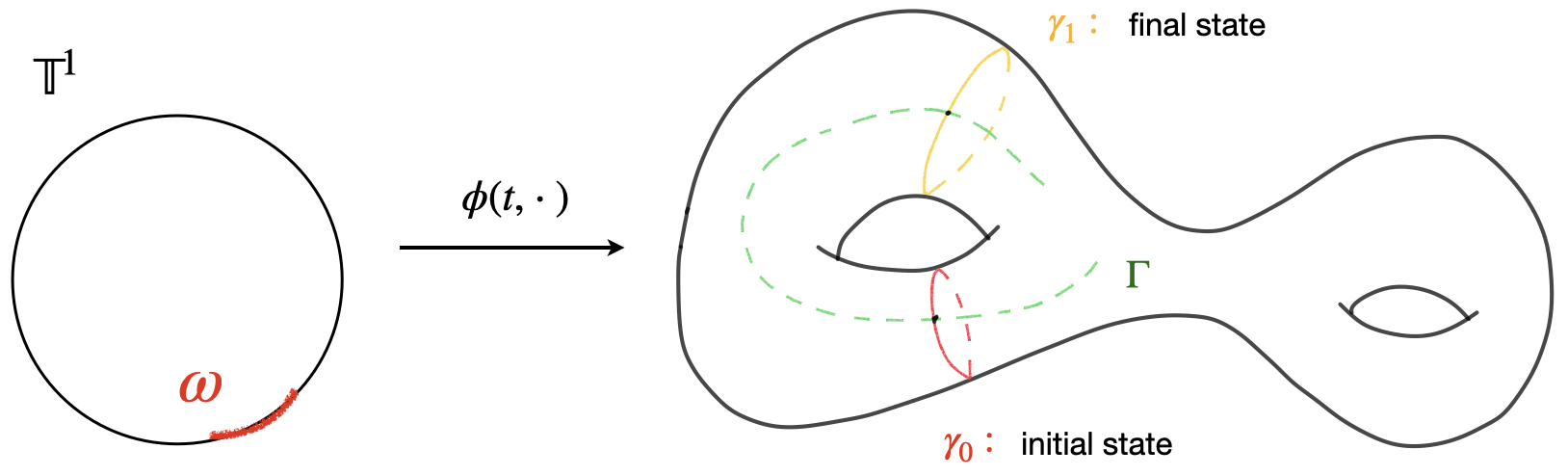}
    \caption{Uniform-time global controllability  between two homotopic steady states. Here $\gamma_0$ and $\gamma_1$ are two closed geodesics.}
    \label{fig:movethreegeo}
\end{figure}

Small-time or uniform-time global controllability remain challenging in nonlinear control theory and there are  many important and longstanding open problems:
\begin{op}[Lions, \cite{Lions-Pb}]
   Small-time global controllability of  Navier–Stokes equations with Dirichlet boundary control.
\end{op}
\begin{op}[Coron, \cite{Coron-Pb}]
   Small-time global controllability between steady states of  nonlinear heat equations, for example,
   \begin{equation*} 
 u_t - u_{xx} - u^3 = \chi_{\omega} f \textrm{ on } \mathbb{T}^1.  
 \end{equation*} 
\end{op}
\begin{op}[Dehman--Lebeau--Zuazua, \cite{Dehman-Lebeau-Zuazua}]
   Uniform-time global controllability of  semilinear wave equations, for example,
   \begin{equation*} 
\partial_t^2 u- \Delta u+ u^3= \chi_{\omega} f.
 \end{equation*} 
\end{op}
By contrast to the first two problems, the last problem on wave equations requires a uniform controlling time instead of a small time. This is due to the finite speed of propagation.  So far the best result on Navier–Stokes equations is given by \cite{Coron-Marbach-Sueur-Zhang}, but this problem is still widely open. 
In the recent work by two of the authors \cite{CX-2024}, the problem on nonlinear heat equations is solved for the harmonic map heat flow case, which is a coupled semilinear heat equation with geometric constraints:
\begin{equation*} 
u_t - u_{xx} + \Pi(u)(u_x, u_x) = \chi_{\omega} f.
\end{equation*}
\vspace{1mm}

Proposition \ref{pro:controlbetween} is a special geometric case of the last open problem concerning wave equations limited to steady states. 
We solve this problem by applying the approach developed in \cite{CX-2024}. This proof hinges on the following key geometric observations; see Section \ref{Sec:gc} for details.
\begin{itemize}[leftmargin=2em]
\item[\tiny$\bullet$]  Exact controllability on closed geodescis and non-closed complete geodesics. 
\item[\tiny$\bullet$]  Mass transport along non-closed complete geodesics.
\item[\tiny$\bullet$]  A gluing strategy: controlling the inner equation and matching it with the outer equation. 
\end{itemize}

\begin{figure}[H]
\centering
\includegraphics[width=1.0\textwidth]{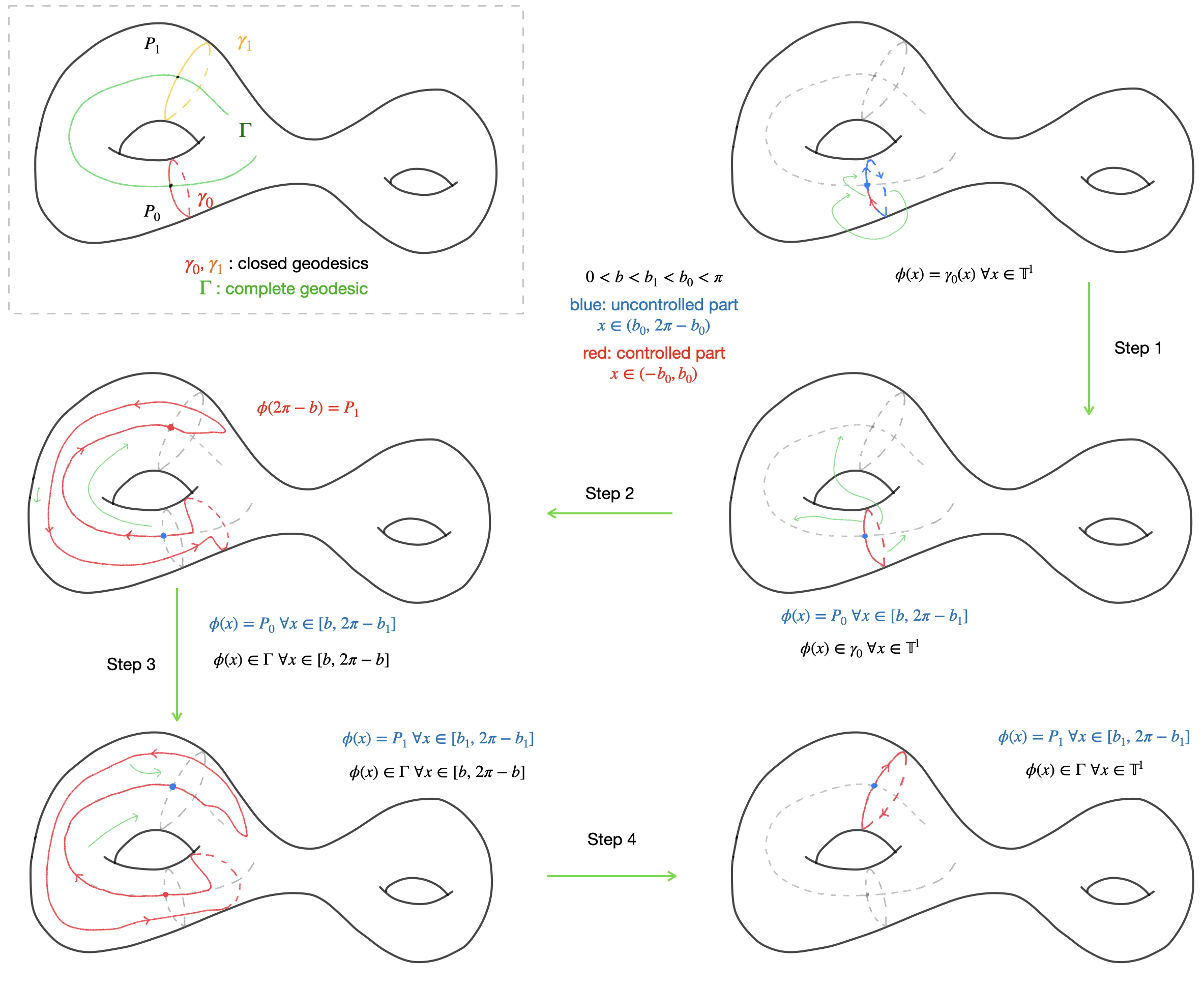}
\caption{The strategy for the uniform-time global controllability between two homotopic closed geodesics.}\label{fig:glo:control:geodesic}
\end{figure}

\vspace{2mm}

\noindent{\bf On the construction of the return trajectory.}

We first establish the following local controllability around any given {\it non-steady}  states.  

\begin{prop}\label{pro:controlaroundcurve}
 Let $\mathcal{N}$ and $\omega$ satisfy  \blackhyperref{setting(S)}{$({\bf S})$}. Let $M>0$. Then there exist $T_3= T_3(M)>0$ and $\delta_3= \delta_3(M)>0$ such that  for any given inital state  $(\phi_0, \phi_{0t})\in \mathcal{H}(\mathbb{T}^1; \mathcal{N})$ with energy smaller than $M$, and any pair of data $
(\phi_1, \phi_{1t}), (\phi_2, \phi_{2t})\in \mathcal{H}(\mathbb{T}^1; \mathcal{N})$
 satisfying 
 \begin{gather*}
      \big\|(\phi_1, \phi_{1t}) -(\phi_0, \phi_{0t}))\big\|_{\mathcal{H}} +  \big\|(\phi_2, \phi_{2t}) - ( \phi_0,  \phi_{0t})\big\|_{\mathcal{H}}<\delta_3,
 \end{gather*}
 there is a control $f\in C([0, T_3]; L^2(\mathbb{T}^1))$ such that the flow associated to \eqref{eq:wminhomo} carries the data $(\phi_1, \phi_{1t})$ at time $t = 0$ into $(\phi_2, \phi_{2t})$ at time $t = T_3$. 
\end{prop}

To establish local controllability around the given state $(\phi_0, \phi_{0t})$, a natural approach is to analyze the linearized equation along a given trajectory. Due to the finite speed of propagation, it is necessary to assume that the time period $T$ be larger than $2\pi$. However, the usual approach runs into two issues:
\begin{itemize}[leftmargin=2em]
    \item[\tiny$\bullet$] Typically, controllability is studied along the free trajectory $\bar \phi$, meaning without any control input. However, in this setting, the final state $\bar\phi[T]$ often differs, or may even be far removed, from the initial state $\bar{\phi}[0]$. 
    \item[\tiny$\bullet$] The trajectory is not explicitly known, and the system is subject to geometric constraints. Consequently, expressing local controllability around non-trivial trajectories of the geometric equation is particularly complicated.
\end{itemize}

To overcome these difficulties, we adapt the return method. The heuristic idea is that when the linearized system around a given state is not controllable, one can construct a trajectory where both the initial and final states coincide with this given state, thereby achieving local controllability around the constructed trajectory. 

In our setting, the construction of a return trajectory relies on the above introduced three intermediate properties, and the time-reversibility of the wave maps equation. See  Figure \ref{returntrajectory}.

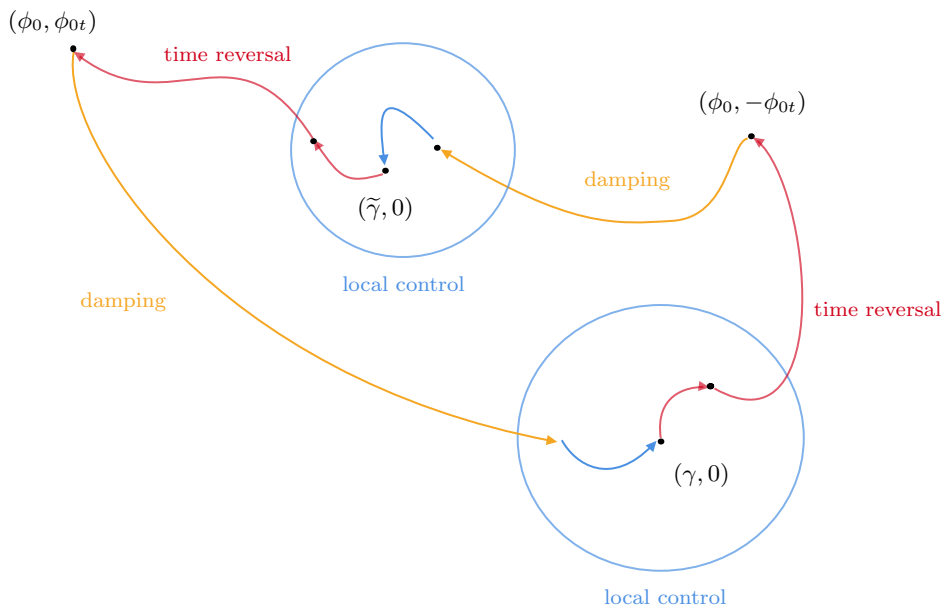
\begin{figure}[H]
    \centering

\tikzset{every picture/.style={line width=0.75pt}} 

\begin{tikzpicture}[x=0.75pt,y=0.75pt,yscale=-1,xscale=1]

\draw  [color={rgb, 255:red, 74; green, 144; blue, 226 }  ,draw opacity=0.7 ] (283.26,229.87) .. controls (283.26,192.96) and (315.23,163.03) .. (354.67,163.03) .. controls (394.1,163.03) and (426.07,192.96) .. (426.07,229.87) .. controls (426.07,266.78) and (394.1,296.7) .. (354.67,296.7) .. controls (315.23,296.7) and (283.26,266.78) .. (283.26,229.87) -- cycle ;
\draw [color={rgb, 255:red, 74; green, 144; blue, 226 }  ,draw opacity=1 ]   (305.3,231.01) .. controls (311.9,243.19) and (329.69,255.18) .. (350.62,233.45) ;
\draw [shift={(352.57,231.34)}, rotate = 131.53] [fill={rgb, 255:red, 74; green, 144; blue, 226 }  ,fill opacity=1 ][line width=0.08]  [draw opacity=0] (5.36,-2.57) -- (0,0) -- (5.36,2.57) -- cycle    ;
\draw [color={rgb, 255:red, 208; green, 2; blue, 27 }  ,draw opacity=0.65 ]   (354.67,229.87) .. controls (351.77,208.11) and (367.45,204.25) .. (376.03,204.26) ;
\draw [shift={(378.95,204.43)}, rotate = 187.11] [fill={rgb, 255:red, 208; green, 2; blue, 27 }  ,fill opacity=0.65 ][line width=0.08]  [draw opacity=0] (5.36,-2.57) -- (0,0) -- (5.36,2.57) -- cycle    ;
\draw  [line width=2.25] [line join = round][line cap = round] (355.02,231.69) .. controls (354.89,231.69) and (354.76,231.69) .. (354.63,231.69) ;
\draw  [line width=2.25] [line join = round][line cap = round] (379.5,203.89) .. controls (380.25,203.89) and (379.85,203.89) .. (379.1,203.89) ;
\draw  [line width=2.25] [line join = round][line cap = round] (399.74,78.15) .. controls (399.87,78.15) and (400,78.15) .. (400.14,78.15) ;
\draw  [line width=2.25] [line join = round][line cap = round] (305.14,231.86) .. controls (305.14,231.86) and (305.14,231.86) .. (305.14,231.86) ;
\draw [color={rgb, 255:red, 245; green, 166; blue, 35 }  ,draw opacity=1 ]   (61.67,35.68) .. controls (53.8,79.75) and (133.12,199.68) .. (300.52,231.42) ;
\draw [shift={(303.05,231.89)}, rotate = 190.39] [fill={rgb, 255:red, 245; green, 166; blue, 35 }  ,fill opacity=1 ][line width=0.08]  [draw opacity=0] (5.36,-2.57) -- (0,0) -- (5.36,2.57) -- cycle    ;
\draw [color={rgb, 255:red, 208; green, 2; blue, 27 }  ,draw opacity=0.65 ]   (381.49,205.01) .. controls (442.63,239.72) and (429.9,103.65) .. (402.88,80.41) ;
\draw [shift={(400.77,78.84)}, rotate = 33.85] [fill={rgb, 255:red, 208; green, 2; blue, 27 }  ,fill opacity=0.65 ][line width=0.08]  [draw opacity=0] (5.36,-2.57) -- (0,0) -- (5.36,2.57) -- cycle    ;
\draw  [line width=2.25] [line join = round][line cap = round] (61.43,33.85) .. controls (61.43,34.03) and (61.43,34.21) .. (61.43,34.38) ;
\draw [color={rgb, 255:red, 245; green, 166; blue, 35 }  ,draw opacity=1 ]   (247.74,86.59) .. controls (310.23,124.63) and (331.9,122.45) .. (360.34,120.71) .. controls (389.21,118.93) and (389.04,80.9) .. (398.08,79.29) ;
\draw [shift={(244.86,84.82)}, rotate = 31.54] [fill={rgb, 255:red, 245; green, 166; blue, 35 }  ,fill opacity=1 ][line width=0.08]  [draw opacity=0] (5.36,-2.57) -- (0,0) -- (5.36,2.57) -- cycle    ;
\draw  [color={rgb, 255:red, 74; green, 144; blue, 226 }  ,draw opacity=0.7 ] (170.16,85.1) .. controls (170.16,55.41) and (195.16,31.34) .. (226.01,31.34) .. controls (256.85,31.34) and (281.85,55.41) .. (281.85,85.1) .. controls (281.85,114.8) and (256.85,138.87) .. (226.01,138.87) .. controls (195.16,138.87) and (170.16,114.8) .. (170.16,85.1) -- cycle ;
\draw  [line width=2.25] [line join = round][line cap = round] (217.48,95.47) .. controls (217.35,95.47) and (217.22,95.47) .. (217.09,95.47) ;
\draw  [line width=2.25] [line join = round][line cap = round] (243.34,83.95) .. controls (243.21,83.95) and (243.08,83.95) .. (242.94,83.95) ;
\draw [color={rgb, 255:red, 74; green, 144; blue, 226 }  ,draw opacity=1 ]   (216.54,89.81) .. controls (214.08,68.74) and (212.38,49.47) .. (241.21,79.51) ;
\draw [shift={(216.89,92.8)}, rotate = 263.36] [fill={rgb, 255:red, 74; green, 144; blue, 226 }  ,fill opacity=1 ][line width=0.08]  [draw opacity=0] (5.36,-2.57) -- (0,0) -- (5.36,2.57) -- cycle    ;
\draw [color={rgb, 255:red, 208; green, 2; blue, 27 }  ,draw opacity=0.65 ]   (183.1,82.32) .. controls (195.02,99.71) and (197.89,102.3) .. (215.99,97.23) ;
\draw [shift={(181.35,79.76)}, rotate = 55.68] [fill={rgb, 255:red, 208; green, 2; blue, 27 }  ,fill opacity=0.65 ][line width=0.08]  [draw opacity=0] (5.36,-2.57) -- (0,0) -- (5.36,2.57) -- cycle    ;
\draw [color={rgb, 255:red, 208; green, 2; blue, 27 }  ,draw opacity=0.65 ]   (64.59,37.98) .. controls (111.83,73.64) and (142.34,17.76) .. (181.35,79.76) ;
\draw [shift={(61.67,35.68)}, rotate = 33.85] [fill={rgb, 255:red, 208; green, 2; blue, 27 }  ,fill opacity=0.65 ][line width=0.08]  [draw opacity=0] (5.36,-2.57) -- (0,0) -- (5.36,2.57) -- cycle    ;
\draw  [line width=2.25] [line join = round][line cap = round] (181.56,80.63) .. controls (181.43,80.63) and (181.3,80.63) .. (181.17,80.63) ;

\draw (27.24,12.97) node [anchor=north west][inner sep=0.75pt]  [font=\footnotesize]  {$( \phi _{0} ,\phi _{0t})$};
\draw (110.12,161.11) node  [font=\scriptsize] [align=left] {\begin{minipage}[lt]{67.7pt}\setlength\topsep{0pt}
\textcolor[rgb]{0.96,0.65,0.14}{damping}
\end{minipage}};
\draw (361.69,100.4) node  [font=\scriptsize] [align=left] {\begin{minipage}[lt]{67.7pt}\setlength\topsep{0pt}
\textcolor[rgb]{0.96,0.65,0.14}{damping}
\end{minipage}};
\draw (151.81,36.75) node  [font=\scriptsize] [align=left] {\begin{minipage}[lt]{67.7pt}\setlength\topsep{0pt}
\textcolor[rgb]{0.82,0.01,0.11}{time reversal}
\end{minipage}};
\draw (476.3,164.69) node  [font=\scriptsize] [align=left] {\begin{minipage}[lt]{67.7pt}\setlength\topsep{0pt}
\textcolor[rgb]{0.82,0.01,0.11}{time reversal}
\end{minipage}};
\draw (392.82,309.78) node  [font=\scriptsize] [align=left] {\begin{minipage}[lt]{99.58pt}\setlength\topsep{0pt}
\textcolor[rgb]{0.29,0.56,0.89}{local control}
\end{minipage}};
\draw (262.13,152.4) node  [font=\scriptsize] [align=left] {\begin{minipage}[lt]{99.58pt}\setlength\topsep{0pt}
\textcolor[rgb]{0.29,0.56,0.89}{local control}
\end{minipage}};
\draw (371.84,53.98) node [anchor=north west][inner sep=0.75pt]  [font=\footnotesize]  {$( \phi _{0} ,-\phi _{0t})$};
\draw (359.32,241.01) node [anchor=north west][inner sep=0.75pt]  [font=\footnotesize]  {$( \gamma ,0)$};
\draw (201.59,107.02) node [anchor=north west][inner sep=0.75pt]  [font=\footnotesize]  {$\left(\tilde{\gamma } ,0\right)$};

\end{tikzpicture}
\caption{ A well-designed return trajectory.  $\gamma$ and $\tilde \gamma$ are closed geodesics. States in the blue ball are approximate closed geodesics. During the orange stage, we set the control as localized damping to use Theorem \ref{thm:dynamics}; when the state is close to a closed geodesic, we apply Proposition \ref{pro:controlaroundgeodesic} during the blue stage. }\label{returntrajectory}
   
\end{figure}

\vspace{2mm}

To obtain Theorem \ref{thm1} with uniform control time from Proposition \ref{pro:controlaroundcurve}, one can rely on  compactness arguments.  
Next, we provide a direct proof of Theorem \ref{thm1}. This proof avoids continuous deformation and compactness arguments, making it more efficient.
Again, it is based on the idea of the return method, namely on the construction of a special trajectory that connects a given initial state $(\phi_0, \phi_{0t})$ and final state $(\phi_1,  \phi_{1t})$ such that the system is controllable around this trajectory. This construction highly relies on the geometric feature of the system, namely Proposition \ref{pro:controlbetween}. See Figure \ref{fig:spe:traject}.
\vspace{2mm}

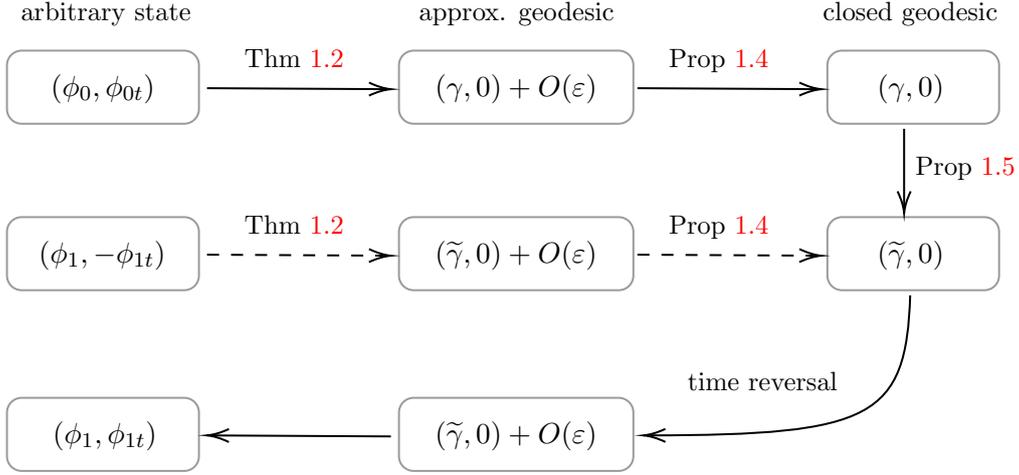
\begin{figure}[H]
    \centering

\tikzset{every picture/.style={line width=0.75pt}} 

\begin{tikzpicture}[x=0.75pt,y=0.75pt,yscale=-1,xscale=1]

\draw  [color={rgb, 255:red, 155; green, 155; blue, 155 }  ,draw opacity=1 ] (35,38.44) .. controls (35,34.38) and (38.29,31.09) .. (42.35,31.09) -- (123.46,31.09) .. controls (127.53,31.09) and (130.82,34.38) .. (130.82,38.44) -- (130.82,60.5) .. controls (130.82,64.56) and (127.53,67.86) .. (123.46,67.86) -- (42.35,67.86) .. controls (38.29,67.86) and (35,64.56) .. (35,60.5) -- cycle ;
\draw  [color={rgb, 255:red, 155; green, 155; blue, 155 }  ,draw opacity=1 ] (230.67,38.44) .. controls (230.67,34.38) and (233.96,31.09) .. (238.02,31.09) -- (340.32,31.09) .. controls (344.38,31.09) and (347.67,34.38) .. (347.67,38.44) -- (347.67,60.5) .. controls (347.67,64.56) and (344.38,67.86) .. (340.32,67.86) -- (238.02,67.86) .. controls (233.96,67.86) and (230.67,64.56) .. (230.67,60.5) -- cycle ;
\draw  [color={rgb, 255:red, 155; green, 155; blue, 155 }  ,draw opacity=1 ] (444.5,38.44) .. controls (444.5,34.38) and (447.79,31.09) .. (451.85,31.09) -- (522.88,31.09) .. controls (526.94,31.09) and (530.23,34.38) .. (530.23,38.44) -- (530.23,60.5) .. controls (530.23,64.56) and (526.94,67.86) .. (522.88,67.86) -- (451.85,67.86) .. controls (447.79,67.86) and (444.5,64.56) .. (444.5,60.5) -- cycle ;
\draw    (134.85,49.98) -- (224.64,50.48) ;
\draw [shift={(226.64,50.49)}, rotate = 180.32] [color={rgb, 255:red, 0; green, 0; blue, 0 }  ][line width=0.75]    (10.93,-3.29) .. controls (6.95,-1.4) and (3.31,-0.3) .. (0,0) .. controls (3.31,0.3) and (6.95,1.4) .. (10.93,3.29)   ;
\draw    (349.69,49.98) -- (439.47,50.48) ;
\draw [shift={(441.47,50.49)}, rotate = 180.32] [color={rgb, 255:red, 0; green, 0; blue, 0 }  ][line width=0.75]    (10.93,-3.29) .. controls (6.95,-1.4) and (3.31,-0.3) .. (0,0) .. controls (3.31,0.3) and (6.95,1.4) .. (10.93,3.29)   ;
\draw  [color={rgb, 255:red, 155; green, 155; blue, 155 }  ,draw opacity=1 ] (35,122.19) .. controls (35,118.13) and (38.29,114.84) .. (42.35,114.84) -- (123.46,114.84) .. controls (127.53,114.84) and (130.82,118.13) .. (130.82,122.19) -- (130.82,144.25) .. controls (130.82,148.31) and (127.53,151.6) .. (123.46,151.6) -- (42.35,151.6) .. controls (38.29,151.6) and (35,148.31) .. (35,144.25) -- cycle ;
\draw  [color={rgb, 255:red, 155; green, 155; blue, 155 }  ,draw opacity=1 ] (230.67,122.19) .. controls (230.67,118.13) and (233.96,114.84) .. (238.02,114.84) -- (340.32,114.84) .. controls (344.38,114.84) and (347.67,118.13) .. (347.67,122.19) -- (347.67,144.25) .. controls (347.67,148.31) and (344.38,151.6) .. (340.32,151.6) -- (238.02,151.6) .. controls (233.96,151.6) and (230.67,148.31) .. (230.67,144.25) -- cycle ;
\draw  [color={rgb, 255:red, 155; green, 155; blue, 155 }  ,draw opacity=1 ] (444.5,122.19) .. controls (444.5,118.13) and (447.79,114.84) .. (451.85,114.84) -- (522.88,114.84) .. controls (526.94,114.84) and (530.23,118.13) .. (530.23,122.19) -- (530.23,144.25) .. controls (530.23,148.31) and (526.94,151.6) .. (522.88,151.6) -- (451.85,151.6) .. controls (447.79,151.6) and (444.5,148.31) .. (444.5,144.25) -- cycle ;
\draw  [dash pattern={on 4.5pt off 4.5pt}]  (134.85,133.73) -- (224.64,134.23) ;
\draw [shift={(226.64,134.24)}, rotate = 180.32] [color={rgb, 255:red, 0; green, 0; blue, 0 }  ][line width=0.75]    (10.93,-3.29) .. controls (6.95,-1.4) and (3.31,-0.3) .. (0,0) .. controls (3.31,0.3) and (6.95,1.4) .. (10.93,3.29)   ;
\draw  [dash pattern={on 4.5pt off 4.5pt}]  (349.69,133.73) -- (439.47,134.23) ;
\draw [shift={(441.47,134.24)}, rotate = 180.32] [color={rgb, 255:red, 0; green, 0; blue, 0 }  ][line width=0.75]    (10.93,-3.29) .. controls (6.95,-1.4) and (3.31,-0.3) .. (0,0) .. controls (3.31,0.3) and (6.95,1.4) .. (10.93,3.29)   ;
\draw    (482.82,70.41) -- (482.82,110.79) ;
\draw [shift={(482.82,112.79)}, rotate = 270] [color={rgb, 255:red, 0; green, 0; blue, 0 }  ][line width=0.75]    (10.93,-3.29) .. controls (6.95,-1.4) and (3.31,-0.3) .. (0,0) .. controls (3.31,0.3) and (6.95,1.4) .. (10.93,3.29)   ;
\draw  [color={rgb, 255:red, 155; green, 155; blue, 155 }  ,draw opacity=1 ] (35,213.09) .. controls (35,209.03) and (38.29,205.73) .. (42.35,205.73) -- (123.46,205.73) .. controls (127.53,205.73) and (130.82,209.03) .. (130.82,213.09) -- (130.82,235.15) .. controls (130.82,239.21) and (127.53,242.5) .. (123.46,242.5) -- (42.35,242.5) .. controls (38.29,242.5) and (35,239.21) .. (35,235.15) -- cycle ;
\draw  [color={rgb, 255:red, 155; green, 155; blue, 155 }  ,draw opacity=1 ] (230.67,213.09) .. controls (230.67,209.03) and (233.96,205.73) .. (238.02,205.73) -- (340.32,205.73) .. controls (344.38,205.73) and (347.67,209.03) .. (347.67,213.09) -- (347.67,235.15) .. controls (347.67,239.21) and (344.38,242.5) .. (340.32,242.5) -- (238.02,242.5) .. controls (233.96,242.5) and (230.67,239.21) .. (230.67,235.15) -- cycle ;
\draw    (136.85,224.64) -- (226.64,225.14) ;
\draw [shift={(134.85,224.63)}, rotate = 0.32] [color={rgb, 255:red, 0; green, 0; blue, 0 }  ][line width=0.75]    (10.93,-3.29) .. controls (6.95,-1.4) and (3.31,-0.3) .. (0,0) .. controls (3.31,0.3) and (6.95,1.4) .. (10.93,3.29)   ;
\draw    (356.1,224.62) .. controls (467.03,224.46) and (483.85,215.83) .. (485.85,154.16) ;
\draw [shift={(352.71,224.63)}, rotate = 0] [color={rgb, 255:red, 0; green, 0; blue, 0 }  ][line width=0.75]    (10.93,-3.29) .. controls (6.95,-1.4) and (3.31,-0.3) .. (0,0) .. controls (3.31,0.3) and (6.95,1.4) .. (10.93,3.29)   ;

\draw (55.43,41.38) node [anchor=north west][inner sep=0.75pt]    {$( \phi_{0}, \phi_{0t})$};
\draw (247.18,41.36) node [anchor=north west][inner sep=0.75pt]    {$( \gamma,0) +O( \varepsilon )$};
\draw (467.86,41.36) node [anchor=north west][inner sep=0.75pt]    {$( \gamma, 0)$};
\draw (152.1,28.7) node [anchor=north west][inner sep=0.75pt]  [font=\small] [align=left] {Thm \ref{thm:dynamics}};
\draw (363.93,28.7) node [anchor=north west][inner sep=0.75pt]  [font=\small] [align=left] {Prop \ref{pro:controlaroundgeodesic} };
\draw (49.43,125.13) node [anchor=north west][inner sep=0.75pt]    {$(\phi_{1}, -\phi_{1t})$};
\draw (247.18,125.11) node [anchor=north west][inner sep=0.75pt]    {$( \tilde\gamma, 0) +O( \varepsilon )$};
\draw (467.86,125.11) node [anchor=north west][inner sep=0.75pt]    {$( \tilde \gamma, 0)$};
\draw (152.1,112.44) node [anchor=north west][inner sep=0.75pt]  [font=\small] [align=left] {Thm \ref{thm:dynamics}};
\draw (363.93,112.44) node [anchor=north west][inner sep=0.75pt]  [font=\small] [align=left] {Prop \ref{pro:controlaroundgeodesic} };
\draw (487.04,82.83) node [anchor=north west][inner sep=0.75pt]  [font=\small] [align=left] {Prop \ref{pro:controlbetween} };
\draw (56.43,216.03) node [anchor=north west][inner sep=0.75pt]    {$( \phi_{1} ,\phi_{1t})$};
\draw (247.18,216.01) node [anchor=north west][inner sep=0.75pt]    {$( \tilde \gamma, 0) +O( \varepsilon )$};
\draw (373.24,191.08) node [anchor=north west][inner sep=0.75pt]  [font=\small] [align=left] {time reversal};
\draw (40.42,5.21) node [anchor=north west][inner sep=0.75pt]  [font=\small] [align=left] {arbitrary state};
\draw (239.17,5.21) node [anchor=north west][inner sep=0.75pt]  [font=\small] [align=left] {approx. geodesic};
\draw (440.86,5.21) node [anchor=north west][inner sep=0.75pt]  [font=\small] [align=left] {closed geodesic};
\end{tikzpicture}
    \caption{ A special trajectory that connects two given states in uniform time.}\label{fig:spe:traject}
\end{figure}

\subsubsection{{\bf Ideas for Theorem \ref{thm:stability}}}\label{sec:strategy2}
\;
\vspace{2mm}

Let $\gamma: \mathbb{T}^1\longrightarrow \mathcal{N}$ be a closed geodesic with negative curvature. In this part we mainly work with the extrinsic formulation; see Section \ref{subsec:geometricsetting} for details.  Recall that the geodesic satisfies 
\begin{equation*}
\Delta \gamma + S_{jk}(\gamma)\partial_x\gamma^j \partial_x\gamma^k = 0.
\end{equation*}
Define the Jacobian operator 
\begin{equation*}
    \mathcal{L}_{\gamma}\varphi := \Delta \varphi + \varphi^r\partial_rS_{jk}(\gamma)\partial_x\gamma^j\partial_x\gamma^k + 2S_{jk}(\gamma)\partial_x\gamma^j\partial_x\varphi^k.
\end{equation*}

\noindent The quantitative stability analysis of the locally damped wave maps is involved for four reasons. 1) The dissipation is localized in a small sub-domain;  2) Due to the rotational symmetry group of closed geodesics, given by $\{(\gamma_p, 0); p\in \mathbb{T}^1\}$, stability cannot be attained throughout any neighborhood of $(\gamma, 0)$;  3) While the geometric equation can be rewritten as a semilinear wave equation in extrinsic formulation, both  linearized operators and nonlinear source terms have complex structures;   
and 4) The stability property is inherently tied to the curvature of  closed geodesics, which is naturally expressed in intrinsic form.  
\vspace{3mm}

 We propose a five-step strategy to prove the exponential stability result.  The key ingredients are summarized as follows:
   \begin{itemize}[leftmargin=2em]
       \item[\tiny$\bullet$]  Describe the state $\phi$ by means of two variables $(\varphi, \alpha)$ where $\varphi$ satisfies an {\it orthogonality/rigidity condition}, while $\alpha$ measures how the geodesic is modulated/rotated. This decomposition offers two advantages: it overcomes the influence of the rotational symmetry group of closed geodesics, and it ensures favorable properties for  $\varphi$. 
       \item[\tiny$\bullet$] The linearized systems for $(\varphi, \alpha)$ is complicated. By introducing an auxiliary function $\Psi= \varphi+ \alpha \gamma_x$, its linearized equation is much simplified.  
       \item[\tiny$\bullet$] Derive a {\it coercive estimate} for $\varphi$ under the operator $\mathcal{L}_{\gamma}$, leveraging the negative curvature and rigidity conditions. Both conditions are essential and necessary.
       \item[\tiny$\bullet$]  Finally, establish the stability of $\Psi$ under a well-chosen energy $\mathcal{E}_{\gamma}(\Psi)$, by combining coercive estimates with the propagation of smallness property.  Note that the first property is intrinsically due to the geometry, while the latter is characterized analytically through extrinsic formulas.
   \end{itemize}
    
\vspace{3mm}
In the sequel, we provide more insights into these five steps.

\noindent {\bf Step 1.}  {\it Decompose the state $\phi$ around the geodesic as $(\varphi, \alpha)$};  see  Section \ref{subsec:decompositiongeode}. 

Let $\gamma$ be a closed geodesic on $\mathcal{N}$. For a state close to $\gamma$, it is natural to decompose it as $ \phi(t, x)= \gamma(x)+ \varphi(t, x)+ \varphi_1(t, x)$  with
 $   \varphi(t, x)\in T_{\gamma(x)}\mathcal{N}, \; \varphi_1(t, x)\in T_{\gamma(x)}^{\perp}\mathcal{N}.$ 
However, the following decompositionturns out to be more effective for our analysis, 
\begin{gather*}
    \phi(t, x)= \gamma(x+ \alpha(t))+ \varphi(t, x)+ \varphi_1(t, x) \textrm{ with } \\
    \varphi(t, x)\in T_{\gamma(x)}\mathcal{N}, \; \varphi_1(t, x)\in T_{\gamma(x)}^{\perp}\mathcal{N},\forall x\in \mathbb{T}^1, \\
    \big\langle \varphi(t, \cdot), \dot\gamma(\cdot) \big\rangle_{L^2(\mathbb{T}^1)}= 0. 
\end{gather*}
We refer to the last equality as the {\it rigidity condition} on $\varphi$, as, without it, infinitely many pairs can satisfy the first two conditions. Note that this rigidity condition will play a significant role in Step 3 and Step 4.

\vspace{3mm}

\noindent {\bf Step 2.}  {\it Characterize the complete system in terms of $(\varphi, \alpha)$};  see Section \ref{subsec:full system varphialpha}.

The linearization for $(\varphi, \alpha)$ is complex and there are problematic coupling terms:
\begin{gather*}
     -  \varphi_{tt}+   \mathcal{L}_{\gamma}\varphi - a  \varphi_t-  \left(a(x)- \frac{l}{L}\right)\dot \alpha \gamma_x 
+ \frac{1}{L} \langle a(\cdot)\varphi_t(t, \cdot), \gamma_x(\cdot)\rangle_{L^2(\mathbb{T}^1)} \gamma_x, \\
 \ddot\alpha+ \frac{l}{L} \dot \alpha+\frac{1}{L}\left\langle  a(\cdot) \varphi_t(t, \cdot), \gamma_x(\cdot) \right\rangle_{L^2(\mathbb{T}^1)}. 
\end{gather*}
The full derivation of the nonlinear system is much more complicated, and we defer this technical aspect to this step, see Proposition  \ref{lem:full:system:varalpha}.
However, passing to $\Psi= \varphi+ \alpha \gamma_x$, the equation becomes much simpler, and no longer involves problematic linear coupling terms:
\begin{equation*}
   -  \Psi_{tt} +   \mathcal{L}_{\gamma}\Psi - a(x) \Psi_t
=  \mathcal{M}(x; \alpha, \dot\alpha, \varphi(\cdot), \varphi_t(\cdot), \varphi_x(\cdot)),
\end{equation*}
where the nonlocal nonlinear term is roughly controlled by
\begin{equation*}
\sum_{j, k} |\partial_{\nu} \varphi^j \partial^{\nu} \varphi^k|+  |\dot \alpha|^2+  |\varphi_t||\dot \alpha|+  (|\varphi|+ |\alpha|) (|\varphi|+ |\varphi_x|+ |\varphi_t|+ |\dot\alpha| ) 
\end{equation*}
Two key aspects of the above {\it quadratic form} should be noted:
\begin{itemize}[leftmargin=2em]
    \item[\tiny$\bullet$] The terms $\partial_{\nu} \varphi^j \partial^{\nu} \varphi^k$ exhibit the null structure. 
    \item[\tiny$\bullet$] The absence of a quadratic term in $\alpha$, $i.e.$ there is no  $\alpha^2$.
\end{itemize}

\vspace{3mm}

\noindent {\bf Step 3.} {\it Obtain a coercive estimate around geodesic with negative curvature}; see Section \ref{sec:coecive}.

To obtain the stability of the equation for $\varphi$ and $\Psi$, we first need to understand the operator $\mathcal{L}_{\gamma}$.
This operator is not self-adjoint for $\varphi\in H^1(\mathbb{T}^1; \mathbb{R}^N)$, but is self-adjoint when restricted to functions in $\gamma^*(T\mathcal{N})$; see Section \ref{subsec:22} for details.
In this step, we benefit from the geometric condition of negative sectional curvature along $\gamma$ to demonstrate the following coercive estimate: 
\[
-\langle \mathcal{L}_{\gamma}\varphi, \varphi\rangle_{L^2(\mathbb{T}^1)}\geq c\big\|\varphi\big\|_{H^1(\mathbb{T}^1)}^2, c>0,
\]
for any  $\varphi\in H^1(\mathbb{T}^1; \gamma^*(T \mathcal{N}))$ satisfying the rigidity condition. 
Both conditions are necessary: see Remark \ref{rem:coer:nece}.

\vspace{3mm}

\noindent {\bf Step 4.}  {\it Investigate the stability of the linearized equation for $\Psi$}; 
see Section \ref{subsec:linearstabilityvarphi}.

In this step, we establish the exponential stability of the linearized equation for $\Psi$.  Introduce an energy function $ \mathcal{E}_{\gamma}$: 
\begin{equation*}
 2 \mathcal{E}_{\gamma}(f, g):= \langle g, g \rangle_{L^2(\mathbb{T}^1)}  -\langle \mathcal{L}_{\gamma}f, f\rangle_{L^2(\mathbb{T}^1)} \;\; \forall (f, g)\in \mathcal{H}_{\gamma}.
\end{equation*}
In light of the coercive estimate, we have 
\begin{equation*}
    \mathcal{E}_{\gamma}(\Psi[t])\sim \mathcal{E}_{\gamma}(\varphi[t])+ |\dot \alpha|^2\sim \|\varphi[t]\|_{\mathcal{H}}^2+ |\dot \alpha|^2.
\end{equation*}

We show that for any function $\Psi\in C([0, 32\pi]; \mathcal{H}_{\gamma})$ satisfying 
\begin{gather*}
    -\Psi_{tt}+ \mathcal{L}_{\gamma} \Psi- a(x)\Psi_t = g,  \\
      \|g\|_{L^2(0, 32\pi; L^2(\mathbb{T}^1))}\leq \delta \mathcal{E}_{\gamma}(\Psi[0]),  
\end{gather*}
it holds that 
\begin{gather*}
      \mathcal{E}_{\gamma}(\Psi[32\pi])\leq (1- c)  \mathcal{E}_{\gamma}(\Psi[0]).
\end{gather*}
The proof relies on two properties: the quantitative propagation of smallness result, and the coercive estimate derived in Step 3.  
\vspace{3mm}

\noindent {\bf Step 5.}      {\it Prove the exponential stability of the full system}; see
Section \ref{subsec:expostabproof}.

Finally, by combining the ideas in Steps 2--4 we obtain the exponential stability of $(\varphi, \alpha)$:
\begin{gather*}
     \|\varphi[t]\|_{\mathcal{H}}^2+ |\dot\alpha(t)|^2 \lesssim e^{- \varepsilon t} \left(\|\varphi[0]\|_{\mathcal{H}}^2+ |\dot\alpha(0)|^2\right)\;\; \forall t\in (0, +\infty), \\
      |\alpha(t)- \bar \alpha|\lesssim e^{- \varepsilon t} \left( \|\varphi[0]\|_{\mathcal{H}}^2+  |\dot\alpha(0)|^2+ |\alpha(0)|^2\right)\;\; \forall t\in (0, +\infty).
\end{gather*}
Note that $\bar \alpha$ may be nonzero due to the existence of infinitely many steady states of the form $(\varphi, \varphi_t, \alpha, \dot \alpha)= (\bar \alpha, 0, 0, 0)$.

\vspace{3mm}

\noindent {\bf Organization of the paper.}
The remainder of the paper is organized as follows.

Section \ref{sec:preliminary} introduces the geometric setting, develops the analysis around closed geodesics, and establishes the relevant well-posedness results.
In Section \ref{sec:propsmall}, we address the quantitative propagation of smallness for wave maps. 

Sections \ref{sec:3}--\ref{Sec:gc} are devoted to proving three key intermediate results: global dynamics of the damped wave maps, Theorem \ref{thm:dynamics}; local exact controllability around closed geodesics, Proposition \ref{pro:controlaroundgeodesic}; and, uniform-time global controllability between homotopic closed geodesics, Proposition \ref{pro:controlbetween}. Building on these results, and by applying the return method, we establish Theorem \ref{thm1} on the global controllability of wave maps in Section \ref{Sec:5}.

Finally, Section \ref{Sec:8} proves Theorem \ref{thm:stability}, which concerns the exponential stability of wave maps around closed geodesics with negative curvature.

The appendices \ref{sec:app:A}--\ref{sec:app:C} gather the proofs of several technical lemmas.

\section{Preliminaries}\label{sec:preliminary}

This preliminary section is devoted to background on the geometric setting and the analysis around geodesics, and the null structure as well as the well-posedness of the controlled wave maps equations.  In particular, the well-posedness results are direct generalizations of those for the case of the sphere target shown in \cite{KX, CKX}.  In Section \ref{Sec:4} concerning the controllability around closed geodesics and Section \ref{Sec:8} concerning the exponential stability around closed geodesics with negative curvature, more geometric settings will be introduced.

\subsection{The geometric setting}\label{subsec:geometricsetting}

Recall from the assumption \blackhyperref{setting(S)}{$({\bf S})$} that $\mathcal{N} \hookrightarrow \mathbb{R}^N$ is a submanifold with dimension $R$.  For $\phi\in \mathcal{N}\subset\mathbb{R}^N$ we denote by $T_{\phi}\mathcal{N}$ the tangent space and by $N_{\phi}\mathcal{N}$ the normal space at this point. Denote by $\Pi_T(\phi) f$ the projection of a vector $f\in T_{\phi}\mathbb{R}^N= \mathbb{R}^N$ to the tangent space $T_{\phi}\mathcal{N}$.  
Let $\nabla$ be the Levi-Civita connection defined in $\mathcal{N}$, and let $\bar \nabla$ be the direction deviation in $\mathbb{R}^N$.  The second fundamental form is a smooth tensor field:
\begin{equation*}
    \Pi: T\mathcal{N}\times T\mathcal{N}\longrightarrow N \mathcal{N}.
\end{equation*}

In a part of this paper, we shall benefit from the extrinsic formulation of wave maps equations and linearized operators and perform explicit global analysis.  Thus, extending the above tensor field to the whole space is useful.   

\begin{lemma}\label{lem:secfunext}
There exist smooth maps
\[
S_{jk}:\mathbb{R}^N\to \mathbb{R}^N,\quad \forall j,k=1,\dots,N,
\]
satisfying
\[
S_{jk}(\phi)= S_{kj}(\phi) \;\; \forall  \phi\in\mathbb{R}^N,
\]
such that the $\mathbb{R}^N$-valued symmetric bilinear form
\begin{equation}\label{eq:2ndfundcoord-ori}
   \tilde\Pi(\phi)(v,w):= -  \sum_{j,k=1}^N S_{jk}(\phi)v^j w^k\quad \forall\, \phi\in\mathbb{R}^N,\; \forall\, v,w\in \mathbb{R}^N, 
\end{equation}
satisfies
\[
\tilde\Pi(\phi)(v,w)= \Pi(\phi)(v,w) \quad \forall\, \phi\in\mathcal{N},\; \forall\, v,w\in T_\phi\mathcal{N}.
\]
\end{lemma}

For example, when $\mathcal{N}= \mathbb{S}^{N-1}\hookrightarrow \mathbb{R}^N$, such an extension can be given by $\tilde \Pi(\phi)(v, w)= -\phi \langle v, w\rangle$.  For a general submanifold, this extension is a consequence of the tubular neighborhood theorem.  The details concerning this extension can be found in Appendix \hyperref[setting(Aex1)]{A.1}. 
In the rest part of this paper, for simplicity of notation we shall simply denote $\tilde \Pi$ by $\Pi$ and use this extended formula. For each component $i$, one has
\begin{equation}\label{eq:2ndfundcoord}
\Pi(\phi)^i(v, w) =  -   \sum_{j, k= 1}^N S^i_{jk}(\phi)v^j w^k \;\; \forall  i=1,2,..., N.
\end{equation}

Using the extrinsic coordinates,  the free wave maps equation, the damped wave maps equation \eqref{eq:dwm:ori}, and the controlled wave maps equation \eqref{eq:wminhomo} can be expressed by the following, where $i = 1,\,2,\ldots,\,N$,
\begin{equation}\label{eq:wm}
\Box \phi^i + S^i_{jk}(\phi)\partial_{\nu}\phi^j \partial^{\nu}\phi^k = 0, 
\end{equation}
\begin{equation}\label{eq:wmdamped}
\Box \phi^i + S^i_{jk}(\phi)\partial_{\nu}\phi^j \partial^{\nu}\phi^k = a(x)\partial_t \phi^i,
\end{equation}
and 
\begin{equation}\label{eq:controlgeneral}
\Box \phi^i + S^i_{jk}(\phi)\partial_{\nu}\phi^j \partial^{\nu}\phi^k = \chi_{\omega} \Pi_T(\phi) f,
\end{equation}
respectively, where we used the Einstein summation convention for $j, k\in \{1, 2,..., N\}$ and $\nu\in \{0, 1\}$ with $(\partial_0, \partial_1, \partial^0, \partial^1)= (\partial_t, \partial_x, -\partial_t, \partial_x)$.

\vspace{2mm} 
The damped equation \eqref{eq:wmdamped} and the controlled wave maps equation \eqref{eq:controlgeneral} are meaningful as illustrated in Section \ref{subsec:wellposedness}.  Throughout the paper, we assume the states are smooth enough, and that the $C^0([0, T]; \mathcal{H})$ solutions can be obtained via a standard passage to the limit argument. 

\subsection{The analysis around closed geodesics}\label{subsec:22}
The  closed geodesics (or harmonic maps), $\gamma: \mathbb{T}^1\rightarrow \mathcal{N}$, are solutions of the elliptic equation 
\begin{equation}\label{eq:geodesic:equation}
\Delta \gamma + S_{jk}(\gamma)\partial_x\gamma^j \partial_x\gamma^k = 0.
\end{equation}
Recall that the rotation group generates a family of harmonic maps  $\{\gamma_p\}_{p\in [0, 2\pi)}$,
\begin{equation}\label{def:closgerota}
    \gamma_{p}(x):= \gamma(x+ p) \; \; \forall p\in [0, 2\pi).
\end{equation}
\begin{defi}\label{definition-isolatedgeo} A closed geodesic $\gamma: \mathbb{T}^1\rightarrow \mathcal{N}$ is called isolated in the $H^1$-topology, if there is a constant $c>0$ with the property such that for any closed geodesic
\begin{align*}
\tilde{\gamma}\notin \{\gamma_p: p\in [0, 2\pi)\}  \textrm{ with }  \gamma_{p}(\cdot):= \gamma(\cdot+ p),
\end{align*}
we have the separation condition 
\begin{align*}
\min_{p\in [0, 2\pi)}\big\|\tilde{\gamma} - \gamma_p\big\|_{H^1(\mathbb{T}^1)}\geq c>0. 
\end{align*}
\end{defi}

\begin{defi}[Approximate closed geodesics]\label{def:app:clo:geo}
A map $\phi(\cdot): \mathbb{T}^1\rightarrow \mathcal{N}$ is said to be a $\varepsilon$-approximate closed geodesic, if there is a closed geodesic    $\gamma: \mathbb{T}^1\rightarrow \mathcal{N}$  such that 
\begin{equation*}
    \|\phi- \gamma\|_{H^1(\mathbb{T}^1)}\leq \varepsilon.
\end{equation*}
A state $(\phi, \phi_t): \mathbb{T}^1\rightarrow T\mathcal{N}$ is said to be a $\varepsilon$-approximate closed geodesic state, if there is a closed geodesic    $\gamma: \mathbb{T}^1\rightarrow \mathcal{N}$  such that 
\begin{equation*}
    \|(\phi, \phi_t)- (\gamma, 0)\|_{\mathcal{H}}\leq \varepsilon.
\end{equation*}
\end{defi}

Theorem \ref{thm:dynamics}, which shall be proved in Section \ref{sec:3}, implies that for any initial state, the unique solution converges to a closed geodesic along a subsequence of times.  To further obtain stability around a geodesic, it is necessary to assume that this geodesic satisfies the above isolation condition.
In order to derive a stronger quantitative rate of convergence, we impose a more stringent condition on the geodesic $\gamma$ toward which the flow converges. The following isolation result is well-known.
\begin{lemma}
\label{cor:isoation} 
Let $\gamma$ be a closed geodesic on $\mathcal{N}$. Assume that sectional curvature is strictly negative on $\gamma$. Then the geodesic $\gamma$ is isolated. 
\end{lemma}
\begin{remark}\label{rmk:iso:negcur}
The negative curvature assumption is essential, since if a portion of a geodesic has positive sectional curvature, then this geodesic is usually not isolated.  This is because in such a case, the Jacobi fields exhibit oscillatory behavior, and small deformation of $\gamma$ could lead to nearby closed geodesics.  To guarantee the isolation of geodesics, this condition can be relaxed to weak negative sectional curvature together with strict negative Ricci curvature, with the help of more delicate analysis.   
\end{remark}

Moreover, the negative curvature condition  entails non-degeneracy of the following operator $\mathcal{L}_{\gamma}$, which shall be explained in Section \ref{sec:coecive}.
\begin{defi}[Jacobi operator around closed geodesics]\label{def:jacobian}
Let $\gamma: \mathbb{T}^1\rightarrow \mathcal{N}$ be a closed geodesic on $\mathcal{N}$. Define  the {\it{Jacobi operator}}, $ \mathcal{L}_{\gamma}$, which arises upon linearising the geodesic equation around $\gamma$:
\begin{equation}
\mathcal{L}_{\gamma}\varphi := \Delta \varphi + \varphi^r\partial_rS_{jk}(\gamma)\partial_x\gamma^j\partial_x\gamma^k + 2S_{jk}(\gamma)\partial_x\gamma^j\partial_x\varphi^k,
\end{equation}  
with the Einstein summation convention for $r, j, k\in \{ 1,2,...,N\}$.
\end{defi}

In principle, the operator $ \mathcal{L}_{\gamma}$ is defined on tangential vector fields along the curve $\gamma$. Thanks to the extended extrinsic representation \ref{lem:secfunext}, the above operator makes sense for every function $\varphi: \mathbb{T}^1\longrightarrow \mathbb{R}^N$. However, due to the geometric feature that the function $\varphi$  should be the perturbation term of the geodesic on the tangent bundle, we shall focus on the action of $\mathcal{L}_{\gamma}$ on tangential vector fields along $\gamma$,  that is $\varphi\in \Gamma\left( \gamma^*(T\mathcal{N})\right)$:
\begin{equation*}
\varphi: \mathbb{T}^1\longrightarrow T\mathcal{N} \textrm{ such that } \varphi(x)\in T_{\gamma(x)} \mathcal{N}.
\end{equation*}
\begin{defi}
 Let $\gamma: \mathbb{T}^1\rightarrow \mathcal{N}$ be a closed geodesic on $\mathcal{N}$. Define 
 \begin{equation*}
    H^1 (\mathbb{T}^1; \gamma^* (T\mathcal{N})):= \{\varphi: \mathbb{T}^1\longrightarrow T\mathcal{N}: \varphi(x)\in T_{\gamma(x)} \mathcal{N}, \varphi  \textrm{ is an $H^1$-section} \}.
 \end{equation*}
\end{defi}

Here, a subtlety needs to be exploited: while the operator $\mathcal{L}_{\gamma}$ does not act in a self-adjoint manner on $H^1\big(\mathbb{T}^1; \mathbb{R}^N\big)$, this operator does act in self-adjoint fashion if we restrict $\varphi$ to functions in $\gamma^*(T\mathcal{N})$.  Recall the Riemann curvature tensor on $\mathcal{N}$:
\begin{equation*}
    R(X, Y)Z:= \nabla_{Y} \nabla_X Z- \nabla_{X} \nabla_Y Z+ \nabla_{[X, Y]}Z.
\end{equation*}
The following result is a consequence of  \cite[Theorem 6.1.4 (Synge’s second variation 
formula)]{2016-Petersen-book}. Its proof is in Appendix \hyperref[setting(A)]{A.2}. 
\begin{lemma}\label{lem:selfadjoint}
    Let $\varphi, \psi\in H^1(\mathbb{T}^1; \gamma^* (T \mathcal{N}))$. The operator  $\mathcal{L}_{\gamma}$ is self-adjoint in the sense that 
        \begin{equation*}
 \langle \mathcal{L}_{\gamma}\varphi, \psi\rangle_{L^2(\mathbb{T}^1)} = - 
\langle \nabla_{\gamma_x}\varphi, 
\nabla_{\gamma_x}\psi\rangle_{L^2(\mathbb{T}^1)}   + \int_{\mathbb{T}^1}\langle R(\gamma_x, 
\varphi)\gamma_x, \psi\rangle\,dx= \langle \varphi, \mathcal{L}_{\gamma}\psi\rangle_{L^2(\mathbb{T}^1)},
        \end{equation*}
and, in particular, $\mathcal{L}_{\gamma} \gamma_x= 0$ and
    \begin{align*}
        -\langle \mathcal{L}_{\gamma}\varphi, \varphi\rangle_{L^2(\mathbb{T}^1)} = \int_{\mathbb{T}^1}\big(\big\|\nabla_{\gamma_x}\varphi\big\|^2 - \langle R(\gamma_x, \varphi)\gamma_x, \varphi\rangle\big)\,dx.
    \end{align*}
\end{lemma}

As a direct consequence of the preceding lemma, one has the following rotational invariance property: let $b\in \mathbb{R}$, then
\begin{equation}\label{eq:mathcalLselfinva}
\langle \mathcal{L}_{\gamma}(\varphi+ b \gamma_x),(\varphi+ b \gamma_x)\rangle_{L^2(\mathbb{T}^1)}= \langle \mathcal{L}_{\gamma}\varphi, \varphi\rangle_{L^2(\mathbb{T}^1)}.
\end{equation}
This operator will be helpful in Section \ref{Sec:8} concerning the stability around geodesics. Note that in Section \ref{Sec:cont:1} and Sections \ref{subsec:decompositiongeode}--\ref{sec:coecive} we shall provide more geometric discussions.

\subsection{The null structure and the well-posedness results}\label{subsec:wellposedness}

Here, we present  well-posedness results for the damped/controlled wave maps equations. Note that in this paper we also examine several variant systems that share the same null structure. Due to the more complex form of these variants, we provide their well-posedness properties in the respective sections. Specifically, in Section \ref{subsubsec:421}, we address the semilinear wave equation with projected controls \eqref{eq:visystem11},  in Section \ref{subsec:full system varphialpha}, we analyze the coupled system for $(\varphi, \alpha)$.
\vspace{2mm}

Extend the equation $2\pi$-periodically to $x\in \mathbb{R}$. Consider the standard {\it null coordinate},
\begin{equation*}
    u= t+x \textrm{ and } v=t- x.
\end{equation*}
Using the $(u, v)$-coordinate we can express the equation on $\phi$ in a way that the propagation of the wave can be easily captured.  Indeed, observing that 
\begin{equation}\label{null-indentity}
    \partial_{\nu}g \partial^{\nu}h=  -2 (g_u h_v+ g_v h_u) \; \textrm{ and } \;
    \Box g= - 4 g_{uv},
\end{equation}
equation \eqref{eq:wmdamped} can be expressed as
\begin{align*}
    -\phi_{u v}^i&=
    S^i_{jk}(\phi) \phi^j_u \phi^k_v + \frac{1}{4}a \phi^i_t   := F^i(t, x)= F^i(u, v).
\end{align*}
Instead of the typical $C([0, T]; \mathcal{H})$ space, the solutions lie in the stronger space $\mathcal{W}_T$. In the following, when no confusion arises, we may denote $\mathcal{W}_T$ simply by $\mathcal{W}$ and a function $\phi\in (\mathcal{W}_T)^N$ by $\phi\in \mathcal{W}$.
\begin{defi}\label{def:WTnorm}
Let $T>0$. Set $D_T:= [0, T]\times \mathbb{T}^1$. Define \footnote{ Note that under the null coordinates of $(u, v)$ the region $D_T$ is not a standard rectangular-shaped region, thus the $L^2_uL^{\infty}_v(D_T)$-norm ($resp$. $L^2_v L^{\infty}_u$-norm,  $L^{\infty}_u L^2_v$-norm, $L^{\infty}_v L^2_u$-norm and  $L^{2}_uL^2_v$-norm) of a function $\phi$ can be understood as $\| \phi_{ex}\|_{L^2_uL^{\infty}_v(\tilde D_T)}$, where we extend $D_T$ to a rectangle region $\tilde D_T$ under $(u, v)$-coordinate and extend $\phi$ to $ \phi_{ex}$ trivially as zero in the region $\tilde D_T\setminus D_T$. }
\begin{gather}
\mathcal{W}_T:= \{\phi: [0, T]\times \mathbb{T}^1\rightarrow \mathbb{R};  \;\; 
\phi\in C(D_T),  \;\;\;\;\;\;\;\;\;\;\;\;\;\;\;\;\;\;\;\;\;\;\;\;\;\;\;\;\;\;\;\;\;\;\;\; \;\;\;\;\;\;\;\;\;\;\;\;\;\;\;\;\;\;\;   \notag \\
\;\;\;\;\;\;\;\;\;\;\;\;\;\;\;\;\;\;\;\;\;\;\;\;\;\;\; \;\;\;\;\;\;\;\;\;(\phi_x, \phi_t)\in C([0, T]; L^2(\mathbb{T}^1)),  \phi_v\in   L^2_v L^{\infty}_u(D_T),  \phi_u\in  L^2_u 
 L^{\infty}_v(D_T)\},  \notag 
\end{gather}
with 
\begin{gather*}
\|\phi\|_{\mathcal{W}_T}:=\|\phi_v\|_{L^2_v L^{\infty}_u\cap L^{\infty}_u L^2_v (D_T)}+ \|\phi_u\|_{L^2_u L^{\infty}_v\cap L^{\infty}_v L^2_u (D_T)}  \;\;\;\;\;\;\;\;\;\;\;\;\;  \;\;\;\;\;\;\;\; \notag \\
\;\;\;\;\;\;\;\;\;\;\;\;\;\;\;\;\;\;\;\;\;\;\;\;\;\;\;\;\;\;\;\;\;\;\;\;\;\;\;  \;\;\;\;\;\;\;\; \;\;\;\;\;\; + \|\phi\|_{C(D_T)}  +  \|(\phi_x, \phi_t)\|_{C([0, T]; L^2(\mathbb{T}^1))}. 
\end{gather*}
\end{defi}
\vspace{2mm}

The well-posedness of the damped/controlled wave maps is based on direct energy estimates and the null structure. The proof of the similar results for the sphere target case in \cite{KX} can be directly generalized to the general Riemannian manifold target case. Thus we omit the detailed proofs but only comment on several minor modifications.

\begin{lemma}\label{well:wm1}
Let $b: \mathbb{T}^1\rightarrow \mathbb{R}$ be a smooth function. Let $\tilde T>0$. There exists an explicit constant $C_w$ such that, for any $T\in (0, \tilde T]$, for any initial state $\phi[0]\in \mathcal{H}(\mathbb{T}^1; \mathcal{N})$,  and for any source term $f: [- T,  T]\times \mathbb{T}^1\rightarrow \mathbb{R}^{N}$ in $L^2_{t,x}(D_T)$, the equation 
\begin{equation*}
        \Box \phi -   \Pi(\phi)\left(\partial_{\nu}\phi, \partial^{\nu}\phi\right)= b \partial_t \phi+   \Pi_T(\phi) f,  
\end{equation*}
admits a unique solution.   This unique solution belongs to $C([0, T]; \mathcal{H}(\mathbb{T}^1; \mathcal{N}))$ and satisfies the following estimates: 
  \begin{gather*}
  \|\phi[t_2]\|_{\dot{H}^1_x\times L^2_x}\leq C_w \left(\|\phi[t_1]\|_{\dot{H}^1_x\times L^2_x}+ T^{1/2}\|f\|_{L^2_{t, x}(D_T)} \right)\;\; \forall t_1, t_2\in [-T, T],\\
  \|\phi_v\|_{L^{2}_v L^{\infty}_u\cap L^{\infty}_u L^{2}_v(D_T)}+  \|\phi_u\|_{L^{2}_u L^{\infty}_v\cap L^{\infty}_v L^{2}_u(D_T)}\leq C_w \left(\|\phi[0]\|_{\dot{H}^1_x\times L^2_x}+  T^{1/2}\|f\|_{L^2_{t, x}(D_T)}\right). 
\end{gather*}
\end{lemma}

This lemma is a generalization of \cite[Equations (2.8)--(2.12) and Lemma 3.2]{KX} with the same proof. Indeed, the energy estimates are straightforward consequences of the time derivation of $E(t)$. The $(u, v)$ coordinate-based estimates are based on the null structure of the wave maps equations. For instance, concerning $\|\phi_u\|_{L^2_u L^{\infty}_v}$ estimates, observe that
\begin{align*}
    \frac{d}{dv} |\phi_u|^2&= 2 \phi_u\cdot \phi_{uv}= -\frac{1}{2} \phi_u \cdot \Pi(\phi)(\partial_{\nu}\phi, \partial^{\nu}\phi) - \frac{b}{2} \phi_u\cdot \phi_t - \frac{1}{2}\phi_u\cdot  \Pi_T(\phi) f \\
    &=  -\frac{b}{2} \phi_u\cdot \phi_t - \frac{1}{2} \phi_u\cdot  \Pi_T(\phi) f.
\end{align*}
Without loss of generality, we assume that $ T\leq 2\pi$, and extend the force term $f$ trivially to $\tilde f$ in  the region  $Q= \{(u, v): |u|\leq 18\pi, |v|\leq 18\pi\}$. We solve the equation in $Q$ and denote the solution $\tilde \phi$. Clearly, by the definition of the $L^2_u L^{\infty}_v$-norm given in the footnote of Definition \ref{def:WTnorm}, one has 
\begin{gather*}
  \|\phi_u\|_{L^2_u L^{\infty}_v((-T, T)\times \mathbb{T}^1)}=  \|(\phi_u)_{ex}\|_{L^2_u L^{\infty}_v(Q)} \leq    \|\tilde\phi_u\|_{L^2_u L^{\infty}_v(Q)}, \\
   \|\phi_u\|_{ L^{\infty}_vL^2_u ((-T, T)\times \mathbb{T}^1)}=  \|(\phi_u)_{ex}\|_{ L^{\infty}_vL^2_u(Q)} \leq    \|\tilde\phi_u\|_{ L^{\infty}_v L^2_u(Q)}
\end{gather*}

Thanks to the direct energy estimate
\begin{equation*}
    \int_{-18\pi}^{18\pi}  \int_{-18\pi}^{18\pi} |\tilde\phi_u|^2(u,  v) \, du dv\lesssim \left( E(0)+ \|\tilde f\|_{L^1_{t}L^2_x(Q)}\right)^2\lesssim  \left( E(0)+ T^{1/2}\| f\|_{L^2_{t, x}(D_T)}\right)^2,
\end{equation*}
hence there exists some $\bar v\in (-18\pi,  18\pi)$ such that 
\begin{equation*}
    \int_{-18\pi}^{18\pi} |\tilde\phi_u|^2(u, \bar v) \, du\lesssim \left( E(0)+ T^{1/2}\| f\|_{L^2_{t, x}(D_T)}\right)^2,
\end{equation*}
We assume that $\bar v=-18\pi$. 
Thus, for any  $u, v\in (-18\pi, 18\pi)$ one has
\begin{align*}
    |\tilde\phi_u(u, v)|^2&\lesssim \left(|\tilde\phi_u|(u, -18\pi)+ \int_{-18\pi}^{v}|\tilde\phi_t(u, v_0)|+|\tilde f(u, v_0)| \, d v_0 \right)^2 \\
    &\lesssim |\tilde\phi_u|^2(u, -18\pi)+ \int_{-18\pi}^{18\pi} |\tilde\phi_t(u, v_0)|^2 \, dv_0+ T\int_{-18\pi}^{v}| \tilde f(u, v_0)|^2 \, d v_0,
\end{align*}
where we have used the fact that $f$ is supported in $t\in [-T, T]$.
 Hence, 
\begin{align*}
    \|\tilde\phi_u\|_{L^2_u L^{\infty}_v(Q)}^2 &\lesssim   \int_{-18\pi}^{18\pi} \left(|\tilde\phi_u|^2(u, -18\pi)+ \int_{-18\pi}^{18\pi} |\tilde\phi_t(u, v_0)|^2 \, dv_0+ T \int_{-18\pi}^{18\pi}|f(u, v_0)|^2 \, d v_0\right) \, d u \\
    &\lesssim \left(\|\tilde\phi[0]\|_{\dot{H}^1_x\times L^2_x}+ T^{1/2}\|\tilde f\|_{L^2_{t, x}(D_T)} \right)^2.
\end{align*}
Finally, it is obvious that 
\begin{equation*}
   \|\tilde\phi_u\|_{ L^{\infty}_vL^2_u(Q)}\leq  \|\tilde\phi_u\|_{L^2_u L^{\infty}_v(Q)}.
\end{equation*}

The following property is an analog of the result shown in \cite[Lemma 2.4]{CKX}.

\begin{lemma}\label{lem-conti-dep-inh}
Let $b: \mathbb{T}^1\rightarrow \mathbb{R}$ be a smooth function. Let $ T >0$ and $M>0$. 
There exists an effectively computable constant 
$C= C( T, M)$ such that,  for any initial states  $\phi[0], \varphi[0]  \in \mathcal{H}(\mathbb{T}^1; \mathcal{N})$, and any source terms $f, g$ in $L^2_{t, x}(D_T)$ satisfying
\begin{gather*}
\lVert \phi[0]\lVert_{\dot{H}^1\times L^2}+ \lVert \varphi[0]\lVert_{\dot{H}^1\times L^2}+ \lVert f\lVert_{L^2_{t, x}(D_T)}+ \lVert g\lVert_{L^2_{t, x}(D_T)}\leq  M,
\end{gather*} the unique solutions of
\begin{equation*}\label{eq:phi1}
    \Box \phi -  \Pi(\phi)\left(\partial_{\nu}\phi, \partial^{\nu}\phi\right)= b\partial_t \phi+  \Pi_T(\phi) f  \; \textrm{ with initial state } \phi[0],
\end{equation*}
and
\begin{equation*}\label{eq:varphi1}
    \Box \varphi - \Pi(\varphi)\left(\partial_{\nu}\varphi, \partial^{\nu}\varphi\right)=  b\partial_t \varphi+  \Pi_T(\phi) g   \; \textrm{ with initial state } \varphi[0],
\end{equation*}
satisfy 
\begin{equation*}
    \|w\|_{\mathcal{W}_T}\leq C \left(\lVert w[0]\lVert_{\mathcal{H}}+ \lVert f-g\lVert_{L^2_{t, x}(D_T)}\right).
\end{equation*}
where $w:= \phi- \varphi$. 
\end{lemma}

Finally, we present a simple result for linear wave systems:
\begin{lemma}\label{lem:linea:well:gen} 
 Let ${\bf A}, {\bf B}, {\bf C}: \mathbb{T}^1\rightarrow \mathbb{R}^{K\times K}$ be $C^2$ matrix valued functions. Let $\tilde T>0$. There exists some $\mathcal{C}>0$ effectively computable such that, for any $T\in (0, \tilde T]$, for any initial state  $\varphi[0]\in \mathcal{H}(\mathbb{T}^1; \mathcal{N})$ and any source term $f\in L^2_{t, x}(D_T)$ the equation   
\begin{equation*}
   \Box \varphi+ {\bf A} \partial_x\varphi+ {\bf C} \partial_t\varphi+ {\bf B}\varphi=  f 
\end{equation*}
admits a unique solution, moreover, 
\begin{gather*} 
  \|{\bf \varphi}\|_{\mathcal{W}_T}\leq \mathcal{C} \left( \|\varphi[0]\|_{\mathcal{H}}+  T^{1/2}\|\alpha\|_{L^2_{t, x}(D_T)}\right).
\end{gather*}
\end{lemma}

\section{Quantitative propagation of smallness for wave maps}\label{sec:propsmall}

In this section we present quantitative propagation of smallness of the (damped) wave maps equation, as formulated in Proposition \ref{lem:2}, Corollary \ref{cor:propaga:linear}, and Proposition \ref{prop:psl2} below. 
This property shows that if  $\phi_t$ is small for $x$ in a small interval, then it is small for every $x$ in $\mathbb{T}^1$. It is 
closely related to unique continuation properties and observability inequalities, which play a central role in control problems.  See, for example, the works mentioned in Section \ref{sec:contro:wave}.

The current quantitative version for geometric wave equations, along with its proof, was originally introduced  by the last two authors in a previous paper \cite{KX} for the case of a sphere target. Here we generalize it to the general manifold case. 
In this paper, these results will serve as auxiliary lemmas in proving various properties: 
\begin{itemize}[leftmargin=2em]
    \item[\tiny$\bullet$] In Section \ref{sec:3}, we will use Proposition \ref{lem:2} in the proof of Theorem \ref{thm:dynamics}, which concerns the convergence of solutions of the damped wave maps equation to closed geodesics. 
    \item[\tiny$\bullet$] In Section \ref{sec:cont:4th}, we will take advantage of Proposition \ref{prop:psl2} to prove the observability inequality formulated in Lemma \ref{lem:ob:key-lin}, which serves as one of the three steps in establishing local exact controllability around closed geodesics. 
    \item[\tiny$\bullet$] In Section \ref{subsec:linearstabilityvarphi}, we will combine the propagation of smallness, Corollary \ref{cor:propaga:linear}, and a coercive estimate around geodesics with negative curvature to prove  Proposition \ref{prop:propagationofsmallness} concerning the stability of the linearized equation governing $(\varphi, \alpha)$.  
\end{itemize}

\begin{prop}\label{lem:2}
Let $M>0$.  Let $J\subset \mathbb{T}^1$ be an open interval.  Let $b, d: \mathbb{T}^1\longrightarrow \mathbb{R}$ be $C^2$ functions.  There exist some $q>0$ and $C_{q}>0$ effectively computable such that, for any $\delta_1, \delta_2 \in [0, 1]$, if  any solution of the inhomogeneous  wave maps equation
\begin{equation}\label{eq:pro21:eq}
   \Box \phi -   \Pi(\phi)\left(\partial_{\nu}\phi, \partial^{\nu}\phi\right)= b(x) \partial_t\phi + d(x) \partial_x\phi+  \Pi_T(\phi)f 
\end{equation}
satisfies
\begin{gather*}
    E(\phi[0])\leq M,\\  \int_{-16\pi}^{16\pi}\int_{\mathbb{T}^1}\chi_J |\phi_t|^2(t, x) \, dx dt\leq \delta_1  E(\phi[0]), \\ \int_{-16\pi}^{16\pi}\int_{\mathbb{T}^1}|f|^2(t, x) \, dx dt\leq \delta_2  E(\phi[0]),
\end{gather*}
then 
\begin{equation*}
\|\phi_t\|_{L^{\infty}_x(\mathbb{T}^1; L^2_t(-3\pi, 3\pi))}^2\leq C_q \left(\delta_1^{1/q}+  \delta_2^{1/q}\right) E(\phi[0]).
\end{equation*}
\end{prop}

This result essentially tells us that once $\phi_t$ is locally small, the state is roughly steady.
Typically, when dealing with wave equations in  euclidean space people investigate the observability inequality of the system
\begin{equation}
    \int_{-16\pi}^{16\pi}\int_{\mathbb{T}^1}\chi_J |\phi_t|^2(t, x) \, dx dt\geq c E(\phi[0]).  \notag
\end{equation}
Such an observability result can be considered only if 0 is the unique steady state. Otherwise, we can assume $\phi[t] \equiv \phi[0]$ as the other non-trivial steady state, which clearly does not satisfy the above mentioned observability inequality.  That is the reason we study the weaker, propagation of the smallness problem.

\vspace{3mm}
Similarly, we have the following analog results on linear wave equations with source terms.
\begin{cor}\label{cor:propaga:linear} 
  Let $J\subset \mathbb{T}^1$ be an open interval. Let ${\bf A}, {\bf B}, {\bf C}: \mathbb{T}^1\rightarrow \mathbb{R}^{K\times K}$ be $C^2$ matrix valued functions.  There exist some $q>0$ and $C_{q}>0$ effectively computable such that, for any $\delta_1, \delta_2 \in [0, 1]$, if  any solution of the inhomogeneous  wave  equation
\begin{equation*}
   \Box \varphi+ {\bf A} \partial_x\varphi+ {\bf C} \partial_t\varphi+ {\bf B}\varphi=  f 
\end{equation*}
satisfies
\begin{gather*}
\int_{-16\pi}^{16\pi}\int_{\mathbb{T}^1}\chi_J |\varphi_t|^2(t, x) \, dx dt\leq \delta_1  \|\varphi[0]\|_{H^1\times L^2(\mathbb{T}^1)}^2, \\ \int_{-16\pi}^{16\pi}\int_{\mathbb{T}^1}|f|^2(t, x) \, dx dt\leq \delta_2 \|\varphi[0]\|_{H^1\times L^2(\mathbb{T}^1)}^2,
\end{gather*}
then 
\begin{equation*}
\|\varphi_t\|_{L^{\infty}_x(\mathbb{T}^1; L^2_t(-3\pi, 3\pi))}^2\leq C_q \left(\delta_1^{1/q}+  \delta_2^{1/q}\right) \|\varphi[0]\|_{\mathcal{H}}^2.   
\end{equation*}
\end{cor}

\begin{prop}\label{prop:psl2}
  Let $J\subset \mathbb{T}^1$ be an open interval. Let ${\bf A}, {\bf B}, {\bf C}: \mathbb{T}^1\rightarrow \mathbb{R}^{K\times K}$ be $C^2$ matrix valued functions.    There exists some effectively computable $q>0$ and $C_q>0$ such that, for any $\varepsilon\in (0, 1)$ if the solution of   \begin{equation}\label{eq:prop:psl2:main}
    \begin{cases}
       \Box \varphi+ {\bf A} \partial_x\varphi+ {\bf C} \partial_t\varphi+ {\bf B}\varphi= 0,\\
        \varphi[0]\in L^2\times H^{-1}(\mathbb{T}^1; \mathbb{R}^R),
    \end{cases}
\end{equation} 
satisfies 
    \begin{equation*}
        \int_{-16\pi}^{16\pi} \int_{\mathbb{T}^1} \chi_{J}| \varphi|^2(t, x) \, dx dt\leq \delta \|\varphi[0]\|_{L^2\times H^{-1}}^2
    \end{equation*}
    then 
    \begin{equation*}
    \|\varphi\|^2_{L^{\infty}_x(\mathbb{T}^1; L^2_t (-3\pi, 3\pi))}\leq C_q \delta^{1/q}  \|\varphi[0]\|_{L^2\times H^{-1}}^2.
    \end{equation*}
\end{prop}

\vspace{3mm}

In the sequel, we outline the proof of Proposition \ref{lem:2}, which is largely inspired by the previous paper \cite{KX}, while postponing some technical calculations to  Appendix \hyperref[setting(B1)]{B.1}. The proof of Corollary \ref{cor:propaga:linear} follows exactly the same argument as Propagation \ref{lem:2}, so we omit it. Although the proof of Proposition \ref{prop:psl2} is also similar, due to differences in regularity settings, we provide a sketch of the proof in  Appendix \hyperref[setting(B2)]{B.2}. 

\begin{proof}[Sketch of the proof of Proposition \ref{lem:2}]
    This proposition  is analogous to one contained in \cite[Proposition 2.2]{KX}. It crucially exploits the one-dimensional setup, and the symmetry between the $x$ and the $t$-variables in the one dimensional wave equation. Since the proof introduced in  \cite{KX} can be applied to the general setting of Riemannian manifolds to obtain Proposition \ref{lem:2}, we only sketch it. 
    Recall that  the proof in \cite{KX} is based on two auxiliary lemmas:
    \begin{itemize}
        \item[Step 1] Lemma 2.4 in \cite{KX}: find a point $x_0$ such that 
        \begin{equation*}
          \|\phi_t(\cdot, x_0)\|_{L^2_t(-16\pi, 16\pi)}^2+   \|\langle\partial_t\rangle^{-1}\left(\eta_{-15\pi}^{15\pi}[\tau] \phi_{tx}\right)(\cdot, x_0)\|_{L^2_t(\R)}^2
        \end{equation*}
        is small, where $\eta$ is a cutoff as specified below in \eqref{eq:def:truncfunc}. This formula roughly represents the value of $ \|
        \phi_t(\cdot, x_0)\|_{L^2_t}^2+  \| \phi_{x}(\cdot, x_0)\|_{L^2_t}^2$. 
         \item[Step 2] Lemma 2.5 in \cite{KX}: propagate the smallness of $ \|\phi_t(\cdot, x_0)\|_{L^2_t}^2+  \| \phi_{x}(\cdot, x_0)\|_{L^2_t}^2$ to $x\in (x_0, x_0+ S_0)$. By iterating this step we obtain the smallness of $\|\phi_t(\cdot, x)\|_{L^2_t}^2$ for every $x\in \mathbb{T}^1$.
    \end{itemize}
In the sequel, we briefly outline the generalization of these two lemmas,  namely Lemma \ref{lem:choosex0} and Lemma \ref{lem:pro:wm} in this paper, and leave further technical details are left to  Appendix \hyperref[setting(B1)]{B.1}. 
\vspace{2mm}

Select an even, smooth, non-negative,  truncated function $\mu$ such that 
\begin{equation*}
     \mu(t)=1 \textrm{ on } [0, 1/2], \;\; \mu(t)=0 \textrm{ on } [1, +\infty].
\end{equation*}
For every $\alpha<\beta$, and for every $\tau\in (0, 1)$ we define
\begin{equation}\label{eq:def:truncfunc}
    \eta_{\alpha}^{\beta}[\tau](t):= 
    \begin{cases}
    \mu(\frac{t-\beta}{\tau}), \;\; \forall t\in (\beta, +\infty), \\
    1, \;\; \forall t\in [\alpha, \beta], \\
    \mu(\frac{t-\alpha}{\tau}), \;\;  \forall t\in (-\infty, \alpha).
    \end{cases}
\end{equation}
For $\alpha+ 2\pi<\beta$ and $l\leq \pi$ we  define
\begin{gather}
    P_{\alpha, \beta}^l(y):= \{(t, x): x\in [y, y+l], t\in [\alpha+ x-y, \beta-x+y]\}, 
\end{gather}
and 
\begin{gather}\label{eq:def:norm:Pab}
    \|\psi\|_{L^{\infty}_x L^2_t(P_{\alpha, \beta}^l(y))}:= \sup_{x\in [y, y+l]}   \|\psi(t, x)\|_{L^2_t(\alpha+ x-y, \beta-x+y)}.
\end{gather}

To generalize the first lemma it suffices to replace the nonlinear term $(|\phi_t|^2- |\phi_x|^2)\phi$ by $\Pi(\phi)\left(\partial_{\nu}\phi, \partial^{\nu}\phi\right)$ and to add the  source term $\Pi_T(\phi)f$ which results in lower order corrections. Thus we omit the proof of this lemma. 
    
    \begin{lemma}\label{lem:choosex0}
Under the assumptions of Proposition \ref{lem:2}. 
There exists some effectively computable $C_0>0$ such that, for any  $\delta_1, \delta_2\in [0, 1], \tau\in (0, 1)$,
and for any solution of \eqref{eq:pro21:eq} satisfying
\begin{gather}
    E(\phi[0])\leq M,  \notag\\  \int_{-16\pi}^{16\pi}\int_{\mathbb{T}^1}\chi_J |\phi_t|^2(t, x) \, dx dt\leq \delta_1  E(\phi[0]), 
     \notag \\ \int_{-16\pi}^{16\pi}\int_{\mathbb{T}^1}|f|^2(t, x) \, dx dt\leq \delta_2  E(\phi[0])  \notag
\end{gather}
there exists some $x_0\in [0, 2\pi)$ such that
\begin{equation}
     \|\phi_t(t, x_0)\|_{L^2_t(-16\pi, 16\pi)}^2+   \|\langle\partial_t\rangle^{-1}\left(\eta_{-15\pi}^{15\pi}[\tau] \phi_{tx}\right)(t, x_0)\|_{L^2_t(\R)}^2 \leq C_0 \frac{\sqrt{\delta_1}}{\tau^2} E(\phi[0]).  \notag
\end{equation}
\end{lemma}
Notice that in the preceding inequality, the estimate is independent of $\delta_2\in [0, 1]$. This is becasue the source term $\Pi_T(\phi)f$ only appears in the estimates of $ \left\|\langle \partial_t\rangle^{-1} \left(\eta(t) \phi_{xx}\right)\right\|_{L^2_x(\mathbb{T}^1; L^2_t(\R))}$ and $ \left\|\langle \partial_t\rangle^{-1} \left(\eta_t(t) \phi_{xx}\right)\right\|_{L^2_x(\mathbb{T}^1; L^2_t(\R))}$, see \cite[equations (2.26)--(2.27)]{KX}, which are uniformly bounded by $ \frac{1}{\tau}\sqrt{E(\phi[0])}$ and $ \frac{1}{\tau^2}\sqrt{E(\phi[0])}$. 
    
\vspace{3mm} 
    
    As for the second lemma, instead of characterizing $\phi_t$ as the sum of 6 terms $\{f_k\}_{1\leq k\leq 6}$ as performed in \cite{KX}, in our current setting for each $i= 1, 2,..., N$ we shall write $\phi_t^i$ in terms of 6 functions.  This is due to the fact that $S^i_{jk}(\phi)\partial_{\nu}\phi^j \partial^{\nu}\phi^k$ is slightly more complicated than the nonlinear term $(|\phi_t|^2- |\phi_x|^2)\phi^i$ appearing in the sphere target case.  Otherwise, all the estimates remain the same despite more complicated formulas, and the appearance of the source term $\Pi_T(\phi)f$ is again irrelevant to the analysis.  We put the explicit description of $\phi^i_t$ and  the proof of this lemma in  Appendix \hyperref[setting(B1)]{B.1}.

\begin{lemma}\label{lem:pro:wm}
Under the assumptions of Proposition \ref{lem:2}. There exist  effectively computable values $S_0$ 
 $C_N$, and $ C_{S_0}$ such that,  for any  $\delta_2\in [0,  1], \tau\in (0, 1), \alpha\in [-15\pi, 0), \beta\in (2\pi, 15\pi], z\in [0, 4\pi]$, and for  any solution of \eqref{eq:pro21:eq}
satisfying
\begin{gather*}
    E(\phi[0])\leq M,\\  \int_{-16\pi}^{16\pi}\int_{\mathbb{T}^1}|f|^2(t, x) \, dx dt\leq \delta_2  E(\phi[0])   \notag
\end{gather*}
we have
\begin{equation*}
    \|\phi_t\|_{L^{\infty}_x L^2_t(P_{\alpha, \beta}^{S_0}(z))}\leq C_N \left(    \|\phi_t(t, z)\|_{L^2_t(\alpha, \beta)}+ \|\langle\partial_t \rangle^{-1}\left(\eta_{\alpha}^{\beta}[\tau](t) \phi_{tx}(t, z)\right)\|_{L^2_t(\R)}+  (\delta_2 E(\phi[0]))^{\frac{1}{2}}\right).  
\end{equation*}
Moreover, by denoting $\tau_0:= S_0/16$, there exists $\bar z\in (z+ S_0/2, z+ S_0)$ such that 
\begin{align*}
   &\;\;\;\;\; \left\| \langle\partial_t\rangle^{-1} \left(\eta_{\alpha+ S_0+\tau_0}^{\beta-S_0-\tau_0}[\tau_0](t)\phi_{tx}(t, \bar z)\right)\right\|_{L^2_t(\R)} \notag \\
   &\leq C_{S_0} \left(E(\phi[0])\right)^{\frac{1}{4}}\left(  \|\phi_t(t, z)\|_{L^2_t(\alpha, \beta)}+ \left\|\langle\partial_t \rangle^{-1}\left(\eta_{\alpha}^{\beta}[\tau_0](t) \phi_{tx}(t, z)\right)\right\|_{L^2_t(\R)}+  (\delta_2 E(\phi[0]))^{\frac{1}{2}} \right)^{\frac{1}{2}}.
\end{align*}
\end{lemma}

Proposition \ref{lem:2} follows by combining Lemma \ref{lem:choosex0} and Lemma \ref{lem:pro:wm}.

\end{proof}

\section{Dynamics of the locally damped wave maps}\label{sec:3}

This section is devoted to the proof of Theorem \ref{thm:dynamics}, which constitutes to the first intermediate result for the proof of Theorem \ref{thm1}. Note that in Section \ref{Sec:8}, a stronger exponential convergence result is shown provided the negative sectional curvature condition, namely  Theorem \ref{thm:stability}.

 Theorem \ref{thm:dynamics} is composed of two properties: 1) There exists a time $T_c = T_c(M, \delta)\geq 0$ such that any solution of the locally damped wave maps with energy smaller than $M$ will be  $\delta$-close to a geodesic for some time $t\in [0, T_c]$; 2) For any given initial state, the solution converges to a closed geodesic along a sequence of times.
 In the sequel, the proof is composed of four steps:
\begin{itemize}[leftmargin=2em]
   \item[\tiny$\bullet$] Give a propagation of smallness result; Lemma \ref{prop:uniformsmallness};  
    \item[\tiny$\bullet$] Present an auxiliary result; Lemma \ref{lem:extraction};  
    \item[\tiny$\bullet$] Prove the second property concerning the limiting closed geodesic for a given solution;
    \item[\tiny$\bullet$] Show the first property concerning the existence of a uniform time $T_c(M, \delta)$ in the preceding. 
\end{itemize}

\vspace{2mm}

\noindent {\bf  Step 1. A propagation of smallness result.}
The following lemma is a direct consequence of Proposition \ref{lem:2}. 
\begin{lemma}\label{prop:uniformsmallness}  Let $\mathcal{N}$ and $a(\cdot)$ satisfy the setting \blackhyperref{setting(S)}{$({\bf S})$}.  Given $M>0, \delta>0$, $A>0$ there is an effective computable $\tilde{\delta} = \tilde{\delta}(\delta,A, M)>0$ with the following property.
For any given $b\in \mathbb{R}$, and any function $\phi$ satisfying the locally damped equation \eqref{eq:dwm:ori} assuming the bound 
\begin{gather*}
    \|\phi[b]\|_{\mathcal{H}}\leq M, \\
\int_{I_{ex}} \int_{\mathbb{T}^1}a(x)\big|\phi_t\big|^2\,dx dt<\tilde{\delta}, 
\end{gather*}
then 
 \begin{align*}
 \big\|\phi_t\big\|_{L_x^\infty L_t^2(I\times \mathbb{T}^1)}<\delta,
 \end{align*}
 where the intervals $I:= (b, b+A)$ and $ I_{ex}:= (b- 13\pi, b+A+13\pi)$.
\end{lemma}

Taking the inner product of \eqref{eq:wmdamped}  against $\phi_t$, and integrating by parts, we infer the basic energy monotonicity property 
\begin{equation}\label{eq:energydecrease} 
E(\phi[t_2]) = E(\phi[t_1]) - 2\int_{t_1}^{t_2} \int_{\mathbb{T}^1}a(x)\big|\phi_t\big|^2\,dx dt. 
\end{equation}
In particular, for any given initial state $\phi[0]$ the limit 
\[
\lim_{t\rightarrow +\infty}E(\phi[t]) =: E_{\infty}\geq 0
\]
exists, and given any $\delta>0$ and $\tilde{\delta} = \tilde{\delta}(\delta, 3, M)$ as in the preceding proposition, we can pick $T = T(\tilde{\delta},\phi)$ such that 
\[
\int_T^\infty \int_{\mathbb{T}^1}a(x)\big|\phi_t\big|^2\,dx dt < \tilde{\delta}. 
\]
Using Lemma \ref{prop:uniformsmallness}, this entails the existence of an interval $I\subset [T, \infty)$, $|I| = 3$, such that we have 
\begin{equation}
\big\|\phi_t\big\|_{L_x^\infty L_t^2(I\times \mathbb{T}^1)}<\delta.  \notag
\end{equation}

\noindent {\bf  Step 2. On the extraction of an approximate closed geodesic.} 
The goal is to extract an approximately closed geodesic $\tilde{\phi}$ provided that $\phi_t$ is small on $I\times \mathbb{T}^1$.
\begin{defi}\label{def:Ipsi}
  Given an interval $I_0= (0, 3)$. Define 
 $\psi_0= \psi_0(I_0)\in C_0^\infty(I)$ be a nonnegative smooth function which equals $1$ on the interval $\tilde{I}_0:= (5/4, 7/4)\subset I$. Furthermore, assume that $\psi_0(x)\leq 1$ for every $x\in I_0$ and that the normalization 
\[
\int_{I_0} \psi(t)\,dt = 1.
\]  
For different interval $I_s= (s, s+3)$, the function $\psi_s= \psi_s(I_s)$ is chosen via simple translation.
\end{defi}
The rest part of this step is devoted to the proof of the following lemma. 
\begin{lemma}\label{lem:extraction}
Let $M, \delta>0$.  Given $|I|= 3$ and let $\psi$ be given by Definition \ref{def:Ipsi}.  Suppose the function $\phi$ satisfies the locally damped equation \eqref{eq:dwm:ori} and has the bounds
\begin{gather}\label{eq:smallphit}
\big\|\phi_t\big\|_{L_x^\infty L_t^2(I\times \mathbb{T}^1)}\leq \delta, \\
\|\phi[t]\|_{\mathcal{H}}\leq M \; \forall t\in I.
\end{gather}
Then the function $\tilde \phi$ defined as 
\begin{equation}\label{eq:tildephi}
\tilde{\phi}(x): = \int_I \phi(t, x)\psi(t)\,dt \;\; \forall x\in \mathbb{T}^1
\end{equation}
satisfies 
\begin{gather}
\begin{split} \label{lem42:1}
     \|\tilde \phi- \phi\|_{L^{\infty}_{t, x}(I\times \mathbb{T}^1)}\leq C_0 \delta, \\
\big\|\tilde{\phi}\big\|_{H^2_x(\mathbb{T}^1)}\leq C_1, \\
 \| \tilde \phi_{xx}+ S_{jk}\tilde \phi_x^j \tilde \phi_x^k\|_{L^2_x(\mathbb{T}^1)}\leq C_1 \sqrt{\delta}, \\
 \big\|\sqrt{\psi(t)}\partial_x (\tilde{\phi} - \phi)\big\|_{L_{x, t}^2(I\times \mathbb{T}^1)}\leq C_1\sqrt{\delta},  
 \end{split}
\end{gather}
where the constants $C_0= C_0(M), C_1= C_1(M)$ do not depend on $\delta$ and $\phi$.
\end{lemma}

\begin{proof}[Proof of Lemma \ref{lem:extraction}]
Without loss of generality, we assume that $I= (0, T)$ with $T= 3$.
By the definition of $\tilde \phi$, there is a constant $C_0$ does not depend on $M, \delta$ and $\phi$ such that 
\begin{align*}
    \|\tilde \phi\|_{L^{\infty}_x(\mathbb{T}^1)}&\leq C_0  \| \phi\|_{L^{\infty}_x(\mathbb{T}^1; L^1_t(I))}, \\
     \|\tilde \phi_x\|_{L^{\infty}_x(\mathbb{T}^1)}&\leq C_0  \| \phi_x\|_{L^{\infty}_x(\mathbb{T}^1; L^1_t(I))}.
\end{align*}
Since 
\begin{equation*}
    \phi(s, x)= \phi(t, x)+ \int_t^s \phi_t (u, x) du,
\end{equation*}
we know that for every $(t, x)\in I \times \mathbb{T}^1$, 
\begin{align*}
 \tilde  \phi(t, x)-     \phi(t, x)
&= \int_I \phi(s, x) \psi(s) ds- \phi(t, x) \\
&= \int_I \int_t^s \phi_t (u, x)  \psi(s) du ds.
\end{align*}
Thus 
\begin{equation*}
    \|\tilde \phi- \phi\|_{L^{\infty}_{t, x}(I\times \mathbb{T}^1)}\leq C_0 \delta.
\end{equation*}

Observe for any test function $\zeta\in C^\infty(\mathbb{T}^1; \mathbb{R}^N)$, one has  the identity 
\begin{equation}\label{eq:tildephizeta1}\begin{split}
-\int_{\mathbb{T}^1}\zeta_x\cdot \tilde{\phi}_x\,dx &= - \int_I \int_{\mathbb{T}^1}\zeta_x\cdot\phi_x(t, x)\psi(t)\,dx dt\\
&=  \int_I \int_{\mathbb{T}^1}\zeta\cdot\phi_{xx}(t, x)\psi(t)\,dx dt\\
& = -\int_I \int_{\mathbb{T}^1}\zeta\cdot\phi_t(t, x)\psi_t\,dx dt -  \sum_{i=1}^N\int_I \int_{\mathbb{T}^1}\psi(t)\zeta^i S^i_{jk}(\phi) \phi^j_x \phi^k_x\,dx dt\\
&\hspace{1cm} + \sum_{i=1}^N\int_I \int_{\mathbb{T}^1}\psi(t)\zeta^i S^i_{jk}(\phi) \phi^j_t \phi^k_t\,dx dt + \int_I \int_{\mathbb{T}^1} \psi(t) a(x) \zeta \cdot \phi_t\,dx dt
\end{split}\end{equation}
We aim to show that $\tilde \phi(\cdot)$ (and eventually, also  $\phi(t, \cdot)$) quantify as approximate closed geodesics, namely 
\begin{equation*}
    \tilde \phi_{xx}+  S^i_{jk}(\tilde{\phi}) \tilde{\phi}^j_x \tilde{\phi}^k_x\approx 0.
\end{equation*}
Thus, we can alternatively write the preceding relation in the form 
\begin{equation}\label{eq:tildephizeta2}\begin{split}
-\int_{\mathbb{T}^1}\zeta_x\cdot \tilde{\phi}_x\,dx +  \sum_{i=1}^N \int_{\mathbb{T}^1}\zeta^i S^i_{jk}(\tilde{\phi}) \tilde{\phi}^j_x \tilde{\phi}^k_x\,dx  = e_1 + e_2, 
\end{split}\end{equation}
where we set 
\begin{align*}
&e_1: = -\int_I \int_{\mathbb{T}^1}\zeta\cdot\phi_t(t, x)\psi_t\,dx dt + \int_I \int_{\mathbb{T}^1} \psi(t) a(x) \zeta \cdot \phi_t\,dx dt  \\
& \;\;\;\;\;\;\;\;\;\;\;\;\; \;\;\;\;\;\;\;\;\;\;\;\;\;   + \sum_{i=1}^N \int_I \int_{\mathbb{T}^1}\psi(t)\zeta^i S^i_{jk}(\phi) \phi^j_t \phi^k_t\,dx dt, \\
&e_2: = -\sum_{i=1}^N \int_I \int_{\mathbb{T}^1}\psi(t)\zeta^i S^i_{jk}(\phi) \phi^j_x \phi^k_x\,dx dt +  \sum_{i=1}^N \int_{\mathbb{T}^1}\zeta^i S^i_{jk}(\tilde{\phi}) \tilde{\phi}^j_x \tilde{\phi}^k_x\,dx.
\end{align*}
Then we proceed in the following four steps. Moreover, we shall successively select constants $C_2= C_2(M), ..., C_9= C_9(M)$, {\it all these constants only depend on the value of $M$, but does not depend on $\delta$ and $\phi$.}

\noindent {\it{(1): Bounding the term $e_1$}}. Using the Cauchy-Schwarz inequality both with respect to space and time, we infer 
\begin{align*}
\big|e_1\big|&\leq \big\|\zeta\big\|_{L_x^2(\mathbb{T}^1)}\cdot \big\|\phi_t\big\|_{L_x^\infty L_t^2(I\times \mathbb{T}^1)}\cdot \big\|\psi_t\big\|_{L_t^2(I)}+ \big\|\zeta\big\|_{L_x^2(\mathbb{T}^1)}\cdot \big\|\phi_t\big\|_{L_x^\infty L_t^2(I\times \mathbb{T}^1)}\cdot \big\|\psi\big\|_{L_t^2(I)} \\
&\;\;\;\;\;\;\;\; + C\cdot \big\|\phi_t\big\|_{L_x^\infty L_t^2(I\times \mathbb{T}^1)}^2\cdot \big\|\zeta\big\|_{L_x^2(\mathbb{T}^1)}\\
&\leq C_2 \delta\cdot \big\|\zeta\big\|_{L_x^2(\mathbb{T}^1)}.
\end{align*}

\noindent {\it{(2): Uniform control over averaged $\phi_x$}}. The following formal computation is made rigorous by approximating $\phi$ by smooth solutions. Note that there is some $x_1\in \mathbb{T}^1$ such that 
\begin{align*}
 \int_I \psi(t)\big|\phi_x(t,x_1)\big|^2\,dt\leq \frac{1}{2\pi}E(0). 
\end{align*}
Then we conclude for arbitrary $\bar{x}\in \mathbb{T}^1$ that\footnote{Observe that the nonlinear term in \eqref{eq:wmdamped} is perpendicular to $\phi_x$.} 
\begin{align*}
 \int_I \psi(t)\big|\phi_x(t,\bar{x})\big|^2\,dt -  \int_I \psi(t)\big|\phi_x(t,x_1)\big|^2\,dt &= 2\int_I\int_{x_1}^{\bar{x}}\psi(t) \phi_x\cdot \phi_{xx}\,dt dx\\
 & = 2\int_I\int_{x_1}^{\bar{x}} \psi(t)\phi_x\cdot [\phi_{tt} + a(x)\phi_t],
\end{align*}
and integration by parts transforms the last term on the right into 
\begin{align*}
2\int_I\int_{x_1}^{\bar{x}} \psi(t)\phi_x\cdot [\phi_{tt} + a(x)\phi_t] = 2\int_I\int_{x_1}^{\bar{x}} [-\phi_x\cdot\psi'(t)\phi_{t} - \phi_{xt}\cdot\psi(t)\phi_t + a(x)\psi(t) \phi_t\cdot\phi_x]\,dx dt.
\end{align*}
Direct integration leads to 
\begin{align*}
 -2\int_I\int_{x_1}^{\bar{x}} \phi_{xt}\cdot\psi(t)\phi_t\,dx dt = -\int_I\psi(t)\big|\phi_t(t,\bar{x})\big|^2\,dt + \int_I\psi(t)\big|\phi_t(t, x_1)\big|^2\,dt. 
\end{align*}
Again taking advantage of \eqref{eq:smallphit} and the a priori control of the energy, we conclude that 
\begin{align*}
\Big|2\int_I\int_{x_1}^{\bar{x}} \psi(t)\phi_x\cdot [\phi_{tt} + a(x)\phi_t] \Big|\leq C_3\delta.
\end{align*}
Picking $\delta>0$ sufficiently small, we can then in any event infer the a priori bound  
\begin{align*}
 \int_I \psi(t)\big|\phi_x(t,\bar{x})\big|^2\,dt\leq \frac{2}{\pi}E(0)\;\; \forall \bar{x}\in \mathbb{T}^1. 
\end{align*}

\noindent {\it{(3): Uniform control over $\tilde{\phi}_{xx}$.}} From \eqref{eq:tildephizeta2}, we infer the relation 
\begin{align*}
-\int_{\mathbb{T}^1}\zeta_x\cdot \tilde{\phi}_x\,dx = e_1 - \sum_i\int_I \int_{\mathbb{T}^1}\psi(t)\zeta^i S^i_{jk}(\phi) \phi^j_x \phi^k_x\, dx dt.
\end{align*}
Applying the Cauchy-Schwarz inequality both in space and time to the double integral, and taking advantage of the preceding two steps, we deduce that 
\begin{align*}
\big|e_1\big| + \big|\sum_i\int_I \int_{\mathbb{T}^1}\psi(t)\zeta^i S^i_{jk}(\phi) \phi^j_x \phi^k_x\, dx dt\big|\leq C_4\cdot \big\|\zeta\big\|_{L_x^2}. 
\end{align*}
Therefore, for any  $\zeta\in C^\infty(\mathbb{T}^1; \mathbb{R}^N)$,
\begin{equation*}
    \big|\langle \zeta, \tilde \phi_{xx} \rangle_{L^2(\mathbb{T}^1)}\big|\leq C_4 \big\|\zeta\big\|_{L_x^2}.
\end{equation*}
This then implies the a priori bound 
\begin{align*}
\big\|\tilde{\phi}_{xx}\big\|_{L^2_x(\mathbb{T}^1)}\leq C_4.
\end{align*}

\noindent {\it{(4): Bounding  the term  $e_2$.}} Write 
\begin{align*}
e_2 &= \sum_i\int_{\mathbb{T}^1}\int_I\psi(t)\zeta^i [S^i_{jk}(\tilde{\phi}) - S^i_{jk}(\phi)] \partial_x \tilde{\phi}^j \partial_x\tilde{\phi}^k\,dx\\
&  \;\;\;\;\;\;\;\;\;\;\;\; + \sum_i\int_{\mathbb{T}^1}\int_I\psi(t)\zeta^i S^i_{jk}(\phi) \partial_x (\tilde{\phi}^j - \phi^j) \partial_x\tilde{\phi}^k\,dx \\
& \;\;\;\;\;\;\;\;\;\;\;\;\;\;\;\;\;\;\;\;\;\;\;\;+ \sum_i\int_{\mathbb{T}^1}\int_I\psi(t)\zeta^i S^i_{jk}(\phi) \partial_x (\tilde{\phi}^k - \phi^k) \partial_x  \phi^j\,dx.
\end{align*}
For the first integral on the right, recall the definition of $\tilde \phi$ given in \eqref{eq:tildephi} and the estimate \eqref{eq:smallphit}, we have   
\begin{equation}\label{eq:phitildephidifference}
\big|S^i_{jk}(\tilde{\phi}) - S^i_{jk}(\phi)\big|\leq C\big|\phi - \tilde{\phi}\big|\leq C\big\|\phi_t\big\|_{L_x^\infty L_t^2}\leq C\delta\;\; \forall (t, x)\in I\times \mathbb{T}^1.
\end{equation}
Thus, using {\it Step {(2)}} and the Cauchy-Schwarz inequality, we deduce that 
\begin{align*}
\Big| \sum_i\int_{\mathbb{T}^1}\int_I\psi(t)\zeta^i [S^i_{jk}(\tilde{\phi}) - S^i_{jk}(\phi)] \partial_x \tilde{\phi}^j \partial_x\tilde{\phi}^k\,dx\Big|\leq C\delta\big\|\zeta\big\|_{L_x^2}, 
\end{align*}
where the constant $C$ does not depend on $M, \delta$ and $\phi$.
\vspace{2mm}

For the second sum above, we use the following estimates 
\[
\big\|\sqrt{\psi}(t)\partial_x (\tilde{\phi} - \phi)\big\|_{L_{t,x}^2(I\times \mathbb{T}^1)}^2 = -\int_I \int_{\mathbb{T}^1}\psi(t)(\tilde{\phi}_{xx} - \phi_{xx})\cdot (\tilde{\phi} - \phi)\,dx dt
\]
and 
\begin{align*}
&\int_I \int_{\mathbb{T}^1}\psi(t)\phi_{xx}\cdot (\tilde{\phi} - \phi)\,dx dt\\& = \int_I \int_{\mathbb{T}^1}\psi(t)[\phi_{tt} - S_{jk}(\phi)\phi_x^j \phi_x^k +  S_{jk}(\phi)\phi_t^j \phi_t^k  + a(x)\phi_t]\cdot (\tilde{\phi} - \phi)\,dx dt
\end{align*}
Using integration by parts with respect to $t$ for the contribution of $\phi_{tt}$ and taking advantage of \eqref{eq:smallphit} as well as the Cauchy-Schwarz inequality and the bound \eqref{eq:phitildephidifference}, we deduce 
\begin{align*}
\big|\int_I \int_{\mathbb{T}^1}\psi(t)\phi_{xx}\cdot (\tilde{\phi} - \phi)\,dx dt\big|\leq C_5\delta. 
\end{align*}
The bound from {\it Step {(3)}} and \eqref{eq:phitildephidifference} on the other hand implies 
\begin{align*}
\big|\int_I \int_{\mathbb{T}^1}\psi(t)\tilde{\phi}_{xx}\cdot (\tilde{\phi} - \phi)\,dx dt\big|\leq C_5\delta. 
\end{align*}
Combining the preceding bounds we infer that 
\begin{align}
\big\|\sqrt{\psi}(t)\partial_x (\tilde{\phi} - \phi)\big\|_{L_{x, t}^2(I\times \mathbb{T}^1)}^2\leq C_6\delta. 
\end{align}

Combining the preceding bound with {\it Part {(2)}} and using the Cauchy-Schwarz inequality, we now infer that 
\begin{align*}
&\;\;\;\; \Big|2\sum_i\int_{\mathbb{T}^1}\int_I\psi(t)\zeta^i S^i_{jk}(\phi) \partial_x (\tilde{\phi}^j - \phi^j) \partial_x\tilde{\phi}^k\,dt dx\Big| \\
&\leq C \sum_{i=1}^N \sum_{j= 1}^N \sum_{k= 1}^N  \|  \zeta^i \psi \partial_x (\tilde{\phi}^j - \phi^j) \partial_x\tilde{\phi}^k  \|_{L^{1}_{x, t}(\mathbb{T}^1\times I)} \\
&\leq C  \sum_{i=1}^N \sum_{j= 1}^N \sum_{k= 1}^N  \|\zeta^i\|_{L^2_x} \|\sqrt{\psi}(t)\partial_x (\tilde{\phi} - \phi)\|_{L^2_{x, t}(\mathbb{T}^1\times I)} \|\partial_x \tilde{\phi}^k \|_{L^{\infty}_{x}(\mathbb{T}^1; L^2_t(I)) }  \\
& \leq C_7\sqrt{\delta}\big\|\zeta\big\|_{L_x^2(\mathbb{T}^1)}.
\end{align*}
Similarly, we also obtain the bound on the third term
\begin{align*}
\Big|2\sum_i\int_{\mathbb{T}^1}\int_I\psi(t)\zeta^i S^i_{jk}(\phi) \partial_x (\tilde{\phi}^k - \phi^k) \partial_x \phi^j\,dx\Big|\leq C_7\sqrt{\delta}\big\|\zeta\big\|_{L_x^2(\mathbb{T}^1)}.
\end{align*}

Therefore, we have now shown the bound 
\begin{align*}
\big|e_2\big|\leq C_8\sqrt{\delta}\big\|\zeta\big\|_{L_x^2(\mathbb{T}^1)}.
\end{align*}
Coming back to \eqref{eq:tildephizeta2}, and using {\it Part {(1)}}, {\it Part {(3)}}, and {\it Part {(4)}}, we now have shown that 
\begin{equation}\label{eq:tildephialmostharmonic}\begin{split}
&\big\|-\int_{\mathbb{T}^1}\zeta_x\cdot \tilde{\phi}_x\,dx +  \sum_i\int_{\mathbb{T}^1}\zeta^i S^i_{jk}(\tilde{\phi}) \tilde{\phi}^j_x \tilde{\phi}^k_x\,dx\big\|_{L_x^2}\leq C_9\sqrt{\delta}\big\|\zeta\big\|_{L_x^2(\mathbb{T}^1)},\\
&\big\| \tilde{\phi}\big\|_{H^2(\mathbb{T}^1)}\leq C_9,
\end{split}\end{equation}
where the constant $C_9= C_9(M)$ does not depend on $\phi$ or $\delta$. This also implies that
\begin{equation*}
    \| \tilde \phi_{xx}+ S_{jk}\tilde \phi_x^j \tilde \phi_x^k\|_{L^2_x(\mathbb{T}^1)}\leq C_9 \sqrt{\delta}.
\end{equation*}
This finishes the proof of  Lemma \ref{lem:extraction}. 
\end{proof}

\noindent {\bf Step 3. The existence of a limiting closed geodesic for a given initial state.} 
Let $\phi[0]\in \mathcal{H}(\mathbb{T}^1; \mathcal{N})$ be an initial state. We know from the previous discussion that for any $\delta$, there exists an interval $I$ with $|I|= 3$ such that
\begin{equation}
\big\|\phi_t\big\|_{L_x^\infty L_t^2(I\times \mathbb{T}^1)}<\delta.  \notag
\end{equation}
We shall now pass to a limit to extract an actual closed geodesics $\gamma: \mathbb{T}^1\longrightarrow \mathcal{N}$. For this we replace $\delta$ before by a sequence $\delta_n= 1/n,\,n\geq 1$. Let $\tilde{\phi}^{(n)}$ the corresponding functions obtained as before, on intervals $I^{(n)}, |I^{(n)}| = 3$.  
Thanks to the a prior bounds given by Lemma \ref{lem:extraction}, for every $n\in \mathbb{N}^*$, 
\begin{gather*}
\big\|\tilde{\phi}^{(n)}\big\|_{H^2_x(\mathbb{T}^1)}\leq C_1, \\
 \| \tilde \phi^{(n)}_{xx}+ S_{jk} (\tilde \phi_x^{(n)})^j (\tilde \phi_x^{(n)})^k\|_{L^2_x(\mathbb{T}^1)}\leq C_1 \sqrt{\delta_n}.
\end{gather*}
By Rellich's compactness theorem, a subsequence of the $\{\tilde{\phi}^{(n)}\}_{n\geq 1}$ converges in the sense of the $H^1$-metric to a limit $\gamma\in H^2(\mathbb{T}^1)$. Furthermore, by passing the limit one checks readily that for any $\zeta\in H^1(\mathbb{T}^1)$ we have 
\begin{align*}
\int_{\mathbb{T}^1}\big[-\zeta_x\cdot\gamma_{x} + \sum_i\zeta^i S^i_{jk}(\gamma) \gamma^j_x  \gamma^k_x\big]\,dx = 0.
\end{align*}
Using standard arguments one infers from this that in fact $\gamma$ is a closed geodesic, and it belongs to $C^\infty(\mathbb{T}^1; \mathcal{N})$. Given $\delta>0$, we can pick $n$ sufficiently large such that 
\[
\delta_n^{\frac14} + \big\|\tilde{\phi}^{(n)} - \gamma\big\|_{H^1(\mathbb{T}^1)}<\frac{\delta}{2}.
\]
Moreover, due to the bounds  from Lemma \ref{lem:extraction}  there exists $t_n\in I^{(n)}$ with the property that 
\[
\big\|\phi(t_n,\cdot) - \tilde{\phi}^{(n)}(\cdot)\big\|_{H^1} + \big\|\phi_t(t_n,\cdot)\big\|_{L_x^2}<C_{10}\delta_n^{\frac12}. 
\]
It follows that for $n$ sufficiently large, we have 
\begin{align*}
\big\|\phi(t_n,\cdot) - \gamma\big\|_{H^1} + \big\|\phi_t(t_n,\cdot)\big\|_{L_x^2}<\delta, 
\end{align*}
as desired.  Namely, $\gamma$ is the limiting closed geodesic that we seek.\\

\noindent {\bf Step 4. Uniform time for $\delta$-approximate closed geodesics.} 
Next, we demonstrate the following lemma. 
\begin{lemma}\label{lem:unif:decay:vaphi}
For any $\delta>0$  and any $M>0$, there exists some $\varepsilon>0$ such that, for any solution of the locally damped wave maps equation \eqref{eq:dwm:ori} satisfying 
\begin{gather*}
    E(0)\leq M,  \\
   \textrm{ for every $t\in [0, 32\pi], \phi[t] $ is not a  $\delta$-approximate closed geodesic, }
\end{gather*}
we have the inequality
\begin{equation*}
    E(0)- E(32\pi)\geq \varepsilon.
\end{equation*}
\end{lemma}
Indeed, to obtain the first property of Theorem \ref{thm:dynamics}, it suffices to select 
\begin{equation*}
    T_c(M, \delta):= 32\pi \left(\frac{M}{\varepsilon}+ 2\right).
\end{equation*}
We claim that for any initial state with energy smaller than $M$, there is some $t\in [0, T_c]$ such that $\phi[t]$ is a $\delta$-approximate closed geodesic. Otherwise, due to Lemma \ref{lem:unif:decay:vaphi}, one has 
\begin{equation*}
    E(32\pi k)- E(32\pi(k+1))\geq \varepsilon, \; \forall k= 0, 1,..., \left[\frac{M}{\varepsilon}\right]+1. 
\end{equation*}
This contradicts the assumption that the initial energy is smaller than $M$. 

\begin{proof}[Proof of Lemma \ref{lem:unif:decay:vaphi}]
We define $T= 32\pi$ and $I= (16\pi- 3/2, 16\pi+3/2)$.
We perform the proof by a contradiction argument. Suppose that for some $\delta>0$, one cannot find such an $\varepsilon$. First, due to Lemma \ref{prop:uniformsmallness}, we select a sequence $\tilde \delta^{(n)}= \tilde \delta^{(n)}(1/n, 3, M)$.
Then, there is a sequence of initial states $\phi^{(n)}[0]$ with energy smaller than $M$  such that 
\begin{gather*}
  \textrm{ for every $t\in [0, 32\pi], \phi^{(n)}[t] $ is not a  $\delta$-approximate closed geodesic, and} \\
E( \phi^{(n)}[0])-  E(\phi^{(n)}[T])< 2 \tilde \delta^{(n)}.
\end{gather*}
This implies, due to \eqref{eq:energydecrease}, that 
\[
\int_0^T \int_{\mathbb{T}^1}a(x)\big|\phi_t^{(n)}(t, x)\big|^2\,dx dt < \tilde \delta^{(n)}. 
\]
Therefore, by Lemma \ref{prop:uniformsmallness}, 
\begin{equation}
 \big\|\phi_t^{(n)}\big\|_{L_x^\infty L_t^2(I\times \mathbb{T}^1)}<\frac{1}{n}.
\end{equation}
In analogy to  Step 3, let $\tilde{\phi}^{(n)}$ be the corresponding functions obtained as before, on the fixed interval $I$.  These functions have uniform $H^2(\mathbb{T}^1)$ bound and strongly converge in the $H^1$-topology to a closed geodesic $\gamma$. For the given $\delta$, one can find some $n$ sufficiently large and some $t_n\in I$ such that 
\begin{align*}
\big\|\phi^{(n)}(t_n,\cdot) - \gamma\big\|_{H^1} + \big\|\phi_t^{(n)}(t_n,\cdot)\big\|_{L^2}<\delta.
\end{align*}
This is in contradiction with our assumption. 

This completes the proof of Theorem \ref{thm:dynamics}.
\end{proof}

\section{Exact controllability around closed geodesics}\label{Sec:4}
This section is devoted to the proof of Proposition \ref{pro:controlaroundgeodesic}, which
constitutes the second intermediate result for the proof of Theorem \ref{thm1}. To control the evolution of $\phi$ near closed geodesics, we shall pass from the {\it extrinsic} equations used in the previous stage to an {\it intrinsic} viewpoint.  As outlined in Section \ref{sec:strategy1}, we propose a three-step argument for this proposition. Accordingly, the whole section is divided into three parts, each addressing one of these steps.
\begin{enumerate}
    \item[Step 1] \; Show that Proposition \ref{pro:controlaroundgeodesic} is equivalent to Proposition \ref{prop:congeo:1st}; see Section \ref{Sec:cont:1}.
    \item[Step 2]  \; Demonstrate that Proposition \ref{prop:congeo:1st} can be derived from  Proposition \ref{prop:congeo:2nd}; see Section \ref{Sec:cont:2}.
    \item[Step 3]  \;  Prove  Proposition  \ref{prop:congeo:2nd};  see Section \ref{subsec:con:3}.
\end{enumerate}

\subsection{On the reduction to a semilinear equation without geometric constraint}\label{Sec:cont:1}
From now on, we fix a closed geodesic $
\gamma: \mathbb{T}^1 \to \mathcal{N}$.  In the end, the choice of $\delta_1$ can be made independent of closed geodesics $\gamma$ with energy smaller than $M$. Actually, relying on the geodesic equation,  the $H^3$-norms of these closed geodesics are uniformly bounded by a constant depending on $M$.

In this first step, we transfer the geometric equation with control $f\in \mathbb{R}^N$ with geometric constraints around the closed geodesic $\gamma$ into a coupled semilinear equation \eqref{eq:visystem11} with control $\{\alpha^p \textrm{ without any constraint}: p= 1,2,..., R\}$.

\subsubsection{The moving frame method under intrinsic coordinates}

Given an initial state $\phi[0]$ that is close to $(\gamma(\cdot), 0)$ in the $\mathcal{H}$ topology. Thanks to the well-posedness result  Lemma \ref{lem-conti-dep-inh}, the solution $\phi[t]$ keeping closing to $(\gamma(\cdot), 0)$ in the $\mathcal{H}$ topology provided that the $L^2_{t, x}$-norm of the control force is small. Thus, in the following, we shall assume that 
\begin{equation*}
  \|\phi(t, \cdot)- \gamma(\cdot)\|_{H^1(\mathbb{T}^1)},  \|\phi(t, \cdot)\|_{L^2(\mathbb{T}^1)}\ll 1 \;\; \forall t\in [0, 64\pi].
\end{equation*}

We start by stating the following result, which provides a smooth moving orthonormal frame around $\gamma$. Its proof is standard and relies on parallel transport and a holonomy correction.  Note that the manifold $\mathcal{N}$ is orientable. For readers' convenience, we put the proofs of Lemmas \ref{lem:movingframe} in Appendix \hyperref[setting(Aex3)]{A.3}. 
\begin{lemma}\label{lem:movingframe}
 There exist $R$ smooth vector valued functions
\[
{\bf f}_p: \mathbb{T}^1 \longrightarrow  T_{\gamma(x)}\mathcal{N}, \quad p=1,\dots,R,
\]
such that for every $x \in \mathbb{T}^1$ the set $\{{\bf f}_1(x),... ,{\bf f}_R(x)\}$ is an orthonormal basis of $T_{\gamma(x)}\mathcal{N}$, and
${\bf f}_1(x) = \gamma_x(x)/|\gamma_x(x)|$.
\end{lemma}
We also have the following lemma, which is a straightforward  consequence of the implicit function theorem and the local graph representation.  
\begin{lemma}\label{lem:implicitcoord}
There exist $\delta>0$ and a smooth mapping 
\[
{\bf n}: \mathbb{T}^1\times B_{\delta}(0)\subset \mathbb{T}^1\times\mathbb{R}^R\longrightarrow N\mathcal{N}
\]
satisfying
\[
 \|{\bf n}(x; {\bf w})\|\leq C\|{\bf{w}}\|^2 \quad \text{and} \quad  {\bf n}(x; {\bf w})\in N_{\gamma(x)}\mathcal{N}, \quad \forall (x, {\bf w})\in \mathbb{T}^1\times B_{\delta}(0),
\]
such that for any $x\in \mathbb{T}^1$ and any $\phi\in \mathcal{N}\cap B_{\delta}(\gamma(x))$, there is a unique vector 
\[
{\bf{w}} = (w^1, w^2,\ldots, w^R)^T\in \mathbb{R}^R, \quad \|{\bf{w}}\|<2\delta,
\]
for which 
\begin{equation}\label{eq:localrep1}
\phi = \gamma(x) + \sum_{p=1}^R w^p {\bf f}_p(x) +  {\bf n}(x; {\bf{w}}).
\end{equation}

Moreover, there exists a smooth mapping 
\[
\Phi: \Bigl\{ (x,\phi): x\in \mathbb{T}^1,  \phi\in \mathcal{N}\cap B_\delta(\gamma(x))\Bigr\}\longrightarrow  \mathbb{T}^1\times \mathbb{R}^R,
\]
such that for any $x\in \mathbb{T}^1$ and any $\phi\in \mathcal{N}\cap B_{\delta}(\gamma(x))$, the unique vector ${\bf w}$ is given by 
\begin{equation}\label{eq:localrep2}
  (x,  {\bf w})= \Phi (x; \phi).
\end{equation}
\end{lemma}

Fix  $x\in \mathbb{T}^1$ and a point $\phi\in \mathcal{N}$ that is close to $\gamma(x)$,  we can introduce the functions ${\bf w}$  via \eqref{eq:localrep1}--\eqref{eq:localrep2}
 with $
(x,  {\bf w})= \Phi (x, \phi)$. 
Define the mapping
\begin{gather}
    \Psi: \mathbb{T}^1 \times B_\delta(0) \longrightarrow  \mathcal{N}\cap\Bigl(\bigcup_{x\in\mathbb{T}^1} B_\delta(\gamma(x))\Bigr) \notag\\
 \label{eq:cons:Psi1}
\Psi(x,{\bf w}): = \gamma(x) + \sum_{p=1}^R w^p {\bf f}_p(x) + {\bf n}(x,{\bf w}).
\end{gather}
Next, we define the tangent frames at $\phi$ using map $\Psi$ restricted to $\{x\}\times \mathbb{R}^{R}$. 
Define $\Psi_x: B_{\delta}(0)\rightarrow \mathcal{N}\cap B_{\delta}(\gamma(x))$
as this restriction map.  Then we define the smooth mapping
\begin{equation*}
    \tilde G_p(x; {\bf w}):= (\Psi_x)_{* {\bf w}} \left( \partial_{w^p}\right)= (\Psi_x)_{* {\bf w}} \left({\bf f}_p(x)\right) \in T_{\phi}\mathcal{N} \;   \; \forall p=1, 2,..., R.
\end{equation*}
For each ${\bf w}$ small enough, the vectors $ \{G_p(x; {\bf w}): p= 1,..., R\}$ form a basis, not necessarily orthonormal, of the tangential space $T_{\phi}\mathcal{N}$. 

 Because the ${\bf w}$-domain is contractible and because the tangent bundle along the geodesic $\gamma$ is trivial over $\mathbb{T}^1$, we can apply the smooth Gram–Schmidt process on  $ \{\tilde G_p(x; {\bf w}): p= 1,..., R\}$,  without meeting any topological obstruction, to obtain an orthonormal basis $\{ G_p(x;{\bf w}): p=1,..., R\}$. Moreover, direct calculation yields 
\begin{gather*}
 G_p(x; {\bf w})= {\bf f}_p(x)+ g_p(x; {\bf w})  
\end{gather*}
with 
\begin{equation*}
    |g_p(x; {\bf w})|\leq C |{\bf w}|,
\end{equation*}
for all $p=1, 2,..., R$.
Hence,  the {\it tangent frames} at $\phi$, which is close to $\gamma(x)$, can be defiend as: 
\begin{equation}
     F_p(x; \phi)=  G_p(x; {\bf w}) \; \; \forall i=1, 2,..., R.
\end{equation}

\vspace{2mm}
 In conclusion, for every $(t, x)$, the solution $\phi(t, x)$ is close to $\gamma(x)$, and
 \begin{itemize}
     \item[\tiny$\bullet$]  the function $\phi(t, x)$  is expressed by $\Psi (x; {\bf w}(t, x))$; 
     \item[\tiny$\bullet$] the tangent frames at $\phi(t, x)$ are given by $\{G_p(x; {\bf w}(t, x)): p=1,...R\}$. 
 \end{itemize}
 Moreover, for any fixed $x_0\in \mathbb{T}^1$, 
 \begin{equation*}
 \phi_t(t, x_0)= 
 \sum_{p= 1}^R \left({\bf f}_p(x_0)+  \partial_{w^p}{\bf n}(x_0; {\bf w}(t, x_0)) \right)\partial_t w^p(t, x_0).
 \end{equation*}
This implies that any $\phi_t(t, x_0)\in T_{\phi(t, x_0)} \mathcal{N}$ determines a unique sequence $\{\partial_t w^p(t, x_0)\}_{p= 1}^R$.
In fact, since we will transform the equation on $\phi$ to an equation on ${\bf w}$, this will be used to fix the initial state ${\bf w}_t(0, \cdot)$.

\subsubsection{The characterization of wave maps}
Now we are in a position to perform the reduction procedure.  Now, the controlled wave maps can be written as 
\eqref{eq:wminhomo}
 \begin{equation}\label{eq:wmcontrol1}
 \Box \phi -  \Pi(\phi)\left(\partial_{\nu}\phi, \partial^{\nu}\phi\right)) = \chi_{\omega}\sum_{p=1}^R\alpha^p(t, x) F_p(x; \phi),
 \end{equation}
 where $\{\alpha^1, \alpha^2,\ldots, \alpha^R\}$ is the set of controls supported in $\omega$.

 Next, we shall take advantage of the local coordinates to transform the equation on $\phi$ into an equation on ${\bf w}$. Indeed, it suffices to  plug the preceding introduced presentation of $\phi,\{F_p\}_{p=1}^R$ into the equation, and one immediately observes that each term  becomes a function of ${\bf w}$.  
Notice that Equation \eqref{eq:wmcontrol1} is equivalent to that $$\textrm{LHS of \eqref{eq:wmcontrol1}}- \textrm{RHS of \eqref{eq:wmcontrol1}}= 0.$$

Since we automatically have the projection of  $\left(\textrm{LHS of \eqref{eq:wmcontrol1}}- \textrm{RHS of \eqref{eq:wmcontrol1}}\right)$ on the normal space $N_{\phi}\mathcal{N}$ are zero, showing $\textrm{LHS} = \textrm{RHS}$ is equivalent to showing the projection of $(\textrm{LHS}- \textrm{RHS})$ on the tangent frames $\{F_p(x; \phi)\}_{p= 1}^R$ are zero:
\begin{equation}\label{eq:projonFi}
   \langle \Box \phi, F_p(x; \phi) \rangle= \chi_{\omega}  \alpha^p(t, x), \; \forall p= 1,2,..., R.
\end{equation}

\vspace{2mm}
Using the representation 
\begin{equation}\label{eq:phirep}
\phi(t, x) = \gamma (x) + \sum_{p=1}^R w^p(t,x) {\bf f}_p(x) + {\bf n}(x; {\bf w}(t, x)),\,
\end{equation}
and keeping in mind that 
\begin{equation*}\label{eq:tildef}
F_p(x; \phi)=   G_p(x; {\bf w})= {\bf f}_p(x)+ g_p(x; {\bf w}),
\end{equation*}
where 
\[
\big|g_p(x, {\bf{w}})\big|\lesssim \|{\bf{w}}\| 
\]
are smooth functions of their arguments. We now reformulate \eqref{eq:projonFi} by using \eqref{eq:phirep}. 
\begin{equation}
    \left\langle \Box\left(\gamma (x) + \sum_{i=1}^R w^i(t,x) {\bf f}_i(x) + {\bf n}\right),  G_p(x; {\bf w}) \right\rangle= \chi_{\omega} \alpha^p(t, x),
\end{equation}
for every $p=1,2,..., R$. 
\vspace{3mm}

Taking advantage of the fact that $\gamma$ is a geodesic, namely, 
\begin{equation}
      \langle \Box(\gamma(x)) ,  G_p(x; 0) \rangle=  \langle \Delta(\gamma(x)) ,  G_p(x; 0) \rangle= 0,  \notag
\end{equation}
the preceding equation becomes
\begin{equation}\label{eq:undereachdir}
     \sum_{i=1}^R \left\langle \Box(w^i(t,x) {\bf f}_i(x) ),  G_p(x; {\bf w}) \right\rangle+   \left\langle \Box {\bf n},  G_p(x; {\bf w}) \right\rangle+ \left\langle \Delta  \gamma(x), g_p(x; {\bf w})\right\rangle=  \chi_{\omega} \alpha^p(t, x),
\end{equation}
for every $p= 1, 2,..., R$.
Successively, we obtain
\begin{align*}
\partial_t ({\bf n} (x; {\bf w}(t, x)))&= \sum_{k= 1}^R \partial_{w^k} {\bf n}(x; {\bf w}) \partial_t w^k, \\
\partial_x ( {\bf n} (x; {\bf w}(t, x)))&= \sum_{k= 1}^R \partial_{w^k} {\bf n}(x; {\bf w}) \partial_x w^k+ \partial_x {\bf n}(x; {\bf w}), 
\end{align*}
as well as the second order
\begin{align*}
\Box(w^i(t,x) {\bf f}_i(x))&= (\Box w^i){\bf f}_i+ 2 (\partial_x w^i)(\partial_x {\bf f}_i)+ w^i (\partial_x^2 {\bf f}_i), \\
\partial_t^2 ({\bf n}(x; {\bf w}(t, x)))&= \sum_{k= 1}^R \partial_{w^k} {\bf n}(x; {\bf w}) \partial_t^2 w^k+ \sum_{k= 1}^R \sum_{r= 1}^R\partial_{w^r}\partial_{w^k} {\bf n}(x; {\bf w})  \partial_t w^k \partial_t w^r, \\
\partial_x^2 ({\bf n}(x; {\bf w}(t, x)))&= \sum_{k= 1}^R \partial_{w^k} {\bf n}(x; {\bf w}) \partial_x^2 w^k+ \sum_{\beta= 0}^1 \sum_{k= 1}^R \sum_{r= 1}^R\partial_{w^r}\partial_{w^k} {\bf n}(x; {\bf w})  \partial_x w^k \partial_x w^r \\
&\;\;\;\;\;\;\;\;\;\; + 2\sum_{k= 1}^R \partial_x \partial_{w^k} {\bf n}(x; {\bf w}) \partial_x w^k+\partial_x^2 {\bf n}(x; {\bf w}), \\
\Box ({\bf n}(x; {\bf w}(t, x)))&= \sum_{k= 1}^R \partial_{w^k} {\bf n}(x; {\bf w}) \Box w^k+ \sum_{\beta= 0}^1 \sum_{k= 1}^R \sum_{r= 1}^R\partial_{w^r}\partial_{w^k} {\bf n}(x; {\bf w})  \partial_{\beta} w^k \partial^{\beta} w^r\\
&\;\;\;\;\;\;\;\;\;\; + 2\sum_{k= 1}^R \partial_x \partial_{w^k} {\bf n}(x; {\bf w}) \partial_x w^k+\partial_x^2 {\bf n}(x; {\bf w}).
\end{align*}

By plugging the preceding formulas into equation \eqref{eq:undereachdir}, it  becomes 
\begin{align*}
   & \Box w^p+ \sum_{i= 1}^R a_{p, i}(x, {\bf w}) \Box w^i+  \sum_{\beta= 0}^1 \sum_{i, j= 1}^R b_{p, i,j}(x, {\bf w}) \partial_{\beta} w^i\partial^{\beta} w^j  \;\;\;\; \;\;\;\;  \;\;\;\; \;\;\;\; \notag \\
 &  \;\;\;\; \;\;\;\;  \;\;\;\; \;\;\;\;   +\sum_{i= i}^R c_{p, i}(x, {\bf w}) \partial_x w^i+ \sum_{i= i}^R d_{p, i}(x, {\bf w})  w^i+ r_p(x, {\bf w})= \chi_{\omega} \alpha^p(t, x)
\end{align*}
for every $p= 1,2,..., R$, and where the functions 
\begin{equation*}
  |a_{p, i}(x, {\bf w})|+    |r_{p}(x, {\bf w})|\lesssim \|\bf{w}\|.
\end{equation*}

\vspace{2mm}
Notice that this coupled equation on ${\bf w}$ is quasilinear. 
By multiplying the inverse matrix of $(\delta_{p, i}+ a_{p, i}(x, {\bf w}))_{1\leq p, i\leq R}$ to the preceding equation we obtain  a coupled system 
\begin{equation}\label{eq:visystem11}
\begin{split}
    \Box w^p = \sum_{\beta= 0}^1 \sum_{j,k= 1}^R
A_{jk}^{p}(x; {\bf{w}})\partial_{\beta} w^j \partial^{\beta} w^k +  \sum_{j= 1}^R B_{j}^{p}(x; {\bf{w}}) \partial_{x}w^j+ \sum_{j= 1}^R  C_j^{p}(x; {\bf w}) w^j\\
+ \chi_{\omega} \alpha^p(t, x) + \chi_{\omega}\sum_{j= 1}^R  D_j^p(x; {\bf{w}}) \alpha^j(t, x),
\end{split}
\end{equation}
for every $p = 1, 2,... , R$, 
where  the functions $A_{jk}^p, B_{j}^{p}, C_j^p$ and $D_j^p$ are smooth.  These functions (matrices) satsify the following Taylor expansion
\begin{align}
\begin{cases} \label{eq:indexcond1}  
       A_{jk}^p(x; {\bf w})=   A_{0,jk}^{p}(x)+  A_{r,jk}^{p}(x; {\bf w}), 
   \\
    B_{j}^{p}(x; {\bf w})=    B_{0, j}^{p}(x)+   B_{r, j}^{p}(x; {\bf w}),\\
     C_{j}^{p}(x; {\bf w})=    C_{0, j}^{p}(x)+   C_{r, j}^{p}(x; {\bf w}),
     \end{cases}
\end{align}
and
\begin{gather}
\begin{cases} \label{eq:indexcond2}
|A^p_{0, jk}(x)|\lesssim 1\\
        |A_{r, jk}^{p}(x; {\bf w})|+ |  B_{r, j}^{p}(x; {\bf w})|+ |  C_{r, j}^{p}(x; {\bf w})|+  |  D_{j}^{p}(x; {\bf w})|\lesssim |{\bf w}|, \\
       |A_{r, jk}^{p}(x; {\bf w}_1)- A_{r, jk}^{p}(x; {\bf w}_2)|+ 
     |B_{r, j}^{p}(x; {\bf w}_1)- B_{r, j}^{p}(x; {\bf w}_2)|\lesssim |{\bf w}_1- {\bf w}_2|, \\
      |C_{r, j}^{p}(x; {\bf w}_1)- C_{r, j}^{p}(x; {\bf w}_2)|+ 
        |D_{j}^{p}(x; {\bf w}_1)- D_{j}^{p}(x; {\bf w}_2)|\lesssim |{\bf w}_1- {\bf w}_2|,
        \end{cases}
\end{gather}
uniformly for every $p, j, k=1,..., R$, for every $x\in \mathbb{T}^1$, and for every ${\bf w}, {\bf w}_1$ and ${\bf w}_2$ in $\mathbb{R}^R$ satisfying  $|{\bf w}|, |{\bf w}_1|, |{\bf w}_2|\leq 1$. 
Since now the state ${\bf w}$ belongs to $\mathbb{R}^R$, in the rest of this section, we slightly abuse notation and use $\mathcal{H}$ to denote
\begin{gather*}
    \mathcal{H}:= \left\{(f, g)\in  H^1(\mathbb{T}^1; \mathbb{R}^R)\times L^2(\mathbb{T}^1; \mathbb{R}^R)\right\} \; \textrm{ with } \\
    \|(f, g)\|_{\mathcal{H}}^2:= \|f\|_{H^1(\mathbb{T}^1)}^2+ \|g\|_{L^2(\mathbb{T}^1)}^2.
\end{gather*}

{\it In conclusion,  Proposition \ref{pro:controlaroundgeodesic} is equivalent to the following control property. }
 \begin{prop}\label{prop:congeo:1st}
Let $T= 64\pi$. The system \eqref{eq:visystem11} with conditions \eqref{eq:indexcond1}--\eqref{eq:indexcond2} is locally exactly controllable in time $T$. More precisely, there is $\delta_2>0$  such that given a pair of data $ ({\bf w}_0, {\bf w}_{0t}), ({\bf w}_1, {\bf w}_{1t})\in H^1(\mathbb{T}^1; \mathbb{R}^R)\times L^2(\mathbb{T}^1; \mathbb{R}^R)$ satisfying
 \[
 \|({\bf w}_0, {\bf w}_{0t})\|_{\mathcal{H}}+  \|({\bf w}_1, {\bf w}_{1t})\|_{\mathcal{H}}\leq \delta_2
 \]
 there is a control $\alpha= (\alpha^1, ..., \alpha^p)\in C([0, T]; L^2(\mathbb{T}^1; \mathbb{R}^R))$ satisfying 
 \[
\big\|\alpha\big\|_{L_t^{\infty} L_x^2([0,T]\times \mathbb{T}^1)}\lesssim \|({\bf w}_0, {\bf w}_{0t})\|_{\mathcal{H}}+  \|({\bf w}_1, {\bf w}_{1t})\|_{\mathcal{H}}, 
 \]
 such that the flow associated to \eqref{eq:visystem11} carries the intial data $({\bf w}_0, {\bf w}_{0t})$ at time $t= 1$ into $({\bf w}_1, {\bf w}_{1t})$ at time $t= T$. 
\end{prop}

\subsection{On the reduction to a linear control problem}\label{Sec:cont:2}
As illustrated in the previous step, Proposition \ref{pro:controlaroundgeodesic} is equivalent to Proposition \ref{prop:congeo:1st}, which addresses the local controllability of wave equation \eqref{eq:visystem11} on ${\bf w}$ without  geometric constraint. The aim of this section is to  further show that Proposition \ref{prop:congeo:1st} can be derived from a linear controllability result, namely,  Proposition \ref{prop:congeo:2nd}. 

To construct this exact control trajectory that connects $({\bf w}_0, {\bf w}_{0t})$ and  $({\bf w}_1, {\bf w}_{1t})$, it suffices to find a solution satisfying 
\begin{equation*}
    {\bf w}[0]= ({\bf w}_0, {\bf w}_{0t})\xrightarrow{ \textrm{  control } \alpha|_{t\in (0, 32\pi)} }  {\bf w}[32\pi]= (0, 0) \xrightarrow{ \textrm{  control } \alpha|_{t\in (32\pi, 64\pi)} }  {\bf w}[64\pi]= ({\bf w}_1, {\bf w}_{1t}).
\end{equation*}
The first step is the null controllability of equation \eqref{eq:visystem11}, while the second step is a direct consequence of the null controllability of the reserve system (which has a similar structure). Indeed,  suppose that  ${\bf w}[t]|_{t\in (0, 2\pi)}$ is a solution of the equation with control ${\bf \alpha}(t)|_{t\in (0, 2\pi)}$, then 
\begin{align*}
    \tilde {\bf w}(t)&:= {\bf w}(32\pi- t)\;\; \forall t\in (0, 32\pi), \\
     \tilde \alpha(t)&:=  \alpha (32\pi- t)\;\; \forall t\in (0, 32\pi),
\end{align*}
satisfy the same equation. Consequently, it suffices to prove the local controllability of \eqref{eq:visystem11} on the time interval $(0, 32\pi)$. 
\vspace{3mm}

By linearizing the system around $0$ we obtain
\begin{equation}\label{eq:visystem11li}
\begin{split}
    \Box w^p- \sum_{1\leq j\leq R} B_{0, j}^{p}(x) \partial_{x}w^j- \sum_{1\leq j\leq R} C_{0, j}^{p}(x) w^j= 
\chi_{\omega} \alpha^p, 
\end{split}
\end{equation}
for every $p = 1, 2,\ldots R$,
where $\alpha$ supp $[0, T]\times \omega$ is the control that we are free to choose.

In the rest of this section, we show that the following property leads to Proposition \ref{prop:congeo:1st}. While the proof of this property is postponed to Sections \ref{subsec:HUM}--\ref{sec:cont:4th}.

\begin{prop} \label{prop:congeo:2nd}
  Let $T= 32\pi$. The linear coupled wave equation \eqref{eq:visystem11li} is exactly controllable in time $T$ in the sense that, there exists an explicit constant $S$ such that for any initial state $({\bf w}_0, {\bf w}_{0t})\in H^1(\mathbb{T}^1; \mathbb{R}^R)\times L^2(\mathbb{T}^1; \mathbb{R}^R)$, one can construct a control $\alpha\in C([0, T]; L^2(\mathbb{T}^1; \mathbb{R}^R))$ such that the unique solution of  
   \begin{equation*}
\begin{cases}
    \Box w^p- \sum_{1\leq j\leq R} B_{0, j}^{p}(x) \partial_{x}w^j- \sum_{1\leq j\leq R} C_{0, j}^{p}(x) w^j= 
\chi_{\omega} \alpha^p \;\; \forall p = 1, 2,\ldots R, \\
{\bf w}[0]= ({\bf w}_0, {\bf w}_{0t}),
\end{cases}
\end{equation*}
satisfies ${\bf w}[T]= (0, 0)$ and 
\begin{equation*}
        \|\alpha\|_{L^{\infty}(0, T; L^2(\mathbb{T}^1))}+ 
\|{\bf w}\|_{\mathcal{W}_T} \leq S  \|({\bf w}_0, {\bf w}_{0t})\|_{\mathcal{H}}.
    \end{equation*}
    \end{prop}

\subsubsection{Well-posedness of equation \eqref{eq:visystem11} }\label{subsubsec:421}

The linear equation \eqref{eq:visystem11li} can be  written as follows:
\begin{gather}\label{eq:linsimiwavesim}
    \mathcal{L}_{{\rm in}} {\bf w}= \chi_{\omega} \alpha, 
\end{gather}
where 
\begin{equation}
\mathcal{L}_{{\rm in}} {\bf w}:= 
\begin{pmatrix}
     \Box w^1- \sum_j B_{0, j}^{1}(x) \partial_{x}w^j- \sum_j C_{0, j}^{1}(x) w^j\\
     \Box w^2- \sum_j B_{0, j}^{2}(x) \partial_{x}w^j- \sum_j C_{0, j}^{2}(x) w^j\\
     ...\;\;\;\;\;\;\; ...\;\;\;\;\;\;\; ...\\
      \Box w^R- \sum_j B_{0, j}^{R}(x) \partial_{x}w^j- \sum_j C_{0, j}^{R}(x) w^j
\end{pmatrix}, \;\;  
\alpha= 
\begin{pmatrix}
    \alpha^1\\
    \alpha^2\\
    ...\\
    \alpha^R
\end{pmatrix}.    
\end{equation}
We also write the matrices ${\bf B}_0(x), {\bf C}_0(x): \mathbb{T}^1\rightarrow M_{R\times R}$ as
\begin{equation}\label{eq:def:BC}
    ({\bf B}_0(x))_{p, j}:= B^p_{0, j}(x) \; \textrm{ and }  \;    ({\bf C}_0 (x))_{p, j}:= C^p_{0, j}(x).
\end{equation}

Using this notations the semilinear wave equation \eqref{eq:visystem11} becomes 
\begin{gather}\label{eq:simiwavesim}
    \mathcal{L}_{{\rm in}} {\bf w}- \mathcal{K}(x; {\bf w})= \chi_{\omega}  \alpha+ \chi_{\omega} \mathcal{K}_1(x; {\bf w}) \alpha, 
\end{gather}
where 
\begin{equation}
\mathcal{K}(x;  {\bf w}):= 
\begin{pmatrix}
   \sum_{j,k}\sum_{\beta} 
A_{jk}^{1}(x; {\bf{w}})\partial_{\beta} w^j \partial^{\beta} w^k +  \sum_j B_{r, j}^{1}(x; {\bf{w}}) \partial_{x}w^j+ \sum_j C_{r, j}^{1}(x; {\bf w}) w^j \\
  \sum_{j,k}\sum_{\beta} 
A_{jk}^{2}(x; {\bf{w}})\partial_{\beta} w^j \partial^{\beta} w^k +  \sum_j B_{r, j}^{2}(x; {\bf{w}}) \partial_{x}w^j+ \sum_j C_{r, j}^{2}(x; {\bf w}) w^j\\
...\;\;\;\;\;\;\; ...\;\;\;\;\;\;\; ...\\
  \sum_{j,k}\sum_{\beta} 
A_{jk}^{R}(x; {\bf{w}})\partial_{\beta} w^j \partial^{\beta} w^k +  \sum_j B_{r, j}^{R}(x; {\bf{w}}) \partial_{x}w^j+ \sum_j C_{r, j}^{R}(x; {\bf w}) w^j
\end{pmatrix},    
\end{equation}
\begin{equation}
\mathcal{K}_1(x; {\bf w}) \alpha:= 
\begin{pmatrix}
 \sum_j D_j^1(x; {\bf{w}})\alpha^j(t, x)\\
 \sum_j D_j^2(x; {\bf{w}})\alpha^j(t, x)\\
...\;\;\;\;\;\;\; ...\;\;\;\;\;\;\; ...\\
 \sum_j D_j^R(x; {\bf{w}})\alpha^j(t, x)
\end{pmatrix},   
\end{equation}
\vspace{1mm}
with the conventions $j, k\in \{1,2,..., R\}$ and $\beta\in \{0, 1\}$.
\vspace{2mm}

Let $T>0$. Recall from Lemma \ref{lem:linea:well:gen} the existence of $\mathcal{C}_T>0$ such that for any  initial state $({\bf w}_0, {\bf w}_{0t})\in \mathcal{H}$, 
and any $\alpha: [0,  T]\times \mathbb{T}^1\rightarrow \mathbb{R}^{R}$ in $L^2_{t,x}(D_T)$,  the unique solution of 
\begin{equation}
        \mathcal{L}_{{\rm in}} {\bf w}=  \alpha \textrm{ with }
        {\bf w}[0]= ({\bf w}_0, {\bf w}_{0t}), \label{eq:li:w:1}
\end{equation}
 satisfies 
  \begin{gather} 
  \|{\bf w}\|_{\mathcal{W}_T}\leq \mathcal{C}_T \left( \|({\bf w}_0, {\bf w}_{0t})\|_{\mathcal{H}}+ \|\alpha\|_{L^2_{t, x}(D_T)}\right). \label{eq:li:w:2}
\end{gather}

Notice that the nonlinear terms $\mathcal{K}(x; {\bf w})$ and $\mathcal{K}_1(x; {\bf w}) \alpha$ satisfy 
\begin{gather}
    \|\mathcal{K}(x; {\bf w})\|_{L^2_{t, x}(D_T)}\leq C_N \|{\bf w}\|_{\mathcal{W}_T}^2, \label{K:p:1} \\
\|\mathcal{K}(x; {\bf w}_1)- \mathcal{K}(x; {\bf w}_2)\|_{L^2_{t, x}(D_T)}\leq C_N \|{\bf w}_1- {\bf w}_2\|_{\mathcal{W}_T} ( \|{\bf w}_1\|_{\mathcal{W}_T}+  \|{\bf w}_2\|_{\mathcal{W}_T}), \\
     \|\mathcal{K}_1(x; {\bf w})\alpha\|_{L^2_{t, x}(D_T)}\leq C_N \|{\bf w}\|_{L^{\infty}(D_T)} \|\alpha\|_{L^2_{t, x}(D_T)}, \\
       \|\mathcal{K}_1(x; {\bf w}_1)\alpha_1- \mathcal{K}_1(x; {\bf w}_2)\alpha_2\|_{L^2_{t, x}(D_T)}\leq C_N \|{\bf w}_1- {\bf w}_2\|_{L^{\infty}(D_T)} \|\alpha_1\|_{L^2_{t, x}(D_T)} \notag\\
       \;\;\;\;\;\;\;\;\;\;\;\;\;\;\;\;\;\;\;\;\;\;\;\;\;\;\;\;\;\;\;\;\;\;\;\;\;\;\;\;\;\;\;\;\;\;\;\;\;\;\;\;\;\;\;\;\;\;\;\; +  C_N\|{\bf w}_2\|_{L^{\infty}(D_T)} \|\alpha_1- \alpha_2\|_{L^2_{t, x}(D_T)}, \label{K:p:2}
\end{gather}
provided that
\begin{equation*}
\|{\bf w}\|_{\mathcal{W}_T}, \|{\bf w}_1\|_{\mathcal{W}_T}, \|{\bf w}_2\|_{\mathcal{W}_T}, \|\alpha\|_{L^2_{t, x}(D_T)}, \|\alpha_1\|_{L^2_{t, x}(D_T)}, \|\alpha_2\|_{L^2_{t, x}(D_T)}\leq 1
\end{equation*}
The proof is straightforward and relies on the  identity \eqref{null-indentity}, which can be found in Appendix \hyperref[setting(C1)]{C.1}.

By combining  \eqref{eq:li:w:1}--\eqref{eq:li:w:2}, the preceding estimates on $\mathcal{K}(x; {\bf w})$ and $\mathcal{K}_1(x; {\bf w}) \alpha$, and the bootstrap argument we obtain the  small data well-posedness result,  whose proof can be found in Appendix \hyperref[setting(C1)]{C.1}.
\begin{lemma}\label{lem:semiwave:1}
Let $T>0$.  There exists an explicit constant $\varepsilon_T$ such that for any initial state $({\bf w}_0, {\bf w}_{0t})\in H^1(\mathbb{T}^1; \mathbb{R}^R)\times L^2(\mathbb{T}^1; \mathbb{R}^R)$, and any source terms $\alpha$ and $e$ satisfying 
\begin{equation}
    \|({\bf w}_0, {\bf w}_{0t})\|_{\mathcal{H}}+ \|\alpha\|_{L^2_{t, x}(D_T)}+ \|e\|_{L^2_{t, x}(D_T)}\leq \varepsilon_T, \label{eq:cond:ualphae}
\end{equation}
the inhomogeneous semilinear  wave  equation 
\begin{equation}
    \begin{cases}
        \mathcal{L}_{{\rm in}} {\bf w}- \mathcal{K}(x; {\bf w})=  \chi_{\omega} \alpha+ \chi_{\omega} \mathcal{K}_1(x; {\bf w}) \alpha+ e, \\
        {\bf w}[0]= ({\bf w}_0, {\bf w}_{0t}),
    \end{cases} \notag
\end{equation}
 admits a unique solution.   This unique solution satisfies the following estimates:
  \begin{gather} 
  \|{\bf w}\|_{\mathcal{W}_T}\leq 2\mathcal{C}_T \left( \|({\bf w}_0, {\bf w}_{0t})\|_{\mathcal{H}}+ \|\alpha\|_{L^2_{t, x}(D_T)}+  \|e\|_{L^2_{t, x}(D_T)}\right). \notag
\end{gather}

Moreover, for any triples $({\bf u}_1[0], \alpha_1, e_1)$ and $({\bf u}_2[0], \alpha_2, e_2)$ satisfying \eqref{eq:cond:ualphae}
the solutions of 
\begin{equation}
    \begin{cases}
        \mathcal{L}_{{\rm in}} {\bf w}_k- \mathcal{K}(x; {\bf w}_k)=  \chi_{\omega} \alpha_k+ \chi_{\omega} \mathcal{K}_1(x; {\bf w}_k) \alpha_k+ e_k, \\
        {\bf w}_k[0]= {\bf u}_k[0],  \textrm{  where } \;k= 1, 2,
    \end{cases} \notag
\end{equation}
 satisfy 
  \begin{gather} 
  \|{\bf w}_1- {\bf w}_2\|_{\mathcal{W}_T}\leq 3\mathcal{C}_T \left( \|{\bf u}_1[0]- {\bf u}_2[0]\|_{\mathcal{H}}+ \|\alpha_1- \alpha_2\|_{L^2_{t, x}(D_T)}+  \|e_1- e_2\|_{L^2_{t, x}(D_T)}\right). \notag
\end{gather}
\end{lemma}

\subsubsection{Proposition \ref{prop:congeo:2nd} implies Proposition \ref{prop:congeo:1st}}

Now we fix $T= 32\pi$.   For ease of notations, we denote the initial state $({\bf w}_0, {\bf w}_{0t})$ by ${\bf u}[0]$. 
Usually, we rely on fixed point arguments to show local controllability results. However, in this framework one notice that the control force is not exactly the functions $\alpha$ that we are choosing but the slightly modified functions $\alpha+ \mathcal{K}_1(x; {\bf w}) \alpha$. This makes it more delicate to adopt fixed point arguments.
Here we use an iteration scheme to prove the controllability of this equation. The idea is to find a sequence of states and controls $\{({\bf w}_n, \alpha_n)\}_{n}$, which forms a Cauchy sequence,  such that for every $n$,
\begin{equation}\label{eq:see:ite}
    \begin{cases}
        \mathcal{L}_{{\rm in}} {\bf w}_n- \mathcal{K}(x; {\bf w}_n)=  \chi_{\omega} \alpha_n+ \chi_{\omega} \mathcal{K}_1(x; {\bf w}_n) \alpha_n, \\
        {\bf w}_n[0]= {\bf u}[0],
    \end{cases} 
\end{equation}
and that ${\bf w}_n[T]$ converges to $(0, 0)$. By passing to the limit of this sequence, we can show the following lemma, which easily yields Proposition \ref{prop:congeo:1st}.
\begin{lemma}\label{lem:conliwavei}
    Let $T= 32\pi$. There exists an explicit $\tilde \varepsilon_T\leq \varepsilon_T$ such that for any initial state ${\bf u}[0]$ satisfying $\|{\bf u}[0]\|_{\mathcal{H}}\leq \tilde \varepsilon_T$, there exist a control $\alpha\in C([0, T]; L^2(\mathbb{T}^1))$ supported in $\omega$ and a solution ${\bf w}$ satisfying 
    \begin{equation*}
    \begin{cases}
        \mathcal{L}_{{\rm in}} {\bf w}- \mathcal{K}(x; {\bf w})=  \chi_{\omega} \alpha+ \chi_{\omega} \mathcal{K}_1(x; {\bf w}) \alpha, \;\; \forall t\in (0, 2\pi), \\
        {\bf w}[0]= {\bf u}[0],\;  {\bf w}[T]= (0, 0),
    \end{cases}
    \end{equation*}
    and 
    \begin{equation*}
        \|\alpha\|_{L^{\infty}(0, T; L^2(\mathbb{T}^1))}, 
\|{\bf w}\|_{\mathcal{W}_T} \leq 3S \mathcal{C}_T  \|{\bf u}[0]\|_{\mathcal{H}},
    \end{equation*}
    where $\mathcal{C}_T$ is the constant given in \eqref{eq:li:w:1}--\eqref{eq:li:w:2}, and $S$ is the constant given in Proposition  \ref{prop:congeo:2nd}.
\end{lemma}

\begin{proof}[Proof of Lemma \ref{lem:conliwavei}]
    
We assume that the initial state is smaller than $\tilde \varepsilon_T$, with the value of $\tilde \varepsilon_T\in (0, \varepsilon_T)$ to be fixed later on. In the {\it zeroth step} we simply let $\alpha_0= 0$ and find a unique solution of 
\begin{equation}
    \begin{cases}
        \mathcal{L}_{{\rm in}} {\bf w}_0- \mathcal{K}(x; {\bf w}_0)=  \chi_{\omega} \alpha_0+ \chi_{\omega} \mathcal{K}_1(x; {\bf w}_0) \alpha_0, \\
        {\bf w}_0[0]= {\bf u}[0].
    \end{cases} \notag
\end{equation}
According to Lemma \ref{lem:semiwave:1} this solution satisfies
\begin{equation*}
    \|{\bf w}_0\|_{\mathcal{W}_T}\leq 2\mathcal{C}_T  \|{\bf u}[0]\|_{\mathcal{H}},
\end{equation*}
and in particular 
\begin{equation}\label{es:errorw0}
    \|{\bf w}_0[T]\|_{H^1\times L^2}\leq 2\mathcal{C}_T  \|{\bf u}[0]\|_{\mathcal{H}}.
\end{equation}

Keeping in mind that the value of $\tilde \varepsilon_T$ can be chosen as small as we want, in {\it the first step } we construct $\alpha_1$ as the sum of $\alpha_0$ and $\beta_0$, where $\beta_0$ is a control to be selected such that 
\begin{equation}
    \begin{cases}
        \mathcal{L}_{{\rm in}} {\bf w}_0^c= \chi_{\omega} \beta_0 \\
        {\bf w}_0^c[0]= 0 \; \textrm{ and } \;   {\bf w}_0^c[T]= - {\bf w}_0[T].
    \end{cases} \notag
\end{equation}
According to Proposition \ref{prop:congeo:2nd} there exists such a $\beta_0$, and it satisfies 
  \begin{equation*}
        \|\beta_0\|_{L^{\infty}(0, T; L^2(\mathbb{T}^1))}+ 
\|{\bf w}_0^c\|_{\mathcal{W}_T} \leq S \|{\bf w}_0[T]\|_{\mathcal{H}} \leq 2 \mathcal{C}_TS \|{\bf u}[0]\|_{\mathcal{H}}. 
    \end{equation*}
Next, we observe that the functions  $\tilde {\bf w}_1:= {\bf w}_0+ {\bf w}_0^c$ and $\alpha_1:=\alpha_0+ \beta_0$ satisfy
\begin{equation}\label{eq:tildw1}
    \begin{cases}
        \mathcal{L}_{{\rm in}} \tilde {\bf w}_1- \mathcal{K}(x; \tilde {\bf w}_1)=  \chi_{\omega} \alpha_1+ \chi_{\omega} \mathcal{K}_1(x; \tilde {\bf w}_1) \alpha_1+ e_{1}, \\
        \tilde {\bf w}_1[0]= {\bf u}[0] \; \textrm{ and } \; \tilde {\bf w}_1[T]= (0, 0),
    \end{cases} 
\end{equation}
where the error term $e_1$ equals 
\begin{align*}
    e_1= \mathcal{K}(x; {\bf w}_0)- \mathcal{K}(x; {\bf w}_0+ {\bf w}_0^c)+ \chi_{\omega} \mathcal{K}_1(x; {\bf w}_0) \alpha_0- \chi_{\omega}\mathcal{K}_1(x;  {\bf w}_0+ {\bf w}_0^c) \alpha_1,
\end{align*}
and  satisfies 
\begin{align*}
    \|e_1\|_{L^2_{t, x}(D_T)}\leq C_N \|{\bf w}_0^c\|_{\mathcal{W}_T}(\|{\bf w}_0^c\|_{\mathcal{W}_T}+ \|{\bf w}_0+ {\bf w}_0^c\|_{\mathcal{W}_T}) + C_N \|{\bf w}_0+ {\bf w}_0^c\|_{\mathcal{W}_T} \|\beta_0\|_{L^2_{t,x}} 
    \lesssim \|{\bf u}[0]\|_{\mathcal{H}}^2.
\end{align*}

 Thus by fixing the control term as $\alpha_1$, the unique solution ${\bf w}_1$ of 
\begin{equation}\label{eq:w1}
    \begin{cases}
        \mathcal{L}_{{\rm in}} {\bf w}_1- \mathcal{K}(x;  {\bf w}_1)= \chi_{\omega}  \alpha_1+  \chi_{\omega} \mathcal{K}_1(x;  {\bf w}_1) \alpha_1, \\
         {\bf w}_1[0]= {\bf u}[0],
    \end{cases} 
\end{equation}
is close to $\tilde {\bf w}_1$. Actually, by comparing equations \eqref{eq:tildw1} and \eqref{eq:w1} and by applying Lemma \ref{lem:semiwave:1}, one obtains
\begin{equation}\label{inefore1}
    \|{\bf w}_1- \tilde {\bf w}_1\|_{\mathcal{W}_T}\leq 3 \mathcal{C}_T \|e_1\|_{L^2_{t, x}}\lesssim \|{\bf u}[0]\|_{\mathcal{H}}^2.
\end{equation}
provided that the value of $\tilde \varepsilon_T$ is smaller than some effectively computable constant $\varepsilon_1$.  This implies that 
\begin{gather*}
    \|{\bf w}_1\|_{\mathcal{W}_T}\leq 2 \mathcal{C}_T(1+ S) \|{\bf u}[0]\|_{\mathcal{H}}+  C_1\|{\bf u}[0]\|_{\mathcal{H}}^2, \\
    \|\alpha_1\|_{L^{\infty}(0, T; L^2(\mathbb{T}^1))}\leq 2 \mathcal{C}_TS \|{\bf u}[0]\|_{\mathcal{H}}, \\
    \|{\bf w}_1[T]\|_{H^1\times L^2}\leq C_1 \|{\bf u}[0]\|_{\mathcal{H}}^2,
\end{gather*}
where $C_1$ is some effectively computable constant. 
To be compared with the estimate \eqref{es:errorw0} of ${\bf w}_0[T]$, the error of ${\bf w}_1[T]$ is much smaller. 
Therefore, this motivates us to construct $\{({\bf w}_n, \alpha_n)\}_{n}$  satisfying equation \eqref{eq:wn} below and the estimates
\begin{gather*}
    \|{\bf w}_n\|_{\mathcal{W}_T}\leq 3 \mathcal{C}_T S \|{\bf u}[0]\|_{\mathcal{H}}, \\
    \|\alpha_n\|_{L^{\infty}(0, T; L^2(\mathbb{T}^1))}\leq 3 \mathcal{C}_TS \|{\bf u}[0]\|_{\mathcal{H}}, \\
    \|{\bf w}_n[T]\|_{\mathcal{H}}\leq C  \|{\bf w}_{n-1}[T]\|_{\mathcal{H}}  \|{\bf u}[0]\|_{\mathcal{H}}.
\end{gather*}
More precisely, we have the following lemma.
\begin{lemma}\label{lem:iteration}
    Let $T= 32\pi$. There exist effectively computable constants $\varepsilon_2$ and $C_2$ such that for any initial state ${\bf u}[0]$ satisfying $\|{\bf u}[0]\|_{\mathcal{H}}\leq \varepsilon_2$ and any solution $({\bf w}_n, \alpha_n)$ of 
\begin{equation}\label{eq:wn}
    \begin{cases}
        \mathcal{L}_{{\rm in}} {\bf w}_n- \mathcal{K}(x;  {\bf w}_n)= \chi_{\omega}  \alpha_n+ \chi_{\omega} \mathcal{K}_1(x;  {\bf w}_n) \alpha_n, \\
         {\bf w}_n[0]= {\bf u}[0]
    \end{cases} 
\end{equation}
satisfying 
\begin{gather*}
    \|{\bf w}_n\|_{\mathcal{W}_T}\leq 3 \mathcal{C}_T S \|{\bf u}[0]\|_{\mathcal{H}}, \\
    \|\alpha_n\|_{L^{\infty}(0, T; L^{2}(\mathbb{T}^1))}\leq 3 \mathcal{C}_TS \|{\bf u}[0]\|_{\mathcal{H}}, 
\end{gather*}
we can construct $({\bf w}_{n+1}, \alpha_{n+1})$ satisfying
\begin{equation}\label{eq:wn+1}
    \begin{cases}
        \mathcal{L}_{{\rm in}} {\bf w}_{n+1}- \mathcal{K}(x;  {\bf w}_{n+1})=   \chi_{\omega} \alpha_{n+1}+ \chi_{\omega} \mathcal{K}_1(x;  {\bf w}_{n+1}) \alpha_{n+1}, \\
         {\bf w}_{n+1}[0]= {\bf u}[0],
    \end{cases} 
\end{equation}
and
\begin{gather*}
     \|{\bf w}_{n+1}- {\bf w}_{n}\|_{\mathcal{W}_T}\leq 2S\|{\bf w}_n[T]\|_{\mathcal{H}} ,      \\
     \|\alpha_{n+1}- \alpha_n\|_{L^{\infty}(0, T; L^2)}\leq  S\|{\bf w}_n[T]\|_{\mathcal{H}},     \\
    \|{\bf w}_{n+1}[T]\|_{\mathcal{H}}\leq C_2  \|{\bf w}_{n}[T]\|_{\mathcal{H}}  \|{\bf u}[0]\|_{\mathcal{H}}.
\end{gather*}
\end{lemma}

\begin{proof}[Proof of Lemma \ref{lem:iteration}]
    Inspired by the first step on constructing $\alpha_1$ as $\alpha_0+ \beta_0$, we directly find the control $\beta_n$ such that 
\begin{equation}
    \begin{cases}
        \mathcal{L}_{{\rm in}} {\bf w}_n^c= \chi_{\omega} \beta_n \\
        {\bf w}_n^c[0]= 0,\;\;  {\bf w}_n^c[T]= - {\bf w}_n[T],
    \end{cases} \notag
\end{equation}
which satisfies 
 \begin{equation*}
        \|\beta_n\|_{L^{\infty}(0, T; L^2(\mathbb{T}^1))}+ 
\|{\bf w}_n^c\|_{\mathcal{W}_T} \leq S \|{\bf w}_n[T]\|_{\mathcal{H}}. 
    \end{equation*}
 Thus we fix $\alpha_{n+1}$ as $\alpha_n+ \beta_n$.  Next, we observe that   $\tilde {\bf w}_{n+1}:= {\bf w}_n+ {\bf w}_n^c$  satisfy
\begin{equation}\label{eq:tildwn1}
    \begin{cases}
        \mathcal{L}_{{\rm in}} \tilde {\bf w}_{n+1}- \mathcal{K}(x; \tilde {\bf w}_{n+1})=  \chi_{\omega} \alpha_{n+1}+ \chi_{\omega} \mathcal{K}_1(x; \tilde {\bf w}_{n+1}) \alpha_{n+1}+ e_{n+1}, \\
        \tilde {\bf w}_{n+1}[0]= {\bf u}[0],  \tilde {\bf w}_{n+1}[T]=(0, 0),
    \end{cases} 
\end{equation}
where the error term $e_{n+1}$ satisfies 
\begin{align*}
    e_{n+1}= \mathcal{K}(x; {\bf w}_n)- \mathcal{K}(x; {\bf w}_n+ {\bf w}_n^c)+\chi_{\omega}  \mathcal{K}_1(x;  {\bf w}_n) \alpha_n- \chi_{\omega} \mathcal{K}_1(x;  {\bf w}_n+ {\bf w}_n^c) \alpha_{n+1}.
\end{align*}
We also define ${\bf w}_{n+1}$ as the unique solution of 
\begin{equation}\label{eq:wn1}
    \begin{cases}
        \mathcal{L}_{{\rm in}}  {\bf w}_{n+1}- \mathcal{K}(x;  {\bf w}_{n+1})=   \chi_{\omega} \alpha_{n+1}+ \chi_{\omega}  \mathcal{K}_1(x;  {\bf w}_{n+1}) \alpha_{n+1}, \\
         {\bf w}_{n+1}[0]= {\bf u}[0].
    \end{cases} 
\end{equation}

By selecting the value of $\varepsilon_2$ smaller than some effective computable constant $\varepsilon_0$, we have
\begin{equation*}
\|{\bf w}_n\|_{\mathcal{W}_T}, \|{\bf w}_n^c\|_{\mathcal{W}_T}, \|\tilde {\bf w}_n\|_{\mathcal{W}_T}, \|\alpha_n\|_{L^2_{t, x}(D_T)}, \|\beta_n\|_{L^2_{t, x}(D_T)}, \|\alpha_{n+1}\|_{L^2_{t, x}(D_T)}\leq 1.
\end{equation*}
Thus 
\begin{align*}
   & \;\;\;\;\;  \|e_{n+1}\|_{L^2_{t, x}(D_T)}\\
    &\leq \|\mathcal{K}({\bf w}_n)- \mathcal{K}({\bf w}_n+ {\bf w}_n^c)\|_{L^2_{t, x}(D_T)}+   \|  \mathcal{K}_1(x;  {\bf w}_n) \alpha_n- \mathcal{K}_1(x;  {\bf w}_n+ {\bf w}_n^c) \alpha_{n+1}  \|_{L^2_{t, x}(D_T)}\\
    &\leq C_N \|{\bf w}_n^c\|_{\mathcal{W}_T}\Big(\|{\bf w}_n^c\|_{\mathcal{W}_T}+  \|{\bf w}_n+ {\bf w}_n^c\|_{\mathcal{W}_T}\Big)\\
    & \;\;\;\;\;\; + C_N   \|{\bf w}_n^c\|_{L^{\infty}(D_T)} \|\alpha_n\|_{L^2_{t, x}(D_T)}+  C_N\| {\bf w}_n+ {\bf w}_n^c\|_{L^{\infty}(D_T)} \|\beta_n\|_{L^2_{t, x}(D_T)}    \\
    &\leq 180 C_N \mathcal{C}_T S^3 \|{\bf u}[0]\|_{\mathcal{H}} \|{\bf w}_n[T]\|_{\mathcal{H}}.
\end{align*}
By selecting the value of $\varepsilon_2$ sufficiently small, one has 
\begin{equation*}
180 C_N \mathcal{C}_T S^3 \|{\bf u}[0]\|_{\mathcal{H}} \|{\bf w}_n[T]\|_{\mathcal{H}}\leq 540 C_N \mathcal{C}_T^2 S^4 \|{\bf u}[0]\|_{\mathcal{H}}^2 \leq 1.
\end{equation*}
Thus, thanks to Lemma \ref{lem:semiwave:1}, by comparing equations \eqref{eq:tildwn1} and \eqref{eq:wn1} one obtains
\begin{align}
    \|{\bf w}_{n+1}- \tilde {\bf w}_{n+1}\|_{\mathcal{W}_T}&\leq 3 \mathcal{C}_T \|e_{n+1}\|_{L^2_{t, x}} \notag \\
    &\leq 540 C_N \mathcal{C}_T^2 S^3 \|{\bf u}[0]\|_{\mathcal{H}} \|{\bf w}_n[T]\|_{\mathcal{H}}.
\end{align}
Therefore, 
\begin{align*}
     \|{\bf w}_{n+1}-  {\bf w}_{n}\|_{\mathcal{W}_T}&\leq  \|{\bf w}_{n+1}- \tilde {\bf w}_{n+1}\|_{\mathcal{W}_T}+  \|{\bf w}_{n}^c\|_{\mathcal{W}_T} \\
     &\leq 540 C_N \mathcal{C}_T^2 S^3 \|{\bf u}[0]\|_{\mathcal{H}} \|{\bf w}_n[T]\|_{\mathcal{H}}+ S \|{\bf w}_n[T]\|_{\mathcal{H}}, 
\end{align*}
and 
\begin{gather*}
     \|\alpha_{n+1}- \alpha_n\|_{L^{\infty}(0, T; L^2)}\leq S \|{\bf w}_n[T]\|_{\mathcal{H}},\\
      \|{\bf w}_{n+1}[T]\|_{\mathcal{H}}=  \|{\bf w}_{n+1}[T]- \tilde {\bf w}_{n+1}[T]\|_{\mathcal{H}} \leq 540 C_N \mathcal{C}_T^2 S^3 \|{\bf u}[0]\|_{\mathcal{H}} \|{\bf w}_n[T]\|_{\mathcal{H}}.
\end{gather*}
Recall the constants $\varepsilon_0$ given after equation \eqref{eq:wn1}, $\varepsilon_1$ given after equation \eqref{inefore1}, and $\varepsilon_T$ given in Lemma \ref{lem:semiwave:1}.  By fixing $\varepsilon_2$ such that  
  \begin{equation}
      540 C_N \mathcal{C}_T^2 S^4 \varepsilon_2^2\leq 1 \textrm{ and } \varepsilon_2\leq \min\{\varepsilon_0, \varepsilon_1, \varepsilon_T\},
  \end{equation}
  and $C_2:=  540 C_N \mathcal{C}_T^2 S^3$, we conclude the proof of the Lemma \ref{lem:iteration}. 
\end{proof}

Now, we come back to the proof of Lemma \ref{lem:conliwavei}.
Recall that $C_1$ is given after equation \eqref{inefore1} and that $C_2$ is given in Lemma \ref{lem:iteration}.
By selecting $\tilde \varepsilon_T\in (0, \varepsilon_T)$ such that
\begin{equation*}
   8 \tilde \varepsilon_T C_1\leq \mathcal{C}_T S \textrm{ and } C_2 \tilde \varepsilon_T\leq 1/2,
\end{equation*}
 we know that $({\bf w}_1, \alpha_1)$ satisfies 
\begin{gather*}
    \|{\bf w}_1\|_{\mathcal{W}_T}\leq \frac{5}{2} \mathcal{C}_TS \|{\bf u}[0]\|_{H^1\times L^2(\mathbb{T}^1)}, \\
    \|\alpha_1\|_{L^{\infty}(0, T; L^2(\mathbb{T}^1))}\leq 2 \mathcal{C}_TS \|{\bf u}[0]\|_{H^1\times L^2(\mathbb{T}^1)}, \\
    \|{\bf w}_1[T]\|_{H^1\times L^2}\leq \frac{C_T}{8} \|{\bf u}[0]\|_{H^1\times L^2(\mathbb{T}^1)}.
\end{gather*}
Then, using Lemma  \ref{lem:iteration} we find a solution $({\bf w}_2, \alpha_2)$ such that 
\begin{gather*}
     \|{\bf w}_{2}- {\bf w}_{1}\|_{\mathcal{W}_T}\leq 2S\|{\bf w}_1[T]\|_{\mathcal{H}}\leq \frac{1}{4} \mathcal{C}_TS \|{\bf u}[0]\|_{\mathcal{H}} ,      \\
     \|\alpha_{2}- \alpha_1\|_{L^{\infty}(0, T; L^2)}\leq  S\|{\bf w}_1[T]\|_{\mathcal{H}}\leq  \frac{1}{8} \mathcal{C}_TS \|{\bf u}[0]\|_{\mathcal{H}}    \\
    \|{\bf w}_{2}[T]\|_{\mathcal{H}}\leq C_2  \|{\bf w}_{1}[T]\|_{\mathcal{H}}  \|{\bf u}[0]\|_{\mathcal{H}}\leq \frac{1}{2} \|{\bf w}_{1}[T]\|_{\mathcal{H}},
\end{gather*}
thus 
\begin{gather*}
    \|{\bf w}_2\|_{\mathcal{W}_T}\leq \frac{11}{4} \mathcal{C}_TS \|{\bf u}[0]\|_{\mathcal{H}}, \\
    \|\alpha_2\|_{L^{\infty}(0, T; L^2(\mathbb{T}^1))}\leq (2+ \frac{1}{8}) \mathcal{C}_TS \|{\bf u}[0]\|_{\mathcal{H}}, \\
    \|{\bf w}_2[T]\|_{H^1\times L^2}\leq \frac{C_T}{16} \|{\bf u}[0]\|_{\mathcal{H}}.
\end{gather*}
Using the induction argument, we can prove that 
\begin{gather*}
    \|{\bf w}_n\|_{\mathcal{W}_T}\leq \left(3- \frac{1}{2^n}\right) \mathcal{C}_TS \|{\bf u}[0]\|_{\mathcal{H}}, \\
    \|\alpha_n\|_{L^{\infty}(0, T; L^2(\mathbb{T}^1))}\leq \left(\frac{9}{4}- \frac{1}{2^{n+1}}\right) \mathcal{C}_TS \|{\bf u}[0]\|_{\mathcal{H}}, \\
    \|{\bf w}_n[T]\|_{H^1\times L^2}\leq \frac{\mathcal{C}_T}{2^{n+2}} \|{\bf u}[0]\|_{\mathcal{H}},
\end{gather*}
and 
\begin{gather*}
     \|{\bf w}_{n+1}- {\bf w}_{n}\|_{\mathcal{W}_T}\leq 2S\|{\bf w}_n[T]\|_{\mathcal{H}}\leq \frac{1}{2^{n+1}} \mathcal{C}_TS \|{\bf u}[0]\|_{\mathcal{H}} ,      \\
     \|\alpha_{n+1}- \alpha_{n}\|_{L^{\infty}(0, T; L^2)}\leq  S\|{\bf w}_n[T]\|_{\mathcal{H}}\leq   \frac{1}{2^{n+2}} \mathcal{C}_TS \|{\bf u}[0]\|_{\mathcal{H}}.
\end{gather*}

Therefore,  the sequence $\{({\bf w}_n, \alpha_n)\}_{n}$ that we find as solutions of \eqref{eq:see:ite} is a   Cauchy sequence in $\mathcal{W}_T\times C^0([0, T]; L^2)$. Hence there exists $({\bf w}, \alpha)$ such that 
\begin{equation*}
    ({\bf w}_n, \alpha_n)\xrightarrow{n\rightarrow \infty} ({\bf w}, \alpha) \textrm{ in } \mathcal{W}_T\times C^0([0, T]; L^2). 
\end{equation*}
Using the standard argument on passing the limit, we know that $({\bf w}, \alpha)$ satisfy 
\begin{equation}
    \begin{cases}
        \mathcal{L} {\bf w}- \mathcal{K}(x;  {\bf w})=  \chi_{\omega} \alpha+ \chi_{\omega} \mathcal{K}_1(x;  {\bf w}) \alpha, \\
         {\bf w}[0]= {\bf u}[0].
    \end{cases} 
\end{equation}
Moreover, we know that ${\bf w}[T]= (0, 0)$.  This finishes the proof of Lemma \ref{lem:conliwavei}.
\end{proof}

\subsection{Controllability of the linearized system}\label{subsec:con:3}

Thanks to the preceding two steps,  it suffices to prove the linear controllability result Proposition \ref{prop:congeo:2nd}.  This type of controllability has been extensively studied in the literature.  See, for instance, the works mentioned in Section \ref{sec:contro:wave}. 
The closest work might be \cite{Li-Yu-2006}, where the controllability of a quasilinear wave equation with boundary control in $C^2\times C^1$-space is shown using the characteristic method. However,  we did not find an exact reference dealing with equation \eqref{eq:obdua}. Thus, we provided a proof, which is inspired by the propagation of smallness argument introduced by the last two authors \cite{KX}.

\subsubsection{Hilbert Uniqueness Method}\label{subsec:HUM}
Using the standard duality argument \cite{Lions},  the exact controllability of wave equations in $H^1\times L^2(\mathbb{T}^1)$ space is equivalent to the observability inequality of the adjoint system for states in $L^2\times H^{-1}(\mathbb{T}^1)$ space.
Then, after a change of time variables (to transform the backward system into a forward system), showing Proposition \ref{prop:congeo:2nd} is equivalent to proving the observability estimates:

\begin{lemma}[Observability inequality]\label{lem:ob:key-lin}
Let $T= 32\pi$. Define the matrix valued functions  
$${\bf b}(x)= \partial_x {\bf B}_{0}^T (x), {\bf c}(x)= {\bf B}_0^{T}(x)- {\bf C}_0^T (x) \; \forall x\in \mathbb{T}^1.$$
Then there exists $C_0>0$ such that every solution of 
    \begin{equation}\label{eq:obdua}
    \begin{cases}
        \Box \varphi+ {\bf b}(x) \partial_x \varphi+ {\bf c}(x) \varphi= 0,\\
        \varphi[0]\in L^2\times H^{-1}(\mathbb{T}^1; \mathbb{R}^R),
    \end{cases}
\end{equation} 
satisfies
\begin{equation}\label{eq:in-ob-l2}
    \int_0^T \|\chi_{\omega} \varphi\|_{L^2}^2 \, dt\geq C_0 \|\varphi[0]\|_{L^2\times H^{-1}}^2.
\end{equation}
\end{lemma}
Note that, typically, duality arguments lead to choosing the control as the optimal one in the space $L^2(0, T; L^2)$. In case of wave equations, this control is explicitly constructed via using the HUM operator, which ensures that it actually belongs to $C([0, T]; L^2)$.

\subsubsection{On the proof of the observability inequality}
\label{sec:cont:4th}

This observability is proved in two steps. 

\vspace{1mm}

\noindent {\it (1)} First, on the one hand, we deliver a lower bound for $\|\varphi\|_{L^2_t(- 3\pi, 3\pi; L^2_x(\mathbb{T}^1))}$. Let us define the energy function of $\varphi$ as 
\begin{equation}\label{def:F:energy}
F(\varphi):= \|\varphi\|_{L^2(\mathbb{T}^1)}^2+ \|\varphi_t\|_{H^{-1}(\mathbb{T}^1)}^2, 
\end{equation}
where the $H^{-1}$-norm is given by 
\begin{equation*}
    \|f\|_{H^{-1}(\mathbb{T}^1)}^2:=  \|\langle \partial_x \rangle^{-1} f\|_{L^{2}(\mathbb{T}^1)}^2.
\end{equation*}

Using straightforward energy estimates we know that for every $T>0$ there exists a constant $C$ such that 
\begin{equation} \label{eq:one:enees}
   C^{-1}F(\varphi[0])\leq  F(\varphi[t])\leq C F(\varphi[0])\;\; \forall t\in [-T, T],
\end{equation}
thus 
\begin{equation}\label{es:Flbd}
 \int_{-2\pi}^{2\pi} \left(\|\varphi_t(t, \cdot)\|_{H^{-1}}^2+ \|\varphi(t, \cdot)\|_{L^2}^2 \right) \, dt\geq C F(0).
\end{equation}

We also have the following standard estimate on wave equations; see, for instance \cite[Lemma 3.4]{Zhang-2000-2} for the case of a one-variable wave equation. Its proof can be found in Appendix \hyperref[setting(C2)]{C.2}. 
\begin{lemma}\label{lem:wave:timegain}
    The solutions of \eqref{eq:obdua} satisfies
    \begin{gather}
        \int_{-2\pi}^{2\pi} \|\varphi_t(t, \cdot)\|_{H^{-1}(\mathbb{T}^1)}^2 \, dt  \lesssim \int_{-3\pi}^{3\pi} \|\varphi(t, \cdot)\|_{L^2(\mathbb{T}^1)}^2 \, dt. \label{es:th-1bd} 
    \end{gather}
\end{lemma}

Estimates \eqref{es:Flbd} and \eqref{es:th-1bd} show that 
\begin{equation}\label{l2lowerbd}
    \int_{-3\pi}^{3\pi}\int_{\mathbb{T}^1}  \varphi^2(t, x) \, dx dt \geq C F(0).
\end{equation} 
\vspace{1mm}

\noindent {\it (2)}
Next,  on the other hand, we use the propagation of smallness result to present an upper bound for $\|\varphi\|_{L^2_t(- 3\pi, 3\pi; L^2_x(\mathbb{T}^1))}$. 
The following is a direct consequence of Proposition \ref{prop:psl2}.
\begin{lemma}\label{lem:psl2}
    There exists some effectively computable $q>0$ and $C>0$ such that, for any $\varepsilon\in (0, 1)$ if the solution of \eqref{eq:obdua} satisfies 
    \begin{equation}
        \int_{-16\pi}^{16\pi} \|\chi_{\omega} \varphi\|_{L^2}^2 \, dt\leq \varepsilon \|\varphi[0]\|_{L^2\times H^{-1}}^2
    \end{equation}
    then 
    \begin{equation}
        \|\varphi\|^2_{L^{\infty}_x(\mathbb{T}^1; L^2_t (-3\pi, 3\pi))}\leq C \varepsilon^{1/p}  \|\varphi[0]\|_{L^2\times H^{-1}}^2.
    \end{equation}
\end{lemma}

By combining this lemma and the estimates \eqref{l2lowerbd},  we deduce the required observability inequality \eqref{eq:in-ob-l2} for $T= 32\pi$ concerning Lemma \ref{lem:ob:key-lin}. Therefore, we conclude the proof of the local exact controllability around geodesics, namely  Proposition \ref{pro:controlaroundgeodesic}.

\section{Uniform-time global exact controllability between homotopic closed geodesics}\label{Sec:gc}

In this section we prove Proposition \ref{pro:controlbetween}, which constitutes to the third intermediate property for the proof of Theorem \ref{thm1}. 
 Looking at Figure \ref{fig:movethreegeo}, our objective is to transfer the state from $\gamma_0$ to $\gamma_1$. The two states may be widely separated, and the maps $\gamma_0, \gamma_1 : \mathbb{T}^1 \to \mathcal{N}$ may exhibit complicated topology.

The proof is inspired by the earlier work \cite{CX-2024}, where the first and third author introduced a novel method to the small-time global controllability between steady states of the harmonic map heat flow. As illustrated in Section \ref{sec:strategy1}, it highlights the control on geodesics and the idea of gluing. 
In Section \ref{subsec:con:gluing}, we establish key control and gluing results for geodesics. Then, in Section \ref{subsec:gloconbewgeo}, we explicitly construct controlled trajectories connecting closed geodesics.

\subsection{The control and gluing on geodesics}\label{subsec:con:gluing}
Without loss of generality, in this entire section,  we fix $0<b< b_1< b_0< \pi$ and define the following domains in $\mathbb{T}^1$.
\begin{gather*}
\omega= (- b_0, b_0):  \textrm{ the controlled domain defined in \blackhyperref{setting(S)}{$({\bf S})$}}, \\
\omega_{\textrm{in}}= [b, b_0)\cup (2\pi- b_0,  2\pi-b]: \textrm{  the controlled domain for the inner equation, } \\
\tilde \omega_{\textrm{in}}= [b, b_1]\cup [2\pi- b_1, 2\pi- b]: \textrm{ a subset of } \omega_{\textrm{in}},   \\
\omega_{\textrm{out}}= (- b, b): \textrm{  the (controlled) domain for the outer equation. }
\end{gather*}

We also define the inner and outer equations:
\begin{equation}
 \Box \phi - \Pi(\phi)\left(\partial_{\nu}\phi, \partial^{\nu}\phi\right)= \chi_{\omega_{\textrm{in}}} \Pi_T(\phi)f \; \textrm{ for } x\in (b, 2\pi- b),  \tag{\textrm{inner}}
\end{equation}
and 
\begin{equation}
 \Box \phi  -  \Pi(\phi)\left(\partial_{\nu}\phi, \partial^{\nu}\phi\right)= \chi_{\omega_{\textrm{out}}} \Pi_T(\phi)f \; \textrm{ for } x\in \omega_{\textrm{out}}. \tag{\textrm{outer}}
\end{equation}
Note that boundary conditions are required to guarantee the uniqueness of solutions to these two equations.
We can glue two classical solutions to these two equations to obtain a classical solution of the original controlled wave maps equation:
\begin{equation*}
     \Box \phi -  \Pi(\phi)\left(\partial_{\nu}\phi, \partial^{\nu}\phi\right)= \chi_{\omega} \Pi_T(\phi)f \; \textrm{ for } x\in \mathbb{T}^1
\end{equation*}
provided that these two solutions coincide on the common boundary, $\{b, 2\pi- b\}$.
\vspace{2mm}

Recall that a {\it closed geodesic} is  a periodic curve $\gamma: \mathbb{T}^1 \rightarrow \mathcal{N}$ satisfying the geodesic equation $\Delta\gamma -  \Pi(\gamma)\left(\partial_{x}\gamma, \partial_{x}\gamma\right)= 0$.  Next, we introduce the so-called non-closed complete geodesics:
\begin{defi}
 A  {\it  non-closed complete geodesic} is a non-periodic curve $\Gamma: \mathbb{R} \rightarrow \mathcal{N}$ satisfying the geodesic equation.
\end{defi}

 One has the following basic fact concerning geodesics, which will be useful in our proof of Proposition \ref{pro:controlbetween}.  Its proof is straightforward. There is a geodesic connecting two points $p_0$ and $p_1$. Then, by the Hopf-Rinow theorem, one can extend this geodesic. The extended curve may be either a closed or a non-closed complete geodesic. 
\begin{lemma} \label{lem:conne:points}
For any two points $p_0$ and $p_1$ in the same connected component of $\mathcal{N}$, there exists a complete or a closed geodesic $\Gamma$ containing both points.
\end{lemma}

Now, we start by proving the following control results on geodesics. Let $\gamma$ be a closed geodesic. We denote $\bar\gamma$ to be the one dimensional submanifold given by this closed geodesic.   The controlled wave maps equation {\it restricted} on $T\bar\gamma$ means the initial state takes value in $T\bar\gamma$, and the control $f$ is always tangent to the geodesic. 
\begin{lemma}[Control on closed geodesics] \label{wave-closed-geodesic}
Let $\gamma: \mathbb{T}^1\rightarrow \mathcal{N}$ be a closed geodesic.  Let $T\geq 2\pi- b_0$. 
We consider the controlled wave maps equation restricted on $T\bar\gamma$: 
\begin{equation}\label{eq:conhmf:whole}
\begin{cases}
 \Box \phi -  \Pi(\phi)\left(\partial_{\nu}\phi, \partial^{\nu}\phi\right)= \chi_{\omega} \Pi_T(\phi)f, \\
 \phi[0](\cdot)\in H^1\times L^2(\mathbb{T}^1; T\bar\gamma).
 \end{cases}
\end{equation}
Then the solution of the controlled wave maps stays in $T\bar\gamma$. 
Moreover, after a change of variables, the geometric equation can be transformed into a linear controlled wave equation. 

Consequently, the system is globally exactly  controllable in time $T$ in $H^1$-topology, $i. e.$, for any two homotopic states $(\phi_{0}, \phi_{0t}), (\phi_{1}, \phi_{1t})\in H^1\times L^2(\mathbb{T}^1; T\bar\gamma)$, there exists a control $f$ such that the unique solution carries the data $(\phi_{0}, \phi_{0t})$ at time $t= 0$ into $(\phi_{1}, \phi_{1t})$ at time $t= T$.
\end{lemma}

\begin{proof}[Proof of Lemma \ref{wave-closed-geodesic}]
It is clear that when  the initial state $\phi[0]$ takes value in $T\bar\gamma$ and the control stays in the tangent direction, then the unique solution stays in $T\bar\gamma$. Assume the degree of the initial state as a map from $\mathbb{T}^1$ to $\gamma$ is
\begin{equation*}
    \textrm{deg} (\phi(0, \cdot); \mathbb{T}^1, \gamma)= K\in \mathbb{Z}.
\end{equation*}
This degree does not change along the evolution of the curve $\phi(t, \cdot): \mathbb{T}^1\rightarrow \gamma$.
From now on, we characterize the solution using a new variable on geodesic:   for every $(t, x)\in [0, T]\times \mathbb{T}^1$ we assume
\begin{gather*}
\phi(t, x)=  \gamma(\varphi(t, x)), \\
f(t, x)=  f_0(t, x)  \gamma_s(\varphi(t, x))\in T_{\bar \gamma(\varphi)}\gamma\subset  T_{\bar \gamma(\varphi)}\mathcal{N}
\end{gather*}
with functions $
    \varphi, f_0: [0, T]\times (0, 2\pi)\rightarrow \mathbb{R}$ satisfying for every $t\in [0, T]$,
    \begin{equation*}
        \varphi(t, 2\pi)= 2\pi K+   \varphi(t, 0)  \;\textrm{ and } \; f_0(t, 2\pi)= f_0(t, 0).
    \end{equation*}
The above compatibility condition ensures that the degree of the closed curve $\phi(t, \cdot)$ does not change.  
Under this new variable, by taking differentiations we obtain
\begin{align*}
\partial_t \phi&=  \gamma_s(\varphi) \varphi_t, \;  \; \; \partial_x \phi=  \gamma_s(\varphi) \varphi_x, \\
\partial_{tt} \phi&=  \gamma_{ss}(\varphi) (\varphi_t)^2+  \gamma_s(\varphi) \varphi_{tt}, \\
\partial_{xx} \phi&=  \gamma_{ss}(\varphi) (\varphi_x)^2+  \gamma_s(\varphi) \varphi_{xx}.
\end{align*}
Substitute these into the controlled wave maps equation,
\begin{align*}
  &\;\;\;\; \Box  \phi -    \Pi( \phi)\left(\partial_{\nu}\phi, \partial^{\nu} \phi\right)- \chi_{\omega} \Pi_T(\phi)f   \\
& =  -\gamma_{ss}(\varphi) (\varphi_t)^2-  \gamma_s(\varphi) \varphi_{tt}+  \gamma_{ss}(\varphi) (\varphi_x)^2+ \gamma_s(\varphi) \varphi_{xx}  +    \Pi( \phi)\left( \gamma_s(\varphi) \varphi_t,  \gamma_s(\varphi) \varphi_t\right)  \\
& \;\;\; \;\;\; \;\;\; \;\;\; \;\;\;\;\;\; \;\;\; \;\;\; \;\;\; \;\;\; -  \Pi( \phi)\left(\gamma_s(\varphi) \varphi_x, \gamma_s(\varphi) \varphi_x\right)-  \chi_{\omega} f_0  \gamma_s(\varphi)\\
& = - (\varphi_t)^2  \left( \gamma_{ss}(\varphi)  -   \Pi( \phi)\left( \gamma_s(\varphi), \gamma_s(\varphi)\right) \right)+  (\varphi_x)^2  \left(\gamma_{ss}(\varphi)  -   \Pi( \phi)\left(\gamma_s(\varphi), \gamma_s(\varphi)\right) \right) \\
& \;\;\;\;\;\;\;\;\;\;\;\;\;\;\;\; +  \gamma_s(\varphi) \Big(\Box \varphi-  \chi_{\omega} f_0   \Big)  \\
&=   \gamma_s(\varphi) \Big(\Box \varphi-  \chi_{\omega} f_0   \Big).
\end{align*}

Therefore, the pair $(\phi, f)$ is a solution of the controlled wave maps equation is equivalent to $(\varphi, f_0)$ is a solution of the linear controlled wave equation
\begin{equation}\label{eq:wave:linear:sec6}
\begin{cases}
\Box \varphi=  \chi_{\omega} f_0 \;\; \forall (t, x)\in (0, T)\times (0, 2\pi), \\
\varphi(t, 2\pi)- \varphi(t, 0)= 2\pi K \;\; \forall t\in (0, T).
\end{cases}
\end{equation}
This linear system is exactly controllable in the sense stated in the lemma. 
\end{proof}

Similarly, we obtain the following result on non-closed complete geodesics. Let $\Gamma$ be a non-closed complete geodesic. We denote $\bar\Gamma$ to be the one dimensional submanifold given by this closed geodesic.   The controlled wave maps equation {\it restricted} on $T\bar\Gamma$ means the initial state takes value in $T\bar\Gamma$, and the control $f$ is always tangent to the geodesic. 
\begin{cor}
\label{wave-complete-geodesic}
Let $T\geq 2\pi- b_0$.  Let $\Gamma$ be a non-closed complete geodesic. Then the same exact controllability result as in Lemma \ref{wave-closed-geodesic} holds for the controlled wave maps equation restricted in $T\bar\Gamma$.
\end{cor}

\begin{proof}[Proof of Corollary \ref{wave-complete-geodesic}]

Recall the definition of a non-closed complete geodesic. 
A {\it non-closed complete geodesic} $\Gamma$ is given by  $\bar \phi (\R)=\bar\Gamma$, where the map $\bar\phi$ 
\begin{align}
\bar \phi: \mathbb{R}&\rightarrow \bar\Gamma \subset \mathcal{N} \label{def:com:geodesic} \\
s& \mapsto \bar \phi(s)
\end{align}
satisfies 
\begin{equation}\label{def:com:geodesic:2}
   \Delta \bar \phi  -   \Pi(\bar \phi)\left(\partial_{x}\bar\phi, \partial_{x}\bar \phi\right)= 0.
\end{equation}

Suppose that the initial state $\phi[0]$ takes value in $T\bar\Gamma$ and the control stays in the tangent direction, then the solutions stays in $T\bar\Gamma$. We may assume for every $(t, x)\in [0, T]\times \mathbb{T}^1$,
\begin{gather*}
\phi(t, x)= \bar \phi(\varphi(t, x)), \\
f(t, x)=  f_0(t, x) \bar \phi_s(\varphi(t, x))\in T_{\bar \phi(\varphi)}\bar\Gamma\subset  T_{\bar \phi(\varphi)}\mathcal{N}
\end{gather*}
with functions
\begin{equation*}
    \varphi, f_0: [0, T]\times \mathbb{T}^1\rightarrow \mathbb{R}.
\end{equation*}

Similarly, $(\phi, f)$ is a solution of the controlled wave maps equation is equivalent to $(\varphi, f_0)$ is a solution of the linear controlled wave equation
\begin{equation*}
    \Box \varphi-  \chi_{\omega} f_0= 0.
\end{equation*}
Thus, the system is exactly controllable in the sense stated in the corollary.
\end{proof}

We also study the following inner controlled system on non-closed complete geodesics. While an analogue exact controllability result can be shown,  we present here only a weaker version that will be used later on. 
\begin{lemma}[Inner controlled system on non-closed complete geodesics]
\label{wave-inner-geodesic}
Let $T\geq 2\pi- b$. Let $\Gamma$ be a non-closed complete geodesic of  $\mathcal{N}$ and let $Q_0, Q_1$ be two points in this geodesic. We consider the controlled wave maps equation restricted on $T\bar\Gamma$ for $(t, x)\in [0, T]\times (b, 2\pi- b)$,
\begin{equation}\label{eq:conhmf:whole:inner}
\begin{cases}
 \Box \phi -  \Pi(\phi)\left(\partial_{\nu}\phi, \partial^{\nu}\phi\right)= \chi_{\omega_{\textrm{in}}} \Pi_T(\phi)f, \\
 \phi(t, b)= Q_0, \;\;  \phi(t, 2\pi- b)= Q_1, \\
 \phi[0](\cdot)\in T\bar\Gamma,
 \end{cases}
\end{equation}
where the initial state takes value in $T\bar\Gamma$, and the control $f$ is always tangent to the geodesic. Then the solution of the controlled wave maps stays in $T\bar\Gamma$. 

Moreover, for any point $P\in \bar\Gamma$ and any initial state $\phi[0]\in C^2\times C^1([b, 2\pi- b]; T\bar\Gamma)$, there exists a control $f\in C^0([0, T]\times [b, 2\pi- b])$ such that the unique strong solution satisfies $\phi\in C^0([0, T]; C^2\times C^1([b, 2\pi- b]; T\bar\Gamma))$
and
\begin{align*}
\phi(T, x)&= P \; \forall x\in [b_1, 2\pi- b_1], \\
\phi_t(T, x)&= 0\; \forall x\in [b, 2\pi- b].
\end{align*}
\end{lemma}
\begin{proof}
Assume this non-closed complete geodesic is given by the map $\bar \phi$ defined in \eqref{def:com:geodesic}--\eqref{def:com:geodesic:2}.
Since $\phi(0,\cdot)$ is a continuous curve on the non-closed complete geodesic $\Gamma$, we assume that
\begin{equation*}
Q_0= \phi(0, b)= \bar \phi (q_0),  \; Q_1= \phi(0, b)= \bar \phi (q_1)  \textrm{ and } P= \bar \phi(p).
\end{equation*}
Similar to the proof of Corollary \ref{wave-complete-geodesic}, we assume 
\begin{gather*}
\phi(t, x)= \bar \phi(\varphi(t, x)), \\
f(t, x)=  f_0(t, x) \bar \phi_s(\varphi(t, x))\in T_{\bar \phi(\varphi)}\bar\Gamma.
\end{gather*}
Thus the pair $(\varphi, f_0)$ satisfies the linear wave equation with Dirichlet boundary condition:
\begin{equation*}
\begin{cases}
 \Box \varphi=  \chi_{\omega_{\textrm{in}}} f_0 \;\; \forall (t, x)\in [0, T]\times (b, 2\pi- b),\\
 \varphi(t, b)= q_0 \; \textrm{ and } \;
 \varphi(t, 2\pi- b)= q_1  \;\; \forall t\in [0, T],\\
 \varphi[0]\in C^2\times C^1([b, 2\pi- b]; \mathbb{R}).
\end{cases}
\end{equation*}

It suffices to find a control $f_0$ such that at time $T$
\begin{gather*}
    \varphi(T, x)= p \; \forall x\in [b_1, 2\pi- b_1], \\
    \varphi_t(T, x)= 0 \; \forall x\in [b, 2\pi- b].
\end{gather*}
This is a direct consequence of the exact controllability of the wave equation. Indeed, we construct a steady pair $(\bar \varphi, g_0)\in  C^5\times C^3([b, 2\pi- b])$ with $g_0$ supported in $\tilde \omega_{\textrm{in}}$,
\begin{equation*}
\begin{cases}
 \Delta \bar \varphi=  \chi_{\tilde \omega_{\textrm{in}}} g_0 \;\;  \forall  x\in  (b, 2\pi- b), \\
 \bar \varphi(b)= q_0, \;\;
\bar  \varphi(2\pi- b)= q_1, \\
\bar  \varphi(x)= p \; \forall x\in [b_1, 2\pi- b_1].
\end{cases}
\end{equation*}
Then, the function $\tilde \varphi:= \varphi- \bar \varphi$ satisfies 
\begin{equation*}
\begin{cases}
 \Box \tilde \varphi=  \chi_{\omega_{\textrm{in}}} f_0- \chi_{\tilde \omega_{\textrm{in}}} g_0, \\
 \tilde \varphi(t, b)= 0, \;\;
\tilde  \varphi(t, 2\pi- b)= 0, \\
\tilde \varphi[0]= \varphi[0]- (\bar \varphi, 0) \in C^2\times C^1([b, 2\pi- b]; \mathbb{R}).
\end{cases}
\end{equation*}
This linear system with zero Dirichlet condition is null controllable. 
Therefore, we conclude the proof of Lemma~\ref{wave-inner-geodesic}.
\end{proof}

Notice that the outer equation has control over its entire domain.  Thus one has the following gluing argument. 
\begin{lemma}[Outer gluing for closed curves]\label{lem:glue}
Let $\mathcal{C}$ be a given closed curve on $\mathcal{N}$. Assume  the pair   $(\phi, \phi_t,  f): [0, T]\times (b, 2\pi- b)\rightarrow T\mathcal{N}\times \mathbb{R}^N$ belonging to  $ C\big([0, T]; C^2\times C^1\times C^0([b, 2\pi- b])\big)$ is a solution to the controlled inner equation: 
\begin{equation*}
    \Box \phi  -\Pi(\phi)\left(\partial_{\nu}\phi, \partial^{\nu}\phi\right)= \chi_{\omega_{\textrm{in}}} \Pi_T(\phi)f \; \; \forall x\in (b, 2\pi- b)
\end{equation*}
Assume that $\phi(0, b)$ and $\mathcal{C}$ belong to the same connected component of $\mathcal{N}$.
Then we can extend it to a periodic pair $(\tilde \phi, \tilde \phi_t, \tilde f): [0, T]\times \mathbb{T}^1 \rightarrow T\mathcal{N}\times \mathbb{R}^N$ belong to $C\big([0, T]; C^2\times C^1\times C^0(\mathbb{T}^1\big)$ such that
\begin{gather*}
   \Box \tilde\phi -   \Pi( \tilde\phi)\left(\partial_{\nu} \tilde\phi, \partial^{\nu} \tilde\phi\right)= \chi_{\omega} \Pi_T(\tilde \phi) \tilde f \; \; \forall x\in \mathbb{T}^1, \\
\tilde \phi(t, \cdot) \textrm{ is homotopic to  } \mathcal{C} \;\;  \forall t\in [0, T].
\end{gather*}
\end{lemma}
\begin{proof}
    It suffices to extend the function $\phi|_{x\in (b, 2\pi- b)}$ to a periodic function $\tilde \phi|_{x\in \mathbb{T}^1}$ which is homotopic to the given curve $\mathcal{C}$. Clearly, this newly constructed function satisfies the periodic controlled wave maps equation with a control $\tilde f$ that is supported in $\omega_{\textrm{in}}\cup \omega_{\textrm{out}}= \omega$.
    \end{proof}
\vspace{2mm}

To summarize, by combining Lemma \ref{wave-inner-geodesic} and Lemma \ref{lem:glue} one obtains the following property, which will play a significant role in the proof of Proposition \ref{pro:controlbetween} (more precisely in {\it Step 3}).  
\begin{prop}[Transport the mass on non-closed complete geodesics]\label{lem:trans:com:geo}
Let $Q_0, Q_1, P$ be three points on $\mathcal{N}$. Let $\Gamma$ be a non-closed complete geodesic. 
Let $T\geq 2\pi- b$. Given an initial state $(u_0, u_{0t})\in (C^2\times C^1(\mathbb{T}^1))\cap \mathcal{H}(\mathbb{T}^1; \mathcal{N})$ satisfying 
\begin{gather}
    u_0(b)= Q_0, \;\; u_0(2\pi- b)= Q_1,  \label{eq:cond:pro:6.6} \\
    (u_0(x), u_{0t}(x))\in T \bar\Gamma \;\;  \forall x\in (b, 2\pi- b), \label{eq:cond:pro:6.6:2}
\end{gather}
 we can construct a pair $(\phi,  f)$  as solution of 
\begin{gather*}
\Box \phi  - \Pi(\phi)\left(\partial_{\nu}\phi, \partial^{\nu}\phi\right)= \chi_{\omega} \Pi_T(\phi)f \; \; \forall (t, x)\in [0, T]\times \mathbb{T}^1,  \\
 (\phi(0, \cdot), \phi_t(0, \cdot))= (u_0, u_{0t})(\cdot) \;\; \forall x\in \mathbb{T}^1,
\end{gather*}
such that
\begin{gather*}
 \phi(t, b)= Q_0, \;\;  \phi(t, 2\pi- b)= Q_1 \; \; \forall t\in [0, T], \\
 \phi(t, x)\in T\bar\Gamma \; \; \forall (t, x)\in [0, T]\times (b, 2\pi- b), \\
 \phi(T, x) = P \; \; \forall x\in [b_1, 2\pi- b_1], \\ 
  \phi_t(T, x)= 0 \; \; \forall x\in \mathbb{T}^1, \\
  \phi(t, \cdot) \textrm{ is homotopic to  }  u_0(\cdot).
\end{gather*}
\end{prop}
\begin{remark}\label{lem:trans:clo:geo}
 A similar result holds if we replace the non-closed complete geodesic $\Gamma$ by a closed geodesic. 
\end{remark}

This result is based on the following intuition; see Figure \ref{fig:transgeodesic}. Suppose the initial state lies on a given non-closed complete geodesic $\Gamma$ for every $x$ in the uncontrolled domain, 
\begin{equation*}
    (\phi(0, x), \phi_t(0, x)) \in T\bar\Gamma \;\; \forall x\in (b, 2\pi- b),
\end{equation*}
and suppose that the target state satisfies
\begin{align*}
    (\phi_1(x), \phi_{1t}(x)) \in T\bar\Gamma \;\; \forall x\in (b_1, 2\pi- b_1).
\end{align*}
Then there exists a control such that the final state coincides with $(\phi_1(x), \phi_{1t}(x))$ for $x\in (b_1, 2\pi- b_1)$, and that the final state is homotopic to the initial state.   Here we only proved the special case that $(\phi_1(x), \phi_{1t}(x))= (P, 0) \; \forall x\in (b_1, 2\pi- b_1)$.
\begin{figure}[H]
    \centering
    \includegraphics[width=1.0\linewidth]{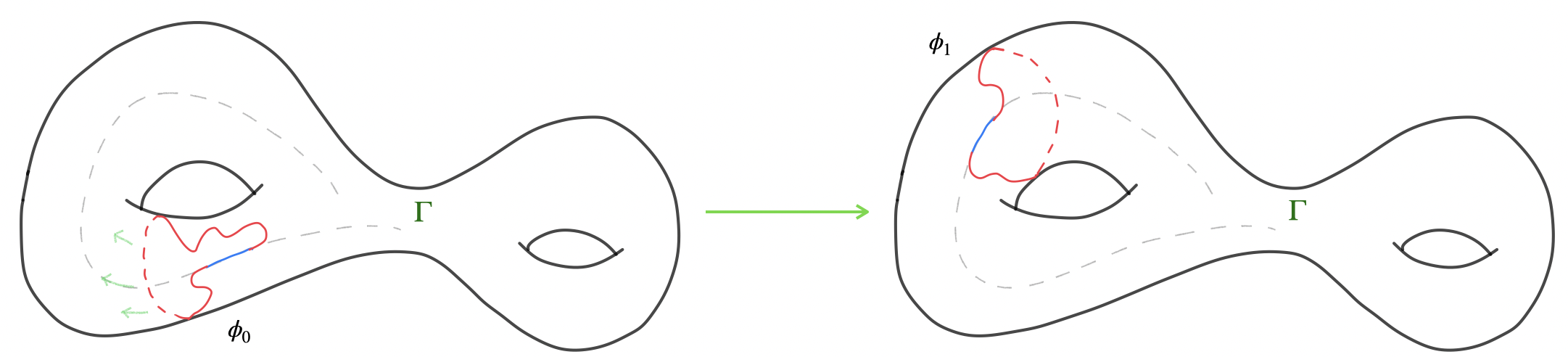}
    \caption{ Consider the map $\phi_0: \mathbb{T}^1 \rightarrow \mathcal{N}$ shown on the left-hand side.  The blue  part corresponds to the map on the uncontrolled region, $x\in (b, 2\pi- b)$, while the red part corresponds to the map on the controlled region, $x\in (-b, b)$. On the uncontrolled domain, the deviation of the map $\phi_{0x}$ is small. By transporting mass along the non-closed complete geodesic $\Gamma$ via the inner equation and performing a gluing with the outer equation, the system evolves from the initial state  $(\phi_0, \phi_{0t})$ to the final state $(\phi_1, \phi_{1t})$, where $\phi_0$ and $\phi_1$ are homotopic. }
    \label{fig:transgeodesic}
\end{figure}

\subsection{The global controllability between closed geodesics}\label{subsec:gloconbewgeo}
Armed with the preceding auxiliary results, we are in a position to construct the explicit trajectory that connects two homotopic closed geodesics. The proof is a combination of Lemmas \ref{lem:conne:points}--\ref{wave-closed-geodesic}, and  Proposition \ref{lem:trans:com:geo}.  The complete process is illustrated in Figure \ref{fig:glo:control:geodesic}. 
\begin{proof}[Proof of Proposition \ref{pro:controlbetween}]
The construction is composed by 5 steps. Given two closed geodesics $\gamma_0, \gamma_1: \mathbb{T}^1\rightarrow \mathcal{N}$ that are homotopic.  Select two points on the given closed geodesics: $Q_0\in \gamma_0$ and $Q_1\in \gamma_1$.  According Lemma \ref{lem:conne:points} there exists either a non-closed complete geodesic or a closed geodesic connecting $Q_0$ and $Q_1$.  We assume here it is a non-closed complete geodesic $\Gamma$.  The case of a closed geodesic can be treated similarly, by replacing Proposition \ref{lem:trans:com:geo} with Remark \ref{lem:trans:clo:geo}. 
\vspace{2mm}

\noindent {\it Step 1. Control on the closed geodesic $\gamma_0$ with mass concentrating on $P_0$. } 
Let $T_1:= 2\pi- b$.  Since the initial state is $(\gamma_0, 0)$, by applying Lemma \ref{wave-closed-geodesic} we find a control to steer the state to  $\phi[T_1]\in C^2\times C^1(\mathbb{T}^1)$ such that the following holds:
\begin{gather*}
    \phi(T_1, x)\in \gamma_0  \;\textrm{ and } \; \phi_t(T_1, x)= 0  \; \; \forall x\in \mathbb{T}^1, \\
    \phi(T_1, x)= P_0 \;\; \forall x\in [b, 2\pi- b_1], \\
    \phi(T_1, \cdot) \textrm{ is homotopic to  }  \gamma_0.
\end{gather*}
In this state, the solution concentrates at a single point $P_0$ for most values of $x$.
\vspace{2mm}

\noindent {\it Step 2. Move towards the non-closed complete geodesic $\Gamma$ while keeping the mass concentrated on $P_0$. }  This step relies on the gluing Lemma \ref{lem:glue} concerning the outer equation.  Let $T_2= 2\pi$. Since $(\phi, \phi_t)(T_1, x)= (P_0, 0) \; \forall x\in [b, 2\pi- b_1]$, we can find a control such that the state becomes  $\phi[T_2]\in C^2\times C^1(\mathbb{T}^1)$:
\begin{gather*}
  \phi_t(T_2, x)= 0  \; \; \forall x\in \mathbb{T}^1, \\
  \phi(T_2, x)\in  \Gamma \;\; \forall x\in [b, 2\pi- b], \\
    \phi(T_2, x)= P_0 \;\; \forall x\in [b, 2\pi- b_1], \\
   \phi(T_2, 2\pi- b)= P_1, \\
    \phi(T_2, \cdot) \textrm{ is homotopic to  }  \gamma_0.
\end{gather*}
Again, the solution concentrates at a single point $P_0$ for most values of $x$.
\vspace{2mm}

\noindent {\it Step 3. Transport the mass from $P_0$ to $P_1$ on the non-closed complete geodesic $\Gamma$. }  Let $T_3= 4\pi- b$. Note that the state $\phi[T_2]$ satisfies the conditions \eqref{eq:cond:pro:6.6}--\eqref{eq:cond:pro:6.6:2} in Proposition \ref{lem:trans:com:geo} with $P_0= Q_0$ and $P_1= Q_1$.
Thus by applying Proposition \ref{lem:trans:com:geo} we can construct a control such that the solution becomes $\phi[T_3]\in C^2\times C^1(\mathbb{T}^1)$:
\begin{gather*}
 \phi(T_3, b)= P_0 \; \textrm{ and } \;  \phi(T_3, 2\pi- b)= P_1, \\
 \phi(T_3, x) = P_1 \; \; \forall x\in [b_1, 2\pi- b_1], \\ 
 \phi(T_3, x)\in  \bar\Gamma \;\; \forall x\in [b, 2\pi- b], \\
  \phi_t(T_3, x)= 0 \; \; \forall x\in \mathbb{T}^1,  \\
    \phi(T_3, \cdot) \textrm{ is homotopic to  }  \gamma_0.
\end{gather*}
Moreover, since the two closed geodesics $\gamma_0$ and $\gamma_1$ are homotopic,
\begin{equation*}
       \phi(T_3, \cdot) \textrm{ is homotopic to  } \gamma_1.
\end{equation*}
Thanks to this step, the solution is transported from $P_0$ to $P_1$ for most values of $x$.
\vspace{2mm}

\noindent {\it Step 4. Move towards the closed geodesic $\gamma_1$ while keeping the mass concentrated on $P_1$.  }   This step is similar to {\it Step 2}.  Let $T_4= 4\pi$.
Using again the gluing Lemma \ref{lem:glue} concerning the outer equation.   Since $(\phi, \phi_t)(T_3, x)= (P_1, 0) \; \forall x\in [b_1, 2\pi- b_1]$, and the controlled domain is $\omega= [- b_0, b_0]$, we can find a control such that the state becomes  $\phi[T_4]\in C^2\times C^1(\mathbb{T}^1)$:
\begin{gather*}
  \phi_t(T_4, x)= 0  \; \; \forall x\in \mathbb{T}^1, \\
  \phi(T_4, x)\in  \gamma_1 \;\; \forall x\in \mathbb{T}^1, \\
    \phi(T_4, x)= P_1 \;\; \forall x\in [b_1, 2\pi- b_1], \\
    \phi(T_4, \cdot) \textrm{ is homotopic to  }  \gamma_1.
\end{gather*}
Hence, the state is now contained in the curve described by $\gamma_1$.
\vspace{2mm}

\noindent {\it Step 5. Control on the closed geodesic $\gamma_1$. } 
Let $T_5= 6\pi$. Finally, since $\phi[T_4]\in T\bar\gamma_1$, we can apply Lemma \ref{wave-closed-geodesic} to find control such that the solution becomes $\phi[T_5]= (\gamma_1, 0)$. 

This finishes the proof of Proposition \ref{pro:controlbetween}.

\end{proof}

\section{Global exact controllability of wave maps}\label{Sec:5}

This section is devoted to the proof of Theorem  \ref{thm1}. 
As illustrated in Section \ref{sec:strategy1}, the proofs are based on the return method,  together with the three intermediate results Theorem \ref{thm:dynamics}, Propositions \ref{pro:controlaroundgeodesic}--\ref{pro:controlbetween}. While the main issue is on the construction of the return trajectory.

\subsection{A well-designed return trajectory for any given state}\label{subsec:globalcontrol:1}

This return trajectory is designed to prove  Proposition \ref{pro:controlaroundcurve}.
Before presenting its proof, we first show that it easily yields the global exact controllability. For any given initial and final states $(\phi_0, \phi_{0t})$ and $(\phi_1, \phi_{1t})$, we can find a continuous trajectory $(\phi_s, \phi_{st})$ for $s\in (0, 1)$. By Proposition \ref{pro:controlaroundcurve} the wave maps equation is controllable around $(\phi_s, \phi_{st})$ for any $s\in (0, 1)$. Then we can move the state along a sequence of states on this given trajectory to achieve the final state. Note that this deformation relies on a compactness argument. 
Thus, the required control time depends on $(\phi_0, \phi_{0t})$ and $(\phi_1, \phi_{1t})$ instead of $M$.

To obtain a uniform control time only depending on $M$ in Theorem \ref{thm1}, one needs to improve Proposition \ref{pro:controlaroundcurve} at a lower regularity  to gain some compactness. Namely, one shall improve Theorem \ref{thm:dynamics} and Propositions \ref{pro:controlaroundgeodesic}.  This task is feasible although it demands additional technical work. Another idea is to first obtain global controllability on a compact set (with respect to $\mathcal{H}$-topoogy), and then benefit on the better regularity of closed geodesics and Theorem \ref{thm:dynamics}.

Indeed, for any given initial and final states  $(\phi_0, \phi_{0t}), (\phi_1, \phi_{1t})$ with energy smaller than $M$. Thanks to Theorem \ref{thm:dynamics}  and Proposition \ref{pro:controlaroundgeodesic}, we can find trajectories to connect 
\begin{gather*}
    (\phi_0, \phi_{0t})\longrightarrow (\gamma_0, 0)   \; \textrm{ and } \;
    (\gamma_1, 0)\longrightarrow  (\phi_1, \phi_{1t}),
\end{gather*}
in a uniform time $(0, T)$, where $\gamma_0, \gamma_1$ are two closed geodesics. Note that the $H^2$-norm of closed geodesics with energy smaller than $M$ is uniformly bounded by $M_0$.  There exists a constant $C(M_0)$ such that any two closed curves with $H^2$-norm smaller than $M_0$ can continuously deform from one to another, and during this deformation the curve always has $H^2$-norm smaller than $C(M_0)$.  Due to the compactness embedding, there are finitely many open balls with radius $\delta_3(C(M_0))$  in $H^1$-topology that cover all closed curves with $H^2$-norm smaller than $C(M_0)$. By Proposition \ref{pro:controlaroundcurve}, for every two curves $l_1$ and  $l_2$ inside the same open ball, there exists a control that steers the state from $(l_1, 0)$ to $(l_2, 0)$ during time $64\pi$. Assume the number of balls for the above open covering is $K$, then set $T(M)$ as $64K \pi$. Therefore, we can construct a control in time interval $T(M)$ such that the solution has initial state $(\gamma_0, 0)$ and final state $(\gamma_1, 0)$.  Thus this finishes the proof.

\vspace{3mm}

Now, let us return to the proof of Proposition \ref{pro:controlaroundcurve}, which is divided into two steps.

{\it Step 1. On the construction of a return trajectory from $(\phi_0, \phi_{0t})$ to $(\phi_0, \phi_{0t})$.}

According to Proposition \ref{pro:controlaroundgeodesic}, there exists $\delta_1= \delta_1(M)>0$ such that for any closed geodesic $\gamma$ with energy smaller than $M$, the system is locally controllable around $(\gamma, 0)$ in time period $(0, 64\pi)$. 
For this value $\delta_1$, thanks to  Theorem \ref{thm:dynamics}, there exist some  $T_c= T_c(M, \delta_1)$ such that:
\begin{itemize}[leftmargin=3em]
    \item[\tiny$\bullet$] for the fixed state $(\phi_0, \phi_{0t})$ with energy smaller than $M$, there exist a closed geodesic $\gamma$ and $T_0\leq T_c$ such that the unique solution of the locally damped equation with  initial state $(\phi_0, \phi_{0t})$, which we denote by $\phi$,  satisfies 
\begin{gather*}
   \|\phi[T_0]- (\gamma, 0)\|_{\mathcal{H}}\leq \frac{\delta_1}{2}. 
\end{gather*}
And, due to  Proposition \ref{pro:controlaroundgeodesic},  there exists a control  $f_0\in L^{\infty}(T_0, T_0+ 64\pi; L^2(\mathbb{T}^1))$ to steer the state from $\phi[T_0]$ to $(\gamma, 0)$.
 \item[\tiny$\bullet$] for the fixed state $(\phi_0, -\phi_{0t})$ with energy smaller than $M$, there exist a closed geodesic $\tilde \gamma$ and $T_1\leq T_c$  such that the unique solution of the locally damped equation with  initial state $(\phi_0, -\phi_{0t})$, which we denote by $\tilde \phi$,  satisfies 
\begin{gather*}
   \|\tilde \phi[T_1]- (\tilde \gamma, 0)\|_{\mathcal{H}}\leq \frac{\delta_1}{2}. 
\end{gather*}
And, due to  Proposition \ref{pro:controlaroundgeodesic},  there exists a control  $f_1\in L^{\infty}(T_1, T_1+ 64\pi; L^2(\mathbb{T}^1))$ to steer the state from $\tilde \phi[T_1]$ to $(\gamma, 0)$.
\end{itemize}

Now we construct the explicit control $\bar h$ and denote the solution by $\bar \phi$.  First, during the time interval $(0, T_0)$ we apply the localized damping control to make the solution close to $(\gamma, 0)$; then, during $(T_0, T_0+ 64\pi)$, we apply the explicit control $h_0$ to steer the state to $(\gamma, 0)$. 
\begin{equation*}
    \bar f(t, x):= 
    \begin{cases}
    a(x) \bar\phi_t(t, x)\;\; \forall x\in \mathbb{T}^1 \;\;  \forall t\in (0, T_c), \\
    f_0(t, x)\;\; \forall x\in \mathbb{T}^1,  \forall t\in (T_c, T_c+ 64\pi).
    \end{cases}
\end{equation*}

Next, during the time $(T_0+ 64\pi, 2T_0+ 128\pi)$ we benefit on the time-reversal property of the wave maps and define the state as
\begin{equation*}
    \bar \phi(t, x):= \bar \phi(2T_0+ 128\pi- t, x),
\end{equation*}
and the control as 
\begin{equation*}
    \bar f(t, x):= 
    \begin{cases}
    f(2T_0+ 128\pi- t, x)\;\; \forall x\in \mathbb{T}^1,  \forall t\in (T_0+ 64\pi, T_0+ 128\pi), \\
     -a(x) \bar\phi_t(t, x)\;\; \forall x\in \mathbb{T}^1,  \forall t\in (T_0+ 128\pi, 2T_0+ 128\pi).
    \end{cases}
\end{equation*}
Hence, at time $t= 2T_0+ 128\pi$ the state becomes 
\begin{equation*}
    (\bar\phi, \bar \phi_t)(2T_0+ 128\pi) = (\phi_0, -\phi_{0t}).
\end{equation*}

Finally, we perform the above construction again, to steer the state from $(\phi_0, -\phi_{0t})$ to $(\phi_0, \phi_{0t})$ via the closed geodesic $(\tilde \gamma, 0)$ during the time interval $(2T_0+ 128\pi, 2T_0+ 128\pi+ 2T_1+ 128\pi)$.

Define $\bar T= 2T_0+ 2T_1+ 256\pi$.  Clearly, this trajectory satisfies 
\begin{gather*}
   \bar \phi[0]= \bar \phi[\bar T]= (\varphi_0, \varphi_{0t}) \textrm{ and } \\
   \bar \phi[T_0+ 64\pi]= (\gamma, 0), \;  \bar \phi[2T_0+ 128\pi]= (\varphi_0, -\varphi_{0t}), \\
   \bar \phi[2T_0+ 128\pi+ T_1+ 64\pi]= (\tilde \gamma, 0).
\end{gather*}

{\it Step 2. On the controllability around the return trajectory.}

The analysis concerning a direct consideration of the controllability around $\bar \phi$ is rather involved. Instead, we benefit from the continuous dependence property and the local exact controllability around closed geodesics. 

By Lemma \ref{lem-conti-dep-inh} there exists some $\delta>0$ such that for any initial state $(\tilde \phi_0, \tilde \phi_{0t})$ satisfying 
\begin{equation*}
    \|(\tilde \phi_0, \tilde \phi_{0t})- (\phi_0, \phi_{0t})\|_{\mathcal{H}}\leq \delta,
\end{equation*}
the unique solution $\tilde \phi$ of the damped equation 
\eqref{eq:dwm:ori} with this initial state satisfies 
\begin{gather*}
   \|(w_x, w_t, w)\|_{L^{\infty}_t L^2_x(Q)}+  \|w_u\|_{L^{2}_u L^{\infty}_v\cap L^{\infty}_v L^{2}_u(Q)}+  \|w_v\|_{L^{2}_v L^{\infty}_u\cap L^{\infty}_u L^{2}_v(Q)} \\
   \leq C \lVert w[0]\lVert_{\mathcal{H}}\leq \frac{\delta_1}{100} ,
\end{gather*}
where $w:= \phi- \bar\phi, Q= [0, T_0]\times \mathbb{T}^1$. In particular, since the constructed trajectory  $\bar \phi$ is close to $(\gamma, 0)$, we have 
\begin{equation*}
      \|\tilde \phi [T_0]- (\gamma, 0)\|_{\mathcal{H}}\leq \delta_1. 
\end{equation*}
According to the definition of $\delta_1$, using Proposition \ref{pro:controlaroundgeodesic} we find a control in period $(T_0, T_0+ 64\pi)$ such that 
\begin{equation*}
    \tilde \phi[T_0+ 64\pi]= (\gamma, 0).
\end{equation*}

Similarly, any state $(\tilde \phi_1, \tilde \phi_{1t})$ satisfying 
\begin{equation*}
    \|(\tilde \phi_1, \tilde \phi_{1t})- (\phi_0, \phi_{0t})\|_{\mathcal{H}}\leq \delta,
\end{equation*} 
we can construct a controlled trajectory in period $(0, T_0+ 64\pi)$ with the initial state $(\tilde \phi_1, \tilde \phi_{1t})$ and final state $(\gamma,0)$. With the time-reversal property, we can construct a controlled trajectory satisfying 
\begin{equation*}
   \tilde  \phi[0]= (\tilde \phi_0, \tilde \phi_{0t}),  \; \;
  \tilde \phi[T_0+ 64\pi]= (\gamma, 0),
   \; \; \tilde \phi[2T_0+ 128 \pi]= (\tilde\phi_1, -\tilde \phi_{1t}).
\end{equation*}
By repeating the above construction on time interval $(2T_0+ 128 \pi, 2T_0+ 128 \pi+ 2T_1+ 128 \pi)$ we can let 
\begin{equation*} 
  \tilde \phi[2T_0+ 128 \pi+ T_1+ 64\pi]= (\tilde \gamma, 0),\;
    \; \tilde \phi[2T_0+ 128 \pi+ 2T_1+ 128\pi]= (\tilde\phi_1, \tilde \phi_{1t}).
\end{equation*}
This finishes the proof of Proposition \ref{pro:controlaroundcurve}.

\subsection{A special trajectory connecting any two given states}\label{subsec:globalcontrol:2}

Again, we adapt the idea of return method to construct a well-designed trajectory such that the system is controllable around it. The proof is similar to the first case, except that now we shall also benefit from the global controllability  result established in  Stage 3 to find a more efficient control and to obtain uniform control time. 
\vspace{2mm}

{\it Step 1. On the construction of a  trajectory from $(\phi_0, \phi_{0t})$ to $(\phi_1, \phi_{1t})$.}

Let $M>0$. Recall the definition of $\delta_1= \delta_1(M)$ and $T_c= T_c(M, \delta_1)$ in {\it Step 1} of Section \ref{subsec:globalcontrol:1}. Thus, for given states $(\phi_0, \phi_{0t})$ and $(\phi_1, \phi_{1t})$ with energy smaller than $M$, 
\begin{itemize}[leftmargin=3em]
    \item[\tiny$\bullet$] there exist a closed geodesic $\gamma$ and $T_0\leq T_c$ such that the unique solution of the locally damped equation with  initial state $(\phi_0, \phi_{0t})$, which we denote by $\phi$,  satisfies 
\begin{gather*}
   \|\phi[T_0]- (\gamma, 0)\|_{\mathcal{H}}\leq \frac{\delta_1}{2}. 
\end{gather*}
And, due to  Proposition \ref{pro:controlaroundgeodesic},  there exists a control  $f_0\in L^{\infty}(T_0, T_0+ 64\pi; L^2(\mathbb{T}^1))$ to steer the state from $\phi[T_0]$ to $(\gamma, 0)$.
 \item[\tiny$\bullet$] There exist a closed geodesic $\tilde \gamma$ and $T_1\leq T_c$  such that the unique solution of the locally damped equation with  initial state $(\phi_1, -\phi_{1t})$, which we denote by $\tilde \phi$,  satisfies 
\begin{gather*}
   \|\tilde \phi[T_1]- (\tilde \gamma, 0)\|_{\mathcal{H}}\leq \frac{\delta_1}{2}. 
\end{gather*}
And, due to  Proposition \ref{pro:controlaroundgeodesic},  there exists a control  $f_1\in L^{\infty}(T_1, T_1+ 64\pi; L^2(\mathbb{T}^1))$ to steer the state from $\tilde \phi[T_1]$ to $(\gamma, 0)$.

\item[\tiny$\bullet$] Moreover, thanks to Proposition \ref{pro:controlbetween}, there exists a control $ f_2\in L^{\infty}(T_0+ 64\pi,  T_0+ 70\pi; L^2(\mathbb{T}^1))$ to steer the state from $(\gamma, 0)$ to $(\tilde \gamma, 0)$. 
\end{itemize}
\vspace{2mm}

After the above preparation,  we are in a position to construct the following explicit control,
\begin{equation*}
    \bar f(t, x):= 
    \begin{cases}
     a(x) \bar\phi_t(t, x) \textrm{ for }  t\in (0, T_0), \\
     f_0(t, x) \textrm{ for }  t\in (T_0, T_0+ 64\pi),\\
     f_2(t, x) \textrm{ for }  t\in (T_0+ 64\pi,  T_0+ 70\pi), \\
       f_1(T_0+ 134\pi+ T_1- t, x)  \textrm{ for }  t\in (T_0+ 70\pi, T_0+ 134\pi), \\
      -a(x) \bar\phi_t(t, x) \textrm{ for }   t\in (T_0+ 134\pi, T_0+ 134\pi+  T_1), 
    \end{cases}
\end{equation*}
such that the unique solution $\bar \phi$ with the initial state $\bar \phi[0]= (\phi_0, \phi_{0t})$ satisfies the following
\begin{gather*}
 \bar \phi[0]= (\phi_0, \phi_{0t}), \\ \bar \phi[T_0]= \phi[T_0] \textrm{ which is close to } (\gamma, 0),  \\
 \bar \phi[T_0+ 64\pi]= (\gamma, 0), \\
  \bar \phi[T_0+ 70\pi]= (\tilde \gamma, 0), \\
   \bar \phi[T_0+ 134\pi]= (\tilde \phi, -\tilde \phi_t)(T_1) \textrm{ which is close to } (\tilde \gamma, 0), \\
    \bar \phi[T_0+ 134\pi+  T_1]= ( \phi_1,  \phi_{1t}). 
\end{gather*}

{\it Step 2. On the controllability around the return trajectory.}

Similar to {\it Step 2} of Section \ref{subsec:globalcontrol:1}, this step is a direct consequence of  the continuous dependence property and the local exact controllability around closed geodesics. Thus we omit it.

\section{Exponential stability around closed geodesic with negative sectional curvature}\label{Sec:8}

In this section we present the proof of Theorem \ref{thm:stability}, as illustrated in Section \ref{sec:strategy2},  following a five-step strategy :
\begin{itemize}
    \item[(1)] {\it Decompose the state $\phi$ as $(\varphi, \alpha)$ around the geodesic};  See Lemma \ref{lem:def:full:decom} in Section \ref{subsec:decompositiongeode}.
\vspace{1mm}
   
 \item[(2)] {\it Express the full system on $(\varphi, \alpha)$};  See Proposition \ref{lem:full:system:varalpha} in Section \ref{subsec:full system varphialpha}.
\vspace{1mm} 

  \item[(3)] {\it A coercive estimate around geodesic with negative curvature};  See Proposition \ref{prop:coercivityofL} in Section \ref{sec:coecive}
  \vspace{1mm}

     \item[(4)] {\it  Stability of the linearized equation on $\Psi$};  See Proposition \ref{prop:propagationofsmallness} in Section \ref{subsec:linearstabilityvarphi}.
\vspace{1mm}

    \item[(5)] {\it Exponential stability of the full system}; see Section \ref{subsec:expostabproof}.

\end{itemize}

\subsection{Decomposition around geodesics: a shifted projection }\label{subsec:decompositiongeode}
To investigate states $\phi(t, x)$ sufficiently close to a given geodesic $\gamma$, we shall perform a decomposition of $\phi$. It is natural to 
write
\begin{gather*}
    \phi(t, x)= \gamma(x)+ \varphi(t, x)+ \varphi_1(t, x) \textrm{ with } \\
    \varphi(t, x)\in T_{\gamma(x)}\mathcal{N}, \; \varphi_1(t, x)\in N_{\gamma(x)}\mathcal{N}.
\end{gather*}
However, this decomposition is not adequate, as we shall require  the tangent component $\varphi$ to satisfy the following {\it orthogonality/rigidity condition}, 
\begin{equation}
    \langle \varphi(t, \cdot), \gamma_x (\cdot)\rangle_{L^2(\mathbb{T}^1)}= 0
\end{equation}
As we will see later on in Section \ref{sec:coecive}, this condition is necessary to obtain coercive estimates for $\langle \mathcal{L}_{\gamma} \varphi, \varphi\rangle_{L^2(\mathbb{T}^1)}$ which we introduced in Definition \ref{def:jacobian}.
\vspace{3mm}

To enforce the preceding rigidity condition, we note that using an implicit function theorem,  for any state $\phi(t, \cdot)$ which is close to $\gamma(\cdot)$ in the $H^1$-topology, one can find a unique $\alpha(t)\in \mathbb{T}^1$, such that 
\begin{equation}\label{eq:orthocondition}
    \big\langle \phi(t, \cdot)- \gamma(\cdot+ \alpha(t)), \gamma_x(\cdot) \big\rangle_{L^2(\mathbb{T}^1)}= 0,
\end{equation}
with
$$|\alpha(t)|\lesssim \|\phi(t)- \gamma\|_{L^1(\mathbb{T}^1)}.$$ Indeed, there is a unique $\alpha(t)$ with $|\alpha(t)|\ll1$ and such that
\begin{equation}\label{es:12321}
    \big\langle \phi(t, \cdot)- \gamma(\cdot), \dot\gamma(\cdot) \big\rangle_{L^2(\mathbb{T}^1)}= \big\langle -\gamma(\cdot)+ \gamma(\cdot+ \alpha(t)), \gamma_x(\cdot) \big\rangle_{L^2(\mathbb{T}^1)}.
\end{equation}
By differentiating the preceding equation with respect to time, we also know that 
\begin{equation*}
   |\dot \alpha|\lesssim \|\phi_t\|_{L^1(\mathbb{T}^1)}. 
\end{equation*}

A variant of the preceding decomposition of $\phi(t, x)$ then is to write
\begin{gather*}
    \phi(t, x)= \gamma(x+ \alpha(t))+ \varphi(t, x)+ \varphi_1(t, x) \textrm{ with } \\
    \varphi(t, x)\in T_{\gamma(x+ \alpha(t))}\mathcal{N},  \; \varphi_1(t, x)\in N_{\gamma(x+ \alpha(t))} \mathcal{N}.
\end{gather*}
This time the evolution of both $\varphi$ and $\varphi_1$ will become harder to express, and in fact, ensuring that $\varphi(t, x)$ stays in $T_{\gamma(x)}\mathcal{N}$ is important. Therefore, we propose yet another decomposition that combines the advantages of the preceding two decompositions, enunciated in Lemma \ref{lem:def:full:decom}. It is motivated by the following simple observation.
\begin{lemma}\label{lem:genprojection}
    Let $\mathcal{N}\subset \mathbb{R}^N$ be a compact Riemannian submanifold. There are constants $c, C>0$ such that 
    for any $p, q, r\in \mathcal{N}$ satisfying 
    \begin{equation*}
        |q- p|+ |r- p|\leq c,
    \end{equation*}
    there is a unique decomposition  
\begin{equation}\label{decom:first}
        r= q+ \psi+ \psi_1, \; \psi\in T_{p}\mathcal{N}, \psi_1\in N_p \mathcal{N}.
    \end{equation}
     The functions $(\psi, \psi_1)$ satisfy
    \begin{gather}
   C^{-1} |q-r|\leq  |\psi|\leq C |q-r|, \label{ine:es1psi} \\
        |\psi_1|\leq C(|\psi|^2+ |q- p||\psi|). \label{ine:es2psi1}
    \end{gather}

    Alternatively, there is a function $F$ depending smoothly on $p, q\in \mathcal{N}$ and $\psi\in T_p\mathcal{N}$ with values in $N_p\mathcal{N}$: 
    \begin{equation}\label{eq:F:ine}
   F(p, q; \psi)\in N_p\mathcal{N}, \;\;   |F(p, q; \psi)|\leq C(|\psi|^2+ |q- p||\psi|),  
    \end{equation}
    such that for any $p, q, r\in \mathcal{N}$ 
     there is a unique decomposition 
    \begin{equation}
        r= q+ \psi+ F(p, q; \psi), \; \psi\in T_{p}\mathcal{N},  \;\psi_1:= F(p, q; \psi)\in N_p \mathcal{N},
    \end{equation}
    and that for any $p, q$ closed enough and $\psi\in T_p\mathcal{N}$ small enough, 
    \begin{equation*}
        q+ \psi+ F(p, q; \psi)\in \mathcal{N}.
    \end{equation*}
    Furthermore, the function $F$ can be extended to a smooth function on $p, q\in \mathbb{R}^N$ and $\psi\in \mathbb{R}^N$ satisfying
     \begin{gather}
    |F|(p, q; \psi)\leq C(|\psi|^2+ |q- p||\psi|),\label{eq:es:F1} \\
   |\nabla_{p, q, \psi}F|(p, q; \psi)\leq C (|\psi|+ |q- p|),\; \; 
   |\nabla_{p, q, \psi}^2 F|(p, q; \psi)\leq C.   \label{eq:es:F2}
    \end{gather}
\end{lemma}

\begin{proof}[Proof of Lemma \ref{lem:genprojection}]
    The proof of the first part is straightforward,  it suffices to select $\psi$ as the projection of $r- q$ on $T_p\mathcal{N}$ and further define $\psi_1 \in N_p \mathcal{N}$. This decomposition is obviously unique.  Inequalities \eqref{ine:es1psi}--\eqref{ine:es2psi1} can be checked directly.   
    
    Next, we turn to the second part of this lemma.  Provided the condition that $q, r$ are sufficiently close to $p$, we notice that for fixed $p$ and $q$ there is a bijection between $r$ and $\psi\in T_p \mathcal{N}$. Indeed, for any $r\in \mathcal{N}$ closed enough to $p$ we define $\psi$ as the projection of $r- q$ on $T_p\mathcal{N}$, this is a locally  smooth diffeomorphism. Conversely,  using the implicit function theorem,  any $\psi\in T_p \mathcal{N}$  uniquely determines a point $r\in \mathcal{N}$: $r= G(p, q; \psi)$. Moreover, by the choice of $r$, $\psi$ is the projection of $r- q$ on $T_p\mathcal{N}$. We can therefore construct the function $F(p, q; \psi):= G(p, q; \psi)- q- \psi$. Thus 
    \begin{equation*}
        G(p, q; \psi)= q+ \psi+ F(p, q; \psi), \psi\in T_{p}\mathcal{N},  \; F(p, q; \psi)\in N_p \mathcal{N}.
    \end{equation*}
    Thanks to the uniqueness of the decomposition in forms of \eqref{decom:first}, $F(p, q; \psi)$ is exactly the value of $\psi_1$ in \eqref{decom:first} with $r= G(p, q; \psi)$, $p$ and $q$. Thus, the estimates \eqref{ine:es1psi}--\eqref{ine:es2psi1} yield \eqref{eq:F:ine}. Finally, we can extend the function $F$ to $p, q, \psi\in \mathbb{R}^N$, and this extended function satisfies the estimates \eqref{eq:es:F1}--\eqref{eq:es:F2}.
\end{proof}

\begin{defi}\label{def:F2}
Fix a closed geodesic $\gamma: \mathbb{T}^1\longrightarrow \mathcal{N}$. Let $\delta> 0$ be a sufficiently small constant.  We define a function $\mathcal{F}$ in terms of $F$ constructed in Lemma \ref{lem:genprojection}:
\begin{equation*}
    \mathcal{F}: (x;  \varphi, \alpha)\in \mathbb{T}^1\times \mathbb{R}^N\times (- \delta, \delta)\mapsto \mathcal{F}(x; \varphi, \alpha):= F(\gamma(x), \gamma(x+ \alpha); \varphi), 
\end{equation*} 
in particular, when restricting $\mathcal{F}$ on $T\mathcal{N}\times (- \delta, \delta)$ we obtain
\begin{equation*}
    \mathcal{F}: (x;  \varphi, \alpha)\in \mathbb{T}^1\times T_{\gamma(x)}\mathcal{N}\times (- \delta, \delta)\mapsto \mathcal{F}(x; \varphi, \alpha)\in N_{\gamma(x)}\mathcal{N}. 
\end{equation*} 
\end{defi}
Clearly, $\mathcal{F}$ is smooth with respect to $(x, \varphi, \alpha)\in \mathbb{T}^1\times \mathbb{R}^N\times (- \delta, \delta)$, and it satisfies 
\begin{equation}\label{esonF2}
\begin{cases}
        |\mathcal{F}|+ |\partial_x \mathcal{F}|+ |\partial_{xx} \mathcal{F}| \lesssim |\varphi|^2+ |\varphi||\alpha|,  \\
        |\partial_{\varphi^l} \mathcal{F}|+  |\partial_{\alpha} \mathcal{F}|+  |\partial_{x, \varphi^l}^2 \mathcal{F}|+  |\partial_{x, \alpha}^2 \mathcal{F}|\lesssim |\varphi|+ |\alpha|, \\
        |\nabla_{\varphi, \alpha}^2 \mathcal{F}|\lesssim 1, 
        \end{cases}
\end{equation}
provided that $|\varphi|+ |\alpha|\ll 1$.
In fact, the parameter $\alpha$ will be chosen such that the orthogonality condition \eqref{eq:orthocondition} holds. Combining \eqref{eq:orthocondition} forcing the unique choice of $\alpha$, Lemma \ref{lem:genprojection} on the projection, and Definition \ref{def:F2} concerning $\mathcal{F}$,  one immediately obtains the following result:
\begin{lemma}\label{lem:def:full:decom}
    Let $\gamma: \mathbb{T}^1\longrightarrow \mathcal{N}$ be a closed geodesic. Let $\mathcal{F}$ given in Definition \ref{def:F2}. There exist  effectively computable $c, C>0$ such that for every closed curve $\phi: \mathbb{T}^1\longrightarrow \mathcal{N}$ satisfying 
    \begin{equation}\label{es:unif:curvebound}
        \|\phi- \gamma\|_{H^1(\mathbb{T}^1)}\leq c,
    \end{equation}
    there is a unique decomposition 
    \begin{equation}\label{eq:def:full:decom}
        \phi(x)= \gamma(x+ \alpha)+ \varphi(x)+ \varphi_1(x) \;\; \forall  x\in \mathbb{T}^1,
    \end{equation}
    such that the following condition holds:
    \begin{equation*}
    \begin{split}
        (P1)& \;\;\; |\alpha|\leq Cc, \; \; \; \varphi(x)\in T_{\gamma(x)}\mathcal{N}\; \textrm{ and }  \; \varphi_1(x)\in N_{\gamma(x)}\mathcal{N} \; \; \forall x\in \mathbb{T}^1,\;\;\;\;\;\;\;\;\;\;\;\;\;\;\;\;\;\;\;\; \;\;\;\;\;\;\;\;\;\;\;\;\;\;\;\;\; \\
        (P2)& \;\;\;  \varphi_1(x)=   \mathcal{F}(x; \varphi, \alpha), 
\;\;\;\;\;\;\;\;\;\;\;\;\;\;\;\;\;\;\;\;\;\;\;\;\;\;\;\\
        (P3)&  \;\;\; \big\langle \varphi(\cdot), \gamma_x(\cdot) \big\rangle_{L^2(\mathbb{T}^1)}= 0.  
    \end{split}
    \end{equation*}
Moreover, one has the following estimates:
\begin{gather}\label{es:alpvarbound1}
      |\alpha|+ \|\varphi\|_{L^2(\mathbb{T}^1)}\leq C \|\phi- \gamma\|_{L^2(\mathbb{T}^1)}, \\ \|\varphi\|_{L^{\infty}(\mathbb{T}^1)}\leq  C \|\phi- \gamma\|_{L^{\infty}(\mathbb{T}^1)}, \; \|\varphi\|_{H^1(\mathbb{T}^1)}\leq C \|\phi- \gamma\|_{H^1(\mathbb{T}^1)},  \label{es:alpvarbound2}   \\
       \|\phi- \gamma\|_{H^1(\mathbb{T}^1)} \leq C \left(\|\varphi\|_{H^1(\mathbb{T}^1)}+ |\alpha|\right).  \label{es:alpvarbound3}  
\end{gather}

    Conversely, there exists effectively computable $\delta>0$ such that, for every pair $(\alpha, \varphi)\in \mathbb{R}\times  H^1(\mathbb{T}^1; \gamma^*(T\mathcal{N}))$ satisfying the condition ($P3$) and 
    \begin{equation*}
         |\alpha|+ \|\varphi\|_{H^1(\mathbb{T}^1)}\leq \delta,
    \end{equation*}
    the function $\phi$ given by \eqref{eq:def:full:decom} with $ \varphi_1(x)=   \mathcal{F}(x; \varphi, \alpha)$ satisfies ($P1$)--($P3$)  as well as  estimates  \eqref{es:unif:curvebound} and \eqref{es:alpvarbound1}--\eqref{es:alpvarbound3}.
\end{lemma}
We emphasize that the rigidity condition ($P3$) makes the above decomposition unique. Otherwise, for every $\alpha$ sufficiently small, one can always find a decomposition such that conditions ($P1$)--($P2$) hold.
\begin{remark}\label{rem:def:full:decom}
Due to this unique decomposition, investigating the flow $\phi(t, \cdot)$ is equivalent to studying the evolution of the pair $$( \varphi(t, \cdot), \alpha(t))\in H^1(\mathbb{T}^1; \gamma^*(T\mathcal{N}))\times \mathbb{R} \textrm{ satisfying the condition } (P3).$$
More precisely, characterizing the state $\phi[t]= (\phi, \phi_t)$ at any given time $t$ is equivalent to characterizing  $(\varphi, \varphi_t, \alpha, \dot \alpha)(t)$ satisfying 
\begin{gather*}
    \varphi(t, \cdot)\in H^1(\mathbb{T}^1; \gamma^*(T\mathcal{N})) \textrm{ with }  \big\langle \varphi(\cdot), \gamma_x(\cdot) \big\rangle_{L^2(\mathbb{T}^1)}= 0, \\
 \varphi_t(t, \cdot)\in L^2(\mathbb{T}^1; \gamma^*(T\mathcal{N})) \textrm{ with }  \big\langle \varphi_t(\cdot), \gamma_x(\cdot) \big\rangle_{L^2(\mathbb{T}^1)}= 0.
\end{gather*}

Moreover,  there exist $\delta_{tr}$ and $C_{tr}$ such that for any state $\phi[t]$ satisfying
\begin{equation*}
    \|\phi[t]- (\gamma, 0)\|_{\mathcal{H}}\leq \delta_{tr}
\end{equation*}
the decomposed component $(\varphi, \varphi_t, \alpha, \dot \alpha)(t)$ satisfies estimates \eqref{es:alpvarbound1}--\eqref{es:alpvarbound3} as well as
\begin{gather}
C_{tr}^{-1} \|\phi_t(t)\|_{L^2(\mathbb{T}^1)}\leq   |\dot \alpha(t)|+ \|\varphi_t(t)\|_{L^2(\mathbb{T}^1)} \leq C_{tr} \|\phi_t(t)\|_{L^2(\mathbb{T}^1)}, \label{eq:equiv:phivarphi:0} \\
 C_{tr}^{-1} \|\phi[t]- (\gamma, 0)\|_{\mathcal{H}}\leq  |\alpha(t)|+ |\dot \alpha(t)|+ \|\varphi[t]\|_{\mathcal{H}} \leq C_{tr} \|\phi[t]- (\gamma, 0)\|_{\mathcal{H}}.\label{eq:equiv:phivarphi}
\end{gather}
Conversely, for any $(\varphi, \varphi_t, \alpha, \dot \alpha)(t)\in H^1\times L^2(\mathbb{T}^1; \gamma^*(T\mathcal{N}))\times \mathbb{R}^2$ satisfying 
\begin{gather*}
\big\langle \varphi(\cdot), \gamma_x(\cdot) \big\rangle_{L^2(\mathbb{T}^1)}= \big\langle \varphi_t(\cdot), \gamma_x(\cdot) \big\rangle_{L^2(\mathbb{T}^1)}= 0, \\
    |\alpha(t)|+ |\dot \alpha(t)|+ \|\varphi[t]\|_{\mathcal{H}}\leq \delta_{tr},
\end{gather*}
the corresponding state  $\phi[t]$ satisfies the inequalities \eqref{es:alpvarbound1}--\eqref{eq:equiv:phivarphi}. 
\end{remark}

\begin{proof}[Proof of Lemma \ref{lem:def:full:decom}]
    The existence and uniqueness of such a decomposition is a direct consequence of the aforementioned results. It suffices to prove the estimates \eqref{es:alpvarbound1}--\eqref{es:alpvarbound3}. According to \eqref{eq:orthocondition}--\eqref{es:12321} one has 
\begin{equation*}
    |\alpha|\lesssim  \|\phi- \gamma\|_{L^2(\mathbb{T}^1)}. 
\end{equation*}
This, together with the bound \eqref{es:unif:curvebound} and the decomposition \eqref{eq:def:full:decom}, yields 
\begin{equation}
    |\gamma(x+ \alpha)- \gamma(x)+ \varphi(x)+ \varphi_1(x)|= |\phi(x)- \gamma(x)|\leq \|\phi- \gamma\|_{L^{\infty}(\mathbb{T}^1)}   \notag
\end{equation}
and 
\begin{align*}
    |\gamma(x+ \alpha)- \gamma(x)+ \varphi(x)+ \varphi_1(x)|&\geq |\varphi(x)|- |\varphi_1(x)|- |\gamma(x+ \alpha)- \gamma(x)|\\
    &\geq |\varphi(x)|- C(|\alpha|+ |\varphi(x)|^2+  |\alpha||\varphi(x)|).
\end{align*}
Thus
\begin{equation*}
    |\varphi(x)|\lesssim  |(\phi- \gamma)(x)|+ |\alpha|,
\end{equation*}
hence
\begin{equation*}
  \|\varphi\|_{L^2(\mathbb{T}^1)}\lesssim  \|\phi- \gamma\|_{L^2(\mathbb{T}^1)} \textrm{ and }    \|\varphi\|_{L^{\infty}(\mathbb{T}^1)}\leq C \|\phi- \gamma\|_{L^{\infty}(\mathbb{T}^1)}.
\end{equation*}

Again, due to the relation  \eqref{eq:def:full:decom}, we have 
\begin{equation*}
    \varphi(x)+ \mathcal{F}(x; \alpha, \varphi(x))= \phi(x)- \gamma(x+ \alpha).
\end{equation*}
Differentiating the preceding equation yields
\begin{equation*}
    \varphi_x(x)+ \partial_x \mathcal{F}(x; \alpha, \varphi)+ \nabla_{\varphi}\mathcal{F} (x; \alpha, \varphi(x))  \cdot \varphi_x(x)= \phi_x(x)- \gamma_x(x+ \alpha),
\end{equation*}
then 
\begin{equation*}
    |\varphi_x(x)|\leq |\varphi(x)|+ |\alpha|+ |\phi_x- \gamma_x|(x).
\end{equation*}
The inequality \eqref{es:alpvarbound3} is a direct consequence of the decomposition \eqref{eq:def:full:decom} and the definition of the mapping $\mathcal{F}$.
This finishes the first part of this lemma. 

Finally, for any given pair $(\alpha, \varphi)\in \mathbb{R}\times \gamma^*(T\mathcal{N})$ sufficiently small satisfying the condition ($P3$), it suffices to 
show the function $\phi$ generated by \eqref{eq:def:full:decom} satisfies \eqref{es:unif:curvebound}. In fact, 
\begin{equation*}
    (\phi- \gamma)(x)= \gamma(x+ \alpha)- \gamma(x)+ \varphi(x)+ \mathcal{F}(x; \alpha, \varphi(x))
\end{equation*}
and 
\begin{equation*}
    (\phi_x- \gamma_x)(x)= \gamma_x(x+\alpha)- \gamma_x(x)+ \varphi_x(x)+ \partial_x \mathcal{F}+ \nabla_{\varphi}\mathcal{F} \cdot \varphi_x.
\end{equation*}
The required estimate immediately follows. 
\end{proof}

Regarding estimates in Remark \ref{rem:def:full:decom}, it only remains to prove \eqref{eq:equiv:phivarphi:0}. Since 
\begin{equation*}
  \phi(t, x)= \gamma(x+ \alpha(t))+   \varphi(t, x)+ \mathcal{F}(x; \alpha(t), \varphi(t, x)),
\end{equation*}
one has
\begin{equation*}
  \phi_t(t, x)= \dot\alpha(t)\gamma_x(x+ \alpha(t))+  \varphi_t(t, x)+ \dot\alpha \partial_{\alpha} \mathcal{F}(x; \alpha, \varphi)+ \nabla_{\varphi}\mathcal{F} (x; \alpha, \varphi(t, x))  \cdot \varphi_t(t, x).
\end{equation*}
The first part of \eqref{eq:equiv:phivarphi:0} is a direct consequence of the above formula. To prove the second part, it suffices to notice that 
\begin{align*}
 \langle \dot\alpha(t)\gamma_x(x+ \alpha(t)),  \varphi_t(t, x)\rangle_{L^2(\mathbb{T}^1)}&=  \langle \dot\alpha(t)\gamma_x(x),  \varphi_t(t, x)\rangle_{L^2(\mathbb{T}^1)}+ O(|\alpha|)|\dot\alpha| \|\varphi_t\|_{L^2} \\
 &= O(|\alpha|)|\dot\alpha| \|\varphi_t\|_{L^2}. 
\end{align*}

To end this section, let us introduce the following spaces and energy function:
\begin{defi}\label{def:ene:varphi}
Let $\gamma$ be a closed geodesic on $\mathcal{N}$.  Define the spaces
\begin{gather*}
\mathcal{H}_{\gamma}:=   \big\{(f, g): f \in H^1(\mathbb{T}^1; \gamma^*(T\mathcal{N})), \; g\in L^2(\mathbb{T}^1; \gamma^*(T\mathcal{N}))  \},\\
  \mathcal{H}_{\gamma, 0}:=   \big\{(f, g)\in \mathcal{H}_{\gamma}:     \big\langle f(\cdot), \dot\gamma(\cdot) \big\rangle_{L^2(\mathbb{T}^1)}= 0 \big\}, \\
  \mathcal{H}_{\gamma, 0, 0}:=   \big\{(f, g)\in \mathcal{H}_{\gamma}:     \big\langle f(\cdot), \dot\gamma(\cdot) \big\rangle_{L^2(\mathbb{T}^1)}=  \big\langle g(\cdot), \dot\gamma(\cdot) \big\rangle_{L^2(\mathbb{T}^1)}= 0 \big\}, 
\end{gather*} 
and the energy (recall Section~\ref{sec:strategy2} for the definition)
\begin{equation}
   2 \mathcal{E}_{\gamma}(f, g):= \langle g, g \rangle_{L^2(\mathbb{T}^1)}  -\langle \mathcal{L}_{\gamma}f, f\rangle_{L^2(S^1)}, \; \forall (f, g)\in \mathcal{H}_{\gamma}.
\end{equation} 
\end{defi}
This function will be discussed in Section \ref{subsec:linearstabilityvarphi}, while the following observation will be useful. 
\begin{lemma}\label{lem:key:ob:Psi}
For every $(\varphi, \varphi_t, \alpha, \dot \alpha)\in \mathcal{H}_{\gamma, 0, 0}\times \mathbb{R}^2$, the functions $(\Psi, \Psi_t)$ defined by 
\begin{align*}
    \Psi(x)&= \varphi(x)+ \alpha \gamma_x(x)\;\; \forall x\in \mathbb{T}^1, \\
      \Psi_t(x)&= \varphi_t(x)+ \dot \alpha \gamma_x(x)\;\; \forall x\in \mathbb{T}^1,
\end{align*}
satisfy
\begin{equation*}
   \mathcal{E}_{\gamma}(\Psi, \Psi_t)=   \mathcal{E}_{\gamma}(\varphi, \varphi_t)+  \frac{1}{2}(\dot \alpha)^2 \|\gamma_x\|_{L^2(\mathbb{T}^1)}^2.
\end{equation*}
\end{lemma} 
Indeed, thanks to the rotational invariance  \eqref{eq:mathcalLselfinva},
\begin{align*}
 2\mathcal{E}_{\gamma}(\Psi, \Psi_t)&= \langle \Psi_t, \Psi_t \rangle_{L^2(\mathbb{T}^1)}  -\langle \mathcal{L}_{\gamma}\Psi, \Psi\rangle_{L^2(S^1)}, \\
 &= \langle \varphi_t+ \dot \alpha \gamma_x, \varphi_t+ \dot \alpha \gamma_x \rangle_{L^2(\mathbb{T}^1)}  -\langle \mathcal{L}_{\gamma}\varphi, \varphi\rangle_{L^2(S^1)}, \\
 &=  \langle \varphi_t, \varphi_t\rangle_{L^2(\mathbb{T}^1)}+ \langle  \dot \alpha \gamma_x,  \dot \alpha \gamma_x \rangle_{L^2(\mathbb{T}^1)}  -\langle \mathcal{L}_{\gamma}\varphi, \varphi\rangle_{L^2(S^1)}\\
 &=  2\mathcal{E}_{\gamma}(\varphi, \varphi_t)+ (\dot \alpha)^2 \|\gamma_x\|_{L^2(\mathbb{T}^1)}^2.
\end{align*}

\vspace{3mm}

\subsection{Characterization of the full system on $(\varphi,  \alpha)$}\label{subsec:full system varphialpha}

In this section, we show that under the decomposition proposed in Lemma \ref{lem:def:full:decom} and Remark \ref{rem:def:full:decom}, the state $(\varphi, \varphi_t, \alpha, \dot \alpha)\in \mathcal{H}_{\gamma, 0,0}\times \mathbb{R}^2$ satisfies a coupled system. 
\begin{prop}\label{lem:full:system:varalpha}
There are constant $\delta_{de}$ and  explicit nonlocal nonlinear mappings 
\begin{align*}
    \mathcal{M}, \mathcal{M}_1&: (x; \alpha, \dot\alpha, \varphi(\cdot), \varphi_t(\cdot), \varphi_x(\cdot))\longrightarrow \mathbb{R}^N, \\
    \mathcal{O}&: (\alpha, \dot\alpha, \varphi(\cdot), \varphi_t(\cdot), \varphi_x(\cdot))\longrightarrow \mathbb{R},
\end{align*}
such that the equation on $(\varphi, \varphi_t, \alpha, \dot \alpha)\in \mathcal{H}_{\gamma, 0,0}\times \mathbb{R}^2$ can be written as 
\begin{equation}\label{eq:varphialphafull}
\begin{cases}
 -  \varphi_{tt}(t, x)+   \mathcal{L}_{\gamma}\varphi(t, x) - a(x) \varphi_t(t, x) 
&=  \left(a(x)- \frac{l}{L}\right)\dot \alpha \gamma_x(x) \\
&\;\;\;\; \;\;- \frac{1}{L} \langle a(\cdot)\varphi_t(t, \cdot), \gamma_x(\cdot)\rangle_{L^2(\mathbb{T}^1)} \gamma_x(x) \\
&\;\;\;\;\;\;\;\;\;\;\; + \mathcal{M}_1(x; \alpha, \dot\alpha, \varphi(\cdot), \varphi_t(\cdot), \varphi_x(\cdot))\\
  \ddot\alpha+ \frac{l}{L} \dot \alpha+\frac{1}{L}\left\langle  a(\cdot) \varphi_t(t, \cdot), \gamma_x(\cdot) \right\rangle_{L^2(\mathbb{T}^1)} &=  \mathcal{O}(\alpha, \dot\alpha, \varphi(\cdot), \varphi_t(\cdot), \varphi_x(\cdot)),
    \end{cases}
\end{equation}
where 
\begin{equation}\label{def:L:l}
 L := \int_{\mathbb{T}^1} |\gamma_x|^2 \, dx \;\textrm{ and } \; l:=    \int_{\mathbb{T}^1} a(x) |\gamma_x|^2 \, dx.
\end{equation}
Moreover, the function $\Psi[t]\in \mathcal{H}_{\gamma}$ defined as
\begin{equation}\label{eq:def:Psi}
    \Psi(t, x):= \varphi(t, x)+ \alpha(t) \gamma_x(x),
\end{equation}
satisfies
\begin{equation}\label{eq:Psi:final:def}
 -  \Psi_{tt}(t, x)+   \mathcal{L}_{\gamma}\Psi(t, x) - a(x) \Psi_t(t, x)
=  \mathcal{M}(x; \alpha, \dot\alpha, \varphi(\cdot), \varphi_t(\cdot), \varphi_x(\cdot)).
\end{equation} 
The mappings are bounded by 
\begin{align}
& |\mathcal{M}(x; \alpha, \dot\alpha, \varphi(\cdot), \varphi_t(\cdot), \varphi_x(\cdot))| + |\mathcal{M}_1(x; \alpha, \dot\alpha, \varphi(\cdot), \varphi_t(\cdot), \varphi_x(\cdot))|  \notag \\
&\;\;\;\;\;\;\;\;\;\;\;\; \lesssim \left( \sum_{j, k} |\partial_{\nu} \varphi^j \partial^{\nu} \varphi^k|+  |\dot \alpha|^2+  |\varphi_t||\dot \alpha|+  (|\varphi|+ |\alpha|) (|\varphi|+ |\varphi_x|+ |\varphi_t|+ |\dot\alpha| )\right)(x)   \notag   \\
&\;\;\;\;\;\;\;\;\;\;\;\;\;\;\;\;\;\;\;\;\;\;\;\; \;\;\;\;+ \left\|\sum_{j, k} |\partial_{\nu} \varphi^j \partial^{\nu} \varphi^k|+   (|\varphi|+ |\alpha|) (|\varphi|+ |\varphi_x|+ |\varphi_t|+ |\dot\alpha|) \right\|_{L^1(\mathbb{T}^1)} \notag \\
&\;\;\;\;\;\;\;\;\;\;\;\;\;\;\;\;\;\;\;\;\;\;\;\;\;\;\;\;\;\;\;\;\;\;\;\;\;\;\;\; \;\;\;\;  + \|\varphi_t\|_{L^1(\mathbb{T}^1)} \Big(|\alpha|+ |\dot \alpha|+ |\varphi|(x)+  \|\varphi\|_{L^1(\mathbb{T}^1)}\Big) + |\dot \alpha|^2  \label{ine:non:varphi}
\end{align}
and 
\begin{align}
& \;\;\;\; |\mathcal{O}(\alpha, \dot\alpha, \varphi(\cdot), \varphi_t(\cdot), \varphi_x(\cdot))|    \notag \\
&\;\;\;\;\;\;\;\;\;\;\;\;\;\;\;\;\;\;\;\;\;\;\;\;\;\;\;\; \;\;\;\;\lesssim \left\|\sum_{j, k} |\partial_{\nu} \varphi^j \partial^{\nu} \varphi^k|+   (|\varphi|+ |\alpha|) (|\varphi|+ |\varphi_x|+ |\varphi_t|+ |\dot\alpha|) \right\|_{L^1(\mathbb{T}^1)} \notag \\
&\;\;\;\;\;\;\;\;\;\;\;\;\;\;\;\;\;\;\;\;\;\;\;\;\;\;\;\;\;\;\;\;\;\;\;\;\;\;\;\;\;\;\;\;\;\;\;\;\;\;\;\; \;\;\;\;  + \|\varphi_t\|_{L^1(\mathbb{T}^1)} \Big( |\dot \alpha|+   \|\varphi\|_{L^1(\mathbb{T}^1)}\Big) + |\dot \alpha|^2 \label{ine:non:alpha} 
\end{align}
provided that 
\begin{equation*}
    |\alpha|+ \|\varphi\|_{L^{\infty}(\mathbb{T}^1)}\leq \delta_{de}.
\end{equation*}
\end{prop}

Let $\phi$ be a solution of 
\begin{equation}\label{eq:sol:full:phi}
    \Box \phi + S_{jk}(\phi)\partial_{\nu}\phi^j \partial^{\nu}\phi^k - a(x)\partial_t\phi= 0,
\end{equation}
sufficiently close to a closed geodesic $\gamma$.
Thanks to Remark \ref{rem:def:full:decom} and Definition \ref{def:ene:varphi}, there is a unique  $(\alpha, \dot \alpha, \varphi, \varphi_t)(t)\in \mathbb{R}^2\times \mathcal{H}_{\gamma, 0, 0}$ such that for every $x\in \mathbb{T}^1$ and every $t\in [0,  T]$,
 \begin{gather}
        \phi(t, x)= \gamma(x+ \alpha(t))+ \varphi(t, x)+ \varphi_1(t, x), \\
         \varphi_1(t, x)=   \mathcal{F}(x; \varphi(t, x), \alpha(t)) \in N_{\gamma(x)} \mathcal{N},
 \end{gather}
and we denote 
\begin{equation}
   \psi(t, x):=  \varphi(t, x)+ \varphi_1(t, x).
\end{equation}

Keep in mind $\Box= -\partial_t^2+ \partial_x^2$, the operator $\mathcal{L}_{\gamma}$ from Definition \ref{def:jacobian} 
\begin{equation}
\mathcal{L}_{\gamma}\varphi = \varphi_{xx} + \varphi^r\partial_rS_{jk}(\gamma)\partial_x\gamma^j\partial_x\gamma^k + 2S_{jk}(\gamma)\partial_x\gamma^j\partial_x\varphi^k=: \varphi_{xx}+ \tilde{\mathcal{L}}_{\gamma} \varphi,
\end{equation} 
and  the geodesic equation 
\begin{equation*}
 \gamma_{xx} + S_{jk}(\gamma)\partial_x\gamma^j \partial_x\gamma^k = 0, \; \forall x\in \mathbb{T}^1,
\end{equation*}
with the Einstein convention $r, j, k= 1,2,..., N$.

\vspace{2mm}
By differentiating $\phi$ we obtain 
\begin{align*}
    \phi(t, x)&= \gamma(x+ \alpha(t))+ \psi(t, x), \\
    \phi_x(t, x)&= \gamma_x(x+ \alpha(t))+ \psi_x(t, x), \\
    \phi_{xx}(t, x)&= \gamma_{xx}(x+ \alpha(t))+ \psi_{xx}(t, x), \\
 \phi_{t}(t, x)&= \gamma_{x}(x+ \alpha(t)) \dot{\alpha}(t)+ \psi_{t}(t, x), \\
 \phi_{tt}(t, x)&= \gamma_{xx}(x+ \alpha(t)) |\dot{\alpha}(t)|^2+ \gamma_{x}(x+\alpha(t)) \ddot{\alpha}(t)+    \psi_{tt}(t, x).
 \end{align*}
Thus 
\begin{align*}
   &\;\;\;\; \Box \phi(t, x)- a(x) \phi_t(t, x) \\
    &= \Box \psi(t, x)- a(x) \psi_t(t, x) \\
    &\;\;\;\; \;\;\;\;  - \left( a(x)\dot \alpha(t) \gamma_x(x+ \alpha(t))+ \gamma_{xx} (x+ \alpha(t))|\dot\alpha(t)|^2+ \gamma_{x}(x+ \alpha(t))\ddot\alpha(t) \right) \\
     &\;\;\;\;\;\;\;\; \;\;\;\;  + \gamma_{xx} (x+ \alpha(t))\\
      &= \Box \psi(t, x)- a(x) \psi_t(t, x) \\
    &\;\;\;\;\;\;\;\;  - \left( a(x)\dot \alpha(t) \gamma_x(x+ \alpha(t))+ \gamma_{xx} (x+ \alpha(t))|\dot\alpha(t)|^2+ \gamma_{x}(x+ \alpha(t))\ddot\alpha(t) \right) \\
     &\;\;\;\;\;\;\;\; \;\;\;\;  - S_{jk}\big(\gamma(x+ \alpha(t))\big) \gamma^j_x(x+ \alpha(t)) \gamma^k_x(x+ \alpha(t)). 
\end{align*}
Substituting this formula into equation \eqref{eq:sol:full:phi} we get 
\begin{align}
&\;\;\;\; \Box \psi(t, x) + \tilde{\mathcal{L}}_{\gamma}\psi(t, x) - a(x) \psi_t(t, x)  \label{eq:fullfirstpsi}  \\
&= a(x) \dot \alpha(t) \gamma_x(x+ \alpha(t))+ \gamma_{xx} (x+ \alpha(t))|\dot\alpha(t)|^2+ \gamma_{x}(x+ \alpha(t))\ddot\alpha(t)    \notag \\
&\;\;\;\;\;\;\;\; \;\;\;\;\;\;\;\;\;\;\;\;\;\;\;\; \;\;\;\;\;\;\;\;  + \psi^r\partial_rS_{jk}\big(\gamma(x)\big)\gamma^j_x(x) \gamma^k_x(x) + 2S_{jk}\big(\gamma(x)\big)\gamma^j_x(x)\psi^k_x(t, x)    \notag  \\
 &\;\;\;\;\;\;\;\;\;\;\;\;\;\;\;\; \;\;\;\;\;\;\;\;\;\;\;\;\;\;\;\;  \;\;\;\;\;\;\;\; \;\;\;\;\;\;\;\; \;\;\;\; \;\;\;\;   +   S_{jk}\big(\gamma(x+ \alpha(t))\big) \gamma^j_x(x+ \alpha(t)) \gamma^k_x(x+ \alpha(t)) 
  \notag  \\
&\;\;\;\;\;\;\;\; \;\;\;\;\;\;\;\;\;\;\;\;\;\; \;\;\;\;\;\;\;\; \underbrace{  \;\;\;\;\;\;\;\; \;\;\;\;\;\;\;\;  \;\;\;\;\;\;\;\; \;\;\;\;\;\;\;\; \;\;\;\;\;\;\;\; \;\;\;\;\;\;\;\; \;\;\;  \;\;\;\;\;\;\;\; \;\;\;\;\;\;\;\; - S_{jk}(\phi)\partial_{\nu}\phi^j \partial^{\nu}\phi^k }_{\textrm{ denoted by } \mathcal{Q}_0(x; \alpha, \dot\alpha, \psi, \psi_t, \psi_x)    } \notag  
\end{align}
In the sequel we show that the nonlinear term $\mathcal{Q}_0(x; \alpha, \dot\alpha, \psi, \psi_t, \psi_x)$ has the following properties:  
\begin{itemize}
  \item[(K1)] it only contains second and higher order terms on $(\alpha, \dot\alpha, \psi, \psi_t, \psi_x)$, and it does not contain third or higher order terms on $(\psi_t, \psi_x)$. Moreover, the quadratic terms on $(\psi_t, \psi_x)$ must appear in forms of $\partial_{\nu} \psi^j \partial^{\nu} \psi^k$. This property is necessary to control the $L^2_{t, x}$-norm of $\mathcal{Q}_0(x; \alpha, \dot\alpha, \psi, \psi_t, \psi_x)$, since 
\begin{equation*}
      \|\psi\|_{L^{\infty}((0, T)\times \mathbb{T}^1)}\lesssim \|\psi\|_{\mathcal{W}_T},\; \|\partial_{\nu} \psi^j \partial^{\nu} \psi^k\|_{L^2((0, T)\times \mathbb{T}^1)}\lesssim \|\psi\|_{\mathcal{W}_T}^2,
\end{equation*}    
while neither $\psi_t \psi_t$ nor $\psi_x \psi_x$ are controlled by the same upper bound. 
\item[(K2)] this term $|\mathcal{Q}_0(x; \alpha, \dot\alpha, \psi, \psi_t, \psi_x)|$  is uniformly bounded by 
\begin{equation}\label{ine:unif:bound:Q}
   \lesssim  \sum_{j, k} |\partial_{\nu} \psi^j \partial^{\nu} \psi^k|+ |\psi_t|  |\dot \alpha| + |\psi_x| (|\psi|+  |\alpha|)+ |\dot \alpha|^2+  |\psi|   |\alpha|, \; \forall x\in \mathbb{T}^1,
\end{equation}
provided that 
\begin{equation*}
    |\alpha|,  |\psi|\leq 1.
\end{equation*}
\end{itemize}

Now, we give a full expansion of $\mathcal{Q}_0(x; \alpha, \dot\alpha, \psi, \psi_t, \psi_x)$ and demonstrate the bound \eqref{ine:unif:bound:Q}. To simplify notation, we simply denote $\alpha(t)$ by $\alpha$.
First, we have 
\begin{align*}
&\;\;\;\; S_{jk}(\phi)\partial_{\nu}\phi^j \partial^{\nu}\phi^k\\
&= S_{jk}\big(\gamma(x+ \alpha)+ \psi\big)\partial_{\nu}\big(\gamma^j(x+ \alpha)+ \psi^j\big) \partial^{\nu}\big(\gamma^k(x+ \alpha)+ \psi^k\big) \\
& =  S_{jk}\big(\gamma(x+ \alpha)+ \psi\big)\partial_{\nu} \psi^j \partial^{\nu} \psi^k+ 2   S_{jk}\big(\gamma(x+ \alpha)+ \psi\big) \partial_{\nu}\big(\gamma^j(x+ \alpha) \big) \partial^{\nu}\big(\psi^k\big) \\
& \;\;\;\; -  S_{jk}\big(\gamma(x+ \alpha)+ \psi\big) \gamma_{x}^j(x+ \alpha) \gamma_{x}^k(x+ \alpha) |\dot \alpha|^2 +  S_{jk}\big(\gamma(x+ \alpha)+ \psi\big) \gamma_{x}^j(x+ \alpha) \gamma_{x}^k(x+ \alpha)\\
& = 2 S_{jk}\big(\gamma(x+ \alpha)+ \psi\big)\gamma^j_x(x+ \alpha)  \psi^k_x +  S_{jk}\big(\gamma(x+ \alpha)+ \psi\big) \gamma_{x}^j(x+ \alpha) \gamma_{x}^k(x+ \alpha)\\
&  \;\;\;\;\;\;\;\;\;\;\;\; \;\;\;\;\;\;\;\;\;\;\;\; \;\; +  S_{jk}\big(\gamma(x+ \alpha)+ \psi\big)\partial_{\nu} \psi^j \partial^{\nu} \psi^k- 2 S_{jk}\big(\gamma(x+ \alpha)+ \psi\big)\gamma^j_x(x+ \alpha) \dot\alpha \psi^k_t \\
& \;\;\;\;  \;\;\;\;\;\;\;\;\;\;\;\;\;\;\;\;\; \;\;\;\;\;\; \underbrace{ \;\;\;\;\;\;\;\;\;\;\;\;\;\;\;\; \;\;\;\;\;\;\;\;\;\;\;\; -  S_{jk}\big(\gamma(x+ \alpha)+ \psi\big) \gamma_{x}^j(x+ \alpha) \gamma_{x}^k(x+ \alpha) |\dot \alpha|^2 \;\;\;\;\;\;\;\;  }_{\textrm{ denoted by }  \mathcal{Q}_{0, 1} (x; \alpha, \dot\alpha, \psi, \psi_t, \psi_x)    } 
\end{align*}
Clearly, the nonlinear term $\mathcal{Q}_{0, 1}(x; \alpha, \dot\alpha, \psi, \psi_t, \psi_x)$ satisfies the two properties: $(K1)$ and $(K2)$.

By substituting the above equation into $\mathcal{Q}_0$, 
\begin{align*}
&\;\;\;\; - \mathcal{Q}_0(x; \alpha, \dot\alpha, \psi, \psi_t, \psi_x)\\
&= \mathcal{Q}_{0, 1}(x; \alpha, \dot\alpha, \psi, \psi_t, \psi_x) +  \underbrace{ 2 S_{jk}\big(\gamma(x+ \alpha)+ \psi\big)\gamma^j_x(x+ \alpha)  \psi^k_x - 2S_{jk}\big(\gamma(x)\big)\gamma^j_x(x)\psi^k_x(t, x) }_{\textrm{ denoted by } \mathcal{Q}_{0, 2}(x; \alpha,  \psi,  \psi_x)    }  \\
& \;\;\;\;\;\;\;\; \;\;\;\;\;\;\;\;+  S_{jk}\big(\gamma(x+ \alpha)+ \psi\big) \gamma_{x}^j(x+ \alpha) \gamma_{x}^k(x+ \alpha)  -\psi^r\partial_rS_{jk}\big(\gamma(x)\big)\gamma^j_x(x) \gamma^k_x(x)    \notag  \\
& \;\;\;\;\;\;\;\; \;\;\;\;\;\;\; \;\underbrace{\;\;\; \;\;\;\;\; \;\;\;\;\;\;\; \;\;\;\;\;\;\; \;\;\;\;\;\;\;\;\;\;\;\;\;\;\;\;\;\;\;\;\;\;\;\;-   S_{jk}\big(\gamma(x+ \alpha)\big) \gamma^j_x(x+ \alpha) \gamma^k_x(x+ \alpha). \;\;\;\;\;  }_{\textrm{ denoted by } \mathcal{Q}_{0, 3}(x; \alpha,  \psi)    } 
\end{align*}
We easily deduce 
\begin{align*}
|\mathcal{Q}_{0, 2}(x; \alpha, \psi, \psi_x)|
\lesssim |\psi_x| \left( |\alpha|+ |\psi|\right)
\end{align*}
and 
\begin{align*}
 |\mathcal{Q}_{0, 3}(x; \alpha,  \psi)| \lesssim |\psi| |\alpha|    
\end{align*}
provided that $
     |\alpha|,  |\psi|\leq 1$.
Thus both $\mathcal{Q}_{0, 2}$ and $\mathcal{Q}_{0, 3}$, and therefore $\mathcal{Q}_{0}$, satisfy the properties $(K1)$ and $(K2)$.
\vspace{4mm}

Since 
\begin{equation*}
     \psi(t, x):=  \varphi(t, x)+ \varphi_1(t, x)=   \varphi(t, x)+  \mathcal{F}\big(x; \varphi(t, x), \alpha(t)\big),
\end{equation*}
with $\mathcal{F}$ smooth and satisfying \eqref{esonF2}. From equation \eqref{eq:fullfirstpsi} on $\psi$, we obtain the equation on $\varphi$. Indeed, 
\begin{align*}
    \psi_x(t, x)=\varphi_x (t, x)+   \partial_x\mathcal{F}\big(x; \varphi(t, x), \alpha(t)\big)  +   \varphi^l_x \partial_{\varphi^l} \mathcal{F} \big(x; \varphi(t, x), \alpha(t)\big)
\end{align*}
with the Einstein convention for $l= 1,2,.., N$. Successively, 
\begin{align*}
     \psi_t(t, x)&=\varphi_t(t, x)+     \varphi^l_t \partial_{\varphi^l} \mathcal{F} \big(x; \varphi(t, x), \alpha(t)\big)+ \dot\alpha \partial_{\alpha} \mathcal{F} \big(x; \varphi(t, x), \alpha(t)\big)  \\
    \psi_{xx}&=\varphi_{xx}+     \varphi^l_{xx} \partial_{\varphi^l} \mathcal{F} + \partial_{xx}  \mathcal{F} + 2 \varphi^l_{x} \partial^2_{x, \varphi^l} \mathcal{F} + \varphi^l_{x} \varphi^m_{x} \partial^2_{\varphi^l, \varphi^m} \mathcal{F}  \\ 
     \psi_{tt}&=\varphi_{tt}+     \varphi^l_{tt} \partial_{\varphi^l} \mathcal{F} +  \varphi^l_{t} \varphi^m_{t} \partial^2_{\varphi^l, \varphi^m}\mathcal{F}+ 2 \dot \alpha \varphi_t^l \partial^2_{\varphi^l, \alpha} \mathcal{F} +  |\dot \alpha|^2 \partial^2_{\alpha, \alpha} \mathcal{F}  + \ddot \alpha \partial_{\alpha} \mathcal{F} \\
    \tilde{\mathcal{L}}_{\gamma}\psi &=  \psi^r\partial_rS_{jk}(\gamma)\partial_x\gamma^j\partial_x\gamma^k + 2S_{jk}(\gamma)\partial_x\gamma^j\partial_x\big(\psi^k\big)\\
    &=  \tilde{\mathcal{L}}_{\gamma}\varphi+  \mathcal{F}^r\partial_rS_{jk}(\gamma)\partial_x\gamma^j\partial_x\gamma^k + 2S_{jk}(\gamma)\partial_x\gamma^j\big(\partial_x \mathcal{F}^k+ \varphi^l_x \partial_{\varphi^l} \mathcal{F}^k \big), 
\end{align*}
with the Einstein convention for $l, m, j, k, r= 1,2,.., N$.

Define the $N\times N$ matrix 
\begin{equation}
    \mathcal{U}_1(x; \varphi, \alpha):= \left(\partial_{\varphi^l} \mathcal{F}^i (x; \varphi, \alpha)\right)_{1\leq i, l\leq N}.
\end{equation}
It satisfies 
\begin{equation*}
    |\mathcal{U}_1(x; \varphi, \alpha)|\lesssim  |\varphi|+ |\alpha|, \textrm{ if }  |\varphi|, |\alpha|\leq 1.
\end{equation*}
We further define the inverse by
\begin{equation*}
    I_n+ \mathcal{U}_2(x; \varphi, \alpha)= (I_n+ \mathcal{U}_1(x; \varphi, \alpha))^{-1},
\end{equation*}
which also satisfies 
\begin{equation}\label{es:macalM}
    |\mathcal{U}_2(x; \varphi, \alpha)|\lesssim  |\varphi|+ |\alpha|, \textrm{ provided that }  |\varphi|, |\alpha|\ll 1.
\end{equation}
Notice that 
\begin{align*}
    \Box \psi&= \big(I_n+ \mathcal{U}_1(x; \varphi(t, x), \alpha(t))\big)  \Box \varphi  - \ddot \alpha \partial_{\alpha} \mathcal{F}   +  \left( \partial_{xx}  \mathcal{F} + 2 \varphi^l_{x} \partial^2_{x, \varphi^l} \mathcal{F} + \varphi^l_{x} \varphi^m_{x} \partial^2_{\varphi^l, \varphi^m} \mathcal{F} \right)    \\
&  - \left(\varphi^l_{t} \varphi^m_{t} \partial^2_{\varphi^l, \varphi^m}\mathcal{F}+ 2 \dot \alpha \varphi_t^l \partial^2_{\varphi^l, \alpha} \mathcal{F} +  |\dot \alpha|^2 \partial^2_{\alpha, \alpha} \mathcal{F} \right).
\end{align*}
Define 
\begin{align*}
 \mathcal{Q}_1(x; \alpha, \dot\alpha, \varphi, \varphi_t, \varphi_x)&:=  \left(\partial_{xx}  \mathcal{F}  + 2 \varphi^l_{x} \partial^2_{x, \varphi^l} \mathcal{F} + \varphi^l_{x} \varphi^m_{x} \partial^2_{\varphi^l, \varphi^m} \mathcal{F} \right)(x; \varphi, \alpha)    \\
&  \;\;\;\;\;\;\;\;\;\;\;\;\;\;\;\;\;\;\;\;  - \left(\varphi^l_{t} \varphi^m_{t} \partial^2_{\varphi^l, \varphi^m}\mathcal{F}+ 2 \dot \alpha \varphi_t^l \partial^2_{\varphi^l, \alpha} \mathcal{F} +  |\dot \alpha|^2 \partial^2_{\alpha, \alpha} \mathcal{F}  \right)(x; \varphi, \alpha),
\end{align*}
which satisfies 
\begin{equation}\label{es:macalq1}
    |\mathcal{Q}_1(x; \alpha, \dot\alpha, \varphi, \varphi_t, \varphi_x)|\lesssim |\varphi|^2+ |\alpha| |\varphi|+ |\varphi_x|(|\varphi|+ |\alpha|)+ |\varphi_t| |\dot\alpha|+ |\dot \alpha|^2+ \sum_{j, k}|\partial_{\nu} \varphi^j \partial^{\nu} \varphi^k|.
\end{equation}
Then 
\begin{equation*}
    \Box \psi(t, x)= \big(I_n+ \mathcal{U}_1(x; \varphi(t, x), \alpha(t))\big)  \Box \varphi(t, x)  - \ddot \alpha \partial_{\alpha} \mathcal{F}  + \mathcal{Q}_1(x; \alpha, \dot\alpha, \varphi, \varphi_t, \varphi_x).
\end{equation*}

We also define $\mathcal{Q}_2$ as the difference between $\tilde{\mathcal{L}}_{\gamma}\psi - a(x) \psi_t$ and $\tilde{\mathcal{L}}_{\gamma}\varphi - a(x) \varphi_t$:
\begin{align*}
\mathcal{Q}_2(x; \alpha, \dot\alpha, \varphi, \varphi_t, \varphi_x)&= \left(\mathcal{F}^r\partial_rS_{jk}(\gamma)\partial_x\gamma^j\partial_x\gamma^k + 2S_{jk}(\gamma)\partial_x\gamma^j\big(\partial_x \mathcal{F}^k+ \varphi^l_x \partial_{\varphi^l} \mathcal{F}^k \big) \right)(x; \varphi, \alpha) \\
&  \;\;\;\;\;\;\;\; \;\;\;\;\;\;\;\;\;\;\;\;\;\;\;\;\;\;\;\;\;\;\;\;   - a(x) \left(      \varphi^l_t \partial_{\varphi^l} \mathcal{F} + \dot\alpha \partial_{\alpha} \mathcal{F}  \right)\big(x; \varphi, \alpha\big),
\end{align*}
which satisfies
\begin{align}
    |\mathcal{Q}_2(x; \alpha, \dot\alpha, \varphi, \varphi_t, \varphi_x)|&\lesssim |\varphi|^2+ |\alpha| |\varphi|+ |\varphi_x|(|\varphi|+ |\alpha|)+ |\varphi_t|(|\varphi|+ |\alpha|)+ |\dot \alpha| (|\varphi|+ |\alpha|) \notag\\
   & \lesssim (|\varphi|+ |\alpha|) (|\varphi|+ |\varphi_x|+ |\varphi_t|+ |\dot\alpha| )  \label{es:macalq2}
\end{align}

Therefore, equation \eqref{eq:fullfirstpsi} is equivalent to 
\begin{align*}
  &  \;\;\;\;\big(I_n+ \mathcal{U}_1(x; \varphi(t, x), \alpha(t))\big)  \Box \varphi(t, x)+   \tilde{\mathcal{L}}_{\gamma}\varphi(t, x) - a(x) \varphi_t(t, x)   \\
&= a(x)\dot \alpha \gamma_x(x+ \alpha)+  \ddot\alpha  \big(\gamma_{x}(x+ \alpha)- \partial_{\alpha}\mathcal{F}(x; \varphi, \alpha) \big) \\
& \;\;\;\;\;\;\;\;  \;\;\;\; + \gamma_{xx} (x+ \alpha)|\dot\alpha|^2-\mathcal{Q}_1(x; \alpha, \dot\alpha, \varphi, \varphi_t, \varphi_x) - \mathcal{Q}_2(x; \alpha, \dot\alpha, \varphi, \varphi_t, \varphi_x)+ \mathcal{Q}_0(x; \alpha, \dot\alpha, \psi, \psi_t, \psi_x)
\end{align*}

By multiplying  both sides of this equation on the left by the matrix $ I_n+ \mathcal{U}_2(x; \varphi, \alpha)$, we obtain  the equation for $\varphi$:
\begin{align}
 - \varphi_{tt}(t, x)+   \mathcal{L}_{\gamma}\varphi(t, x) - a(x) \varphi_t(t, x) 
&= a(x) \dot \alpha \gamma_x(x+ \alpha)+  \mathcal{Q}(x; \alpha, \dot\alpha, \varphi, \varphi_t, \varphi_x)   \notag \\
&\;\;\;\; +\ddot\alpha  \underbrace{\big(I_n+ \mathcal{U}_2(x; \varphi, \alpha)\big)   \big(\gamma_{x}(x+ \alpha)- \partial_{\alpha}\mathcal{F}(x; \varphi, \alpha) \big)}_{\textrm{ denoted by } \mathcal{P}(x; \varphi, \alpha)+ \gamma_x(x)},    \label{equationfullonvarphi:1st}
\end{align} 
where $\mathcal{P}(x; \varphi, \alpha)$ is bounded by $|\alpha|+ |\varphi|$, and
$\mathcal{Q}$  is composed of second or higher order terms and is given by 
\begin{align*}
  &\;\;\;\;   \mathcal{Q}(x; \alpha, \dot\alpha, \varphi, \varphi_t, \varphi_x)  \\
  &= 
  \mathcal{U}_2(x; \varphi, \alpha) \left( -\tilde{\mathcal{L}}_{\gamma}\varphi(t, x) + a(x) \varphi_t(t, x)+  a(x) \dot \alpha \gamma_x(x+ \alpha) \right)   \\
  & \;\;\;\;\;\;\;\;\;\;\;\; \;\;\;\;\;\;\;\;\;\; \;\;\;\;\;\;\;\;\;\;  + \big( I_n+ \mathcal{U}_2(x; \varphi, \alpha)\big)  \Big(  \gamma_{xx} (x+ \alpha)|\dot\alpha|^2-\mathcal{Q}_1(x; \alpha, \dot\alpha, \varphi, \varphi_t, \varphi_x) \\
  &\;\;\;\;\;\;\;\;\;\; \;\;\;\;\;\;\;\;\;\; \;\;\;\;\;\;\;\;\;\; \;\;\;\;\;\;\;\;\;\; \;\;\;\;\;\;\;\;\;\; \;\;\;\;\;\;\;\;\;\;  - \mathcal{Q}_2(x; \alpha, \dot\alpha, \varphi, \varphi_t, \varphi_x)+ \mathcal{Q}_0(x; \alpha, \dot\alpha, \psi, \psi_t, \psi_x)  \Big)   
\end{align*}
Indeed, concerning $\mathcal{Q}_0(x; \alpha, \dot\alpha, \psi, \psi_t, \psi_x)$, direct calculation yields 
\begin{align}
&  |\mathcal{Q}_0(x; \alpha, \dot\alpha, \psi, \psi_t, \psi_x)| \notag \\
 &\;\;\;\;\;\;\;\;\lesssim \sum_{j, k}|\partial_{\nu} \psi^j \partial^{\nu} \psi^k|+ |\psi_t|  |\dot \alpha| + |\psi_x| (|\psi|+  |\alpha|)+ |\dot \alpha|^2+  |\psi|   |\alpha|  \notag\\
  & \;\;\;\;\;\;\;\;\;\;\;\;\;\;\;\;\lesssim \sum_{j, k} |\partial_{\nu} \varphi^j \partial^{\nu} \varphi^k| +  |\varphi_t|  |\dot \alpha|+ |\dot \alpha|^2+  (|\varphi|+ |\alpha|) (|\varphi|+ |\varphi_x|+ |\dot\alpha|).\label{es:macalq0}
\end{align}
By combining \eqref{es:macalM}--\eqref{es:macalq0}, we obtain estimates on the nonlinear term $\mathcal{Q}$:
\begin{align}
    &|\mathcal{Q}(x; \alpha, \dot\alpha, \varphi, \varphi_t, \varphi_x)| \notag \\
  &\;\;\;\;\;\;\;\;\lesssim \sum_{j, k} |\partial_{\nu} \varphi^j \partial^{\nu} \varphi^k|+  |\dot \alpha|^2+  |\varphi_t||\dot \alpha|+  (|\varphi|+ |\alpha|) (|\varphi|+ |\varphi_x|+ |\varphi_t|+ |\dot\alpha| ).  
\end{align}
 provided that  $|\varphi|, |\alpha|\ll 1$.
\vspace{4mm}

 Next, we turn to the equation for $\alpha$.
Recall that $\alpha(t)$ is chosen such that the orthogonality condition holds: 
\begin{equation}\label{eq:var:gamx1}
    \langle \varphi(t, \cdot), \gamma_x(\cdot) \rangle_{L^2(\mathbb{T}^1)}= 0\;\; \forall t\in [0, T].
\end{equation}
Differentiating twice equation \eqref{eq:var:gamx1} we obtain 
\begin{equation}
    \langle \varphi_{t}(t, \cdot), \gamma_x(\cdot) \rangle_{L^2(\mathbb{T}^1)}=   \langle \varphi_{tt}(t, \cdot), \gamma_x(\cdot) \rangle_{L^2(\mathbb{T}^1)}= 0\;\; \forall t\in [0, T],
\end{equation}
Thanks to Lemma \ref{lem:selfadjoint}, we have
\begin{equation}
   \langle \mathcal{L}_{\gamma} \varphi(t, \cdot), \gamma_x(\cdot) \rangle_{L^2(\mathbb{T}^1)}=  \langle \varphi(t, \cdot), \mathcal{L}_{\gamma} 
 \gamma_x(\cdot) \rangle_{L^2(\mathbb{T}^1)}= 0\;\; \forall t\in [0, T]. 
\end{equation}

By integrating equation \eqref{equationfullonvarphi:1st} against $\gamma_x$, we obtain in light of the definition of $L$ and $l$ in \eqref{def:L:l}
\begin{align*}
&\;\;\;\;\;\; \ddot\alpha    \underbrace{\left\langle \mathcal{P}(x; \varphi(x), \alpha)+ \gamma_x(x), \gamma_x(x) \right\rangle_{L^2(\mathbb{T}^1)} }_{ \approx L+ |\alpha|+ \|\varphi\|_{L^1}}\\
&= - \left\langle a(x) \dot \alpha \gamma_x(x+ \alpha)+ a(x) \varphi_t(t, x)+   \mathcal{Q}(x; \alpha, \dot\alpha, \varphi, \varphi_t, \varphi_x), \gamma_x(\cdot) \right\rangle_{L^2(\mathbb{T}^1)} \\
&= - \dot \alpha  \underbrace{\left\langle a(x) \gamma_x(x+ \alpha), \gamma_x(x) \right\rangle_{L^2(\mathbb{T}^1)} }_{\approx 
 l+ |\alpha|} \\
& \;\;\;\;\;\;\;\;\;\;\;\;\;\;\;\;\;\;\;\;\; \;\;\;\;\;\;\;\; \;\;\;\;\;\;\;\; \;\;\;\;\;\;\;\;  - \underbrace{\left\langle  a(x) \varphi_t(t, x), \gamma_x(x) \right\rangle_{L^2(\mathbb{T}^1)} }_{\textrm{ first order term $\|\varphi_t\|_{L^1}$ }} \\
&\;\;\;\;\;\;\;\;\;\;\;\;\;\;\;\;\;\;\;\;\;\;\;\;\;\; \;\;\;\;\;\;\;\; \;\;\;\;\;\;\;\; \;\;\;\;\;\;\;\; \;\;\;\;\;\;\;\; \;\;\;\;\;\;\;\; \;\;\;\;\;\;\;\; \;\;\;\;\;\;\;\; -\underbrace{ \left\langle    \mathcal{Q}(x; \alpha, \dot\alpha, \varphi, \varphi_t, \varphi_x), \gamma_x(\cdot) \right\rangle_{L^2(\mathbb{T}^1)} }_{\textrm{ higher order terms }} 
\end{align*}

Hence $\alpha$ satisfies 
\begin{equation}\label{eq:final:alpha}
  \ddot\alpha+ \frac{l}{L} \dot \alpha= -\frac{1}{L}\left\langle  a(\cdot) \varphi_t(t, \cdot), \gamma_x(\cdot) \right\rangle_{L^2(\mathbb{T}^1)} +  \mathcal{O}(\alpha, \dot\alpha, \varphi(\cdot), \varphi_t(\cdot), \varphi_x(\cdot)) 
\end{equation}
with 
\begin{align*}
 &  \mathcal{O}(\alpha, \dot\alpha, \varphi(\cdot), \varphi_t(\cdot), \varphi_x(\cdot))  \\
  & \;\;\;\; \;\;\;\; =  \dot \alpha \left(\frac{\langle a(x)\gamma_x(x), \gamma_x(x)\rangle_{L^2(\mathbb{T}^1)}}{\langle \gamma_x(x), \gamma_x(x)\rangle_{L^2(\mathbb{T}^1)}} -  \frac{\left\langle a(x) \gamma_x(x+ \alpha), \gamma_x(x) \right\rangle_{L^2(\mathbb{T}^1)}}{\left\langle \mathcal{P}(x; \varphi(x), \alpha)+ \gamma_x(x), \gamma_x(x) \right\rangle_{L^2(\mathbb{T}^1)}} \right) \\
  &\;\;\;\;\;\;\;\;\;\; \;\;\;\;    +  \left\langle  a(x) \varphi_t(t, x), \gamma_x(x) \right\rangle_{L^2(\mathbb{T}^1)}  \left(\frac{1}{L}-  \frac{1}{\left\langle \mathcal{P}(x; \varphi(x), \alpha)+ \gamma_x(x), \gamma_x(x) \right\rangle_{L^2(\mathbb{T}^1)}} \right)\\
&  \;\;\;\;\;\; \;\;\;\;\;\; \;\;\;\; \;\;\;\; \;\;\;\; \;\;\;\;   -   \frac{ \left\langle    \mathcal{Q}(x; \alpha, \dot\alpha, \varphi, \varphi_t, \varphi_x), \gamma_x(x) \right\rangle_{L^2(\mathbb{T}^1)}   }{\left\langle \mathcal{P}(x; \varphi(x), \alpha)+ \gamma_x(x), \gamma_x(x). \right\rangle_{L^2(\mathbb{T}^1)}}
\end{align*}
One easily checks that estimate \eqref{ine:non:alpha} holds for $\mathcal{O}$.

\vspace{3mm}

Now, by substituting \eqref{eq:final:alpha} into equation \eqref{equationfullonvarphi:1st} for $\varphi$,  we obtain 
\begin{align}
 -  \varphi_{tt}(t, x)+   \mathcal{L}_{\gamma}\varphi(t, x) - a(x) \varphi_t(t, x) 
&=  \left(a(x)- \frac{l}{L}\right)\dot \alpha \gamma_x(x)- \frac{1}{L} \langle a(\cdot)\varphi_t(t, \cdot), \gamma_x(\cdot)\rangle_{L^2(\mathbb{T}^1)} \gamma_x(x) \notag\\
&\;\;\;\;\;\;\;\;\;\;\;\;\;\;\;\;\;\;\;\; + \mathcal{M}_1(x; \alpha, \dot\alpha, \varphi(\cdot), \varphi_t(\cdot), \varphi_x(\cdot)).    \label{eq:final:varphi}
\end{align} 
Here the high order term is given by
\begin{align}
 &  \mathcal{M}_1(x; \alpha, \dot\alpha, \varphi(\cdot), \varphi_t(\cdot), \varphi_x(\cdot)) \notag\\
 &\;\;\;\; \;\;\;\;  \;\;\;\;  = - \left(\frac{l}{L} \dot \alpha+ \frac{1}{L}\left\langle  a(\cdot) \varphi_t(t, \cdot), \gamma_x(\cdot) \right\rangle_{L^2(\mathbb{T}^1)} \right) \mathcal{P}(x; \varphi, \alpha) \notag\\
 &\;\;\;\;  \;\;\;\; \;\;\;\; \;\;\; \;\;\;\;\;\;\; +  \Big(\mathcal{P}(x; \varphi, \alpha)+ \gamma_x(x)\Big) \mathcal{O}(\alpha, \dot\alpha, \varphi(\cdot), \varphi_t(\cdot), \varphi_x(\cdot))   \notag \\
 &\;\;\;\; \;\;\;\;  \;\;\;\; \;\;\;\; \;\;\;\;\;\;\;\;\;\;\;\;\;\;\;\;\;\; + a(x) \dot \alpha \left(\gamma_x(x+ \alpha)- \gamma_x(x) \right) +\mathcal{Q}(x; \alpha, \dot\alpha, \varphi, \varphi_t, \varphi_x), \label{eq:def:M3}
\end{align}
and satisfies estimate \eqref{ine:non:varphi}.
\vspace{3mm}

The coupled system \eqref{eq:final:alpha}--\eqref{eq:final:varphi}  for $(\varphi, \alpha)$ encapsulates the evolution of $\phi(t,\cdot)$. However, the linear coupling terms on the right cannot be  treated as a perturbation. Thus we introduce a new function 
\begin{equation}
  \Psi(t, x):= \varphi(t, x)+ \alpha(t) \gamma_x(t, x).
\end{equation}
We observe that this is governed by a simpler equation:
\begin{align}
&\;\;\;\; -  \Psi_{tt}(t, x)+   \mathcal{L}_{\gamma}\Psi(t, x) - a(x) \Psi_t(t, x) \notag \\
&=  - \varphi_{tt}(t, x)+   \mathcal{L}_{\gamma}\varphi(t, x) - a(x) \varphi_t(t, x)- \ddot \alpha(t) \gamma_x(x)+ \alpha(t)\mathcal{L}_{\gamma} \gamma_x- \dot  \alpha(t) a(x)\gamma_x(x) \notag \\
&=  a(x) \dot \alpha \Big(\gamma_x(x+ \alpha)- \gamma_x(x)\Big)+  \ddot\alpha \mathcal{P}(x; \varphi, \alpha)+  \mathcal{Q}(x; \alpha, \dot\alpha, \varphi, \varphi_t, \varphi_x)   \notag \\
& =  \mathcal{M}(x; \alpha, \dot\alpha, \varphi(\cdot), \varphi_t(\cdot), \varphi_x(\cdot)),
\end{align}
where, by plugging  \eqref{eq:final:alpha} into above equation, the nonlinear term is given by 
\begin{align}
&  \mathcal{M}(x; \alpha, \dot\alpha, \varphi(\cdot), \varphi_t(\cdot), \varphi_x(\cdot)) \notag \\
 &\;\;\;\; = a(x) \dot \alpha \Big(\gamma_x(x+ \alpha)- \gamma_x(x)\Big)+  \mathcal{Q}(x; \alpha, \dot\alpha, \varphi, \varphi_t, \varphi_x)  \notag \\
 & \;\;\;\; \;\;\;\; - \left(\frac{l}{L} \dot \alpha + \frac{1}{L}\left\langle  a(\cdot) \varphi_t(t, \cdot), \gamma_x(\cdot) \right\rangle_{L^2(\mathbb{T}^1)} -  \mathcal{O}(\alpha, \dot\alpha, \varphi(\cdot), \varphi_t(\cdot), \varphi_x(\cdot)) \right) \mathcal{P}(x; \varphi, \alpha),
\end{align}
and satisfies estimate \eqref{ine:non:varphi}. 
This concludes the proof of Proposition \ref{lem:full:system:varalpha}.  

\vspace{3mm}

\subsection{Coercive estimates around closed geodesics}\label{sec:coecive}

Recall from Lemma  \ref{lem:selfadjoint} that the Jacobian operator $\mathcal{L}_{\gamma}$ is self-adjoint on functions in $\gamma^*(T\mathcal{N})$. Under the negative curvature assumption, we further demonstrate it is positive definite in a codimension one submanifold, more precisely, for functions $\varphi\in \gamma^*(T \mathcal{N})$ satisfying the condition ($P3$). 
\begin{prop}\label{prop:coercivityofL} 
Let $\gamma$ be a closed geodesic on $\mathcal{N}$. Assume that sectional curvature is strictly negative on $\gamma$.
 Then the operator $\mathcal{L}_{\gamma}$ given in Definition \ref{def:jacobian} is coercive in the following sense: 
there exists a constant $c_{co}>0$ such that for any  $\varphi\in H^1(\mathbb{T}^1; \gamma^*(T \mathcal{N}))$ satisfying the rigidity condition ($P3$),
we have the lower bound 
\[
-\langle \mathcal{L}_{\gamma}\varphi, \varphi\rangle_{L^2(\mathbb{T}^1)}\geq c_{co} \big\|\varphi\big\|_{H^1(\mathbb{T}^1)}^2. 
\]
\end{prop}
Before presenting its proof, we offer some comments. 
\begin{remark}\label{rem:coer:nece}
Both the rigidity condition ($P3$) and the negative curvature assumptions are essential for this proposition. Indeed, if the first condition is dropped,  we can take $\varphi(\cdot)= \gamma_x(\cdot)$ and get 
\begin{equation*}
    -\langle \mathcal{L}_{\gamma}\varphi, \varphi\rangle_{L^2(\mathbb{T}^1)}= \int_{\mathbb{T}^1} \big(\big\|\nabla_{\gamma_x}\varphi\big\|^2 - \langle R(\gamma_x, \varphi)\gamma_x, \varphi\rangle\big)\,dx= 0.
\end{equation*}
If the curvature is positive, then assumption may also fail. For example, let the target manifold be $\mathbb{S}^k\subset \mathbb{R}^{k+1}$ and suppose that $\gamma$ lies entirely in the plane $\mathbb{R}^k\times \{0\}$, then we can select $\varphi$ as $(0,0,..., 1)$ and yields the same degeneracy.
\end{remark}

\begin{remark}\label{rem:equi:norms}
   When the sectional curvature is strictly negative on $\gamma$, then there exists $C_{co}>0$ such that 
\begin{equation*}
  C_{co}^{-1}  \|(f, g)\|_{\mathcal{H}}\leq   \left(\mathcal{E}_{\gamma_p}(f, g)\right)^{1/2}\leq C_{co} \|(f, g)\|_{\mathcal{H}},
\end{equation*}
for every $p\in [0, 2\pi)$ and for every $ (f, g)\in \mathcal{H}_{\gamma_p, 0}$.
\end{remark} 

\begin{remark}
Thanks to Lemma \ref{lem:key:ob:Psi} and the above proposition,  the quantity $\mathcal{E}_{\gamma}(\Psi[t])$ can characterize the distance between $(\phi, \phi_t)$ and the rotation family of closed geodesics $\{(\gamma_p, 0)\}_{p\in [0, 2\pi)}$.
\end{remark}

The rest part of this section is devoted to
\begin{proof}[Proof of Proposition \ref{prop:coercivityofL}]

Without loss of generality, we assume that $|\gamma_x(x)|= 1$ for every $x\in \mathbb{T}^1$.  Define the inner product 
\[
\mathcal{S}(\varphi): = -\langle \mathcal{L}_{\gamma}\varphi, \varphi\rangle_{L^2(\mathbb{T}^1)}.
\]
It equals up to a constant the second variation of the energy with respect to a variation $h(\cdot, \varphi): (-\delta,\delta)\times \mathbb{T}^1\longrightarrow \mathcal{N}$ with $h_s(0,x) = \varphi(x)$ and $h(0,x) = \gamma(x) \; \forall x\in \mathbb{T}^1$. Due to Lemma \ref{lem:selfadjoint}, this can also be expressed in the intrinsic coordinates 
\[
\mathcal{S}(\varphi) = \int_{\mathbb{T}^1} \big(|\nabla_{\gamma_x(x)}\varphi(x)|^2 - \langle R(\gamma_x, \varphi)\gamma_x, \varphi\rangle(x)\big)\,dx. 
\]

The negative curvature assumption implies that function $\mathcal{S}(\varphi)$ is non-negative for any $\varphi\in H^1(\mathbb{T}^1; \gamma^*(T\mathcal{N}))$.
To prove the lower bound, we proceed by contradiction. Assume there exists a sequence $\{\varphi\}_{n\geq 1}\subset H^1(\mathbb{T}^1; \gamma^*(TM))$ with 
\[
\big\|\varphi_n\big\|_{H^1(\mathbb{T}^1)} = 1,\,\langle \varphi_n, \gamma_x\rangle_{L^2(\mathbb{T}^1)} = 0\, 
\]
and such that 
\[
\lim_{n\rightarrow \infty} \mathcal{S}(\varphi_n) = 0 
\]
Using the Banach-Alaoglu lemma as well as the Rellich compactness lemma we can select a subsequence and $\varphi_*\in H^1(\mathbb{T}^1; \gamma^*(T \mathcal{N}))$, which we label in the same way, such that 
\begin{align}
    \varphi_{n, x}&\rightharpoonup \varphi_{*, x} \textrm{ weakly in } L^2(\mathbb{T}^1; \mathbb{R}^N), \label{eq:weak:conv:inRN} \\
     \varphi_{n}&\rightarrow \varphi_{*} \textrm{  in } C^0(\mathbb{T}^1; \mathbb{R}^N).
\end{align}
We also immediately know from the limit of $\mathcal{S}(\varphi_n)$ that 
\begin{gather}
\int_{\mathbb{T}^1} |\nabla_{\gamma_x(x)}\varphi_n(x)|^2 \,dx \longrightarrow 0, \label{ine:con:nab:varphin} \\
\langle R(\gamma_x, \varphi_*)\gamma_x, \varphi_*\rangle(x)= 0,  \; \forall x\in \mathbb{T}^1.
\end{gather}
The preceding equality shows that for any $x\in \mathbb{T}^1$, the vector $\varphi_*$ is parallel to $\gamma_x$.
\vspace{2mm}

Next, we first show that the weak convergence of $\{\varphi_{n, x}\}$ yields the weak convergence of $ \nabla_{\gamma_x(x)}\varphi_{n}$,
\begin{align} \label{eq:weak:nabla}
    \nabla_{\gamma_x(x)}\varphi_{n}&\rightharpoonup  \nabla_{\dot \gamma(x)} \varphi_{*} \textrm{ weakly in } L^2(\mathbb{T}^1; \mathbb{R}^N).
\end{align}
Indeed, by the definition of $\nabla_{ \gamma_x(x)}\varphi_{n}(x)$ we have 
\begin{gather*}
    \varphi_{n, x}(x)= \nabla_{\gamma_x(x)}\varphi_{n}(x)+ g_{n}(x) \textrm{ with }  g_n(x)\in T^{\perp}_{\gamma(x)} \mathcal{N}\; \forall x\in \mathbb{T}^1, \\
    \varphi_{*, x}(x)= \nabla_{\gamma_x(x)}\varphi_{*}(x)+ g_{*}(x) \textrm{ with }  g_*(x)\in T^{\perp}_{\gamma(x)} \mathcal{N}\; \forall x\in \mathbb{T}^1.
\end{gather*}
Thanks to the weak convergence \eqref{eq:weak:conv:inRN} of $\varphi_{n, x}$ in $L^2(\mathbb{T}^1; \mathbb{R}^N)$, by choosing the test function $f\in L^2(\mathbb{T}^1; \gamma^*(T\mathcal{N}))\subset L^2(\mathbb{T}^1; \mathbb{R}^N)$, one has
\begin{equation*}
    \langle \varphi_{n, x}, f\rangle_{L^2(\mathbb{T}^1; \mathbb{R}^N)} \longrightarrow \langle \varphi_{*, x}, f\rangle_{L^2(\mathbb{T}^1; \mathbb{R}^N)}. 
\end{equation*}
Thus,  
\begin{equation*}
    \langle  \nabla_{\gamma_x(x)}\varphi_{n}, f\rangle_{L^2(\mathbb{T}^1; \mathbb{R}^N)} \longrightarrow \langle  \nabla_{\gamma_x(x)}\varphi_{*}, f\rangle_{L^2(\mathbb{T}^1; \mathbb{R}^N)}\;\; \forall f\in L^2(\mathbb{T}^1; \gamma^*(T\mathcal{N})).
\end{equation*}
For any $f\in L^2(\mathbb{T}^1; \mathbb{R}^N)$, we define $f_1(x)$ as the projection of $f(x)$ on the tangent space $T_{\gamma(x)}\mathcal{N}$, then 
\begin{gather*}
  \langle  \nabla_{\gamma_x(x)}\varphi_{n}, f\rangle_{L^2(\mathbb{T}^1; \mathbb{R}^N)}=   \langle  \nabla_{\gamma_x(x)}\varphi_{n}, f_1\rangle_{L^2(\mathbb{T}^1; \mathbb{R}^N)} \;\; \;\;\;\;\;\;\; \;\; \;\;\;\;\;\;\;\;\;\;\;\;   \;\; \;\;\;\;\;\;\; \;\; \;\;\\  
 \;\;\;\;\; \;\; \;\; \;\; \;\;\;\;\;\;\; \;\; \;\;\;\;\;\;\; \;\; \;\;\;\;\;\;\; \longrightarrow \langle  \nabla_{\gamma_x(x)}\varphi_{*}, f_1\rangle_{L^2(\mathbb{T}^1; \mathbb{R}^N)}= \langle  \nabla_{\gamma_x(x)}\varphi_{*}, f\rangle_{L^2(\mathbb{T}^1; \mathbb{R}^N)}.
\end{gather*}
This finishes the proof of \eqref{eq:weak:nabla}.  From the lower semi-continuity, the weak convergence \eqref{eq:weak:nabla} yields 
\begin{equation}
     \|\nabla_{\gamma_x} \varphi_{*}\|_{L^2(\mathbb{T}^1)}
 \leq \liminf_{n}  \|\nabla_{\gamma_x}\varphi_{n}\|_{L^2(\mathbb{T}^1)}.
\end{equation}
This combined with \eqref{ine:con:nab:varphin} yield 
\begin{equation}
    \nabla_{\gamma_x} \varphi_{*}(x)= 0  \; \; \forall x\in \mathbb{T}^1.
\end{equation}

Using the fact that $\gamma_x(x)$ and $\varphi_*(x)$ belong to $T_{\gamma(x)}\mathcal{N}$, we conclude that 
 \begin{align*}
\partial_x\big( \varphi_*\cdot\gamma_x\big)= \partial_x \varphi_*\cdot\gamma_x + \varphi_*\cdot \partial_x \gamma_x 
= \nabla_{\gamma_x}\varphi_*\cdot\gamma_x + \varphi_*\cdot \nabla_{\gamma_x} \gamma_x = 0,
\end{align*}
and so we obtain that 
\[
 \varphi_*\cdot\gamma_x \equiv  \text{const} \; \; \forall x\in \mathbb{T}^1.
 \]
 Observing from the ($P3$) condition that
 \[
 \langle \varphi_*, \gamma_x\rangle_{L^2(\mathbb{T}^1)} = \lim_{n\rightarrow\infty} \langle \varphi_n, \gamma_x\rangle_{L^2(\mathbb{T}^1)} = 0,
 \]
we then conclude that $ \varphi_*\cdot\gamma_x= 0$ for every $x$. 
\vspace{2mm}
 
The negative sectional curvature assumption implies that 
  \begin{align*}
      0 &= \lim_{n}\int_{\mathbb{T}^1} - \langle R(\gamma_x, \varphi_n)\gamma_x, \varphi_n\rangle(x)\,dx. \\
&=  -\int_{\mathbb{T}^1} \big(\langle R(\gamma_x, \varphi_*)\gamma_x, \varphi_*\rangle(x)\big)\,dx\\
&\geq c  \int_{\mathbb{T}^1} \big ( \langle  \gamma_x, \gamma_x \rangle\langle  \varphi_*, \varphi_* \rangle - (\langle \gamma_x, \varphi_* \rangle)^2\big)(x)  \,dx \\
&= c  \int_{\mathbb{T}^1} \langle  \varphi_*, \varphi_* \rangle(x)\,dx.
\end{align*}
This implies that $\varphi_* = 0$.  Therefore, 
\begin{equation*}
    \|\varphi_{n, x}\|_{L^2(\mathbb{T}^1)}\longrightarrow 1.
\end{equation*}

To arrive at a contradiction, recall the orthogonal tangent frame $\{{\bf f_j(x)}\}_{j= 1}^L$ of $T_{\gamma(x)}\mathcal{N}$ defined in Section \ref{Sec:cont:1},  we  write 
\[
\varphi_n(x) = \sum_{j=1}^L a_{j,n}(x)\cdot {\bf{f}}_j(x).
\] 
Then the fact that $\varphi_*=0$ yields
\begin{equation*}
    \|a_{j, n}\|_{L^2(\mathbb{T}^1)}\longrightarrow 0\;\; \forall j= 1,2,..., L.
\end{equation*}
Next, writing 
\begin{align*}
\partial_x\varphi_n (x) &= \sum_{j=1}^L \partial_x\big(a_{j,n}\big)\cdot {\bf{f}}_j(x) + \sum_{j=1}^L a_{j,n}(x)\cdot \partial_x{\bf{f}}_j(x),\\
\nabla_{\dot{\gamma}}\varphi_n (x) &= \sum_{j=1}^L \partial_x\big(a_{j,n}\big)\cdot {\bf{f}}_j(x) + \sum_{j=1}^L a_{j,n}(x)\cdot \nabla_{\gamma_x}{\bf{f}}_j(x),
\end{align*}
the convergence of $\{a_{j, n}\}$ leads to 
\begin{align*}
\|\partial_x \varphi_n- \nabla_{\gamma_x} \varphi_n\|_{L^2(\mathbb{T}^1)}\longrightarrow 0.
\end{align*}
Thus 
\begin{align*}
\|\nabla_{\gamma_x} \varphi_n\|_{L^2(\mathbb{T}^1)}\longrightarrow 1.
\end{align*}
This is in contradiction with \eqref{ine:con:nab:varphin}, and ends the proof of Proposition \ref{prop:coercivityofL}.
\end{proof}

\vspace{3mm}

\subsection{Exponential stability of the linearized equation on $\Psi$ based on propagation of smallness and coercive estimates}\label{subsec:linearstabilityvarphi}
We prove Proposition \ref{prop:propagationofsmallness} concerning the  equation: 
\begin{equation}\label{eq:linearkey}
-\Psi_{tt}+ \mathcal{L}_{\gamma} \Psi- a(x)\Psi_t = g \textrm{ with } \Psi[t]\in \mathcal{H}_{\gamma}.
\end{equation}

The key point shall be to understand how the friction term $ a(x)\Psi_t$ causes the exponential decay of a suitable energy functional, namely $\mathcal{E}_{\gamma}(\Psi[t])$ which we defined in Definition \ref{def:ene:varphi}. 
Recall that 
\[ \Psi(t, x)\in T_{\gamma(x)}\mathcal{N}\;\; \forall x\in \mathbb{T}^1, \] 
and the orthogonal tangent frame $\{{\bf f_j(x)}\}_{j= 1}^L$ of $T_{\gamma(x)}\mathcal{N}$ defined in Section \ref{Sec:cont:1}. Set 
\[
\Psi(t, x) = \sum_{j=1}^L a_j(t, x)\cdot {\bf{f}}_j(x), \; \forall x\in \mathbb{T}^1.
\]
Then $\Psi_t(t, x)$ stays in $T_{\gamma(x)}\mathcal{N}$: 
\[
\Psi_t(t, x) = \sum_{j=1}^n \partial_ta_j(t, x)\cdot {\bf{f}}_j(x)\in T_{\gamma(x)}\mathcal{N} \;\; \forall x\in \mathbb{T}^1.
\]
Thus   
\[
\partial_t\big(\nabla_{\dot{\gamma}}\Psi\big) = \nabla_{\dot{\gamma}}\big(\partial_t\Psi\big) \textrm{ and } \partial_t(\mathcal{L}_{\gamma}\Psi)= \mathcal{L}_{\gamma}\Psi_t.
\]
Using the fact that $\mathcal{L}_{\gamma}$ is self-adjoint for functions in $H^1(\mathbb{T}^1; \gamma^*(T\mathcal{N}))$, namely Lemma \ref{lem:selfadjoint}, 
\begin{align*}
-\frac{d}{dt}\langle \mathcal{L}_{\gamma}\Psi, \Psi\rangle_{L^2(\mathbb{T}^1)} &= -\langle  \partial_t(\mathcal{L}_{\gamma}\Psi), \Psi\rangle_{L^2(\mathbb{T}^1)}-\langle  \mathcal{L}_{\gamma}\Psi, \partial_t\Psi\rangle_{L^2(\mathbb{T}^1)} \\
&= -\langle  \mathcal{L}_{\gamma}\Psi_t, \Psi\rangle_{L^2(\mathbb{T}^1)}-\langle  \mathcal{L}_{\gamma}\Psi, \Psi_t\rangle_{L^2(\mathbb{T}^1)} \\
&= -2 \langle  \mathcal{L}_{\gamma}\Psi, \Psi_t\rangle_{L^2(\mathbb{T}^1)}.
\end{align*}

Hence the variation of the energy functional $\mathcal{E}_{\gamma}(\Psi[t])$ reads
\begin{align}
\frac{d}{dt} \mathcal{E}_{\gamma}(\Psi[t]) &= \langle \Psi_{tt}, \Psi_{t} \rangle_{L^2(\mathbb{T}^1)}  -\langle \mathcal{L}_{\gamma}\Psi, \Psi_t\rangle_{L^2(\mathbb{T}^1)}  \notag\\
&= \langle  \mathcal{L}_{\gamma} \Psi- a(x)\Psi_t- g, \Psi_{t} \rangle_{L^2(\mathbb{T}^1)}  -\langle \mathcal{L}_{\gamma}\Psi, \Psi_t\rangle_{L^2(\mathbb{T}^1)} \notag\\
&=  -\langle a(x)\Psi_t, \Psi_{t} \rangle_{L^2(\mathbb{T}^1)}- \langle g, \Psi_{t} \rangle_{L^2(\mathbb{T}^1)}. \label{eq:varivarEvarphi}
\end{align}
Heuristically, if during some period $t\in (0, T)$ the following two hold simultanesly 
\begin{gather*}
    \|g\|_{L^2(0, T; L^2(\mathbb{T}^1))}\ll \mathcal{E}_{\gamma}^{\frac12}(\Psi[0]),  \\
   \int_0^{T} \langle a(x)\Psi_t, \Psi_{t} \rangle_{L^2(\mathbb{T}^1)} \, dt\geq c \mathcal{E}_{\gamma}(\Psi[0]),
\end{gather*}
then $\mathcal{E}_{\gamma}(\Psi[t])$ decays.
\vspace{3mm}

In light of the preceding discussion, by combining the propagation of the smallness property (Corollary \ref{cor:propaga:linear}) and the coercive estimate Proposition \ref{prop:coercivityofL}, we deduce the following result concerning the asymptotic stability of the component $\Psi[t]$:

\begin{prop}\label{prop:propagationofsmallness}
There exist positive constants $\delta_{li}$ and $c_{li}$ such that for any pair $\Psi$ and $g$ satisfying the following properties 
\begin{gather}
    -\Psi_{tt}+ \mathcal{L}_{\gamma} \Psi- a(x)\Psi_t = g,  \label{eq:varphi:linear} \\
  \Psi[t]\in \mathcal{H}_{\gamma},\; \forall t\in [0, 32\pi], \label{eq:varphi:linear:2rd}  \\
      \|g\|_{L^2(0, 32\pi; L^2(\mathbb{T}^1))}^2\leq \delta_{li} \mathcal{E}_{\gamma}(\Psi[0]),  \label{eq:gcond} 
\end{gather}
one has 
\begin{equation}\label{ine:avarli1st:ori}
      \int_0^{32\pi} \langle a(x)\Psi_t, \Psi_{t} \rangle_{L^2(\mathbb{T}^1)} \, dt\geq c_{li} \mathcal{E}_{\gamma}(\Psi[0]),
\end{equation}
and 
\begin{gather}\label{ine:avarli2nd}
      \mathcal{E}_{\gamma}(\Psi[32\pi])\leq (1- c_{li})  \mathcal{E}_{\gamma}(\Psi[0]), \\
    \mathcal{E}_{\gamma}(\Psi[t])\leq 2  \mathcal{E}_{\gamma}(\Psi[0]), \; \forall t\in [0, 32\pi]. \label{ine:avarli3rd}
\end{gather}
\end{prop}

In the sequel, we reduce the proposition to the Lemma \ref{lem:propagationofsmallnessvarphi} and then establish this auxiliary result. 
\subsubsection{A reduction step: Lemma \ref{lem:propagationofsmallnessvarphi} implies Proposition \ref{prop:propagationofsmallness}}
\begin{lemma}\label{lem:propagationofsmallnessvarphi}
There exist constants $\tilde \delta_{li}>0$ and $\tilde c_{li}>0$ such that for any pair $\Psi$ and $g$ satisfying the following properties 
\begin{gather}
    -\Psi_{tt}+ \mathcal{L}_{\gamma} \Psi- a(x)\Psi_t = g,  \label{eq:varphi:linear:new} \\
    \Psi[t]\in \mathcal{H}_{\gamma},\; \forall t\in [0, 32\pi], \\
    \Psi[16 \pi]\in \mathcal{H}_{\gamma, 0}, \\
      \|g\|_{L^2(0, 32\pi; L^2(\mathbb{T}^1))}^2\leq \tilde \delta_{li} \|\Psi[16\pi]\|_{\mathcal{H}}^2,  \label{eq:g:linear:new} 
\end{gather}
one has 
\begin{equation}\label{ine:avarli1st}
      \int_0^{32\pi} \langle a(x)\Psi_t, \Psi_{t} \rangle_{L^2(\mathbb{T}^1)} \, dt\geq \tilde c_{li} \|\Psi[16\pi]\|_{\mathcal{H}}^2.
\end{equation}
\end{lemma}

We  first present the following basic energy estimate:
\begin{lemma}\label{lem:equinormEgamma}
    There exist $\delta>0$ and $C>0$ such that for every $s\in [0, 32\pi]$, and for every pair $\Psi$ and $g$ satisfying the following properties 
\begin{gather*}
    -\Psi_{tt}+ \mathcal{L}_{\gamma} \Psi- a(x)\Psi_t = g,   \\
    \Psi[t]\in \mathcal{H}_{\gamma},\; \forall t\in [0, 32\pi], \\
      \|g\|_{L^2(0, 32\pi; L^2(\mathbb{T}^1))}^2\leq \delta \mathcal{E}_{\gamma}(\Psi[s]),  
\end{gather*}
one has 
\begin{equation*}
   C^{-1}  \mathcal{E}_{\gamma}(\Psi[t_2]) \leq   \mathcal{E}_{\gamma}(\Psi[t_1])\leq C  \mathcal{E}_{\gamma}(\Psi[t_2]) \; \; \forall t_1, t_2\in [0, 32\pi]. 
\end{equation*}
and
\begin{equation*}
     \mathcal{E}_{\gamma}(\Psi[t_2]) \leq   2\mathcal{E}_{\gamma}(\Psi[t_1]) \;\; \forall 0\leq t_1\leq t_2\leq 32\pi. 
\end{equation*}
\end{lemma}
\begin{proof}
To prove the first inequality, it suffices to show that 
\begin{equation*}
     C^{-1}  \mathcal{E}_{\gamma}(\Psi[s]) \leq  \mathcal{E}_{\gamma}(\Psi[t])\leq C  \mathcal{E}_{\gamma}(\Psi[s]) \; \; \forall t\in [0, 32\pi]. 
\end{equation*}

Thanks to \eqref{eq:varivarEvarphi}, there exists $C$ such that 
\begin{equation*}
      \mathcal{E}_{\gamma}(\Psi[s_1])\leq C\left(  \mathcal{E}_{\gamma}(\Psi[s_2])+ \|g\|_{L^2_{t, x}}^2 \right) \; \; \forall s_1, s_2\in [0, 32\pi]. 
\end{equation*}

On the one hand, we have
\begin{equation*}
      \mathcal{E}_{\gamma}(\Psi[t])\leq C\left(  \mathcal{E}_{\gamma}(\Psi[s])+ \|g\|_{L^2_{t, x}}^2 \right)\leq (C+ \delta)  \mathcal{E}_{\gamma}(\Psi[s]),
\end{equation*}
On the other hand, by selecting $\delta$ small enough, we obtain
\begin{equation*}
      \mathcal{E}_{\gamma}(\Psi[s])\leq C\left(  \mathcal{E}_{\gamma}(\Psi[t])+ \|g\|_{L^2_{t, x}}^2 \right)\leq C\left(  \mathcal{E}_{\gamma}(\Psi[t])+ \delta  \mathcal{E}_{\gamma}(\Psi[s]) \right)\leq 2C \mathcal{E}_{\gamma}(\Psi[t]).
\end{equation*}

Next, we show the second inequality. Thanks to  \eqref{eq:varivarEvarphi}, 
\begin{align*}
     \mathcal{E}_{\gamma}(\Psi[t_2]) &\leq   \frac{3}{2}\mathcal{E}_{\gamma}(\Psi[t_1])+ C \|g\|_{L^2_{t, x}}^2 \\
     &\leq \frac{3}{2}\mathcal{E}_{\gamma}(\Psi[t_1])+ C \delta \mathcal{E}_{\gamma}(\Psi[s])\leq 2\mathcal{E}_{\gamma}(\Psi[t_1]),
\end{align*}
provided that $\delta$ is sufficiently small.
\end{proof}
Similarly, based on the well-posedness property Lemma \ref{lem:linea:well:gen}, we also obtain the following result without giving its proof:
\begin{lemma}\label{lem:equinormH1}
    There exist $\delta>0$ and $C>0$ such that for any $s\in [0, 32\pi]$, and for any pair $\Psi$ and $g$ satisfying equation \eqref{eq:varphi:linear:new} and  
\begin{gather*}
      \|g\|_{L^2(0, 32\pi; L^2(\mathbb{T}^1))}^2\leq \delta \|\Psi[s]\|_{\mathcal{H}}^2, 
\end{gather*}
one has 
\begin{equation*}
    C^{-1}  \|\Psi[t_2]\|_{\mathcal{H}}^2 \leq  \|\Psi[t_1]\|_{\mathcal{H}}^2 \leq C  \|\Psi[t_2]\|_{\mathcal{H}}^2 \; \; \forall t_1, t_2\in [0, 32\pi]. 
\end{equation*}
\end{lemma}
\vspace{2mm}
 Now, assume that Lemma \ref{lem:propagationofsmallnessvarphi} holds with some $\tilde \delta_{li}$ and $\tilde c_{li}$, we will find the constants $\delta_{li}$ and $c_{li}$ such that Proposition \ref{prop:propagationofsmallness} is fulfilled.  Suppose that $(\Psi, g)$ satisfy \eqref{eq:varphi:linear}--\eqref{eq:gcond}, and define 
 \begin{equation*}
     \tilde \Psi(t, x):= \Psi(t, x)- b \dot \gamma(x)\;\; \forall t\in [0, 32\pi]\;\; \forall x\in \mathbb{T}^1,
 \end{equation*}
where $b$ is a constant chosen such that $\tilde \Psi[16\pi]\in \mathcal{H}_{\gamma, 0}$.  Due to the observation \eqref{eq:mathcalLselfinva}, 
\begin{equation}
    \mathcal{E}_{\gamma}(\tilde \Psi[t])= \mathcal{E}_{\gamma}(\Psi[t])\;\; \forall t\in [0, 32\pi].
\end{equation}

Since $\mathcal{L}_{\gamma} \dot \gamma= 0$ and $\Psi_t= \tilde \Psi_t$, the new function $\tilde \Psi$ also satisfies 
\begin{equation*}
     -\tilde \Psi_{tt}+ \mathcal{L}_{\gamma} \tilde  \Psi- a(x) \tilde \Psi_t = g.
\end{equation*}
Moreover, by Lemma \ref{lem:equinormEgamma},  assuming $\delta_{li}$ small enough we have 
\begin{equation*}
    \|g\|_{L^2(0, 32\pi; L^2(\mathbb{T}^1))}^2\leq \delta_{li} \mathcal{E}_{\gamma}(\Psi[0])\leq C\delta_{li} \mathcal{E}_{\gamma}(\Psi[16\pi])= C\delta_{li} \mathcal{E}_{\gamma}(\tilde \Psi[16\pi])
\end{equation*}
Recall from Remark \ref{rem:equi:norms} that 
\begin{equation*}
  C^{-1}  \|(f, g)\|_{\mathcal{H}}^2\leq   \mathcal{E}_{\gamma}(f, g)\leq C \|(f, g)\|_{\mathcal{H}}^2 \; \; \forall (f, g)\in \mathcal{H}_{\gamma, 0}.
\end{equation*}
Because $\tilde \Psi[16\pi]$ belongs to $\mathcal{H}_{\gamma, 0}$, we get
\begin{equation*}
    \|g\|_{L^2(0, 32\pi; L^2(\mathbb{T}^1))}^2\leq  C_0\delta_{li} \|\tilde \Psi[16\pi]\|_{\mathcal{H}}^2.
\end{equation*}

Assuming that $ C_0\delta_{li}\leq \tilde \delta_{li}$, then $(\tilde \Psi, g)$ satsify the conditions given in Lemma \ref{lem:propagationofsmallnessvarphi}. Thus 
\begin{equation*}
      \int_0^{32\pi} \langle a(x)\tilde \Psi_t, \tilde \Psi_{t} \rangle_{L^2(\mathbb{T}^1)} \, dt\geq \tilde c_{li} \|\tilde \Psi[16\pi]\|_{\mathcal{H}}^2.
\end{equation*}
Therefore, 
\begin{equation*}
      \int_0^{32\pi} \langle a(x) \Psi_t,  \Psi_{t} \rangle_{L^2(\mathbb{T}^1)} \, dt\geq \tilde c_{li}  C^{-1} \mathcal{E}_{\gamma}(\tilde \Psi[16\pi])= \tilde c_{li}  C^{-1} \mathcal{E}_{\gamma}( \Psi[16\pi])\geq  \tilde c_{li}  C^{-1} \mathcal{E}_{\gamma}( \Psi[0]).
\end{equation*}
This is exactly inequality \eqref{ine:avarli1st:ori} by selecting $c_{li}$ as $ \tilde c_{li}  C^{-1}$.
\vspace{2mm}

We also notice that \eqref{ine:avarli2nd} is a direct consequence of \eqref{ine:avarli1st:ori} by decreasing the values of $c_{li}$ and $\delta_{li}$ if necessary.
Indeed,  
\begin{align*}
    \mathcal{E}_{\gamma}(\Psi[0])- \mathcal{E}_{\gamma}(\Psi[32\pi])&= \int_0^{32\pi}  \langle a(x)\Psi_t, \Psi_{t} \rangle_{L^2(\mathbb{T}^1)}+ \langle g, \Psi_{t} \rangle_{L^2(\mathbb{T}^1)} \, dt \\
    &\geq  c_{li} \mathcal{E}_{\gamma}( \Psi[0])- C \delta_{li}^{1/2} \mathcal{E}_{\gamma}(\Psi[0]).
\end{align*}
Finally, inequality \eqref{ine:avarli3rd} is already shown in Lemma \ref{lem:equinormEgamma}.
This finishes the proof of this reduction procedure.

\subsubsection{The proof of Lemma \ref{lem:propagationofsmallnessvarphi}}

 This proof is based on two aspects.
 \begin{itemize}
     \item[1)]  The propagation of smallness of this linear system with source terms. In this step we do not use the specific structure of the solution, namely that the solution belongs to $\mathcal{H}_{\gamma}$ and that the geodesic $\gamma$ has negative curvature.
     \item[2)] If some solution satisfies $\Psi_t$ small in a time period $J$, then $\langle \mathcal{L}_{\gamma}\Psi, \Psi\rangle$ must be small in a smaller interval inside $J$. This leads to a contradiction with the coercive estimate; Proposition \ref{prop:coercivityofL}.
 \end{itemize} 
\noindent{\it Step 1}. Assume that for some $\tilde \delta_{li}$ and $\tilde c_{li}$ small, there is some non-trivial pair $\Psi, g$ satisfying 
\begin{gather}
    -\Psi_{tt}+ \mathcal{L}_{\gamma} \Psi- a(x)\Psi_t = g, \label{111111111}\\
    \Psi[t]\in \mathcal{H}_{\gamma} \; \; \forall t\in [0, 32\pi], \\
      \Psi[16 \pi]\in \mathcal{H}_{\gamma, 0}, \\
      \|g\|_{L^2(0, 32\pi; L^2(\mathbb{T}^1))}^2\leq \tilde\delta_{li} \|\Psi[16\pi]\|_{\mathcal{H}}^2, 
\end{gather}
and 
\begin{equation}
      \int_0^{32\pi} \langle a(x)\Psi_t, \Psi_{t} \rangle_{L^2(\mathbb{T}^1)} \, dt\leq \tilde  c_{li} \|\Psi[16\pi]\|_{\mathcal{H}}^2. \label{2222222222}
\end{equation}

Then, the propagation of smallness property, Corollary \ref{cor:propaga:linear}, implies that 
\begin{equation}\label{eq:varphitsmall}
    \|\Psi_t\|_{L^{\infty}_x(\mathbb{T}^1; L^2_t(13\pi, 19\pi))}^2\leq C_q   \left(\tilde\delta_{li}^{1/q}+  \tilde c_{li}^{1/q}\right) \|\Psi[16\pi]\|_{\mathcal{H}}^2.
\end{equation}

Note from Lemma \ref{lem:equinormH1} that 
\begin{equation*}
    C^{-1}  \|\Psi[t_2]\|_{\mathcal{H}}^2 \leq  \|\Psi[t_1]\|_{\mathcal{H}}^2 \leq C  \|\Psi[t_2]\|_{\mathcal{H}}^2 \; \; \forall t_1, t_2\in [0, 32\pi]. 
\end{equation*}

Set for every $t\in [0, 32\pi]$,
\begin{gather*}
    \Psi(t, x)= \tilde\Psi(t, x)+ b(t) \dot\gamma(x) \textrm{ with } \\
    \tilde \Psi[t] \in \mathcal{H}_{\gamma, 0} \textrm{ and } b(t)= \langle \Psi(t, \cdot), \gamma_x(\cdot) \rangle_{L^2(\mathbb{T}^1)}.
\end{gather*}
Notice that in particular $b(16\pi)= 0$.  Furthermore, for any $t\in [13\pi, 19\pi]$ one has 
\begin{equation*}
    \Psi(t, x)= \Psi(16\pi, x)+ \int_{16\pi}^{t} \Psi_t (s, x) ds,
 \end{equation*}
hence
\begin{align*}
 b(t)&=    \langle \Psi(t, \cdot), \gamma_x(\cdot) \rangle_{L^2(\mathbb{T}^1)} \\
 &=  \left\langle \Psi(16\pi, \cdot)+ \int_{16\pi}^{t} \Psi_t (s, \cdot)ds, \gamma_x(\cdot) \right\rangle_{L^2(\mathbb{T}^1)} \\
 &=  \left\langle  \int_{16\pi}^{t} \Psi_t (s, \cdot)ds, \gamma_x(\cdot) \right\rangle_{L^2(\mathbb{T}^1)}.
\end{align*}
This, combined with inequality \eqref{eq:varphitsmall}, implies that for every $t\in [13\pi, 19\pi]$ 
\begin{equation}\label{eq:btsmall}
  |b(t)|^2\leq \|\Psi_t\|_{L^1_{t, x}((13\pi, 19\pi)^2\times \mathbb{T}^1)}^2\lesssim  \left(\tilde\delta_{li}^{1/q}+  \tilde c_{li}^{1/q}\right) \|\Psi[16\pi]\|_{\mathcal{H}}^2. 
\end{equation}

\noindent{\it Step 2}. Keep in mind the preceding estimate on $\|\Psi_t\|_{L^{\infty}_x(\mathbb{T}^1; L^2_t(13\pi, 19\pi))}$, we then come back to \eqref{eq:varphi:linear:new}, which we integrate against $\Psi$. The idea is to use the smallness of $\Psi_t$ on the time interval $(13\pi, 19\pi)$.   We define a non-negative smooth truncated function  $\eta$ such that
\[
\eta \;  \supp \; (13\pi, 19\pi), \; \eta= 1 \textrm{ on } [14\pi, 18\pi].
\]
Integrating equation \eqref{eq:varphi:linear}  against $\eta \Psi$, we obtain 
\begin{equation*}
    \int_{\mathbb{R}} \langle \mathcal{L}_{\gamma} \Psi, \eta \Psi \rangle_{L^2(\mathbb{T}^1)} \, dt= \int_{\mathbb{R}} \int_{\mathbb{T}^1} (\Psi_{tt}+ a(x)\Psi_t+ g)\cdot \eta \Psi \, dx  dt. 
\end{equation*}
Concerning the right hand side, we have 
\begin{align*}
 \left|  \int_{\mathbb{T}^1} \int_{\mathbb{R}} \Psi_{tt} \cdot \eta \Psi  \, dt dx \right|&=    \left|  \int_{\mathbb{T}^1} \int_{\mathbb{R}} \Psi_{t} \cdot (\eta_t \Psi+  \eta \Psi_t) \, dt dx  \right|\\
   &=  \left| \int_{13\pi}^{19\pi}  \langle  \Psi_t, \eta_t \Psi+  \eta \Psi_t \rangle_{L^2(\mathbb{T}^1)} \, dt \right|\\
   &\lesssim    \|\Psi_t\|_{L^2(13\pi, 19\pi; L^2(\mathbb{T}^1))} \|\Psi[t]\|_{L^2(13\pi, 19\pi; \mathcal{H})}    \\
    &\lesssim \left( \tilde\delta_{li}^{1/2q}+  \tilde c_{li}^{1/2q}\right) \|\Psi[16\pi]\|_{\mathcal{H}}^2.
\end{align*}
Similarly,
\begin{align*}
 \left|  \int_{\mathbb{T}^1} \int_{\mathbb{R}} (a(x)\Psi_t+ g)\cdot \eta \Psi \,dt dx \right| \lesssim \left( \tilde \delta_{li}^{1/2q}+  \tilde c_{li}^{1/2q
 }\right) \|\Psi[16\pi]\|_{\mathcal{H}}^2.
\end{align*}
Thus 
\begin{equation}
    \int_{\mathbb{R}} \langle \mathcal{L}_{\gamma} \Psi, \eta \Psi \rangle_{L^2(\mathbb{T}^1)} \, dt\leq \tilde{C} \left(\tilde  \delta_{li}^{1/2q}+  \tilde c_{li}^{1/2q}\right) \|\Psi[16\pi]\|_{\mathcal{H}}^2,
\end{equation}
where the constant $\tilde{C}$ is independent of the choice of $\tilde \delta_{li}$ and $\tilde c_{li}$.
\vspace{2mm}

On the other hand, thanks to the coercive estimate provided in Proposition \ref{prop:coercivityofL} and the observation \eqref{eq:mathcalLselfinva}, we have 
\begin{align*}
    \int_{\mathbb{R}} \langle \mathcal{L}_{\gamma} \Psi, \eta \Psi \rangle_{L^2(\mathbb{T}^1)} \, dt&=   \int_{13\pi}^{19\pi} \eta \langle \mathcal{L}_{\gamma} \Psi,  \Psi \rangle_{L^2(\mathbb{T}^1)} \, dt  \\ 
    &= \int_{13\pi}^{19\pi} \eta \langle \mathcal{L}_{\gamma} \tilde  \Psi,  \tilde \Psi \rangle_{L^2(\mathbb{T}^1)} \, dt  \\ 
    & \geq c \int_{14\pi}^{18\pi} \|\tilde \Psi(t)\|_{H^1(\mathbb{T}^1)}^2 \, dt \\
     & \geq \frac{c}{2} \int_{14\pi}^{18\pi} \| \Psi(t)\|_{H^1(\mathbb{T}^1)}^2 \, dt- C \int_{14\pi}^{18\pi}|b(t)|^2 \, dt \\
    &= \frac{c}{2} \int_{14\pi}^{18\pi} \left(\|\Psi[t]\|_{\mathcal{H}}^2- \|\Psi_t(t)\|_{L^2(\mathbb{T}^1)}^2 \right)\, dt- C \int_{14\pi}^{18\pi}|b(t)|^2 \, dt.
\end{align*}  
Due to the smallness of $\Psi_t$ and $b(t)$, namely inequalities \eqref{eq:varphitsmall} and \eqref{eq:btsmall}, by assuming $\tilde \delta_{li}$ and $\tilde c_{li}$ sufficiently small we have 
\begin{align*}
    \int_{\mathbb{R}} \langle \mathcal{L}_{\gamma} \Psi, \eta \Psi \rangle_{L^2(\mathbb{T}^1)} \, dt\geq  \tilde{c} \|\Psi[16\pi]\|_{\mathcal{H}}^2,
\end{align*}
where the constant $\tilde{c}$ is independent of the choice of $\tilde \delta_{li}$ and $\tilde c_{li}$. 

Hence any non-trivial solution satisfying \eqref{111111111}--\eqref{2222222222} must implies
\begin{equation*}
    \tilde \delta_{li}^{1/2q}+  \tilde c_{li}^{1/2q}\geq \tilde{c} \tilde{C}^{-1}.
\end{equation*}
This leads to the proof of the estimate \eqref{ine:avarli1st} by selecting $\tilde \delta_{li}$ and $\tilde c_{li}$ sufficiently small such that the preceding inequality fails. 
Thus we finish the proof of Lemma \ref{lem:propagationofsmallnessvarphi}.

\vspace{3mm}

\subsection{Exponential stability of the full system on $(\varphi, \alpha)$} \label{subsec:expostabproof}

\subsubsection{Fixed time stability}

This section is devoted to the following result. Its  proof is based on the following 
ingredients:  Proposition \ref{prop:propagationofsmallness} concerning the stability of the linearized equation on $\Psi$, the observation Lemma \ref{lem:key:ob:Psi} on $\mathcal{E}_{\gamma}(\Psi[t])$, the structure of the nonlinear term (namely the lack of the quadratic term $\alpha^2$), and the standard bootstrap argument.   
\begin{prop}\label{prop:main:stabilityT0}
    Let $\gamma$ be a closed geodesic on $\mathcal{N}$ along which the sectional curvature is negative. Let $c_{li}$ be the constant given in Proposition \ref{prop:propagationofsmallness}. There exist constants $\delta_{fi}>0$ and $M>1$ such that for every initial state $(\varphi, \varphi_t, \alpha, \dot\alpha)(0)\in \mathcal{H}_{\gamma, 0, 0}\times \mathbb{R}^2$ satisfying 
    \begin{equation}\label{eq:in:varEalpha}
        \mathcal{E}_{\gamma}(\varphi[0])+ \dot \alpha^2(0)+ \alpha^2(0) \leq \delta_{fi}^2, 
    \end{equation}
    the system \eqref{eq:varphialphafull} admits a unique solution in period $[0, 32\pi]$.  The function $ \varphi[t]$ belongs to $\mathcal{H}_{\gamma, 0, 0}$ for every $t\in [0, 32\pi]$, the function  $\Psi$ given as \eqref{eq:def:Psi} satisfies \eqref{eq:Psi:final:def}. Moreover, 
     \begin{gather}
     \mathcal{E}_{\gamma}(\Psi[32\pi])\leq (1- c_{li})  \mathcal{E}_{\gamma}(\Psi[0]), \label{estim:mu;decay} \\
    \mathcal{E}_{\gamma}(\Psi[t])\leq 2  \mathcal{E}_{\gamma}(\Psi[0]), \; \forall t\in [0, 32\pi], \label{estim:mu;unifboun} \\
    |\alpha(t)|\leq |\alpha(0)|+ M \mathcal{E}_{\gamma}^{1/2}(\Psi[0]), \; \forall t\in [0, 32\pi].  \label{estim:alpha;unifboun}
\end{gather}
\end{prop}

\begin{proof}[Proof of Proposition \ref{prop:main:stabilityT0}]

Define $D:= [0, 32\pi]\times \mathbb{T}^1$.  Recall the definition of space $\mathcal{W}= \mathcal{W}_{32\pi}$ given in Definition \ref{def:WTnorm}, the constant $L$ in \eqref{def:L:l},  the definition of constants $c_{li}$ and $\delta_{li}$ in   Proposition \ref{prop:propagationofsmallness}, and the constant $C_{co}>1$ from Remark \ref{rem:equi:norms} such that 
\begin{equation*}
  C^{-1}_{co}  \|(f, g)\|_{\mathcal{H}}\leq   \left(\mathcal{E}_{\gamma}(f, g)\right)^{1/2}\leq C_{co} \|(f, g)\|_{\mathcal{H}} \; \; \forall (f, g)\in \mathcal{H}_{\gamma, 0},
\end{equation*}
and the identity from Lemma \ref{lem:key:ob:Psi},
    \begin{equation}\label{eq:use:ide:EPsi}
   \mathcal{E}_{\gamma}(\Psi[t])=   \mathcal{E}_{\gamma}(\varphi[t])+  \frac{L}{2}\dot \alpha^2(t).
\end{equation}

We first present the following lemma concerning some useful estimates for the coupled system.
\begin{lemma}\label{lem:liwave:1:vari2}
There exist constants $\delta_2>0$ and $C_2>0$ such that,  for every initial state $(\varphi, \varphi_t, \alpha, \dot\alpha)(0)\in \mathcal{H}_{\gamma, 0, 0}\times \mathbb{R}^2$ satisfying 
    \begin{equation} \label{eq:cond:varphialphadosl:1}
       \mathcal{E}_{\gamma}(\varphi[0])+ \dot \alpha^2(0)+ \alpha^2(0) \leq \delta_2^2,  
    \end{equation}
    the system \eqref{eq:varphialphafull} admits a unique solution in period $[0, 32\pi]$.  The function $ \varphi[t]$ belongs to $\mathcal{H}_{\gamma, 0, 0}$ for every $t\in [0, 32\pi]$, the function  $\Psi$ given as \eqref{eq:def:Psi} satisfies \eqref{eq:Psi:final:def}. Moreover, 
      \begin{gather} 
  \|\varphi\|_{\mathcal{W}_{32\pi}}+ \|\alpha\|_{C^1(0, 32\pi)} \leq C_2 \big( \|\varphi[0]\|_{\mathcal{H}}+ |\alpha(0)|+ |\dot\alpha(0)|\big), \label{eq:vralpha1stbound}  \\
   \|\varphi\|_{\mathcal{W}_{32\pi}}+ \|\dot\alpha\|_{C(0, 32\pi)} \leq C_2 \big( \|\varphi[0]\|_{\mathcal{H}}+ |\dot\alpha(0)|\big), \label{eq:vralpha2ndbound}
\end{gather}
and the nonlinear terms on the right-hand side of \eqref{eq:varphialphafull} and \eqref{eq:Psi:final:def} satisfy
\begin{align}
&\;\;\;\;  \left\|\big(\mathcal{M}, \mathcal{M}_1\big)(x; \alpha, \dot\alpha, \varphi(\cdot), \varphi_t(\cdot), \varphi_x(\cdot))\right\|_{L^2_{t, x}(D)} + \|\mathcal{O}(\alpha, \dot\alpha, \varphi(\cdot), \varphi_t(\cdot), \varphi_x(\cdot)) \|_{L^2(0, 32\pi)} \notag
   \\
 &\leq C_2 \big( \|\varphi[0]\|_{\mathcal{H}}+ |\alpha(0)|+ |\dot\alpha(0)|\big)\big( \|\varphi[0]\|_{\mathcal{H}}+  |\dot\alpha(0)|\big).  \label{eq:vralpha3rdbound}
\end{align}
\end{lemma}  
\begin{proof}[Proof of Lemma \ref{lem:liwave:1:vari2}]

The small data well-posedness of this coupled system on $(\varphi, \alpha)$ is a consequence of the well-posedness of the system on $\phi$ and the equivalence between the flows of $\phi$ and $(\varphi, \alpha)$. Indeed,  thanks to the decomposition results Lemma \ref{lem:def:full:decom} and Remark \ref{rem:def:full:decom},  any small initial state $(\varphi, \varphi_t, \alpha, \dot\alpha)(0)\in \mathcal{H}_{\gamma, 0, 0}\times \mathbb{R}^2$ corresponds to an initial state $(\phi, \phi_t)(0)$ close to $(\gamma, 0)$. According to Lemma \ref{well:wm1}, the wave maps equation admits a unique solution $\phi[t]$ on the time interval $(0, 32\pi)$.  Thus due to the rule of decomposition we have that $ \varphi[t]$ belongs to $\mathcal{H}_{\gamma, 0, 0}$ for every $t\in [0, 32\pi]$, and according to Proposition \ref{lem:full:system:varalpha}  the function  $\Psi$ given as \eqref{eq:def:Psi} satisfies \eqref{eq:Psi:final:def}.

Notice that $(\gamma, 0)$ is a steady state of the damped wave maps equation, thus using the continuous dependence result Lemma \ref{lem-conti-dep-inh},
    \begin{gather*}
   \|(w_x, w_t, w)\|_{L^{\infty}_t L^2_x(D)}+  \|w_u\|_{L^{2}_u L^{\infty}_v\cap L^{\infty}_v L^{2}_u(D)}+  \|w_v\|_{L^{2}_v L^{\infty}_u\cap L^{\infty}_u L^{2}_v(D)} 
   \leq C \lVert w[0]\lVert_{\mathcal{H}},
\end{gather*}
 where $w(t, x):= \phi(t, x)- \gamma(x)$. This implies the unique solution $\phi[t]$ is close enough to the geodesic $(\gamma, 0)$ during this period. Using again Lemma \ref{lem:def:full:decom} and  Remark \ref{rem:def:full:decom}, the projection $(\varphi, \varphi_t, \alpha, \dot\alpha)(t)\in \mathcal{H}_{\gamma, 0, 0}\times \mathbb{R}^2$ is the unique solution of \eqref{eq:varphialphafull} with given data at time $t = 0$ and is small: for every $t\in [0, 32\pi]$,
 \begin{align}
     |\alpha(t)|+ |\dot \alpha(t)|+ \|\varphi[t]\|_{\mathcal{H}} &\leq C_{tr} \|\phi[t]- (\gamma, 0)\|_{\mathcal{H}} \notag\\
     &\leq C C_{tr}\lVert w[0]\lVert_{\mathcal{H}}\leq C C_{tr}^2 (|\alpha(0)|+ |\dot \alpha(0)|+ \|\varphi[0]\|_{\mathcal{H}}). \label{eq:bound:alphavarphi:t}
 \end{align}

To demonstrate estimates \eqref{eq:vralpha1stbound}--\eqref{eq:vralpha3rdbound}, we first  
estimate from the upper bound of $\mathcal{M}, \mathcal{M}_1, \mathcal{O}$  in  \eqref{ine:non:varphi}--\eqref{ine:non:alpha}  that 
\begin{align}
  &\;\;\;\; \left\|\Big(\mathcal{M}, \mathcal{M}_1\Big)
  (x; \alpha, \dot\alpha, \varphi(\cdot), \varphi_t(\cdot), \varphi_x(\cdot))\right\|_{L^2_{t, x}(D)}+ \left\| \mathcal{O}
  (\alpha, \dot\alpha, \varphi(\cdot), \varphi_t(\cdot), \varphi_x(\cdot))\right\|_{L^2_{t}(0, 32\pi)}  \notag\\
  &\lesssim  
  \left\|\sum_{j, k} |\partial_{\nu} \varphi^j \partial^{\nu} \varphi^k|+  |\dot \alpha|^2+  |\varphi_t||\dot \alpha|+  (|\varphi|+ |\alpha|) (|\varphi|+ |\varphi_x|+ |\varphi_t|+ |\dot\alpha| )\right\|_{L^2_{t, x}(D)}  \notag\\
  &\;\;\;\;\;\;\;\;\;\;\;\;\;\;\;\;\;\;\;\;\;\;\;\;  + \left\| \|\varphi_t\|_{L^1(\mathbb{T}^1)} \left(|\alpha|+ |\varphi|(x)+  \|\varphi\|_{L^1(\mathbb{T}^1)}\right)\right\|_{L^2_{t, x}(D)}  \notag\\
  &\lesssim \|\varphi\|_{\mathcal{W}_{32\pi}}^2+ \|\dot\alpha\|_{C(0, 32\pi)}^2+ \|\dot\alpha\|_{C(0, 32\pi)} \|\varphi\|_{\mathcal{W}_{32\pi}}+ (\|\varphi\|_{\mathcal{W}_{32\pi}}+ \|\alpha\|_{C(0, 32\pi)}) (\|\varphi\|_{\mathcal{W}_{32\pi}}+ \|\dot\alpha\|_{C(0, 32\pi)})  \notag\\
  &\lesssim   (\|\varphi\|_{\mathcal{W}_{32\pi}}+ \|\alpha\|_{C^1(0, 32\pi)}) (\|\varphi\|_{\mathcal{W}_{32\pi}}+ \|\dot\alpha\|_{C(0, 32\pi)}).  \label{es:N:123}
\end{align}

We first prove \eqref{eq:vralpha1stbound}.  The function $\varphi$ is governed by \eqref{eq:final:varphi}.
We regard this equation as a linear wave equation on $\varphi$ with the right-hand side as a given source term. Thus by the well-posedness result of the linear wave equation Lemma \ref{lem:linea:well:gen},  as well as  estimates \eqref{eq:bound:alphavarphi:t}-- \eqref{es:N:123},   we obtain 
\begin{align*} \|\varphi\|_{\mathcal{W}_{32\pi}}&\lesssim \|\varphi[0]\|_{\mathcal{H}}+ \|\mathcal{M}_3(x; \alpha, \dot\alpha, \varphi, \varphi_t, \varphi_x)\|_{L^2_{t, x}(D)}\\
&\;\;\;\;\;\; + \left\|\left(a(x)- \frac{l}{L}\right)\dot \alpha \gamma_x(x)- \frac{1}{L} \langle a(\cdot)\varphi_t(t, \cdot), \gamma_x(\cdot)\rangle_{L^2(\mathbb{T}^1)} \gamma_x(x) \right\|_{L^2_{t, x}(D)} \\
&\lesssim  \left(|\alpha(0)|+ |\dot \alpha(0)|+  \|\varphi[0]\|_{\mathcal{H}}\right)+ \left(\|\varphi\|_{\mathcal{W}_{32\pi}}+ \|\alpha\|_{C^1(0, 32\pi)} \right)\left(\|\varphi\|_{\mathcal{W}_{32\pi}}+ \|\dot\alpha\|_{C(0, 32\pi)} \right).
\end{align*}
Recall that $\alpha$ is bounded by 
\begin{equation*}
 \|\alpha\|_{C^1(0, 32\pi)}\lesssim   |\alpha(0)|+ |\dot \alpha(0)|+ \|\varphi[0]\|_{\mathcal{H}}. 
\end{equation*}
Thus, estimate \eqref{eq:vralpha1stbound} follows by combining the preceding two estimates and by assuming $\delta_2= \delta_{2, 0}$ sufficiently small.

\vspace{2mm}

Next, we turn to the proof of inequality \eqref{eq:vralpha2ndbound}. 
This time we shall first prove the desired estimate on a small time interval  $(0, T)$, and then iterate the procedure to $(0, 32\pi)$. Let the value of $T$ to be chosen later on. We consider the equation on $D_T= [0, T]\times \mathbb{T}^1$ and work with the  $\mathcal{W}_T$ norm. 

 In analogy to the estimate \eqref{es:N:123}, and using estimate \eqref{eq:vralpha1stbound}, 
\begin{align}
  & \left\|\Big(\mathcal{M}, \mathcal{M}_1\Big)
  (x; \alpha, \dot\alpha, \varphi(\cdot), \varphi_t(\cdot), \varphi_x(\cdot))\right\|_{L^2_{t, x}(D_T)}+ \left\| \mathcal{O}
  (\alpha, \dot\alpha, \varphi(\cdot), \varphi_t(\cdot), \varphi_x(\cdot))\right\|_{L^2_{t}(0, T)}  \notag\\
  & \;\;\;\;\;\; \lesssim   (\|\varphi\|_{\mathcal{W}_{T}}+ \|\alpha\|_{C^1(0, T)}) (\|\varphi\|_{\mathcal{W}_{T}}+ \|\dot\alpha\|_{C(0, T)}) \notag \\
  &\;\;\;\;\;\;\;\;\;\;\;\;\lesssim   \left(|\alpha(0)|+ |\dot \alpha(0)|+  \|\varphi[0]\|_{\mathcal{H}}\right)(\|\varphi\|_{\mathcal{W}_{T}}+ \|\dot\alpha\|_{C(0, T)})  \label{es:N:123:new}
\end{align}

Again, thanks to the well-posedness result of the linear wave equation Lemma \ref{lem:linea:well:gen},  as well as  estimates \eqref{eq:vralpha1stbound} and \eqref{es:N:123:new} we obtain 
\begin{align*} \|\varphi\|_{\mathcal{W}_{T}}&\lesssim \|\varphi[0]\|_{\mathcal{H}}+ T^{1/2} \|\mathcal{M}_1(x; \alpha, \dot\alpha, \varphi, \varphi_t, \varphi_x)\|_{L^2_{t, x}(D_T)}\\
&\;\;\;\;\;\; + T^{1/2} \left\|\left(a(x)- \frac{l}{L}\right)\dot \alpha \gamma_x(x)- \frac{1}{L} \langle a(\cdot)\varphi_t(t, \cdot), \gamma_x(\cdot)\rangle_{L^2(\mathbb{T}^1)} \gamma_x(x) \right\|_{L^2_{t, x}(D_T)} \\
&\lesssim \|\varphi[0]\|_{\mathcal{H}}+   T^{1/2} \left(|\alpha(0)|+ |\dot \alpha(0)|+  \|\varphi[0]\|_{\mathcal{H}}\right)(\|\varphi\|_{\mathcal{W}_{T}}+ \|\dot\alpha\|_{C(0, T)})     \\
&\;\;\;\; \;\;\;\;\;\;\;\; +  T \|\dot\alpha\|_{C(0, T)}+ T \|\varphi\|_{\mathcal{W}_T}    
\end{align*}
Next we turn to the equation governing $\alpha$ in \eqref{eq:varphialphafull}. Notice that it can be regarded as a first order equation on $\dot\alpha$, and
\begin{align*} 
\|\dot \alpha\|_{C(0, T)}&\lesssim |\dot\alpha(0)|+ \left\|\frac{1}{L}\left\langle  a(\cdot) \varphi_t(t, \cdot), \gamma_x(\cdot) \right\rangle_{L^2(\mathbb{T}^1)}+ \mathcal{O}(\alpha, \dot\alpha, \varphi(\cdot), \varphi_t(\cdot), \varphi_x(\cdot)) \right\|_{L^1(0, T)}\\
&\lesssim |\dot\alpha(0)|+ \|\varphi_t\|_{L^1_{t, x}(D_T)} + T^{1/2}\left\| \mathcal{O}(\alpha, \dot\alpha, \varphi(\cdot), \varphi_t(\cdot), \varphi_x(\cdot)) \right\|_{L^2(0, T)}\\
&\lesssim |\dot\alpha(0)|+  T \|\varphi\|_{\mathcal{W}_T} + T^{1/2} \left(|\alpha(0)|+ |\dot \alpha(0)|+  \|\varphi[0]\|_{\mathcal{H}}\right)(\|\varphi\|_{\mathcal{W}_{T}}+ \|\dot\alpha\|_{C(0, T)}).   
\end{align*}
By combining the above two estimates, we find effective computable $T>0$ and $C>0$ such that for any initial state satisfying \eqref{eq:cond:varphialphadosl:1} with $\delta_{2, 0}$, we infer 
\begin{equation*}
\|\varphi\|_{\mathcal{W}_{T}}+ \|\dot\alpha\|_{C(T)} \leq C_2 \big( \|\varphi[0]\|_{\mathcal{H}}+ |\dot\alpha(0)|\big),
\end{equation*}
and, in particular 
\begin{equation*}
\|\varphi\|_{C(0, T;\mathcal{H})}+ \|\dot\alpha\|_{C(T)}\leq C_2 \big( \|\varphi[0]\|_{\mathcal{H}}+ |\dot\alpha(0)|\big).  
\end{equation*}

Thanks to the bound  \eqref{eq:vralpha1stbound}, by reducing the value of $\delta_2$ we deduce that 
\begin{equation*}
\mathcal{E}_{\gamma}(\varphi[t])+ \dot \alpha^2(t)+ \alpha^2(t) \leq \delta_{2, 0}^2
\end{equation*}
holds for every $t\in [0, 32\pi]$ and for every initial state satisfying \eqref{eq:cond:varphialphadosl:1}. Therefore, by repeating the procedure on $\{[kT, (k+1T)]: k= 0, 1, ..., [32\pi/T]\}$, we conclude the estimate \eqref{eq:vralpha2ndbound}.

Finally, the inequality \eqref{eq:vralpha3rdbound} is a direct consequence of \eqref{es:N:123} together with the bounds \eqref{eq:vralpha1stbound}--\eqref{eq:vralpha2ndbound}. 
This finishes the proof of Lemma \ref{lem:liwave:1:vari2}.
\end{proof}

Now we come back to the proof of Proposition \ref{prop:main:stabilityT0}. Fix the value of $\delta_{fi}$ and $M$ as 
\begin{equation}
    \delta_{fi}:= \min \left\{\frac{\delta_{li}^{1/2}}{8 C_{co} C_2 (C_{co}+ 2/\sqrt{L})}, \delta_2\right\}, \;\; M:= 32\pi  C_2 \left( C_{co}+ 
2/\sqrt{L} \right).
\end{equation} 

The main issue is to obtain the estimate \eqref{estim:mu;decay}. Since $\delta_{fi}$ is smaller than $\delta_2$, thanks to Lemma \ref{lem:liwave:1:vari2}, the unique solution $(\varphi, \varphi_t, \alpha, \dot \alpha)$  and the  nonlinear terms 
$\mathcal{M}(x; \alpha, \dot\alpha, \varphi, \varphi_t, \varphi_x)$, $\mathcal{M}_1(x; \alpha, \dot\alpha, \varphi, \varphi_t, \varphi_x)$, and $\mathcal{O}(\alpha, \dot\alpha, \varphi(\cdot), \varphi_t(\cdot), \varphi_x(\cdot))$ satisfy
  \begin{gather} 
  \|\varphi\|_{\mathcal{W}}+ \|\alpha\|_{C^1(0, 32\pi)} \leq C_2 \big( \|\varphi[0]\|_{\mathcal{H}}+ |\alpha(0)|+ |\dot\alpha(0)|\big)\leq 2C_{co} C_2 \delta_{fi}, \notag  \\
   \|\varphi\|_{\mathcal{W}}+ \|\dot\alpha\|_{C(0, 32\pi)} \leq C_2 \big( \|\varphi[0]\|_{\mathcal{H}}+ |\dot\alpha(0)|\big), \notag \\
\|\mathcal{M}\|_{L^2_{t, x}(D)} + \|\mathcal{M}_1\|_{L^2_{t, x}(D)}+ \|\mathcal{O} \|_{L^2(0, 32\pi)} \leq 2C_{co} C_2 \delta \big( \|\varphi[0]\|_{\mathcal{H}}+ |\dot\alpha(0)|\big).  \notag
\end{gather}

The idea is to use Proposition \ref{prop:propagationofsmallness}. Clearly, by the definition of $\Psi$ in \eqref{eq:def:Psi} and equation \eqref{eq:Psi:final:def}, the function $\Psi$ satisfies conditions \eqref{eq:varphi:linear}--\eqref{eq:varphi:linear:2rd} with 
\begin{equation*}
    g= \mathcal{M}(x; \alpha, \dot\alpha, \varphi(\cdot), \varphi_t(\cdot), \varphi_x(\cdot)).
\end{equation*}
Furthermore, one has from the definition of $\delta_{fi}$ that 
\begin{align*}
  \|g\|_{L^2_{t, x}(D)}^2&=    \|\mathcal{M}\|_{L^2_{t, x}(D)}^2\leq 4C_{co}^2 C_2^2 \delta_{fi}^2 \big( \|\varphi[0]\|_{\mathcal{H}}+ |\dot\alpha(0)|\big)^2 \\
  &\leq  8C_{co}^2 C_2^2 \delta_{fi}^2  \left(C_{co}^2+ \frac{2}{L}\right) \mathcal{E}_{\gamma}(\Psi[0])\leq \frac{\delta_{li}}{2}\mathcal{E}_{\gamma}(\Psi[0]),
\end{align*}
namely, the condition \eqref{eq:gcond} also holds.  Thus, by applying Proposition \ref{prop:propagationofsmallness}, we obtain estimates \eqref{estim:mu;decay}--\eqref{estim:mu;unifboun}. 
Finally, the inequality \eqref{estim:alpha;unifboun} directly follows from 
\begin{align*}
    |\alpha(t)|&\leq |\alpha(0)|+ \int_0^t |\dot\alpha (s)| ds \\
    &\leq |\alpha(0)|+ 32\pi C_2 \big( \|\varphi[0]\|_{\mathcal{H}}+ |\dot\alpha(0)|\big) \\
    &\leq |\alpha(0)|+ 32\pi  C_2 \left( C_{co}+ 
2/\sqrt{L} \right)\mathcal{E}^{1/2}_{\gamma}(\Psi[0])
\end{align*}

This finishes the proof of Proposition \ref{prop:main:stabilityT0}.
\end{proof}

\subsubsection{The proof of Theorem \ref{thm:stability} Part $(ii)$}
Armed with the fixed time stability result Proposition \ref{prop:main:stabilityT0}, we are now in a position to present the proof of Theorem \ref{thm:stability}.

Recall the following constants: $\delta_{tr}$ and $C_{tr}$ in Remark \ref{rem:def:full:decom},  $C_{co}$ in Remark \ref{rem:equi:norms}, $ c_{li}$  in Proposition \ref{prop:propagationofsmallness},  $L$ in \eqref{def:L:l}, $\delta_{fi}$ and $M$ in Proposition \ref{prop:main:stabilityT0}.
Let us fix the value of $\varepsilon$ as
\begin{equation}
    \varepsilon:= \frac{\delta_{fi}}{2 C_{tr} L (C_{co}+ \sqrt{L}/2) (1+ 2/\sqrt{L}+ 2M/c_{li}) }.
\end{equation}

Thanks to Proposition \ref{prop:main:stabilityT0},  Remark \ref{rem:equi:norms}, the identity \eqref{eq:use:ide:EPsi}, using a simple induction argument we obtain the following: for any initial state $\phi[0]$ satisfying the condition,
 \begin{equation}\label{eq:pro:Thm2part2:ini}
        \|\phi[0]- (\gamma, 0)\|_{\mathcal{H}}\leq \varepsilon,
    \end{equation}
the system admits a unique solution in period $(0, +\infty)$. Moreover, under the decomposition $(\varphi, \varphi_t, \alpha, \dot \alpha)\in \mathcal{H}_{\gamma, 0, 0}\times \mathbb{R}^2$, we have that  for every $n\in \mathbb{N}$, 
 \begin{gather}
      \mathcal{E}_{\gamma}(\Psi[32\pi (n+1)])\leq (1- c_{li}) \mathcal{E}_{\gamma}(\Psi[32\pi n]), \label{eq:induc:1} \\
   \mathcal{E}_{\gamma}(\Psi[t])\leq 2 \mathcal{E}_{\gamma}(\Psi[32\pi n]) \; \; \forall t\in [32\pi n, 32\pi (n+1)], \label{eq:induc:2} \\
    |\alpha(t)|\leq |\alpha(32\pi n)|+ M \mathcal{E}_{\gamma}^{1/2}(\Psi[32\pi n]) \; \; \forall t\in [32\pi n, 32\pi (n+1)], \label{eq:induc:3}\\
    \mathcal{E}_{\gamma}(\varphi[32\pi (n+1)])+ \dot \alpha^2(32\pi (n+1))+ \alpha^2(32\pi (n+1)) < \delta_{fi}^2. \label{eq:induc:4}
\end{gather}

The case $n$ equals 0 is a direct consequence of Proposition \ref{prop:main:stabilityT0}. Suppose that the argument is correct for every $n\in \{0,..., k\}$. In particular, inequality \eqref{eq:induc:4} for the case $n= k$ is exactly the smallness condition  of the initial state at time $32\pi (k+1)$ for Proposition \ref{prop:main:stabilityT0}. Then, for the case $n$ equals to $k+1$,  the existence of a unique solution as well as the estimates \eqref{eq:induc:1}--\eqref{eq:induc:3}  are guaranteed by Proposition \ref{prop:main:stabilityT0}. It suffices to show  inequality \eqref{eq:induc:4}  to conclude the induction argument.

Indeed, thanks to the inequality \eqref{eq:induc:1} for the cases $n\in \{0,..., k+1\}$, 
\begin{align*}
\mathcal{E}_{\gamma}(\Psi[ 32\pi (k+2)])&\leq (1- c_{li})  \mathcal{E}_{\gamma}(\Psi[ 32\pi (k+1)]) \\
&\leq (1- c_{li})^2  \mathcal{E}_{\gamma}(\Psi[ 32\pi k]) \leq (1- c_{li})^{k+2}  \mathcal{E}_{\gamma}(\Psi[0])
\end{align*}
These bounds, to be combined with the inequality \eqref{eq:induc:3} for the cases $n\in \{0,..., k+1\}$, yield
\begin{align*}
|\alpha(32\pi (k+2))|&\leq |\alpha(32\pi (k+1))|+ M  \mathcal{E}_{\gamma}^{1/2}(\Psi[32\pi (k+1)])  \\
&\leq |\alpha(32\pi k)|+ M  \left(\mathcal{E}_{\gamma}^{1/2}(\Psi[32\pi (k+1)])+  \mathcal{E}_{\gamma}^{1/2}(\Psi[32\pi k])\right) \\
&\leq |\alpha(0)|+ M \sum_{n= 0}^{k+1} \mathcal{E}_{\gamma}^{1/2}(\Psi[32\pi n])\\ 
&\leq |\alpha(0)|+ M \sum_{n= 0}^{k+1}  (1- c_{li})^{n/2}   \mathcal{E}_{\gamma}^{1/2}(\Psi[0])  \\
&\leq |\alpha(0)|+ \frac{2M}{c_{li}} \mathcal{E}_{\gamma}^{1/2}(\Psi[0]).
\end{align*}
Hence,
\begin{align*}
    &\;\;\;\; \mathcal{E}_{\gamma}(\varphi[32\pi (k+2)])+ \dot \alpha^2(32\pi (k+2))+ \alpha^2(32\pi (k+2)) \\
    &\leq  \left(1+ \frac{2}{L}+  \frac{4 M^2}{c_{li}^2}\right)  \mathcal{E}_{\gamma}(\Psi[0]) +  2|\alpha(0)|^2  \\
    &= \left(1+ \frac{2}{L}+  \frac{4 M^2}{c_{li}^2}\right)  \left( C_{co}^2\|\varphi[0]\|_{\mathcal{H}}^2+ \frac{L}{2} \dot\alpha^2(0)\right) +  2|\alpha(0)|^2 \\
    & \leq \left(1+ \frac{2}{L}+  \frac{4 M^2}{c_{li}^2}\right)  \left( C_{co}^2+ \frac{L}{2} \right)\left(|\alpha(0)|+ |\dot \alpha(0)|+ \|\varphi[0]\|_{\mathcal{H}}\right)^2   \\
    & \leq \left(1+ \frac{2}{L}+  \frac{4 M^2}{c_{li}^2}\right)  \left( C_{co}^2+ \frac{L}{2} \right) C_{tr}^2 \varepsilon^2< \delta_{fi}^2. 
\end{align*}
This concludes the induction procedure. 

\vspace{3mm}

Now, we are in a position to prove the exponential stability property. 
First, we observe that $|\dot \alpha(t)|+ \|\varphi[t]\|_{\mathcal{H}}$ decays exponentially: for every $t\in (0, +\infty)$,
\begin{align*}
 |\dot \alpha(t)|+ \|\varphi[t]\|_{\mathcal{H}}&\lesssim  \mathcal{E}_{\gamma}^{1/2}(\Psi[ t]) 
\lesssim  e^{- r t}   \mathcal{E}_{\gamma}^{1/2}(\Psi[0]) \lesssim  e^{- r t} \left(|\dot \alpha(0)|+ \|\varphi[0]\|_{\mathcal{H}}\right).
\end{align*}
Next, we fix the value of rotation $p$, which is indeed the limit of $\alpha(t)$.  Due to the estimates \eqref{eq:induc:1} and \eqref{eq:induc:4}, and the Cauchy convergence principle, the sequence $\{\alpha(t)\}_{t\in (0, +\infty)}$ admits a limit:
\begin{gather*}
  |p|\leq C \left(|\alpha(0)|+ |\dot \alpha(0)|+ \|\varphi[0]\|_{\mathcal{H}}\right), \\
    |\alpha(t)- p|\leq C e^{- r t} \left(|\alpha(0)|+ |\dot \alpha(0)|+ \|\varphi[0]\|_{\mathcal{H}}\right)\;\; \forall t\in (0, +\infty).
\end{gather*}

Recall from Lemma \ref{lem:def:full:decom} concerning the decomposition of $\phi$: 
\begin{equation}
        \phi(t, x)= \gamma(x+ \alpha(t))+ \varphi(t, x)+ \mathcal{F}(x; \varphi(t, x), \alpha(t)) \;\; \forall t\in (0, +\infty)\;\; \forall  x\in \mathbb{T}^1, \notag
    \end{equation}
    where the nonlinear mapping $\mathcal{F}$ satisfies \eqref{esonF2}.
Thus,  
\begin{align*}
  \|\phi(t)- \gamma_p\|_{L^2(\mathbb{T}^1)}
   &\lesssim  \|\gamma_{\alpha(t)}- \gamma_p\|_{L^2(\mathbb{T}^1)}+ \|\varphi(t)\|_{L^2(\mathbb{T}^1)}+ \|\mathcal{F}(x; \varphi(t, x), \alpha(t))\|_{L^2(\mathbb{T}^1)} \\
   &\lesssim |\alpha(t)- p|+ \|\varphi(t)\|_{L^2(\mathbb{T}^1)} \\
   &\lesssim e^{- rt} \left(|\alpha(0)|+ |\dot \alpha(0)|+ \|\varphi[0]\|_{\mathcal{H}}\right)\;\; \forall t\in (0, +\infty).
\end{align*}
And, similarly, 
\begin{align*}
   &\;\;\;\; \|(\phi(t)- \gamma_p)_x\|_{L^2(\mathbb{T}^1)}\\
   &\lesssim  \|(\gamma_{\alpha(t)}- \gamma_p)_x\|_{L^2(\mathbb{T}^1)}+ \|\varphi_x(t)\|_{L^2(\mathbb{T}^1)}+ \|\mathcal{F}_x(x; \varphi(t, x), \alpha(t))\|_{L^2(\mathbb{T}^1)}+ \|\mathcal{F}_{\varphi}(x; \varphi(t, x), \alpha(t)) \varphi_x(t, x)\|_{L^2(\mathbb{T}^1)} \\
   &\lesssim  \|(\gamma_{\alpha(t)}- \gamma_p)_x\|_{L^2(\mathbb{T}^1)}+ \|\varphi_x(t)\|_{L^2(\mathbb{T}^1)}+ \||\varphi|^2+ |\varphi||\alpha| \|_{L^2(\mathbb{T}^1)}+ \|(|\varphi|+ |\alpha|)\varphi_x\|_{L^2(\mathbb{T}^1)} \\
   &\lesssim |\alpha(t)- p|+ \|\varphi(t)\|_{H^1(\mathbb{T}^1)} \\
   &\lesssim e^{- rt} \left(|\alpha(0)|+ |\dot \alpha(0)|+ \|\varphi[0]\|_{\mathcal{H}}\right)\;\; \forall t\in (0, +\infty).
\end{align*}
as well as 
\begin{align*}
 &\;\;\;\; \|\phi_t(t)\|_{L^2(\mathbb{T}^1)}\\
  &\lesssim \|(\gamma_{\alpha(t)})_x \dot\alpha\|_{L^2(\mathbb{T}^1)}+ \|\varphi_t(t)\|_{L^2(\mathbb{T}^1)}+  \|\mathcal{F}_{\varphi}(x; \varphi(t, x), \alpha(t)) \varphi_t(t, x)\|_{L^2(\mathbb{T}^1)}+ \|\mathcal{F}_{\alpha}(x; \varphi(t, x), \alpha(t)) \dot\alpha\|_{L^2(\mathbb{T}^1)}\\
   &\lesssim |\dot\alpha(t)|+ \|\varphi_t(t)\|_{L^2(\mathbb{T}^1)} \\
  &\lesssim  e^{- r t} \left(|\dot \alpha(0)|+ \|\varphi[0]\|_{\mathcal{H}}\right)\;\; \forall t\in (0, +\infty).
\end{align*}

By combining the above three inequalities, we get the exponential stability result:
\begin{align*}
    \|\phi[t]- (\gamma_p, 0)\|_{\mathcal{H}} \lesssim e^{- rt} \left(|\alpha(0)|+ |\dot \alpha(0)|+ \|\varphi[0]\|_{\mathcal{H}}\right)\;\; \forall t\in (0, +\infty),
\end{align*}
provided that the initial condition \eqref{eq:pro:Thm2part2:ini} is verified.

Thus, this finishes the proof of Theorem \ref{thm:stability}.
\vspace{5mm}

\noindent\textbf{Acknowledgments} \;
Part of this work was accomplished during visits between the three authors, and they would like to thank Sorbonne Université, EPFL, and Peking University for the warm hospitality. Jean-Michel Coron would like to thank Xu Zhang for the kind invitation to visit Sichuan University.  This visit, during which part of this article was carried out,  was supported by his contract with the New Cornerstone Science Foundation.  Joachim Krieger would like to thank the hospitality of IHP in June 2025.  Shengquan Xiang was supported by the NSF of China under Grant No. 12301562.

\appendix

\section{Proofs of some geometric lemmas}\label{sec:app:A}

This section is devoted to the proofs of some geometric  Lemmas \ref{lem:secfunext}, \ref{lem:selfadjoint}, and \ref{lem:movingframe}.

\noindent {\bf A.1.}\label{setting(Aex1)}  {\bf On the extension of the second fundamental form.}

\begin{proof}[Proof of Lemma \ref{lem:secfunext}]
Since $\mathcal{N}\subset \mathbb{R}^N$ is a compact submanifold of dimension $R$, there exists an open tubular neighborhood $\mathcal{U}$ of $\mathcal{N}$ in $\mathbb{R}^N$ and a smooth projection $\mathcal{P}: \mathcal{N}\rightarrow \mathcal{N}$ such that every point $x\in \mathcal{U}$ can be uniquely written as 
\begin{equation*}
    x= y+ v \textrm{ with } y= \mathcal{P}(x)\in \mathcal{N}, v\in N_{y} \mathcal{N}.
\end{equation*}
We have the smooth orthogonal projections, for every   $\phi\in \mathcal{N}$, 
\begin{align*}
    \Pi_T(\phi)&: \mathbb{R}^N\longrightarrow T_{\phi} \mathcal{N}, \\
     \Pi_N(\phi) &: \mathbb{R}^N\longrightarrow N_{\phi} \mathcal{N}.
\end{align*}

Recall that the second fundamental form $\Pi$ is  a smooth tensor field: $
    \Pi: T\mathcal{N}\times T\mathcal{N}\rightarrow N \mathcal{N}$. Now we extend it to $\tilde{\tilde \Pi}$ as follows. For $(x, u, v)\in \mathcal{U}\times \mathbb{R}^N\times \mathbb{R}^N$, define 
    \begin{equation*}
      \tilde{\tilde \Pi}(x) (u, v):= \Pi\left(\mathcal{P}(x)\right)  \Big(\Pi_T(\mathcal{P}(x)) u, \Pi_T(\mathcal{P}(x)) v \Big)\in N_{\mathcal{P}(x)} \mathcal{N} \subset \mathbb{R}^N.
    \end{equation*}
This new map $\tilde \Pi$ is smooth on $\mathcal{U}\times \mathbb{R}^N\times \mathbb{R}^N$. For every $(\phi, u, v)$ with $\phi\in \mathcal{N}$ and $u, v\in T_{\phi}\mathcal{N}$ we recover the original map:
\begin{equation*}
 \tilde{\tilde \Pi}(\phi) (u, v)= \Pi(\phi) (u, v).   
\end{equation*}

Next, we can define a map $\tilde \Pi$ from $\mathbb{R}^N\times \mathbb{R}^N\times \mathbb{R}^N$ to $\mathbb{R}^N$. Choose a smooth cutoff function $\chi: \mathbb{R}^N \rightarrow [0, 1]$ such that $\chi= 1$ on a smaller tube $\mathcal{U}_0\subset \mathcal{U}$ and that $\supp \chi\subset \mathcal{U}$. Define, for every $(x, u, v)\in \mathbb{R}^N\times \mathbb{R}^N\times \mathbb{R}^N$,
\begin{equation*}
  \tilde{\Pi} (x) (u, v):= 
  \begin{cases}
      \chi(x) \tilde{\tilde \Pi}(x) (u, v) \; \textrm{ if } x\in \mathcal{U}, \\
      0 \; \textrm{ otherwise.}
  \end{cases}
\end{equation*}
This new map is smooth and agrees with $\Pi$ for every $(\phi, u, v)$ with $\phi\in \mathcal{N}$ and $u, v\in T_{\phi}\mathcal{N}$. By the construction, this new map is clearly bilinear and symmetric.  

Now we define the smooth functions $\{S_{jk}: \mathbb{R}^N \rightarrow \mathbb{R}^N\}$ as follows. Let the vectors $\{\iota_1, ..., \iota_N\}$ be the standard orthonormal basis of $\mathbb{R}^N$. Set 
\begin{equation*}
    S_{jk}(x):= -\tilde \Pi(x) (\iota_j, \iota_k)\in \mathbb{R}^N.
\end{equation*}
It is automatically symmetric due to the symmetry of $\tilde \Pi$. Moreover, by the definition of $\{S_{jk}\}$ and the bilinearity of $\tilde \Pi$, one has for every $ x\in\mathbb{R}^N,\; \forall\, v,w\in \mathbb{R}^N, $
\begin{equation*}
   \tilde\Pi(x)(v,w)= \sum_{j,k=1}^N  \tilde\Pi(x)(\iota_j, \iota_k) v^j w^k = -  \sum_{j,k=1}^N S_{jk}(x)v^j w^k.
\end{equation*}
This finishes the proof of Lemma \ref{lem:secfunext}.
\end{proof} 

\vspace{2mm}

\noindent {\bf A.2.}\label{setting(A)}  {\bf  Proof of Lemma \ref{lem:selfadjoint}. } 
Let us now check the operator $\mathcal{L}_{\gamma}$ given in Definition \ref{def:jacobian} satisfies the identities in Lemma \ref{lem:selfadjoint}. Without loss of generality, we assume that $|\gamma_x(x)|= 1$. Let $\varphi$
 be a vector field along the geodesic $\gamma$. Our first goal is to show that 
 \begin{align}\label{eq:Lvarvarshow}
        -\langle \mathcal{L}_{\gamma}\varphi, \varphi\rangle_{L^2(\mathbb{T}^1)} = \int_{\mathbb{T}^1}\big(\big\|\nabla_{\gamma_x}\varphi\big\|^2 - \langle R(\gamma_x, \varphi)\gamma_x, \varphi\rangle\big)\,dx.
    \end{align}
 
 Construct a map 
\begin{align*}
    c(\cdot, \cdot): (- \varepsilon, \varepsilon)\times \mathbb{T}^1&\rightarrow \mathcal{N}, \\
    (s, x)&\mapsto \exp_{\gamma(x)} \left(s \varphi(x) \right).
\end{align*}
For any $s$, $c(s, \cdot)$ is a closed curve on $\mathcal{N}$ that is close to $\gamma$. By the construction of $c$ and the extension of exponential maps,  we have
\begin{equation*}
    c(s, x)= \gamma(x)+ s \varphi(x)- \frac{s^2}{2} \Pi\left(\varphi(x), \varphi(x)\right)+ O(s^3).
\end{equation*}
Define the energy for the curve $c(s, \cdot)$ as
\begin{equation*}
    E(c(s, \cdot)):= \frac{1}{2}\int_{\mathbb{T}^1} \left|\frac{\partial}{\partial x} c(s, x)\right|^2 \, dx 
\end{equation*}

On the one hand, thanks to Synge’s second variation 
formula for the energy (see \cite[Theorem 6.1.4]{2016-Petersen-book}),
\begin{equation*}
    \frac{d^2}{ d s^2} E(c(s, \cdot))|_{s= 0}= \int_{\mathbb{T}^1}\big(\big\|\nabla_{\gamma_x}\varphi\big\|^2 - \langle R(\gamma_x, \varphi)\gamma_x, \varphi\rangle\big)\,dx.
\end{equation*}

On the other hand, direct calculation yields
\begin{align*}
    2 E(c(s, \cdot))= \int_{\mathbb{T}^1}   \Big|\gamma_x+ s \varphi_x- \frac{s^2}{2} \partial_x \left( \Pi\left(\varphi(x), \varphi(x)\right)\right)+ O(s^3) \Big|^2    \, dx.
\end{align*}
Recall that 
\begin{equation*}
    \gamma_{xx}= \nabla_{\gamma_x} \gamma_x+ \Pi(\gamma_x, \gamma_x).
\end{equation*}
Thus 
\begin{align*}
    \frac{d^2}{ d s^2} E(c(s, \cdot))|_{s= 0}&=   \int_{\mathbb{T}^1}  |\varphi_x|^2 \, dx- \int_{\mathbb{T}^1}  \left \langle  \partial_x \left( \Pi\left(\varphi(x), \varphi(x)\right)\right), \gamma_x(x) \right\rangle  \, dx, \\
    &=   \int_{\mathbb{T}^1}  |\varphi_x|^2 \, dx+ \int_{\mathbb{T}^1}  \left \langle   \Pi\left(\varphi(x), \varphi(x)\right),\partial_x (\gamma_x)(x) \right\rangle  \, dx, \\
    &=   \int_{\mathbb{T}^1}  |\varphi_x|^2 \, dx+ \int_{\mathbb{T}^1}  \left \langle  \Pi\left(\varphi, \varphi\right),\Pi\left(\gamma_x, \gamma_x\right) \right\rangle  \, dx, 
\end{align*}
Thus, to show \eqref{eq:Lvarvarshow} it suffices to prove
\begin{align}
\begin{split}\label{eq:SPivarphi}
      & \;\;\; \int_{\mathbb{T}^1} \langle \varphi^r\partial_rS_{jk}(\gamma)\partial_x\gamma^j\partial_x\gamma^k + 2S_{jk}(\gamma)\partial_x\gamma^j\partial_x\varphi^k, \varphi\rangle \, dx \\
  & =  \int_{\mathbb{T}^1}  \left \langle  \partial_x \left( \Pi\left(\varphi(x), \varphi(x)\right)\right), \gamma_x(x) \right\rangle  \, dx     \\
 \;\; & =  -\int_{\mathbb{T}^1}  \left \langle  \Pi\left(\varphi, \varphi\right),\Pi\left(\gamma_x, \gamma_x\right) \right\rangle  \, dx. 
\end{split}
\end{align}

The equality between the first and second terms is not straightforward to establish. Thus we turn to show the equivalence between the first and third. 
Consider an auxiliary function
\begin{align*}
 I(s):= \int_{\mathbb{T}^1} \left\langle - \Pi\left(c(s, x)\right)\left(\frac{\partial c}{\partial x} (s, x), \frac{\partial c}{\partial x} (s, x)\right), \frac{\partial c}{\partial s} (s, x) \right\rangle  \, dx. 
\end{align*}
Note that the above equation is well-defined, since for every given $(s, x)$, the vectors $\frac{\partial c}{\partial x} (s, x)$ and $\frac{\partial c}{\partial s} (s, x)$ belong to $T_{c(s, x)}\mathcal{N}$. Recall the extrinsic formula of the second fundamental form, equation \eqref{eq:2ndfundcoord}. Thus 
\begin{equation*}
I(s):= \int_{\mathbb{T}^1} \left\langle S_{jk}\left(c(s, x)\right)\left(\frac{\partial c}{\partial x}\right)^j (s, x) \left(\frac{\partial c}{\partial x}\right)^k(s, x), \frac{\partial c}{\partial s} (s, x) \right\rangle  \, dx.
\end{equation*}
Note that the second derivatives such as $\frac{\partial^2 c}{ \partial s \partial x}$ shall be understood in $\mathbb{R}^N$.
Direct calculation yields 
\begin{align*}
    \frac{d}{ds} I(0)&=  \int_{\mathbb{T}^1} \langle \varphi^r\partial_rS_{jk}(\gamma)\partial_x\gamma^j\partial_x\gamma^k + 2S_{jk}(\gamma)\partial_x\gamma^j\partial_x\varphi^k, \varphi\rangle \, dx \\
    &\;\;\;\;  + \int_{\mathbb{T}^1} \left\langle - \Pi\left(\gamma\right)\left(\gamma_x, \gamma_x\right), - \Pi\left(\gamma(x))\right)\left(\varphi(x), \varphi(x)\right) \right\rangle  \, dx, \\
   & =  \int_{\mathbb{T}^1} \langle \varphi^r\partial_rS_{jk}(\gamma)\partial_x\gamma^j\partial_x\gamma^k + 2S_{jk}(\gamma)\partial_x\gamma^j\partial_x\varphi^k, \varphi\rangle \, dx \\
     &\;\;\;\;  + \int_{\mathbb{T}^1} \left\langle  \Pi\left(\gamma\right)\left(\gamma_x, \gamma_x\right),  \Pi\left(\gamma\right)\left(\varphi, \varphi\right) \right\rangle  \, dx.
\end{align*}
Meanwhile, we also notice from the definition of $I(s)$ that for every $s$, 
\begin{equation*}
    I(s)= \int_{\mathbb{T}^1} 0\, dx= 0.
\end{equation*}
Hence, the equality \eqref{eq:SPivarphi} holds. This finishes the proof of the identity \eqref{eq:Lvarvarshow}.
\vspace{2mm}

Next, we show that $\mathcal{L}_{\gamma}$ is self-adjoint. Given two tangent fields $\varphi, \psi$ along the geodesic $\gamma$.  Similarly, we construct a map 
\begin{align*}
    c(\cdot, \cdot, \cdot): (- \varepsilon, \varepsilon)\times (- \varepsilon, \varepsilon)\times \mathbb{T}^1&\rightarrow \mathcal{N}, \\
    (s, t,  x)&\mapsto \exp_{\gamma(x)} \left(s \varphi(x)+ t\psi(x) \right).
\end{align*}
For any $s, t\in (-\varepsilon, \varepsilon)$, $c(s, t, \cdot)$ is a closed curve on $\mathcal{N}$ that is close to $\gamma$. We have
\begin{equation*}
    c(s, x)= \gamma(x)+ s \varphi(x)+ t\psi(x) - \frac{1}{2} \Pi\left(s\varphi(x)+ t\psi(x), s\varphi(x)+ t\psi(x)\right)+ O(|(s, t)|^3).
\end{equation*}
Consider an auxiliary function
\begin{align*}
 II(s, t):= \int_{\mathbb{T}^1} \left\langle - \Pi\left(c(s, t, x)\right)\left(\frac{\partial c}{\partial x} (s, t, x), \frac{\partial c}{\partial x} (s, t,  x)\right), \frac{\partial c}{\partial t} (s, t, x) \right\rangle  \, dx= 0. 
\end{align*}
Then, direct calculation yields 
\begin{align*}
    \frac{\partial}{\partial s} II(0, 0)&= \int_{\mathbb{T}^1} \langle \varphi^r\partial_rS_{jk}(\gamma)\partial_x\gamma^j\partial_x\gamma^k + 2S_{jk}(\gamma)\partial_x\gamma^j\partial_x\varphi^k, \psi\rangle \, dx \\
    &\;\;\;\;\;\;\;\;  +  \int_{\mathbb{T}^1} \left\langle  \Pi\left(\gamma\right)\left(\gamma_x, \gamma_x\right), \Pi\left(\gamma\right)\left(\varphi, \psi\right) \right\rangle  \, dx.
\end{align*}
By the symmetry of the second fundamental form, we obtain 
\begin{align*}
  & \;\;\;\;  \int_{\mathbb{T}^1} \langle \varphi^r\partial_rS_{jk}(\gamma)\partial_x\gamma^j\partial_x\gamma^k + 2S_{jk}(\gamma)\partial_x\gamma^j\partial_x\varphi^k, \psi\rangle \, dx  \\
 &   =  -\int_{\mathbb{T}^1} \left\langle  \Pi\left(\gamma\right)\left(\gamma_x, \gamma_x\right), \Pi\left(\gamma\right)\left(\varphi, \psi\right) \right\rangle  \, dx\\
 & = \int_{\mathbb{T}^1} \langle \psi^r\partial_rS_{jk}(\gamma)\partial_x\gamma^j\partial_x\gamma^k + 2S_{jk}(\gamma)\partial_x\gamma^j\partial_x\psi^k, \varphi\rangle \, dx,
\end{align*}
thus $\mathcal{L}_{\gamma}$ is self-adjoint. This, together with the identity  \eqref{eq:Lvarvarshow} as well as the symmetry 
\begin{equation*}
   \langle R(\gamma_x, 
\varphi)\gamma_x, \psi\rangle= \langle R(\gamma_x, 
\psi)\gamma_x, \varphi\rangle 
\end{equation*}
yields 
  \begin{equation*}
 \langle \mathcal{L}_{\gamma}\varphi, \psi\rangle_{L^2(\mathbb{T}^1)} = - 
\langle \nabla_{\gamma_x}\varphi, 
\nabla_{\gamma_x}\psi\rangle_{L^2(\mathbb{T}^1)}   + \int_{\mathbb{T}^1}\langle R(\gamma_x, 
\varphi)\gamma_x, \psi\rangle\,dx= \langle \varphi, \mathcal{L}_{\gamma}\psi\rangle_{L^2(\mathbb{T}^1)}.
 \end{equation*}

 Finally, we show that $\mathcal{L}_{\gamma} \gamma_x= 0$. Using the above identity for $\varphi= \gamma_x$ we obtain 
 \begin{equation*}
    \langle \mathcal{L}_{\gamma}\gamma_x, \psi\rangle_{L^2(\mathbb{T}^1)}= 0. 
 \end{equation*}
 Recall that $\gamma$ satisfies the geodesic equation \eqref{eq:geodesic:equation}. Then, we have 
\begin{equation*}
\Delta \gamma(x+ s) + S_{jk}(\gamma(x+ s))\partial_x\gamma^j(x+s) \partial_x\gamma^k(x+s) = 0\;\; \forall x\in \mathbb{T}^1 \; \forall s\in (- \varepsilon, \varepsilon).
\end{equation*}
Differentiating the preceding equation with respect to the variable $s$ at $s= 0$, we obtain 
\begin{equation*}
0=   \Delta \gamma_x + \gamma_x^r\partial_rS_{jk}(\gamma)\partial_x\gamma^j\partial_x\gamma^k + 2S_{jk}(\gamma)\partial_x\gamma^j\partial_x\gamma_x^k= \mathcal{L}_{\gamma} \gamma_x\;\; \forall x\in \mathbb{T}^1.
\end{equation*}
This finishes the proof of Lemma \ref{lem:selfadjoint}.

 \vspace{2mm}

\noindent {\bf A.3.}\label{setting(Aex3)}  {\bf  Proof of Lemma \ref{lem:movingframe}. }
To simplify notations, we identify $\mathbb{T}^1$ with the interval $[0,1]$ with endpoints identified and denote $\gamma_x$ by $\dot \gamma$. The proof is composed of two steps.

\textit{Step 1. Constructing a frame via parallel transport.}  
Choose a base point $x_0= 0 \in [0,1]$. Select an orthonormal basis $
\{{\bf h}_1, {\bf h}_2, \dots, {\bf h}_R\}$
of $T_{\gamma(0)}\mathcal{N}$ with the property that $
e_1 = \dot{\gamma}(0).$
For each $x\in [0,1]$, we denote the  parallel transport along $\gamma$ (with respect to the Levi--Civita connection) as 
\[
P_{0\to x}: T_{\gamma(0)}\mathcal{N} \to T_{\gamma(x)}\mathcal{N}
\]
Define, for $p=1,... ,R$,
\[
{\bf f}_p(x) = P_{0\to x}({\bf h}_p).
\]
Because parallel transport is an isometry, for every $x$ the set $\{{\bf f}_1(x),... ,{\bf f}_R(x)\}$ forms an orthonormal basis of $T_{\gamma(x)}\mathcal{N}$. Moreover, since $\gamma$ is a geodesic, its velocity is parallel along $\gamma$,
$\nabla_{\dot{\gamma}} \dot{\gamma} = 0.$
It follows that
\[
{\bf f}_1(x) = P_{0\to x}({\bf h}_1) = P_{0\to x}(\dot{\gamma}(0)) = \dot{\gamma}(x)\;\; \forall x \in [0,1].
\]

\medskip

\textit{Step 2. Adjusting the frame via holonomy correction to achieve periodicity.}  
The construction above yields a smooth frame $\{{\bf f}_i(x)\}$ on $[0,1]$. However, in general the holonomy
\[
P_{0\to 1}: T_{\gamma(0)}\mathcal{N} \to T_{\gamma(1)}\mathcal{N} = T_{\gamma(0)}\mathcal{N}
\]
need not equal the identity. 
Since $P_{0\to 1}$ is an orthogonal transformation (an element of $O(R)$), there exists a smooth path
\[
R: [0,1] \to O(R)
\]
such that
\[
R(0) = I \quad \text{and} \quad R(1) = P_{0\to 1}^{-1}.
\]
Moreover, because we want to preserve the fact that $f_1(x) = \dot{\gamma}(x)$, we may choose $R(x)$ such that
\[
R(x){\bf h}_1 = {\bf h}_1 \quad \forall x\in [0,1].
\]
Define the modified frame
\[
\tilde {\bf f}_p(x) = R(x) {\bf f}_p(x) \quad \forall x\in [0,1] \text{ and } p=1,\dots,R.
\]
Then $\tilde {\bf f}_i(x)$ is smooth on $[0,1]$. At $x=0$,
\[
\tilde {\bf f}_p(0) = R(0){\bf f}_p(0) = I\cdot {\bf h}_p = {\bf h}_p,
\]
and at $x=1$,
\[
\tilde {\bf f}_p(1) = R(1) f_p(1) = P_{0\to 1}^{-1}\bigl(P_{0\to 1}({\bf h}_p)\bigr) = {\bf h}_p.
\]
Thus, $\tilde {\bf f}_p(0) = \tilde {\bf f}_p(1)$ for each $p$, and the modified frame is a smooth, periodic frame on $\mathbb{T}^1$.

Moreover, since $R(x){\bf h}_1 = {\bf h}_1$ for all $x$, we have
\[
\tilde {\bf f}_1(x) = R(x){\bf f}_1(x) = R(x)P_{0\to x}({\bf h}_1) = P_{0\to x}({\bf h}_1) = \dot{\gamma}(x)
\]
for all $x\in \mathbb{T}^1$. To simplify the notations we rename $\tilde {\bf f}_p$ as ${\bf f}_p$.
This finishes the proof of this construction.

\section{More details on the propagation of smallness}\label{sec:app:B}

In this section we provide more details on the proofs of Propositions \ref{lem:2}, \ref{prop:psl2}.

\noindent {\bf B.1.}\label{setting(B1)}  {\bf   More details on the proof of Proposition \ref{lem:2}.}
This part is devoted to the proof of Lemma \ref{lem:pro:wm}.
First, we  define the $i^{th}$-component of 
$\Pi_T(\phi)f$ as $g^i \; \forall i=1,2,...,N$.
Recall that under the null coordinate equation \eqref{eq:wmdamped} can be expressed as
\begin{equation}
    -\phi_{u v}^i= S^i_{jk}(\phi) \phi^j_u \phi^k_v+ \frac{1}{4}a \phi^i_t+ g^i:= F^i\;\; \forall i= 1,..., N. \notag
\end{equation}
Therefore $\phi^i$ can be characterized as follows using direct integration
\begin{align*}
  \phi^i(u, v)
  &=   2\int_0^t \int_{x-t+s}^{x+t-s} F^i(s, y) dy ds+ \int_{-v}^u \phi^i_u(u_0, -u_0) du_0+ \phi^i(-v, v),
\end{align*}
 where, $(u_0, v_0)$ is related  to the $(u, v)$ coordinate, while $(s, y)$ is to be understand in the original $(t,  x)$ coordinate.  

\noindent {\it On the characterization of $\phi_t$}.

By taking the time derivative the function $\phi_t$ is governed by 
\begin{equation}
    -\phi_{t, uv}^i=  \left(\nabla S^i_{jk}(\phi) \cdot \phi_t \right)  \phi^j_u \phi^k_v+ S^i_{jk}(\phi) \left(\phi^j_u \phi^k_{t, v}+ \phi^j_v \phi^k_{t, u} \right)+ \frac{1}{4}a \phi^i_{tt}+ g^i_t := G^i, \notag
\end{equation}
for every $ i= 1,..., N$. Taking advantage of the preceding formula we can express the value of $\phi_t$, basically due to the finite speed of propagation. This is linked to the proof of Proposition 2.2 in \cite{KX}, more precisely, to the ``propagation of the smallness'' part as Section 2.4.2 therein. To characterize of propagation of the wave, recall the setting in \cite[equations (2.30)--(2.31)]{KX}, we work with the vertical trapezoidal regions  
for $\alpha+ 2\pi<\beta$ and $l\leq \pi$
\begin{gather}
    P_{\alpha, \beta}^l(y):= \{(t, x): x\in [y, y+l], t\in [\alpha+ x-y, \beta-x+y]\},  \notag
\end{gather}
and  the $L^{\infty}_x L^2_t$-norm on it as
\begin{gather}
    \|\psi\|_{L^{\infty}_x L^2_t(P_{\alpha, \beta}^l(y))}:= \sup_{x\in [y, y+l]}   \|\psi(t, x)\|_{L^2_t(\alpha+ x-y, \beta-x+y)}.\notag
\end{gather}
Thus, recalling the calculation in \cite[Section 2.4.2]{KX}, for every point $(u, v)\in  P_{\alpha, \beta}^l(x_0)$ 
\begin{align*}
    -\phi^i_{t}(u, v)
    &= -\phi^i_{t}(2x_0+ v, v)- \int_{2x_0+v}^u \phi^i_{t, u}(u_0, u_0- 2x_0) d u_0+ \int_{2x_0+v}^u \int_{u_0- 2x_0}^v G^i(u_0, v_0) dv_0 d u_0, \\
    &=  -\frac{1}{2}\Big(\phi^i_t(u- x_0, x_0)+ \phi^i_t(v+x_0, x_0)+ \phi^i_x(u- x_0, x_0)- \phi^i_x(v+ x_0, x_0)\Big)\\
     &\;\;\;\;\;\;\;\;\;\;\;\;+ \int_{2x_0+v}^u \int_{u_0- 2x_0}^v \left( \left(\nabla S^i_{jk}(\phi) \cdot \phi_t \right)  \phi^j_u \phi^k_v+ \frac{1}{4}a \phi^i_{tt}+ g^i_t\right) (u_0, v_0) \,dv_0 d u_0\\
      &\;\;\;\;\;\;\;\;\;\;\;\;\;\;\;\;\;\;\;\;\;\;\;\;+ \int_{2x_0+v}^u \int_{u_0- 2x_0}^v  S^i_{jk}(\phi) \left(\phi^j_u \phi^k_{t, v}+ \phi^j_v \phi^k_{t, u} \right)(u_0, v_0) \,dv_0 d u_0.
\end{align*}

Thanks to the integration by parts
\begin{align*}
   & \int_{2x_0+v}^u \int_{u_0- 2x_0}^v \left( S^i_{jk}(\phi) \phi^j_u \phi^k_{t, v}\right)(u_0, v_0) \,dv_0 d u_0 \\
   & = -\int_{2x_0+v}^u \int_{u_0- 2x_0}^v \left(\nabla S^i_{jk}(\phi)\cdot \phi_v \right) \phi^j_u \phi^k_t(u_0, v_0) \,dv_0 d u_0 \\
    &\;\;\;\;\;\;\;\;\;\;\;\;\;\;\;\; -\int_{2x_0+v}^u \int_{u_0- 2x_0}^v  S^i_{jk}(\phi) \phi^j_{uv} \phi^k_t  (u_0, v_0) \,dv_0 d u_0 \\
   &\;\;\;\;\;\;\;\;\;\;\;\;\;\;\;\;\;\;\;\;\;\;\;+ \int_{2x_0+v}^u S^i_{jk}(\phi) \phi^j_u \phi^k_{t}(u_0, v)- S^i_{jk}(\phi) \phi^j_u \phi^k_{t}(u_0, u_0- 2x_0)  \,  du_0 \\
    &= -\int_{2x_0+v}^u \int_{u_0- 2x_0}^v \left(\nabla S^i_{jk}(\phi)\cdot \phi_v \right) \phi^j_u \phi^k_t(u_0, v_0) \,dv_0 d u_0 \\
    &\;\;\;\;\;\;\;\;\;\;\;\;\;\;\;\;\;\; +\int_{2x_0+v}^u \int_{u_0- 2x_0}^v  S^i_{jk}(\phi) \phi^k_t \left(S^j_{lm}(\phi)\phi^l_u \phi^m_v+ \frac{1}{4}a \phi^j_t+ g^j \right)  (u_0, v_0) \, dv_0 d u_0 \\
   &\;\;\;\;\;\;\;\;\;\;\;\;\;\;\;\;\;\;\;\;\;\;\;\;\;\;\;\;\;\;+ \int_{2x_0+v}^u S^i_{jk}(\phi) \phi^j_u \phi^k_{t}(u_0, v)- S^i_{jk}(\phi) \phi^j_u \phi^k_{t}(u_0, u_0- 2x_0)  \,  du_0 \\
\end{align*}
and 
\begin{align*}
   &\;\;\;\; \int_{2x_0+v}^u \int_{u_0- 2x_0}^v \left( S^i_{jk}(\phi) \phi^j_v \phi^k_{t, u}\right)(u_0, v_0) \, dv_0 d u_0 \\
   & =  -\int_{v}^{u- 2x_0} \int_{v_0+ 2x_0}^u \left( S^i_{jk}(\phi) \phi^j_v \phi^k_{t, u}\right)(u_0, v_0) \, du_0 d v_0 \\
   &= \int_{v}^{u- 2x_0} \int_{v_0+ 2x_0}^u  \left(\nabla S^i_{jk}(\phi)\cdot \phi_u \right) \phi^j_v \phi^k_t(u_0, v_0) \, du_0 d v_0 \\
    &\;\;\;\;\;\;\;\;\;\;\;\;\;\;\;\;\;\; +\int_{v}^{u- 2x_0} \int_{v_0+ 2x_0}^u  S^i_{jk}(\phi) \phi^j_{uv} \phi^k_t  (u_0, v_0) \, du_0 d v_0 \\
   &\;\;\;\;\;\;\;\;\;\;\;\;\;\;\;\;\;\;\;\;\;\;\;\;\;\;\;\;\;\;+ \int_v^{u- 2 x_0}  S^i_{jk}(\phi) \phi^j_v \phi^k_{t}(v_0+ 2x_0, v_0)- S^i_{jk}(\phi) \phi^j_u \phi^k_{t}(u, v_0)   \, dv_0 \\
   &= \int_{v}^{u- 2x_0} \int_{v_0+ 2x_0}^u  \left(\nabla 
  S^i_{jk}(\phi)\cdot \phi_u \right) \phi^j_v \phi^k_t(u_0, v_0) \, du_0 d v_0 \\
    &\;\;\;\;\;\;\;\;\;\;\;\;\;\;\;\;\;\; -\int_{v}^{u- 2x_0} \int_{v_0+ 2x_0}^u  S^i_{jk}(\phi) \phi^k_t \left(S^j_{lm}(\phi)\phi^l_u \phi^m_v+ \frac{1}{4}a \phi^j_t+ g^j \right)  (u_0, v_0)  \, du_0 d v_0 \\
   &\;\;\;\;\;\;\;\;\;\;\;\;\;\;\;\;\;\;\;\;\;\;\;\;\;\;\;\;\;\;+ \int_v^{u- 2 x_0}  S^i_{jk}(\phi) \phi^j_v \phi^k_{t}(v_0+ 2x_0, v_0)- S^i_{jk}(\phi) \phi^j_u \phi^k_{t}(u, v_0)  \,  dv_0 \\
     &=  -\int_{2x_0+v}^u \int_{u_0- 2x_0}^v  \left(\nabla 
  S^i_{jk}(\phi)\cdot \phi_u \right) \phi^j_v \phi^k_t(u_0, v_0) \, dv_0 d u_0 \\
    &\;\;\;\;\;\;\;\;\;\;\;\;\;\;\;\;\;\;  +\int_{2x_0+v}^u \int_{u_0- 2x_0}^v S^i_{jk}(\phi) \phi^k_t \left(S^j_{lm}(\phi)\phi^l_u \phi^m_v+ \frac{1}{4}a \phi^j_t+ g^j \right)  (u_0, v_0)\,  dv_0 d u_0 \\
   &\;\;\;\;\;\;\;\;\;\;\;\;\;\;\;\;\;\;\;\;\;\;\;\;\;\;\;\;\;\;+ \int_v^{u- 2 x_0}  S^i_{jk}(\phi) \phi^j_v \phi^k_{t}(v_0+ 2x_0, v_0)- S^i_{jk}(\phi) \phi^j_u \phi^k_{t}(u, v_0)  \,  dv_0 
\end{align*}
These equations further deduce that for any $(t, x)\in P_{\alpha, \beta}^l(x_0)$ with $x= x_0+d$,
\begin{align*}
    -\phi_t^i (t, x)&=
    \underbrace{ -2\int_{x_0}^x \int_{t-x+y}^{t+x-y} W^i(s, y)  \, ds dy }_{=: H^i_1(t, x)}\\
     &\;\; \;\;\;  \underbrace{ -\frac{1}{2}\Big(\phi^i_t(u- x_0, x_0)+ \phi^i_t(v+x_0, x_0)\Big) }_{=: H^i_2(t, x)} \;\; 
     \underbrace{ -\frac{1}{2} \int_{t- d}^{t+ d} \phi^i_{tx}(s, x_0) \, ds }_{=: H^i_3(t, x)}\\
    &\;\; \;\;\;\underbrace{ -\frac{1}{2} \int_{x_0}^x a(y)\phi^i_t(t+x-y, y)+ g^i(y) \, dy +\frac{1}{2} \int_{x_0}^x a(y)\phi^i_t(t-x+y, y)+ g^i(y) \, dy }_{=: H^i_4(t, x)}\\
    &\;\; \;\;\; \underbrace{  + \int_{2x_0+v}^u S^i_{jk}(\phi) \phi^j_u \phi^k_{t}(u_0, v) \, d u_0 -  \int_{2x_0+v}^u  S^i_{jk}(\phi) \phi^j_u \phi^k_{t}(u_0, u_0- 2x_0)    \, du_0 }_{=: H^i_5(t, x)}\\
    & \;\; \;\;\; \underbrace{  +  \int_v^{u- 2 x_0}  S^i_{jk}(\phi) \phi^j_v \phi^k_{t}(v_0+ 2x_0, v_0)  \, d v_0- \int_v^{u- 2 x_0} S^i_{jk}(\phi) \phi^j_u \phi^k_{t}(u, v_0)  \,  dv_0 }_{=: H^i_6(t, x)}\\
    &=: \sum_{j=1}^6 H_j^i(t, x),
\end{align*}
where $W^i$ is given by 
\begin{align*}
W^i&:= \left(\nabla S^i_{jk}(\phi) \cdot \phi_t \right)  \phi^j_u \phi^k_v- \left(\nabla S^i_{jk}(\phi)\cdot \phi_v \right) \phi^j_u \phi^k_t-  \left(\nabla 
  S^i_{jk}(\phi)\cdot \phi_u \right) \phi^j_v \phi^k_t   \\
  & \;\;\;\;\;\;\;\;\;\;\;\;\;\;\;\;\;\;\;\;\;\;\;\;
  + S^i_{jk}(\phi) \phi^k_t \left( 2S^j_{lm}(\phi)\phi^l_u \phi^m_v+ \frac{1}{2}a \phi^j_t+ 2g^j \right).
\end{align*}

\noindent {\it On the estimates of $\|\phi_t\|_{L^{\infty}_x L^2_t(\alpha+d, \beta-d)}$}.

To be compared with \cite[Step 1 of the proof of Lemma 2.5]{KX}, we notice that the estimates on $H_2^i, H_3^i, H_5^i, H_6^i$ remain the same: for $d\in (0, S)$, there are
\begin{gather*}
 \|H_2^i(t, x)\|_{L^2_t(\alpha+d, \beta-d)}\lesssim \|\phi_t(t, x_0)\|_{L^2_t(\alpha, \beta)},\\
   \|H_3^i(t, x)\|_{L^2_t(\alpha+d, \beta-d)}\lesssim\|\langle \partial_t\rangle^{-1} \left(\eta_{\alpha}^{\beta}[\tau](t) \phi_{tx}(t, x_0)\right)
    \|_{L^2_{t}(\R)},\\
  \|H_5^i(t, x)\|_{L^2_t(\alpha+d, \beta-d)}+ \|H_6(t, x)\|_{L^2_t(\alpha+d, \beta-d)}
\lesssim \sqrt{d} \sqrt{E(0)} \|\phi_t\|_{L^{\infty}_x L^2_t(P)}.
\end{gather*}

As for $H_2^i$ and $H_4^i$, the contribution of the extra term $g^j$ is bounded by its $L^2_{t, x}$-norm which is further controlled by $\delta_2 E(\phi[0])$, otherwise the other estimates remain the same. More precisely, 
\begin{gather*}
   \;\;\;\;  \int_{\alpha+d}^{\beta-d}\left(\int_{x_0}^x a(y)\phi_t^i(t+x-y, y)+ g^i(y) dy\right)^2 \, dt  \;\;\;\;\;\; \;\;\;\;\;\;\;\;\;\;\;\;\;\;\;\;\;\;\\
   \;\;\;\;\;\;\;\;\;\;\;\; \lesssim  d \int_{\alpha+d}^{\beta-d}\int_{x_0}^x \left((\phi^i_t)^2(t+x-y, y) + \left(g^i(y)\right)^2 \right)\, dy dt\\
 \;\;\;\;\;\;\;\;\;\;\;\;\;\;\;\;\;\;\;\;\;\;\;\;   \lesssim d^2 \|\phi_t\|_{L^{\infty}_x L^2_t(P)}^2+ d \|g^i\|_{L^{2}_{t, x}(P)}^2
\end{gather*}
thus
\begin{equation}\label{es:f4}
     \|H_4(t, x)\|_{L^2_t(\alpha+d, \beta-d)}\lesssim d \|\phi_t\|_{L^{\infty}_x L^2_t(P)} +  \left(d \delta_2 E(\phi[0])\right)^{1/2}.  \notag
\end{equation}
We also know that the contribution of $g$ in $H^i_1$ is bounded by 
\begin{align*}
\int_{x_0}^x \int_{t-x+y}^{t+x-y} S^i_{jk}(\phi) \phi^k_t g^j (s, y) ds dy
\lesssim \|\phi_t^k\|_{L^2_{t, x}} \|g\|_{L^2_{t, x}} 
\lesssim  (\delta_2)^{1/2} E(\phi[0])
\end{align*}
thus 
\begin{align*}
   |H_1|(t, x)
    &\lesssim \left(\sqrt{d}E(\phi[0])+ d\sqrt{E(\phi[0])}\right)\|\phi_t\|_{L^{\infty}_x L^2_t(P)}+ (\delta_2)^{1/2} E(\phi[0]).
\end{align*}

In conclusion, for $x= x_0+ d$ and for any $\tau\in (0, 1)$, we have
\begin{align*}
   \|\phi_t(t, x)\|_{L^2_t(\alpha+d, \beta-d)}
    &\lesssim  \|\phi_t(t, x_0)\|_{L^2_t(\alpha, \beta)}+ \left\|\langle \partial_t\rangle^{-1} \left(\eta_{\alpha}^{\beta}[\tau](t) \phi_{tx}(t, x_0)\right)
    \right\|_{L^2_{t}(\R)} \\
    & \;\;\;\;\;\;\;\;\;\;\;\;\;\;\;\;\;\;\;\; +  \sqrt{d} \|\phi_t\|_{L^{\infty}_x L^2_t(P)}+ (\delta_2 E(\phi[0]))^{1/2} \notag
\end{align*}
Hence, by choosing $d$ small enough, we obtain
\begin{align}
   &\;\;\;\; \|\phi_t(t, x)\|_{L^{\infty}_x L^2_t(P^{d}_{\alpha, \beta}(x_0))} \notag \\
    &\lesssim  \|\phi_t(t, x_0)\|_{L^2_t(\alpha, \beta)}+ \left\|\langle \partial_t\rangle^{-1} \left(\eta_{\alpha}^{\beta}[\tau](t) \phi_{tx}(t, x_0)\right)
    \right\|_{L^2_{t}(\R)}  +  (\delta_2 E(\phi[0]))^{1/2}. \notag
\end{align}

This yields the first estimate  in Lemma \ref{lem:pro:wm}.

\vspace{2mm}

\noindent {\it On the choice of $z$ such that $\phi_{tx}(t, z)$ is small.} 

This part is the same as  \cite[Step 2 of the proof of Lemma 2.5]{KX}. We can find  some point $\bar z\in [x_0+ \frac{5 S_0}{8}, x_0+ \frac{7 S_0}{8}]$ such that 
\begin{gather*}
\;\;\;\;\;  \int_{\mathbb{R}} \left(\langle \partial_t\rangle^{-1} \left(\eta_{\alpha+ S_0+ \tau_0}^{\beta-S_0-\tau_0}[\tau_0](t) \phi_{tx}(t, \bar z)\right)\right)^2 \, dt\;\;\;\;\;\;\;\;\;\;\;\;\;\;\;\;\;\;\;\;\;\;\;\;\;\;\;\;\;\;\;\;\; \\
\;\;\;\;\;\;\;\;\;\;\;\;\;\;\;\;\;\;\;\;\;\; \;\;\;\;\;\;\;\;\;\;\;\lesssim  \sqrt{E(\phi[0])}\|\phi_t(t, x)\|_{L^{\infty}_x(x_0+S_0/2, x_0+ S_0; L^2_t(\alpha+ S_0, \beta- S_0))}.
\end{gather*}

Therefore, we finish the proof of Lemma \ref{lem:pro:wm}.

\vspace{3mm}

\noindent {\bf B.2.}\label{setting(B2)}  {\bf  More details on the proof of Proposition \ref{prop:psl2}.}
Similar to \cite[Proposition 2.2]{KX}, it is a direct consequence of the following two auxilary properties: Lemma \ref{lemma:choose} and Lemma \ref{lem:keypropa}. Since these two lemmas are analog to \cite[Lemma 2.4, Lemma 2.5]{KX} concerning $\phi_t$, we omit the detailed proofs.
Actually, the analysis here is even simpler:  in \cite{KX} the authors have studied a nonlinear equation on $\phi_t$: 
\begin{equation}
    \Box \phi_t = (|\phi_t|^2-|\phi_x|^2)\phi_t+ 2(\phi_t\cdot \phi_{tt}- \phi_x\cdot \phi_{tx})\phi+ a \phi_{tt}, \notag
\end{equation}
while, in the current setting, we deal with a  linear equation \eqref{eq:prop:psl2:main} concerning the function $\varphi$:
\begin{equation*}
     \Box \varphi+ {\bf A} \partial_x\varphi+ {\bf C} \partial_t\varphi+ {\bf B}\varphi= 0.
\end{equation*}

Recall the truncated function $\mu$, the cutoff function $\eta_{\alpha}^{\beta}[\tau]$, 
 the domain $P_{\alpha, \beta}^l(y)$, and the norm $\|\cdot\|_{L^{\infty}_x L^2_t(P_{\alpha, \beta}^l(y))}$ defined in \eqref{eq:def:truncfunc}--\eqref{eq:def:norm:Pab}.

\begin{lemma}\label{lemma:choose}
There exists some effectively computable $C_0>0$ such that, for any  $\varepsilon\in (0, 1)$, for any $\tau\in (0, 1)$,
if there is some solution  $\varphi$  of \eqref{eq:prop:psl2:main}
\begin{gather}
\int_{-16\pi}^{16\pi}\int_{\mathbb{T}^1}\chi_{\omega}|\varphi|^2(t, x) \, dx dt\leq   \varepsilon \|\varphi[0]\|_{L^2\times H^{-1}}^2, \label{es:mainasum}  \notag
\end{gather}
then,  there exists some $x_0\in [0, 2\pi)$ such that
\begin{equation}\label{es:lem_31}
     \|\varphi(t, x_0)\|_{L^2_t(-16\pi, 16\pi)}^2+   \|\langle\partial_t\rangle^{-1}\left(\eta_{-15\pi}^{15\pi}[\tau] \varphi_{x}\right)(t, x_0)\|_{L^2_t(\R)}^2 \leq C_0 \frac{\sqrt{\varepsilon}}{\tau^2} \|\varphi[0]\|_{L^2\times H^{-1}}^2.  \notag
\end{equation}
\end{lemma}
This lemma is similar to \cite[Lemma 2.4]{KX}, where the authors  obtained estimates on $\|\phi_t(t, x_0)\|_{L^2_t(-16\pi, 16\pi)}^2$ and $   \|\langle\partial_t\rangle^{-1}\left(\eta_{-15\pi}^{15\pi}[\tau] \phi_{tx}\right)(t, x_0)\|_{L^2_t(\R)}^2$.

\begin{lemma}\label{lem:keypropa}
There exist some effectively computable values $S_0$ 
and $C_{S_0}$ such that, for any $\tau\in (0, 1)$,  for any $\alpha\in [-15\pi, 0), \beta\in (2\pi, 15\pi]$, for any $z\in [0, 4\pi]$, 
one has
\begin{equation}\label{eq:bound:p181}
    \|\varphi\|_{L^{\infty}_x L^2_t(P_{\alpha, \beta}^{S_0}(z))}\leq 2  \|\varphi(t, z)\|_{L^2_t(\alpha, \beta)}+ 6\left\|\langle\partial_t \rangle^{-1}\left(\eta_{\alpha}^{\beta}[\tau](t) \varphi_{x}(t, z)\right)\right\|_{L^2_t(\R)},   \notag
\end{equation}
and, moreover, by denoting $\tau_0:= S_0/16$, there exists some $\bar x\in (z+ S_0/2, z+ S_0)$ such that 
\begin{align}\label{eq:bound:p182}
   &\;\;\;\;\; \left\| \langle\partial_t\rangle^{-1} \left(\eta_{\alpha+ S_0+\tau_0}^{\beta-S_0-\tau_0}[\tau_0](t)\varphi_{x}(t, \bar x)\right)\right\|_{L^2_t(\R)} \notag \\
   &\leq C_{S_0} \left(\|\varphi[0]\|_{L^2\times H^{-1}}\right)^{\frac{1}{2}}\left(  \|\varphi(t, z)\|_{L^2_t(\alpha, \beta)}+ \left\|\langle\partial_t \rangle^{-1}\left(\eta_{\alpha}^{\beta}[\tau_0](t) \varphi_{x}(t, z)\right)\right\|_{L^2_t(\R)} \right)^{\frac{1}{2}}.   \notag
\end{align}
\end{lemma}
This lemma is similar to \cite[Lemma 2.5]{KX}, where the authors  obtained estimates on $ \|\phi_t\|_{L^{\infty}_x L^2_t(P_{\alpha, \beta}^{S_0}(z))}$ and  $\left\| \langle\partial_t\rangle^{-1} \left(\eta_{\alpha+ S_0+\tau_0}^{\beta-S_0-\tau_0}[\tau_0](t)\phi_{tx}(t, \bar x)\right)\right\|_{L^2_t(\R)}$.

By combining the above two properties, we obtain Proposition \ref{prop:psl2}.

\section{More details on the well-posedness}\label{sec:app:C}

This section is devoted to the proofs of some lemmas related to well-posedness issues, including  Lemma   \ref{lem:semiwave:1}, and 
Lemma \ref{lem:wave:timegain}.

\vspace{2mm}

\noindent {\bf C.1.}\label{setting(C1)}  {\bf
 The proof of Lemma \ref{lem:semiwave:1}.} 
First we prove the basic properties of the nonlinear terms $\mathcal{K}(x; {\bf w})$ and $\mathcal{K}_1(x; {\bf w}) \alpha$, namely inequalities \eqref{K:p:1}--\eqref{K:p:2},
\begin{align*}
 \| \sum_{j,k}
A_{jk}^{p}(x; {\bf{w}})\partial_{\beta} w^j \partial^{\beta} w^k \|_{L^2_{x,t}(D_T)}&\lesssim \sum_{j,k} \|\sum_{\beta} 
A_{jk}^{i}(x; {\bf{w}})\partial_{\beta} w^j \partial^{\beta} w^k \|_{L^2_{u, v}(D_T)}  \\
&\lesssim \sum_{j,k} \|
A_{jk}^{p}(x; {\bf{w}}) (w^j_u  w^k_v+ w^j_v w^k_u) \|_{L^2_{u, v}(D_T)}  \\
&\lesssim \sum_{j,k} \|
A_{jk}^{p}(x; {\bf{w}})\|_{L^{\infty}(D_T)} \|w^j_u\|_{L^{\infty}_v L^2_u(D_T)}  \|w^k_v \|_{L^2_v L^{\infty}_u(D_T)} \\
& \;\;\;\;\;\; +  \sum_{j,k} \|
A_{jk}^{p}(x; {\bf{w}})\|_{L^{\infty}(D_T)} \|w^j_v\|_{L^{\infty}_u L^2_v(D_T)}  \|w^k_u \|_{L^2_u L^{\infty}_v(D_T)} \\
&\lesssim \|{\bf w}\|_{\mathcal{W}_T}^2,
\end{align*}
where we have used the identity \eqref{null-indentity},   and  then successively obtain
\begin{align*}
\|  \sum_j B_{r, j}^{p}(x; {\bf{w}}) \partial_{x}w^j\|_{L^2_{x,t}(D_T)}&\lesssim \sum_j \|{\bf w}\|_{L^{\infty}(D_T)}  \|  \partial_{x}w^j\|_{L^2_{x,t}(D_T)}\lesssim \|{\bf w}\|_{\mathcal{W}_T}^2,  \\
\| \sum_j C_{r, j}^{p}(x; {\bf w}) w^j \|_{L^2_{x,t}(D_T)} &\lesssim \sum_j \|{\bf w}\|_{L^{\infty}(D_T)}  \| w^j\|_{L^2_{x,t}(D_T)}\lesssim \|{\bf w}\|_{\mathcal{W}_T}^2, \\
\|\sum_j D_j^p(x; {\bf{w}})\alpha^j(t, x)\|_{L^2_{x,t}(D_T)} &\lesssim \sum_j \|{\bf w}\|_{L^{\infty}(D_T)}  \| \alpha^j\|_{L^2_{x,t}(D_T)}\lesssim \|{\bf w}\|_{L^{\infty}(D_T)}  \| \alpha\|_{L^2_{x,t}(D_T)}
\end{align*}
for every $p, j, k= 1,2,..., R$. Similarly,  
\begin{align*}
 & \;\;\;\;\; \| \sum_{\beta} A_{jk}^{p}(x; {\bf{w}}_1)\partial_{\beta} w^j_1 \partial^{\beta} w^k_1- \sum_{\beta}  A_{jk}^{p}(x; {\bf{w}}_2)\partial_{\beta} w^j_2 \partial^{\beta} w^k_2 \|_{L^2_{x,t}(D_T)}\\
 &\lesssim \| A_{jk}^{p}(x; {\bf{w}}_1) w^j_{1u}  w^k_{1v}+ A_{jk}^{p}(x; {\bf{w}}_1) w^j_{1v}  w^k_{1u} -A_{jk}^{p}(x; {\bf{w}}_2) w^j_{2u}  w^k_{2v}- A_{jk}^{p}(x; {\bf{w}}_2) w^j_{2v}  w^k_{2u} \|_{L^2_{u, v}(D_T)}\\
 &\lesssim \|A_{jk}^{p}(x; {\bf{w}}_1) w^j_{1u}  w^k_{1v}- A_{jk}^{p}(x; {\bf{w}}_2) w^j_{2u}  w^k_{2v} \|_{L^2_{u, v}(D_T)} \\
 &\;\;\;\;\;\;  \;\;\;\;\;\;  \;\;\;\;\;\; + \|  A_{jk}^{p}(x; {\bf{w}}_1) w^j_{1v}  w^k_{1u} - A_{jk}^{p}(x; {\bf{w}}_2) w^j_{2v}  w^k_{2u} \|_{L^2_{u, v}(D_T)} \\
 & \lesssim \|{\bf w}_1- {\bf w}_2\|_{L^{\infty}(D_T)} \|{\bf w}_1\|_{\mathcal{W}_T}^2+ \|{\bf w}_1- {\bf w}_2\|_{\mathcal{W}_T} (\|{\bf w}_1\|_{\mathcal{W}_T}+ \|{\bf w}_2\|_{\mathcal{W}_T})\\
 & \lesssim \|{\bf w}_1- {\bf w}_2\|_{\mathcal{W}_T} (\|{\bf w}_1\|_{\mathcal{W}_T}+ \|{\bf w}_2\|_{\mathcal{W}_T}).
\end{align*}
as well as
\begin{align*}
&\;\;\;\;\; \|  B_{r, j}^{p}(x; {\bf{w}}_1) \partial_{x}w^j_1-  B_{r, j}^{p}(x; {\bf{w}}_2) \partial_{x}w^j_2\|_{L^2_{x,t}(D_T)} \\
&\lesssim \|  B_{r, j}^{p}(x; {\bf{w}}_1) \partial_{x}w^j_1- B_{r, j}^{p}(x; {\bf{w}}_1) \partial_{x}w^j_2 \|_{L^2_{x,t}(D_T)}\\
&\;\;\;\;\;\;\;\;\;\; \;\;\;\;\;\;\;\;\;\; \;\;\;\;\;\;\;\;\;\; \;\;\;\;\;\;\;\;\;\;+ \|   B_{r, j}^{p}(x; {\bf{w}}_1) \partial_{x}w^j_2-  B_{r, j}^{p}(x; {\bf{w}}_2) \partial_{x}w^j_2\|_{L^2_{x,t}(D_T)} \\
&\lesssim \|{\bf w}_1- {\bf w}_2\|_{\mathcal{W}_T} (\|{\bf w}_1\|_{\mathcal{W}_T}+ \|{\bf w}_2\|_{\mathcal{W}_T}),\\
&\;\;\;\;\; \|  C_{r, j}^{p}(x; {\bf{w}}_1) w^j_1-  C_{r, j}^{p}(x; {\bf{w}}_2) w^j_2\|_{L^2_{x,t}(D_T)} \\
&\lesssim  \|{\bf w}_1- {\bf w}_2\|_{\mathcal{W}_T} (\|{\bf w}_1\|_{\mathcal{W}_T}+ \|{\bf w}_2\|_{\mathcal{W}_T}),
\end{align*}
and finally,
\begin{align*}
 &\;\;\;\;\; \|\mathcal{K}_1(x; {\bf w}_1)\alpha_1- \mathcal{K}_1(x; {\bf w}_2)\alpha_2\|_{L^2_{t, x}(D_T)} \\
&\leq \|\mathcal{K}_1(x; {\bf w}_1)\alpha_1- \mathcal{K}_1(x; {\bf w}_2)\alpha_1\|_{L^2_{t, x}(D_T)}+ \|\mathcal{K}_1(x; {\bf w}_2)\alpha_1- \mathcal{K}_1(x; {\bf w}_2)\alpha_2\|_{L^2_{t, x}(D_T)} \\
 &\lesssim \|{\bf w}_1- {\bf w}_2\|_{L^{\infty}(D_T)} \|\alpha_1\|_{L^2_{t, x}(D_T)}+ \|{\bf w}_2\|_{L^{\infty}(D_T)} \|\alpha_1- \alpha_2\|_{L^2_{t, x}(D_T)}.
\end{align*}

Now we come back to the proof of Lemma \ref{lem:semiwave:1}.
    Suppose the unique solution exists, then by \eqref{eq:li:w:1}--\eqref{eq:li:w:2},
\begin{equation*}
   \|{\bf w}\|_{\mathcal{W}_T}\leq \mathcal{C}_T \left( \|({\bf w}_0, {\bf w}_{0t})\|_{\mathcal{H}}+ \|\mathcal{K}(x; {\bf w})+   \alpha+ \mathcal{K}_1(x; {\bf w}) \alpha+ e\|_{L^2_{t, x}(D_T)}\right). 
\end{equation*}
Next, we a priori assume that the unique solution satisfies $\|{\bf w}\|_{\mathcal{W}_T}\leq 1$ and obtain 
\begin{align*}
 \|{\bf w}\|_{\mathcal{W}_T}
&\leq \mathcal{C}_T \left(  \|({\bf w}_0, {\bf w}_{0t})\|_{\mathcal{H}}+ \|\alpha+  e\|_{L^2_{t, x}(D_T)}\right)\\
&\;\;\;\;\;\;\;\;\;\;\;\;\;\; + \mathcal{C}_T \left(  \|\mathcal{K}(x; {\bf w})\|_{L^2_{t, x}(D_T)}+ \| \mathcal{K}_1(x; {\bf w}) \alpha\|_{L^2_{t, x}(D_T)}\right)\\
& \leq \mathcal{C}_T \left(  \|({\bf w}_0, {\bf w}_{0t})\|_{\mathcal{H}}+ \|\alpha\|_{L^2_{t, x}(D_T)}+  \|e\|_{L^2_{t, x}(D_T)}\right)\\
&\;\;\;\;\;\;\;\;\;\;\;\;\;\; + C \mathcal{C}_T \left(  \|{\bf w}\|_{\mathcal{W}_T}^2+\|{\bf w}\|_{\mathcal{W}_T} \| \alpha\|_{L^2_{t, x}(D_T)}\right).
\end{align*}

Next, we turn to the second part of this lemma. Indeed, the difference of the two solutions, ${\bf w}:= {\bf w}_1- {\bf w}_2 $, satisfies the following linear equation:
\begin{equation}
    \begin{cases}
        \mathcal{L}_{{\rm in}} {\bf w}= \mathcal{K}(x; {\bf w}_1)- \mathcal{K}(x; {\bf w}_2)+  \chi_{\omega}\big( \alpha_1- \alpha_2+ \mathcal{K}_1(x; {\bf w}_1) \alpha_1- \mathcal{K}_1(x; {\bf w}_2) \alpha_2\big) + e_1- e_2, \\
        {\bf w}[0]= {\bf u}_1[0]- {\bf u}_2[0].
    \end{cases} \notag
\end{equation}
Thanks to the first part of this lemma, we know that $\|{\bf w}_1\|_{\mathcal{W}_T}$ and $\|{\bf w}_2\|_{\mathcal{W}_T}$ are small. Thus,
due to the well-posedness of the linear equation \eqref{eq:li:w:1}--\eqref{eq:li:w:2}, one has
\begin{align*}
     \|{\bf w}\|_{\mathcal{W}_T}&\leq \mathcal{C}_T \left( \|{\bf u}_1[0]- {\bf u}_2[0]\|_{H^1\times L^2(\mathbb{T}^1)}+ \|\alpha_1- \alpha_2+ e_1- e_2\|_{L^2_{t, x}(D_T)}\right) \\
     & \;\;\;\;\; + \mathcal{C}_T \Big( \mathcal{N}(x; {\bf w}_1)- \mathcal{N}(x; {\bf w}_2)+   \mathcal{N}_1(x; {\bf w}_1) \alpha_1- \mathcal{N}_1(x; {\bf w}_2) \alpha_2\Big) \\
     & \leq \mathcal{C}_T \left( \|{\bf u}_1[0]- {\bf u}_2[0]\|_{H^1\times L^2(\mathbb{T}^1)}+ \|\alpha_1- \alpha_2+ e_1- e_2\|_{L^2_{t, x}(D_T)}\right) \\
     & \;\;\;\;\; + C \mathcal{C}_T \|{\bf w}\|_{\mathcal{W}_T} ( \|{\bf w}_1\|_{\mathcal{W}_T}+  \|{\bf w}_2\|_{\mathcal{W}_T})\\
    & \;\;\;\;\; + C \mathcal{C}_T \Big(\|{\bf w}\|_{L^{\infty}(D_T)} \|\alpha_1\|_{L^2_{t, x}(D_T)}+ \|{\bf w}_2\|_{L^{\infty}(D_T)} \|\alpha_1- \alpha_2\|_{L^2_{t, x}(D_T)}\Big).
\end{align*}
This concludes the proof of this lemma. 

\vspace{2mm}

\noindent {\bf C.2.}\label{setting(C2)}  {\bf
  The proof of Lemma \ref{lem:wave:timegain}.} 
This proof is directly inspired by \cite[Lemma 3.4]{Zhang-2000-2}, where the author considered a single variable wave equation with potential terms. 

Define a non-negative cutoff function $g\in C^{\infty}_c(-3\pi, 3\pi)$ such that $g(x)=1$ in $[-2\pi, 2\pi]$.  Using integration by parts we get
\begin{align*}
   & \;\;\;\;  \int_{-3\pi}^{3\pi} g \partial_t^2 \varphi \cdot \left((-\Delta+ 1)^{-1} \varphi\right) \, dx dt \\
    &= -\int_{-3\pi}^{3\pi}  g \partial_t \varphi \cdot \left((-\Delta+ 1)^{-1} \partial_t \varphi\right) \, dx dt- \int_{-3\pi}^{3\pi} \dot g \partial_t \varphi \cdot \left((-\Delta+ 1)^{-1} \varphi\right) \, dx dt\\
    &= -\int_{-3\pi}^{3\pi} g \|\varphi_t\|_{H^{-1}(\mathbb{T}^1)}^2 \, dt- \frac{1}{2} \int_{-3\pi}^{3\pi} \dot g \partial_t \left( \big((-\Delta+ 1)^{- \frac{1}{2}} \varphi\big)^2\right) \, dx dt\\
    &= -\int_{-3\pi}^{3\pi} g \|\varphi_t\|_{H^{-1}(\mathbb{T}^1)}^2 \, dt+ \frac{1}{2} \int_{-3\pi}^{3\pi} \ddot g  \big((-\Delta+ 1)^{- \frac{1}{2}} \varphi\big)^2 \, dx dt.
\end{align*}
By plugging equation \eqref{eq:obdua} into the above identity we obtain 
\begin{align*}
 &\;\;\;\; \int_{-2\pi}^{2\pi}  \|\varphi_t\|_{H^{-1}(\mathbb{T}^1)}^2 \, dt\\
 &\leq   \int_{-3\pi}^{3\pi} g \|\varphi_t\|_{H^{-1}(\mathbb{T}^1)}^2 \, dt \\
  &=  -\int_{-3\pi}^{3\pi} g \partial_t^2 \varphi \cdot \left((-\Delta+ 1)^{-1} \varphi\right) \, dx dt+ \frac{1}{2} \int_{-3\pi}^{3\pi} \ddot g  \big((-\Delta+ 1)^{- \frac{1}{2}} \varphi\big)^2 \, dx dt\\
  &= -\int_{-3\pi}^{3\pi} g (\Delta \varphi+ {\bf b}\partial_x \varphi+ {\bf c}\varphi) \cdot \left((-\Delta+ 1)^{-1} \varphi\right) \, dx dt+ \frac{1}{2} \int_{-3\pi}^{3\pi} \ddot g  \big((-\Delta+ 1)^{- \frac{1}{2}} \varphi\big)^2 \, dx dt\\
   &= -\int_{-3\pi}^{3\pi} g ((\Delta-1) \varphi+ {\bf b}\partial_x \varphi+ ({\bf c}+1)\varphi) \cdot \left((-\Delta+ 1)^{-1} \varphi\right) \, dx dt+ \frac{1}{2} \int_{-3\pi}^{3\pi} \ddot g  \big((-\Delta+ 1)^{- \frac{1}{2}} \varphi\big)^2 \, dx dt\\
  & \lesssim \int_{-3\pi}^{3\pi} \int_{\mathbb{T}^1} \varphi^2(t, x) \, dx dt.
\end{align*}
Therefore, we finish the proof of Lemma \ref{lem:wave:timegain}.

\bibliographystyle{abbrv}
\bibliography{biblio}

\end{document}